\documentclass[aos]{imsart}

\RequirePackage{amsthm,amsmath,amsfonts,amssymb}
\RequirePackage[authoryear]{natbib}
\RequirePackage[colorlinks,citecolor=blue,urlcolor=blue]{hyperref}

\usepackage{3_preamble}
\defcitealias{chernozhukov2020nearly}{CCK23}
\defcitealias{kuchibhotla2020high}{KR20}

\begin{document}
\begin{frontmatter}
\title{Dual Induction CLT for High-dimensional \texorpdfstring{$m$}{m}-dependent Data}
\runtitle{Dual Induction CLT for High-dim \texorpdfstring{$m$}{m}-dep Data}

\begin{aug}
\author[A]{\fnms{Heejong} \snm{Bong} \ead[label=e1]{hbong@umich.edu}},
\author[B]{\fnms{Arun Kumar} \snm{Kuchibhotla} \ead[label=e2]{arunku@stat.cmu.edu}}
\and
\author[C]{\fnms{Alessandro} \snm{Rinaldo} \ead[label=e3]{alessandro.rinaldo@austin.utexas.edu}}
\address[A]{Department of Statistics,
University of Michigan,
\printead{e1}}

\address[B]{Department of Statistics and Data Science,
Carnegie Mellon University,
\printead{e2}}

\address[C]{Department of Statistics and Data Science,
The University of Texas at Austin,
\printead{e3}}
\end{aug}

\begin{abstract}
We derive novel and sharp high-dimensional Berry--Esseen bounds for the sum of \texorpdfstring{$m$}{m}-dependent random vectors over the class of hyper-rectangles exhibiting only a poly-logarithmic dependence in the dimension. Our results hold under minimal assumptions, such as non-degenerate covariances and finite third moments, and exhibit an optimal sample complexity of order \texorpdfstring{$m^{(q-1)/(q-2)}/\sqrt{n}$}{m^((q-1)/(q-2))/n^(1/2)}. 
Aside from logarithmic terms, the resulting rates match the optimal rates established in the univariate case. When specialized to the sums of independent non-degenerate random vectors, our results produce sharp and, in some  cases, optimal rates under the weakest possible conditions. We develop a novel inductive relationship between anti-concentration inequalities and Berry--Esseen bounds inspired by the classical Lindeberg swapping method and the concentration inequality approach for dependent data that may be of independent interest.
\end{abstract}



\end{frontmatter}


\section{Introduction}

The accuracy of the central limit theorem in growing and even in infinite dimensions is a classic topic in probability theory that has been extensively studied since at least the 1950's; see, e.g., \cite{paulauskas2012approximation} and \cite{MR643968} for  historical accounts of the central limit theorem in its various forms.  In a seminal contribution, \cite{bentkus03} derived a high-dimensional Berry--Esseen bound for normalized sums of $n$ independent and isotropic $p$-dimensional random vectors over appropriate classes of subsets of $\mathbb{R}^p$. In details, if $X_1,\ldots,X_n$ are i.i.d. isotropic centered random vectors in $\mathbb{R}^p$  and $Y_1,\ldots,Y_n$ are independent $p$-dimensional  standard Gaussian random vectors, \cite{bentkus03} proved that 
\[
    \sup_{A \in \mathcal{A}} \left| \mathbb{P}\left( \frac{1}{\sqrt{n}} \sum_{i=1}^n X_i \in A \right) - \mathbb{P}\left( \frac{1}{\sqrt{n}} \sum_{i=1}^n Y_i \in A\right) \right| \leq \frac{C_p(\mathcal{A}) \mathbb{E}\| X_1\|^3}{\sqrt{n}},
\]
where $\mathcal{A}$ is a class subsets of $\mathbb{R}^p$ satisfying mild regularity conditions,  $C_p(\mathcal{A})$ is its Gaussian isoperimetric constant and $\| \cdot \|$ denotes the Euclidean norm. The bound is at least of order $p^{3/2}/\sqrt{n}$ (e.g. when $\mathcal{A}$ consists of all Euclidean balls) and is of order at least $p^{7/4}/\sqrt{n}$ for the larger class of convex sets \citep[see][]{nazarov2003maximal}. While, in principle, this result allows for applications of the central limit theorem in growing dimensions, its usefulness is limited to settings in which the dimension $p$ is of smaller order than the sample size $n$. 


The class of hyper-rectangles -- and, more generally, of polyhedra with small enough combinatorial complexity and of sparsely convex sets \citep[see][for details]{chernozhukov2017central} --  are not covered by Bentkus' result. Yet, remarkably, they admit high-dimensional Berry--Esseen bounds  with only a poly-logarithmically dependence on the dimension. 
Results of this type, initially put forward by \cite{chernozhukov2013gaussian}, demonstrate that formal statistical inference via Gaussian approximations is feasible even when the dimension is much larger than the sample size; see, e.g., \cite{belloni2018high} for a selection of notable applications in high-dimensional and non-parametric statistics. Largely motivated by the increasing need to certify the validity and effectiveness of statistical inference in high-dimensional models, over recent years, there has been a renewed impetus to derive sharp Berry--Esseen bounds for Gaussian and bootstrap approximations for hyper-rectangles with explicit dependence on the dimension. This flurry of efforts has led to novel and optimal high-dimensional Berry--Esseen bounds. See Section~\ref{sec:literature} below for a brief summary of the literature and \cite{chernozhukov2023high} for a recent, more extensive review. These bounds can be described as follows, in their most simplified and unified form. Let
\[
\mathcal{R}_p \equiv \left\{  \prod_{k=1}^p [a_k,b_k] \; \colon  -\infty \leq a_k < b_k \leq \infty, k \in [p] \right\}
\]
be the class of all $p$-dimensional hyper-rectangles, with possibly infinite sides. If $X_1,\ldots,X_n$ are independent, centered $p$-dimensional random vectors with non-degenerate (i.e. positive definite) covariances and eigenvalues bounded away from zero and $Y_1,\ldots, Y_n$ are independent, centered Gaussian random vectors such that $\mathrm{Var}[X_i] = \mathrm{Var}[Y_i]$ for all $i$, then 
\begin{equation}\label{eq:uno}
    \sup_{A \in \mathcal{R}_p} \left| \mathbb{P}\left( \frac{1}{\sqrt{n}} \sum_{i=1}^n X_i \in A \right) - \mathbb{P}\left( \frac{1}{\sqrt{n}} \sum_{i=1}^n Y_i \in A\right) \right| \leq C \frac{\mathrm{polylog}(p,n)}{\sqrt{n}} B_{q},
\end{equation}
where $C$ is a positive universal constant and $\mathrm{polylog}(p,n)$ a quantity that is  polylogarithmic in $n$ and $p$.
Above, $B_q$ is a uniform bound on the $q$-th moment of the $X_i$'s for $q \geq 3$, e.g.,
\[
\max_{i \in [n]} \mathbb{E} \left[ \| X_i \|^q_\infty \right] \leq B_q,
\]
where for a vector $x = (x_1,\ldots, x_p) \in \mathbb{R}^p$, $\| x \|_\infty \equiv \max_{k \in [p]}|x_k|$.  Assuming at least 4-moments (i.e. $q \geq 4$),  \cite{chernozhukov2020nearly} (henceforth \citetalias{chernozhukov2020nearly}) proved that the term $\mathrm{polylog}(p)$ can be taken to be  $\log^{3/2}(p)$, which is optimal when the $X_i$'s are almost surely bounded. \cite{kuchibhotla2020high} (henceforth \citetalias{kuchibhotla2020high}) showed that the bound \eqref{eq:uno} holds assuming only third-moment conditions, with a worse poly-logarithmic dependence on $p$; in analogy with the univariate Berry--Esseen bound, this is the weakest possible condition. Furthermore, the restriction that the minimal eigenvalues are bounded away from zero may be relaxed to allow for vanishing eigenvalues. Recently, \cite{koike.degenerate} showed that, under additional restrictions, the assumption of degenerate covariances can be lifted altogether.

The sharp high-dimensional Gaussian approximation bound \eqref{eq:uno} is predicated on the assumption of independent summands. In contrast, for normalized sums of {\it dependent} random vectors, fewer results are available and relatively little is known about the optimal dependence on the sample size and the dimension. See Section \ref{sec:literature} below for a review of the current literature. In this paper, we close this gap by establishing a high-dimensional Berry--Esseen bound of the form \eqref{eq:uno} for $m$-dependent $p$-dimensional random vectors. In detail, following the same settings outlined above, \Cref{thm:m_dep_berry_esseen} shows that 
\begin{equation}\label{eq:due}
    \sup_{A \in \mathcal{R}_p} \left| \mathbb{P}\left( \frac{1}{\sqrt{n}} \sum_{i=1}^n X_i \in A \right) - \mathbb{P}\left( \frac{1}{\sqrt{n}} \sum_{i=1}^n Y_i \in A\right) \right| \leq C 
    \frac{m^{(q-1)/(q-2)} }{\sqrt{n}}
    \mathrm{polylog}(p,n) B_{q}.
\end{equation}
The sequence $X_1,\ldots,X_n$ of random vectors satisfies the same conditions described above except for that of mutual independence; instead, they form an $m$-dependent sequence, i.e. 
\begin{equation}\label{eq:m.dependence}
 X_i \indep X_j \quad \text{if} \quad |i-j|>m,
\end{equation}
for some integer $m$, which is allowed to depend on $n$, where the symbol $\indep$ signifies stochastic independence. Analogously, the centered $p$-dimensional Gaussian random vectors $Y_1,\ldots,Y_n$ have the same second moments as the $X_i$'s, in the sense that, for each $i, j \in \{1,\ldots,n\}$,
\[
\mathbb{E}[X_i X_j^\top] = \mathbb{E}[Y_i Y_j^\top].
\]
In particular, $\mathrm{Var}[X_i] = \mathrm{Var}[Y_i]$ for all $i$, and  the $Y_i$'s are also $m$-dependent. 
A few remarks are in order.
\begin{itemize}
    \item The bound in \Cref{eq:due} displays the optimal sample dependence rate $m^{(q-1)/(q-2)}/\sqrt{n}$,
    originally established in the univariate case \citep{shergin1980convergence,chen2004normal}. 
    In comparison, the previously known best rate in the high-dimensional case was $m\log^{5/4}(ep)/n^{1/4}$, assuming third-moment conditions \citep[Theorem 1.4]{fang2021high}.
    \item In the most general form of our result, we only require an average third moment condition, namely $ B _3 \equiv \frac{1}{n} \sum_i \mathbb{E}[\| X_i\|^3_\infty] < \infty$. As shown in \citetalias{kuchibhotla2020high}, this assumption cannot be dispensed with to ensure a $1/\sqrt{n}$ rate. 
    \item A better dependence on the dimension can be demonstrated by allowing for higher order moments; see \Cref{thm:1_dep_berry_esseen_4P}.
    \item We further require the $Y_i$'s to satisfy minimal conditional variance restrictions, as discussed in Section~\ref{sec:theoretical_result}. To the best of our knowledge, these conditions are novel.
     \item Our results also apply to sums of independent random vectors as a special case, delivering sharp or nearly sharp rates depending on the settings. In particular, we recover and, in fact, improve the bounds that a subset of the authors obtained in the preprint \citetalias{kuchibhotla2020high}. We further elaborate on this point in Section~\ref{sec:comparison_indep}.
    \item We derive our bounds by formulating a novel, general class of symmetric stochastic dependence, which we refer to as $m$-ring dependence, of which $m$-dependence is a special case; see Section~\ref{sec:setting_assumption}.
\end{itemize}

In terms of the technical contributions, we deploy novel proof techniques generalizing the inductive arguments used in \citetalias{kuchibhotla2020high} that appear to be well suited to handle $m$-dependence settings and may be of independent interest. 
\citetalias{kuchibhotla2020high} used a Lindeberg interpolation scheme coupled with induction (over $n$) to establish a high-dimensional Berry--Esseen bound for independent random variables with rates of order (ignoring terms depending on the minimal eigenvalues of the covariances) $\log^{3/2}n\log^4(ep)/\sqrt{n}$. Though this approach is not new \citep[see, e.g.][]{SENATOV2011NORMAL,lopes2020central}, the authors managed to refine it to obtain an optimal dependence on the sample size under minimal assumptions: non-degenerate covariances and finite third moments. 
A key step in the proof of \citetalias{kuchibhotla2020high} is the use of the anti-concentration inequality for the Gaussian distribution \citep{nazarov2003maximal}. The importance of anti-concentration inequalities in obtaining Berry--Esseen bounds is well recognized in the literature; see, e.g., \citet{bentkus2000accuracy} for an elucidation of how the Berry--Esseen bound follows from combining an induction over the number of samples along with anti-concentration inequalities. However, implicit in  the arguments of \citetalias{kuchibhotla2020high} was the formulation of a backward relationship: anti-concentration bounds can, in turn, be derived from  Berry--Esseen bounds. We further leverage this duality and deploy a novel induction argument on both the Berry--Esseen bound and the anti-concentration inequalities to obtain sharper bounds than those originally established by \citetalias{kuchibhotla2020high} for normalized sums of independent random vectors while, at the same time, allowing for $m$-dependence. See~\cref{sec:pf_sketch} for a more detailed description of the dual induction.

The paper is organized as follows. In the next section, we summarize related works. We describe the settings and notation and present our main results in \cref{sec:theoretical_result}, which showcase the optimal scaling of $1/\sqrt{n}$ under non-degenerate covariance and finite third moments conditions. To establish these results, we used a combination of the Lindeberg swapping technique and a novel dual induction argument, overcoming the challenges posed by the $m$-dependence structure; we outline the proof techniques in \cref{sec:pf_sketch}. We then compare our results with those existing in the literature under independence and $m$-dependence; see \cref{sec:comparison}. In \cref{sec:m_dep_bootstrap} we introduce a new bootstrap procedure for $m$-dependent random vectors that builds on our refined Berry--Esseen bound. Finally, in \cref{sec:discussion}, we conclude the paper by summarizing the main findings and outlining avenues for future research.

\subsection{Related work}\label{sec:literature}
Below, we summarize some notable recent results about high-dimensional Berry--Esseen bounds for hyper-rectangles, for independent and dependent summands. We refer the reader to \cite{chernozhukov2023high} for a thorough review article.

As already mentioned, \cite{bentkus03} derived a general high-dimensional Berry--Esseen bound for a wide variety of classes of Euclidean sets under only third moment conditions, assuming independent and isotropic summands. When applied to the important class convex sets, the resulting bound is non trivial only when $p = o( n^{2/7})$. See \cite{bentkus2005lyapunov, raic2019, fang2021high, zhilova_22} for extensions and sharpening of this original, fundamental results. In their pathbreaking contribution, \cite{chernozhukov2013gaussian}  obtained a Berry--Esseen bound of order $(\log(np))^{7/8} / n^{1/8}$ for the class of $p$-dimensional hyper-rectangles under appropriate conditions. (It is worth noting that, though the authors develop novel non-trivial and valuable techniques, in hindsight the same bound could have also been established using more direct and standard methods; see~\citet{MR1115160}.) 
Subsequently, several authors generalized the settings of \citet{chernozhukov2013gaussian} and improved on the original rate  both in terms of dependence on the dimension $p$ and of the exponent for the sample size $n$: see, in particular, \citet{chernozhukov2017central,koike2019high,koike2019notes,deng2020beyond,kuchibhotla2021high}. For some time, it was a conjecture that the best dependence on the sample size, demonstrated by~\citet{chernozhukov2017central}, should be of order $1/n^{1/6}$, matching  the optimal sample size dependence in general Banach spaces found  by \citet{bentkus1987}. This conjecture was later on disproved by multiple authors, who managed to improve the sample complexity beyond $1/n^{1/6}$. In particular, \cite{chernozhukov2019improved} managed to achieve a better dependence on the sample size of order $1/n^{1/4}$ for sub-Gaussian vectors while only requiring $\log(ep) = o(n^{1/5})$. Next, \cite{fang2021high},~\cite{das2020central}, and~\cite{lopes2020central} succeeded in going beyond a $1/n^{1/4}$ dependence.  The most noticeable difference between these latest contributions compared to the earlier papers is that the random vectors are assumed to be non-degenerate, i.e., the minimum eigenvalue of the covariance matrix of the normalized sum is bounded away from zero.  In detail, \citet[Corollary 1.3]{fang2021high} used Stein's method to sharpen the dependence on the sample size to $1/n^{1/3}$, while at the same time  weakening the requirement on the dimension to $\log(ep) = o(n^{1/4})$. \citet[Corollary 1.1]{fang2021high} obtained an optimal $1/\sqrt{n}$ dependence on the sample size along with a dimension requirement of $\log(ep) = o(n^{1/3})$, for sums of random vectors with log-concave distributions. 
\cite{das2020central} also established a $1/\sqrt{n}$ dependence when the random vectors have independent and sub-Gaussian coordinates (among other assumptions); they also investigated the optimal dependence on the dimension. \cite{lopes2020central} finally succeeded in showing the optimal $1/\sqrt{n}$ rate in general settings, assuming only sub-Gaussian vectors. A subset of the authors of this paper weakened considerably  Lopes' assumptions and demonstrated  a bound of order $\log^{3/2}n \log^4(ep)/ \sqrt{n}$ under only third-moment conditions; see \citetalias{kuchibhotla2020high}. Though the moment requirements are shown to be the weakest possible, the rate suffers from a suboptimal dependence on the dimension. Finally, \citetalias{chernozhukov2020nearly} obtained the sharper rate  $\log^{3/2}(ep)/\sqrt{n}$, which is, in fact, optimal for distributions with bounded support under fourth-moment conditions. \cite{koike.degenerate} further shows that the same rate was shown to hold also in the degenerate case under appropriate restrictions. 


Overall, the aforementioned contributions offer a near-complete picture of optimal Gaussian and bootstrap approximations for independent normalized sums of random vectors with respect to the class of hyper-rectangles, with rates exhibiting optimal sample and dimension complexity. 
On the other hand, the literature on high-dimensional central limit theorems -- over hyper-rectangles or other classes of sets in $\mathbb{R}^p$ -- for sums of {\it dependent random vectors} is still scarce, with relatively few results available and limited understanding of rate optimality. 
Furthermore, some of the best-known results at this time have been established by adapting arguments and techniques from \cite{chernozhukov2013gaussian, chernozhukov2017central}, which yield sub-optimal sample complexity. In turn, this is likely affecting the sharpness of current rates derived under dependence settings. Below we provide a (necessarily) short summary of the existing contributions, which are all very recent. 
\cite{zhang2017gaussian} derived asymptotic Gaussian approximations over hyper-rectangles for high-dimensional stationary time series under a polynomial decay of the physical dependence measure, including $m$-dependent sequences, obtaining rates with sample size dependence of order $1/n^{1/8}$. 
\cite{zhang2018gaussian} considered instead the broader class of weakly dependent, possibly non-stationary sequences as well as  $m$-dependent data  assuming a geometric moment condition under physical/functional dependence.  
\cite{chang2021central} provided high-dimensional Berry--Esseen bounds for various classes of sets, including hyper-rectangles and the larger class of convex sets, under $\alpha$ mixing, physical dependence and $m$-dependence conditions.   
\cite{fang2021high} improved the sample size dependence to $1/n^{1/4}$ under local dependence conditions, including $m$-dependent sequences, using finite third-moment conditions.
\cite{chernozhukov2019inference} investigated the validity of the  block multiplier bootstrap under $\beta$-mixing. 
\cite{chiang2021inference} derive high-dimensional Berry--Esseen bounds for sums of high-dimensional exchangeable array assuming sub-exponential conditions or $4 + \delta$ moments, for $\delta > 0$.
\cite{kojevnikov2022berry} demonstrated, under some conditions, a finite sample central limit approximation over hyper-rectangles of order $\log^{5/4}(ep)/n^{1/4}$ for martingale difference sequences. 
Finally, \citet{rinott1999some,kurisu2024gaussian} studied Gaussian approximations over high-dimensional random fields under $\beta$-mixing spatial dependence structure, obtaining a rate of $1/n^{1/12}$.

\section{High-dimensional Berry--Esseen bound for \texorpdfstring{$m$}{m}-ring dependent random samples with nondegenerate covariance matrices} \label{sec:theoretical_result}

In this section, we introduce novel high-dimensional Berry--Esseen bounds for $m$-dependent random vectors with non-degenerate covariance matrices. It is sufficient to consider the simplest case of $m=1$; the general case will follow from a standard argument: see Corollary~\ref{thm:m_dep_berry_esseen}.
Before stating our main results, we introduce some necessary notation and concepts. 

 \subsection{Setting and Assumptions}\label{sec:setting_assumption}

Throughout, $X_1,\ldots,X_n$ denotes a time series of centered random vectors in $\mathbb{R}^p$ with at least $q \geq 3$ moments, i.e. $\max_i \mathbb{E}[\| X_i \|^q_\infty] < \infty$. We assume the $X_i$'s satisfy a novel notion of stochastic dependence that we refer to as {\it $m$-ring dependence.} In detail, $m$-ring dependence is in effect if $X_i \indep X_j$ for any $i$ and $j$ satisfying $\min\{\abs{i-j}, n-\abs{i-j}\} > m$. It is easy to see that $m$-ring dependence is a generalization of $m$-dependence, as defined in~\eqref{eq:m.dependence}. Indeed, if $X_1, \ldots, X_n$ are $m$-dependent, then
\begin{equation*}
    \min\{\abs{i-j}, n-\abs{i-j}\} > m\quad
    \Rightarrow\quad \abs{i-j} > m\quad
    \Rightarrow\quad X_i \indep X_j.
\end{equation*}
Thus, if $X_1, \dots, X_n$ are $m$-dependent, they are also $m$-ring dependent. The notion of $m$-ring dependence is technically advantageous because of its higher degree of symmetry compared to $m$-dependence.

We will couple the $m$-ring dependent sequence $X_1,\ldots,X_n$ with an independent sequence of centered $p$-dimensional Gaussian random vectors  $Y_1, \dots, Y_n$  with matching covariances, in the sense that
\begin{equation*}
    \Var[(Y_1^\top, \dots, Y_n^\top)^\top] = \Var[(X_1^\top, \dots, X_n^\top)^\top].
\end{equation*}
Thus, if the $X_i$'s are $m$-ring dependent, so are the $Y_i$'s.

For any non-empty subset $I \subseteq \{1, 2, \ldots, n\}$, define
\begin{equation*}
    \mathscr{X}_I \equiv \{ X_i: i \in I \}, ~~
    \mathscr{Y}_I \equiv \{ Y_i: i \in I \}.
\end{equation*}
and
\begin{equation} \label{eq:X_I_Y_I}
    X_I \equiv \sum_{i \in I} X_i, ~~
    Y_I \equiv \sum_{i \in I} Y_i.
\end{equation}
To streamline our discussion, we introduce a special notation for integral intervals. For any pair of integers $i$ and $j$ that satisfy $1 \leq i < j \leq n$, we will denote the ordered index set $\{i, \dots, j\}$ as $[i,j]$. Accordingly, $[1,n]$ will refers to the complete time course. To align with the conventions of real intervals, we will employ parentheses to represent open-ended intervals; e.g., $(1,n]$ denotes the set $\{2, \dots, n\}$. 
With this notation in place, our goal is to derive finite sample bounds for the  quantity
\begin{equation}\label{eq:mu}
\mu\left( X_{[1,n]}, Y_{[1,n]} \right) \equiv \sup_{A \in \mathcal{R}_p} \left| \mathbb{P} \left( \frac{1}{\sqrt{n}} X_{[1,n]} \in A \right) - \mathbb{P} \left( \frac{1}{\sqrt{n}} Y_{[1,n]} \in A \right) \right|.
\end{equation}
To that effect, in order to ensure non-trivial results, we impose appropriate conditions on the conditional variances of the $Y_i$'s which, to the best of our knowledge, are novel. 
Specifically, for any $I, I' \subset [1,n]$, let 
\begin{equation*} \begin{aligned}
    {\sigma}^2_{\min,I} \equiv
    & \min_{k \in [p]} \Var[Y_{I,k}], 
    \quad\mbox{and}\quad
    \underline{\sigma}^2_{I|I'} \equiv
     \lambda_{\min}(\Var[Y_{I}|\mathscr{Y}_{I'}]),
\end{aligned} \end{equation*}
where the superscript $Y_{I,k}$ denotes each $k$-th element of $p$-dimensional random vector $Y_{I}$.
We recall that the conditional variance coincides with the Schur complement of the marginal covariance matrix, i.e.
\begin{equation*}
\begin{aligned}
    \Var[Y_{I}|\mathscr{Y}_{I'}] 
    = \Var[Y_{I}]
    - \Cov[Y_{I}, \mathrm{vec}(\mathscr{Y}_{I'})] \Var[\mathrm{vec}(\mathscr{Y}_{I'})]^{-1} \Cov[\mathrm{vec}(\mathscr{Y}_{I'}), Y_{I}],
\end{aligned}
\end{equation*}
where $\mathrm{vec}(\mathscr{Y}_{I'})$ indicates the vectorized representation by concatenation. i.e., $\mathrm{vec}(\mathscr{Y}_{I'}) \equiv (Y_i^\top: i \in I')^\top$. 
In the case $m=0$, i.e. when $X_1,\dots,X_n$ are independent, then 
$\Var[Y_{I}|\mathscr{Y}_{I'}] = \Var[Y_{I}].$
We assume that there exist constants $\sigma_{\min}, \underline{\sigma} > 0$ such that 
for all $I, I' \subset [1,n]$,
\begin{align}
    \sigma^2_{\min,I} & \geq \sigma^2_{\min} \cdot \abs{I}, 
    \label{assmp:min_var} \tag{MIN-VAR} \\
    \underline{\sigma}^2_{I|I'} & \geq \underline{\sigma}^2 \cdot \abs{I \cap I'^{\indep}}, 
    \label{assmp:min_ev} \tag{MIN-EV} \\
    \sigma_{\min} & \le \underline{\sigma} \sqrt{\log(4ep)/2},
    \label{assmp:var_ev}\tag{VAR-EV}
\end{align}
where $\abs{I}$ is the number of elements in $I$, and $I'^{\indep} \;:=\; \bigl\{\,j \in [n] : X_j \;\indep\; \mathscr{X}_{I'}\bigr\}$ is the set of indices whose corresponding variables are independent of $\mathscr{X}_{I'}$. Thus $I \cup I'^{\indep}$ is the subset of $I$ which are separated from $I'$ by at least $m$ positions. 
The assumption of strongly non-degenerate covariance (i.e., $\lambda_{\min}(\Var[X_i])$ is bounded away from zero) has been commonly used in high-dimensional CLTs under independence (\citealp{kuchibhotla2020high,fang2021high,lopes2022central}; \citetalias{chernozhukov2020nearly}). Assumptions \eqref{assmp:min_var} and \eqref{assmp:min_ev} impose analogous restrictions to $m$-(ring) dependent sequences; they are inspired by Assumption (3) in \citet{shergin1980convergence} in the univariate case. 
It is worth highlighting that Assumption~\eqref{assmp:min_ev} accommodates the possibility of complete dependence between $X_{I}$ and $\mathscr{X}_{I'}$. 
For example, consider a scenario where $X_{2j-1} = X_{2j}$ for $j \in [\ceil{\frac{n}{2}}]$, while $X_1, X_3, \dots$ are independent. In this case, $\sigma^2_{[2,3]|\{1\}\cup[4,n]} = \Var[Y_{[2,3]} | Y_1, Y_4, Y_5, \dots, Y_n] = 0$ because $Y_2$ and $Y_3$ exhibit total dependence with $Y_1$ and $Y_4$, respectively. However,  $\sum_{i=1}^n X_i$ still  convergences weakly to $\sum_{i=1}^n Y_i$. Our Berry--Esseen bounds apply to this scenario.
Finally, Assumption~\eqref{assmp:var_ev} is made out of technical convenience and may be removed, leading to a more complicated upper bound expression; see~\cref{rmk:var_ev} for details.

We will also consider two different types of moment bounds: a marginal one on the individual coordinates of the $X_i$'s and of the $Y_i$'s, and a stronger one involving their $L_\infty$ norm. Specifically, for $q \geq 1$ and $i \in [n]$, let 
\begin{equation}\label{eq:definition-L-nu}
    L_{q,i} \equiv \max_{k \in [p]} \Exp[\abs{X_{i,k}}^q] + \max_{k \in [p]} \Exp[\abs{Y_{i,k}}^q],\quad\mbox{and}\quad
    \nu_{q,i} \equiv \Exp[\norm{X_{i}}_\infty^q] + \Exp[\norm{Y_{i}}_\infty^q].
\end{equation}
Clearly, $L_{q,i} \leq \nu_{q,i}$ for all $q$ and $i$.
For any subset $I \subset [1,n]$, we denote the average marginal and joint moments over the indices in $I$ with
\begin{equation}\label{eq:definition-L-bar-I-and-nu-bar-I}
    \bar{L}_{q,I} \equiv \frac{1}{\abs{I}} \sum_{i \in I} L_{q,i}
    \textand
    \bar\nu_{q,I} \equiv \frac{1}{\abs{I}} \sum_{i \in I} \nu_{q,i}.
\end{equation}
and write $\bar{L}_q \equiv \bar{L}_{q,[1,n]}$ and $\bar{\nu}_q \equiv \bar{\nu}_{q,[1,n]}$ for the global average moments.
We note that, due to Jensen's inequality, $\max_{i \in [n]} \nu_{2,i} \leq \nu_{q,i}^{2/q}$  and $\bar\nu_2 \leq \bar\nu_q^{2/q}$ (since  $q > 2$). 

\subsection{Main Results}\label{sec:main}

With the notations and assumptions in place, we are now ready to present our first bound on the quantity $\mu\left( X_{[1,n]}, Y_{[1,n]} \right)$ (see~equation \ref{eq:mu}),  assuming only finite third moments. 


\begin{theorem} \label{thm:1_dep_berry_esseen_3P}
Suppose that Assumptions \eqref{assmp:min_var}, \eqref{assmp:min_ev} and \eqref{assmp:var_ev} hold. Then, for $m=1$ and $q \geq 3$, 
\begin{equation*}
\begin{aligned}
    & \mu\left( X_{[1,n]}, Y_{[1,n]} \right) \\
    & \leq \frac{C \log(en)}{\sigma_{\min}} \sqrt{\frac{\log(pn)}{n}}  \left[
        \frac{\bar{L}_3}{\underline{\sigma}^2} \log^{2}(ep) 
        + \left( \frac{\bar\nu_q}{\underline\sigma^2} 
        \right)^{1/(q-2)} (\log(ep))^{\max\{2/(q-2),1\}}
    \right], \\
\end{aligned}
\end{equation*}
for some universal constant $C > 0$.
\end{theorem}


 
The main steps of the proof of a simplified and weaker version of the above statement that depends on the {\it maximal}, as opposed to the average, third moments are described in detail in \cref{sec:pf_sketch}. The complete proof calls for a more complicated and cumbersome Lindeberg swapping argument that will be explained in \cref{sec:permutation_argument}. For the comprehensive proof, please refer to \cref{sec:pf_1_dep_3P}.

It is worth commenting on the dimension complexity (i.e. the overall dependence on $p$) of the bound in \Cref{thm:1_dep_berry_esseen_3P}. The distance $\mu(X_{[1,n]}, Y_{[1,n]})$, will vanish as $n$, $p$ and all the quantities involved vary, provided that  
\begin{equation}\label{eq:rate}
\max\left\{ \frac{\bar{L}_3^2}{\sigma_{\min}^2 \underline{\sigma}^4} \log^5 p, \frac{\bar{\nu}_q^{2/(q-2)}}{\sigma_{\min}^2 \underline{\sigma}^{4/(q-2)}} (\log p)^{\max\{(q+2)/(q-2), 3\}}\right\} = o(n).
\end{equation}
To illustrate, if $\sigma_{\min}$ and $\underline{\sigma}$ are of constant order, the dimension complexity will depend on how  the terms $\bar{L}_3^2 \log^5(p)$ and $\bar{\nu}_q^{2/(q-2)} (\log p)^{\max\{(q+2)/(q-2), 3\}}$ scale with $n$.
For instance, if the $X_i$'s are sub-Gaussian, then  \cref{eq:rate} with $q=4$ implies that the dimension complexity is dominated by the first term on the left hand side of \eqref{eq:rate}, and the requirement becomes $\log^5 p = o(n)$.
%
%
%
%
%
%
Under the same settings and assuming independent summands (a special case of the above theorem; see  \cref{sec:comparison_indep} below), the best-known dimension complexity in the literature is $\log^4 p = o(n)$; see Corollary 2.1 of \citetalias{chernozhukov2020nearly}.
Interestingly, Theorem 2.1 can be improved if $q \ge 4$ (instead of $q \ge 3$) and this improvement implies the best-known dimension complexity for the sub-Gaussian case.
This is not entirely surprising, as the existence of higher moments has been shown to improve the dimension complexity: for example, \citet{fang2020large} sharpened the the high-dimensional Berry--Esseen bound for convex sets established by \citet{bentkus2005lyapunov}  assuming finite fourth moment conditions.
In the present setting, which is concerned with hyper-rectangles, higher-order moments similarly enables a more precise control of the higher-order remainder terms from the various Taylor series expansions, resulting in the improved convergence rates, as stated in the next result.

\begin{theorem} \label{thm:1_dep_berry_esseen_4P}
Suppose that Assumptions \eqref{assmp:min_var}, \eqref{assmp:min_ev} and \eqref{assmp:var_ev} hold. Then, 
if $m=1$ and $q \geq 4$,
\begin{equation*}
\begin{aligned}
    & \mu\left( X_{[1,n]}, Y_{[1,n]} \right) \\
    & \leq \frac{C \log\left(en\right)}{\sigma_{\min}} \sqrt{\frac{\log(pn)}{n}}  \left[
        \frac{\bar{L}_3}{\underline{\sigma}^2} \log^{3/2}(ep)
        + \frac{\bar{L}_4^{1/2}}{\underline{\sigma}} \log(ep)
        + \left( \frac{\bar\nu_q}{\underline\sigma^2} 
        \right)^{1/(q-2)} \log(ep) 
    \right], \\
\end{aligned}
\end{equation*}
for some universal constant $C > 0$.
\end{theorem}

The proof is similar to that of \Cref{thm:1_dep_berry_esseen_3P}, as outlined in \cref{sec:pf_sketch}, with an additional step that takes advantage of the extra moment condition; for details, see \cref{sec:pf_sketch_4N}. The complete proof can be found in \cref{sec:pf_1_dep_4P}.

In terms of rates, \cref{thm:1_dep_berry_esseen_4P} will require the  asymptotic scaling
$$\max\left\{ \frac{\bar{L}_3^2}{\sigma_{\min}^2 \underline{\sigma}^4} \log^4 p, \frac{\bar{L}_4}{\sigma_{\min}^2 \underline{\sigma}^2} \log^3 p, \frac{\bar{\nu}_q^{2/(q-2)}}{\sigma_{\min}^2 \underline{\sigma}^{4/(q-2)}} \log^{3} p \right\} = o(n).$$ 
When the $X_i$'s are sub-Gaussian and $\sigma_{\min}$ and $\underline{\sigma}$ are of constant order, the above scaling reduces to $\log^4 p = o(n)$, which matches that of \citetalias{chernozhukov2020nearly}. In \cref{sec:comparison_indep}, we will provide more general and in-depth comparisons of our rates with those of \citetalias{chernozhukov2020nearly}.

\begin{remark} \label{rmk:var_ev}
   Assumption \eqref{assmp:var_ev} can be dispensed of in both \cref{thm:1_dep_berry_esseen_3P,thm:1_dep_berry_esseen_4P}, though the proofs become more involved and the final bounds more cumbersome. Specifically, the factor ${C}/{\sigma_{\min}}$ can be replaced by $C/\min\{ \sigma_{\min}, \underline{\sigma} \sqrt{\log(ep)}\}$; see~\cref{sec:pf_without_var_ev} for an explanation of the changes required in the proof.
\end{remark}


Finally, Berry--Esseen bounds for general $m>1$ can be obtained as corollaries of \cref{thm:1_dep_berry_esseen_3P,thm:1_dep_berry_esseen_4P} using a standard argument that we outline next; see also \citet[Theorem 2]{shergin1980convergence} and \citet[Theorem 2.6]{chen2004normal}.
Let 
\begin{equation} \label{eq:block_setting}
    X'_i \equiv \begin{cases}
        X_{((i-1)m, im]}, & i \in [1,n') \\
        X_{((n'-1)m, n]}, & i = n',
    \end{cases}
\end{equation}
where $n' \equiv \floor{{n}/{m}}$.
We define $Y'_i$ similarly for $i \in [1,n']$. 
The newly defined random vectors satisfy the following lemma, whose proof is in \cref{sec:pf_1_to_m}.

\begin{lemma} \label{thm:1_to_m}
    Suppose that Assumptions \eqref{assmp:min_var}, \eqref{assmp:min_ev}, and \eqref{assmp:var_ev} hold and that $m > 1$. Let 
        \begin{equation*}
            L'_{q,i} \equiv \max_{k \in [p]} \Exp[\abs{X'_{i,k}}^q] + \max_{k \in [p]} \Exp[\abs{Y'_{i,k}}^q],\quad\mbox{and}\quad
            \nu'_{q,i} \equiv \Exp[\norm{X'_{i}}_\infty^q] + \Exp[\norm{Y'_{i}}_\infty^q],
        \end{equation*}
        and define $\bar{L}'_{q,I}$, $\bar{\nu}'_{q,I}$, $\sigma'_{\min,I}$ and $\underline{\sigma}'_I$ similarly. Then,
    \begin{enumerate}
        \item $(X'_1, \dots, X'_{n'})$ is a $1$-ring dependent sequence of random vectors in $\reals^p$ with
        \begin{equation*}
            \bar{L}'_q \leq (2m)^q \bar{L}_q,
            \textand
            \bar{\nu}'_q \leq (2m)^q \bar{\nu}_q.
        \end{equation*}
        
        \item Assumptions \eqref{assmp:min_var}, \eqref{assmp:min_ev}, and \eqref{assmp:var_ev} hold for $(X'_1, \dots, X'_{n'})$ with $(\sigma'_{\min,I}, \underline{\sigma}'_I, \sigma'_{\min}, \underline{\sigma}')$ instead of $(\sigma_{\min,I}, \underline{\sigma}_I, \sigma_{\min}, \underline{\sigma})$, where
        \begin{equation*}
            \sigma'^2_{\min} \equiv m \sigma^2_{\min}
            \textand
            \underline{\sigma}'^2 \equiv m \underline{\sigma}^2.
        \end{equation*}
    \end{enumerate} 
    
\end{lemma}


Applying \cref{thm:1_dep_berry_esseen_3P,thm:1_dep_berry_esseen_4P} to $\mu(X'_{[1,n']}, Y'_{[1,n']}) = \mu(X_{[1,n]}, Y_{[1,n]})$, we immediately obtain the following results. 

\begin{corollary} \label{thm:m_dep_berry_esseen}
Suppose that Assumptions \eqref{assmp:min_var}, \eqref{assmp:min_ev}, and \eqref{assmp:var_ev} hold and that $m > 1$. 
If $q \geq 3$, 
\begin{equation*}
\begin{aligned}
    & \mu\left( X_{[1,n]}, Y_{[1,n]} \right) \\
    & \leq \frac{C\log\left(en/m\right)}{\sigma_{\min}} \sqrt{\frac{\log(pn/m)}{n}} 
    \left[
        m^{2} \frac{\bar{L}_3}{\underline{\sigma}^2} \log^{2}(ep) 
        + \left( m^{q-1} \frac{\bar\nu_q}{\underline\sigma^2} 
        \log^2(ep) \right)^{1/(q-2)} 
    \right], \\
\end{aligned}
\end{equation*}
for some universal constant $C > 0$. If $q \geq 4$,
\begin{equation*}
\begin{aligned}
    & \mu\left( X_{[1,n]}, Y_{[1,n]} \right) \\
    & \leq \frac{C\log\left(en/m\right)}{\sigma_{\min}} \sqrt{\frac{\log(pn/m)}{n}}  \\ 
    & \quad \times \left[
        m^{2} \frac{\bar{L}_3}{\underline{\sigma}^2} \log^{3/2}(ep)
        + m^{3/2} \frac{\bar{L}_4^{1/2}}{\underline{\sigma}} \log(ep)
        + \left( m^{q-1} \frac{\bar\nu_q}{\underline\sigma^2} 
        \right)^{1/(q-2)} \log(ep) 
    \right], \\
\end{aligned}
\end{equation*}
for some universal constant $C > 0$.
\end{corollary}
{We emphasize that, ignoring logarithmic factors, the sample complexity in the previous bounds is of order
\begin{equation*}
    \frac{m^{(q-1)/(q-2)}}{\sqrt{n}},
\end{equation*}
which matches the optimal Berry--Esseen rates for univariate $m$-dependent time series established in \cite{shergin1980convergence}; we  elaborate further below in \Cref{sec:comparison}.}


\subsection{Comparisons with the existing literature} \label{sec:comparison}


In this section, we compare our main results with existing results in the high-dimensional CLT literature. In \cref{sec:comparison_indep}, we explore the implication of our bounds under independence using as baseline the nearly optimal rates derived by \citetalias{chernozhukov2020nearly}. 
Though our results are derived for the general $m$-dependent setting, they match the dimension complexity of \citetalias{chernozhukov2020nearly} for the important case of sub-Weibull $X_i$'s (including sub-Gaussian and sub-exponential cases).
Importantly, beyond independent summands, our work delivers a significant improvement over existing results about Berry--Esseen bound under $m$-dependence; see \cref{sec:comparison_m_dep}.

\subsubsection{Independent summands} \label{sec:comparison_indep}
Let's consider the case where the random variables $X_i$ are independent. Since independence holds for all pairs of $X_i$'s, regardless of the difference between their indices, we can say that $X_i \indep X_j$ when $\abs{i-j}$ is greater than or equal to 1 or, equivalently, when $i \neq j$. Thus, according to the definition of $1$-dependence, independence is a special case of $1$-dependence. As a result, \cref{thm:1_dep_berry_esseen_3P} and \cref{thm:1_dep_berry_esseen_4P} readily hold for sums of independent random vectors. Below, we compare our bounds with the those established by \citetalias{chernozhukov2020nearly} and \citet{koike.degenerate}, which deliver the sharpest rates to date. 
We first note that, under independence, assumptions \eqref{assmp:min_ev} and \eqref{assmp:min_var} reduce to
\begin{equation*}
\begin{aligned}
\min_{k \in [p]} \Var[Y_{[i,j],k}] & \geq \sigma^2_{\min} \cdot \abs{[i,j]}, \\
\lambda_{\min}(\Var[Y_{[i,j]}]) & \geq \underline{\sigma}^2 \cdot \abs{[i,j]}, ~\forall i, j.\\
\end{aligned}
\end{equation*}

These assumptions require the covariance matrix of $X_{[i,j]}$ to be strongly non-degenerate for all pairs $(i, j)$, a condition that is commonly imposed in order to derive high-dimensional CLTs. \citetalias{chernozhukov2020nearly} do not require non-degenerate covariance matrices but assume that the covariance matrix of the scaled average is well-approximated by a positive-definite matrix. Thus, the resulting Berry--Esseen bound depends on the quality of this approximation.  
Importantly,  a minimal eigenvalue condition is required on {\it the average} of the covariance matrices, while we take the minimal eigenvalue of all covariances; thus their result is more general in this important aspect.
Recently, \citet{koike.degenerate} pursued a more direct approach and succeeded in obtaining a Berry--Esseen bound with the optimal $1/\sqrt{n}$-rate allowing for degenerate covariance matrices, provided that they satisfy additional assumptions. 
It would be interesting to explore an extension of their work in our setting.
In terms of moment assumption, our result is more general than those of   \citetalias{chernozhukov2020nearly} and \citet{koike.degenerate}, which are predicated on finite fourth moments. In contrast, \cref{thm:1_dep_berry_esseen_3P} only assumes finite average third moments, thus accommodating a broader class of models. 
For example, this scenario can occur in high-dimensional linear regression problems where $X_i = \xi_iW_i$ for heavy-tailed univariate errors $\xi_i$ and light-tailed covariates $W_i$. If $\xi_i$'s have a finite third (conditional on $W_i$) moment but an infinite fourth moment, then $\mathbb{E}\|X_i\|_{\infty}^4$ can be infinite; see~\citet[Section 4.1]{chernozhukov2023high} for an application to penalty parameter selection in lasso using bootstrap.

We now shift our focus to the overall convergence rate implied by our Berry--Esseen bounds, taking into account the dimension. Our theorems yields  rates of order $1/\sqrt{n}$ up to logarithmic factors when $p$ remains fixed. Therefore, a more meaningful comparison lies in the dimension complexities imposed by the theorems. As previously discussed, these dimension complexities are dependent on the tail behavior of the $X_i$'s. To exemplify, assume that the $X_i$'s are i.i.d. and that $\Var[X_{i,k}] = 1$ for all $k \in [p]$ and consider the  scenarios
\begin{enumerate}[label = {\bf (E\arabic*)}]
    \item $\abs{X_{i,k}} \leq B$ for all {$i \in [n]$ and $k \in [p]$,} almost surely;
    \item $\norm{X_{i,k}}_{\psi_\alpha} \leq B$ and $\frac{1}{n} \sum_{i=1}^n \Exp[\abs{X_{i,k}}^4] \leq B^2$ 
    for all $i \in [n]$ and $k \in [p]$, where $\norm{\cdot}_{\psi_\alpha}$ is the Orlicz norm with respect to $\psi_\alpha(x) \equiv \exp(x^\alpha)$ for $\alpha \leq 2$.
\end{enumerate}
These cases correspond to conditions (E1) and (E2) respectively of \citetalias{chernozhukov2020nearly}.
Comparing the Berry--Esseen bounds implied by \cref{thm:1_dep_berry_esseen_4P} with those from Corollary 2.1 of \citetalias{chernozhukov2020nearly}sc, we have, in Scenario {\bf (E1)}, that 
\begin{equation*}
    \frac{\sqrt{n}\underline{\sigma}^2}{\log^{3/2}(ep)}\mu(X_{[1,n]}, Y_{[1,n]}) \leq 
    \begin{cases}
    C B 
    \log(en), &\mbox{ from Corollary 2.1 of \citetalias{chernozhukov2020nearly}},\\
    C B
    \sqrt{\log(pn)} \log\left(en\right), &\mbox{ from \cref{thm:1_dep_berry_esseen_4P} above.}
    \end{cases}
\end{equation*}
The  dimension complexity obtained by \citetalias{chernozhukov2020nearly} is $\log ^3 p = o(n)$, which is shown to be optimal; see Remark 2.1 therein. For our result, the bound is derived using the fact that $L_{3,j} \leq B L_{2,j} = B$, $L_{4,j} \leq B^2 L_{2,j} = B^2$, and $\nu_{q,j} \leq B^q$ for all $q \geq 4$. 
The resulting dimension complexity is $\log^4 p = o(n)$, which is suboptimal when the $X_i$'s are bounded random vectors. 
Next, in Scenario {\bf (E2)} suppose that the individual elements of $X_i$ are sub-Weibull$(\alpha)$ for $\alpha \leq 2$ (we recall that the cases $\alpha = 2$ and $\alpha = 1$ correspond to sub-Gaussian and sub-Exponential coordinates, respectively). Then
\begin{equation*}
\begin{aligned}
    & \frac{\sqrt{n}\underline{\sigma}^2}{\log^{3/2}(ep)} \mu(X_{[1,n]}, Y_{[1,n]}) \\
    & \leq 
    \begin{cases}
    C B
    \left(
        \log(en) + \log^{(1/\alpha)}(ep)
    \right), 
    & \mbox{ from Corollary 2.1 of \citetalias{chernozhukov2020nearly}},\\
    C B
     \log\left(en\right) 
    \left(\sqrt{\log(pn)} + \log^{(1/\alpha)}(pn) \sqrt{\log\log(ep)} \right), 
    &\mbox{ from \cref{thm:1_dep_berry_esseen_4P} above.}
    \end{cases}
\end{aligned}
\end{equation*}
Corollary 2.1 of \citetalias{chernozhukov2020nearly} only consider  sub-Gaussian $X_i$'s, but the proof readily extends to the other cases with $\alpha < 2$.
This leads to the dimension complexity of $\log^{3 + 2/\alpha}(ep) = o(n)$.
Our bounds stem from the fact that $L_{3,j} \leq \sqrt{L_{2,j} L_{4,j}} \leq B$, $L_{4,j} \leq B^2$, and $\nu_{q,j} \leq C q^{q/2} B^{q} \log^{q/\alpha}(ep)$ for any $q \geq 4$ (see Corollary 7.4, \citealp{zhang2020concentration}). Taking $q-2 = \log\log(ep)$ yields the above bound. 
The resulting dimension complexity is the same as \citetalias{chernozhukov2020nearly} up to $\log(en)$ factors.

    



\subsubsection{\texorpdfstring{$m$}{m}-dependent summands} \label{sec:comparison_m_dep}
In this section, we compare \cref{thm:m_dep_berry_esseen} to existing CLT results for $m$-dependent random variables. 
Towards that goal, we find it helpful to recall the well-known Berry--Esseen bounds of \citet{shergin1980convergence} for $m$-dependent univariate ($p=1$) sequences. Specifically, Theorem 2 of \citet{shergin1980convergence} yields that, for $q \geq 3$,
\begin{equation} \label{eq:m_dep_berry_esseen_Shergin}
\begin{aligned}
    \mu\left( X_{[1,n]}, Y_{[1,n]} \right)
    & \leq C(q,M,n_0) \left[ (m+1)^{q-1} \frac{\sum_{i=1}^n \Exp[\abs{X_i}^q]}{(\Exp[X_{[1,n]}^2])^{q/2}} \right]^{1/(q-2)},
\end{aligned}
\end{equation}
under the assumption that $\sum_{i=1}^n \Var[X_i] \leq M \cdot \Var[\sum_{i=1}^n X_i]$ for $n$ larger than an appropriate $n_0 > 0$. 
We note that this assumption is similar to our Assumptions \eqref{assmp:min_var} and \eqref{assmp:min_ev}. 
Furthermore, the dependence on $m$ cannot be improved, as shown in \citet{berk1973central}. Next, using the fact that, in the univariate case, $\Exp[\abs{X_i}^q] = L_{q,i} = \nu_{q,i}$ for all $i$, we conclude that the bounds we obtain in \Cref{thm:m_dep_berry_esseen} match Shergin's bound \eqref{eq:m_dep_berry_esseen_Shergin}, up to universal constants and poly-logarithmic factors. What is more, the same remains true in the multivariate case, as long as $p$ is fixed: the upper bound of \cref{thm:m_dep_berry_esseen} exhibits the same scaling in $n$ and $m$ as in the expression \eqref{eq:m_dep_berry_esseen_Shergin}, up to logarithmic factors. This indicates that, in fixed dimensional settings, our bounds are essentially sharp.

 Moving on to the high-dimensional case in which the dimension has to be explicitly accounted for in the rates,
the literature on high-dimensional Berry--Esseen bounds for $m$-dependent random vectors is relatively scarce. \citet{zhang2018gaussian} used a Berry--Esseen bound for $m$-dependent series to obtain a bound for weak physical processes. Though their results do not seem to provide explicit rates for $m$-dependent cases, those rates appear to be no faster $m^{1/2}/n^{1/8}$ (ignoring the impact of the dimension), which is slower than the rates we achieve; see their Theorem 2.1. 
More recently, by employing the large-small-block approach similar to the one used in \citet{romano2000more} in the univariate cases, \citet[Section 2.1.2]{chang2021central} derived a bound of order $O({m^{2/3}\mathrm{polylog}(p,n)}/{n^{1/6}})$, under the strong assumption of 
sub-exponential random vectors. 
\citet[Theorem 1.4]{fang2021high} improved the sample size dependence to be of order $m / n^{1/4}$ under local dependence conditions, including $m$-dependent sequences, using finite third-moment conditions.

To summarize, \cref{thm:m_dep_berry_esseen} yields a  rate of order 
\[
O({\mathrm{polylog}(p,n)}{m^{(q-1)/(q-2)} / \sqrt{n}}),
\]
which is essentially optimal and is achieved under minimal assumptions.

\section{\texorpdfstring{$m$}{m}-dependent bootstrap} \label{sec:m_dep_bootstrap}

Although \cref{thm:m_dep_berry_esseen} provides a Berry--Esseen bound for the $m$-dependent sum, $X_{[1,n]}$, its direct use is limited in practice because the covariance matrix of $X_{[1,n]}$ is often times unknown. Below we present a practicable high-dimensional bootstrap scheme that overcomes this issue.
Now, there are typically two standard approaches to deploy Corollary 2.5 in practice: (1) by replacing the unknown covariance matrix of the approximating Gaussian distribution with an estimated one and then using Gaussian comparison inequalities ; or (2) by developing a bootstrap scheme to estimate directly the limiting distribution directly without an explicit covariance estimation. In classical (fixed-dimensional) setting, it is well-known that the naive covariance estimator $\hat{\Sigma}$ given below in in~\eqref{eq:sigma_hat} can fail to be positive semi-definite, even though it is consistent and unbiased \citep{newey1987simple}. Hence, we cannot replace the unknown $\Sigma$ with the naive estimator. To overcome this issue,~\cite{newey1987simple} proposed the heteroscedasticity and autocorrelation consistent (HAC) estimator that requires pooling covariances from distant observations even when the data is $1$-dependent. This yields a slower rate for HAC estimator as shown in~\cite{andrews1991heteroskedasticity,zhang2017gaussian}. Without an explicit covariance estimation, in the high-dimensional settings,~\citep{zhang2014bootstrapping} have proposed block multiplier bootstrap schemes following the idea of HAC estimator. This implies a suboptimal rate for all such bootstrap schemes as well.
%
To briefly elaborate on the existing bootstrap schemes, we describe the block-based bootstrap schemes of \citet{zhang2014bootstrapping}.
First, they partition the sample into contiguous blocks  
\begin{equation*}
  I_k \;=\; [i_{k-1}+1,\,i_k], 
  \quad k=1,\dots,K, 
  \qquad 0=i_0<i_1<\dots<i_K=n,
\end{equation*}
and apply the classical bootstrap to the block sums $\{X_{I_k}\}_{k=1}^K$.  
The construction ensures $X_{I_{k_1}} \indep X_{I_{k_2}}$ whenever $\abs{k_1-k_2} > 1$, but neighboring blocks $(X_{I_{k-1}},X_{I_k})$ remain dependent.  
Because of this residual dependence the resulting variance estimator is biased.  
Controlling the bias forces the block length, $\abs{I_k}$ to increase with $n$, even when $m$ is fixed, yielding an error rate no better than $n^{-1/8}$ in their Eq.~(40).
To overcome this issue, we introduce a new bootstrap procedure specifically tailored to $m$‑dependent random vectors.  
Leveraging the bounds of Theorem~\ref{thm:m_dep_berry_esseen}, we show that this methodology delivers non‑asymptotic coverage guarantees without any efficiency loss. 

To motivate the proposed procedure, we begin with the simplest case of $m=1$. In this case a natural unbiased estimator for the covariance matrix $\Sigma$ of $Y_{[1,n]}$ is 
\begin{equation} \label{eq:sigma_hat}
    \hat{\Sigma} 
    \equiv \sum_{i=1}^n \left( X_i X_i^\top 
    + X_i X_{i+1}^\top + X_{i+1} X_i^\top \right).
\end{equation}
A straightforward approach is to replace the infeasible distribution of $Y_{[1,n]}$ by the feasible multi-variate Gaussian distribution with mean $\vec{0}$ and covariance $\hat\Sigma$.
However, 
$\hat\Sigma$ is not guaranteed to be positive semidefinite, and hence, this Gaussian distribution is not well-defined. 
To address this, we propose projecting $\hat\Sigma$ onto the cone of positive semidefinite matrices. This approach is closely related to Remark D.1 of \citet{kock2025high}.
We begin by 
computing the closest positive semidefinite matrix to $\hat\Sigma$ in the elementwise $\ell_\infty$ norm:
\begin{equation} \label{eq:rho_tilde}
    \tilde\Sigma 
    \equiv \underset{\Sigma' \in \mathbb{S}_{+}^p}{\arg\min} \;\norm{\Sigma' - \hat\Sigma}_\infty,
\end{equation}
where $\mathbb{S}_{+}^p$ denotes the cone of $p \times p$ positive semidefinite matrices, and $\norm{\Sigma' - \hat\Sigma}_\infty \equiv \max_{i, j \in [p]} \abs{\Sigma'_{ij} - \hat\Sigma_{ij}}$ is the entrywise maximum norm.
This construction ensures that $\tilde\Sigma$ is positive semidefinite and elementwise close to $\hat\Sigma$.
Once $\tilde\Sigma$ is obtained, we define the bootstrap sample as 
\begin{equation*}
    \tilde{Y} | \mathscr{X}_{[1,n]} \sim \distNorm(\vec{0}, \tilde\Sigma).
\end{equation*}
For statistical inference over hyper-rectangles, we follow the standard approach: we replace the unknown distribution of $X_{[1,n]}$ with the computable bootstrap distribution of $\tilde{Y}$ given the observed data. Inference is then carried out using samples from the bootstrap distribution. For example, to estimate the $\alpha$-quantile of $\lVert X_{[1,n]} \rVert_\infty$, we compute
\begin{equation*}
    \hat{Q}_{\norm{X_{[1,n]}}_\infty}(\alpha) 
    \equiv \inf\left\{ t \in \reals: \Pr[\norm{\tilde{Y}}_\infty \leq t | \mathscr{X}_{[1,n]}] \geq \alpha \right\},
\end{equation*}
as in Eq. (35) of \citet{zhang2014bootstrapping}. 
The accuracy of our inference hinges on the distributional distance (evaluated over hyper-rectangles) between the sampling distribution of $X_{[1,n]}$ and the bootstrap distribution of $\tilde{Y}$ \citep[Eq. (38)]{zhang2014bootstrapping}, given by
\begin{equation*}
    \mu(X_{[1,n]}, \tilde{Y}) \equiv \sup_{A \in \mathcal{R}_p} \abs*{ \mathbb{P} \left( X_{[1,n]} \in A \right) - \mathbb{P} \left( \left. \tilde{Y} \in A  \right| \mathscr{X}_{[1,n]} \right) }.
\end{equation*}
We note that $\mu(X_{[1,n]}, \tilde{Y})$ is itself random and that the accuracy of the bootstrap procedure depends on how fast $\mu(X_{[1,n]}, \tilde{Y})$ converges to $0$. The following result provides a high-probability bound on such approximation error. 
\begin{theorem} \label{thm:1_dep_bootstrap}
    Suppose that Assumptions \eqref{assmp:min_var}, \eqref{assmp:min_ev} and \eqref{assmp:var_ev} hold. Then, for $m=1$, $q \geq 3$ and $\delta \in (0, 1)$, 
    \begin{equation} \label{eq:bootstrap_error_bound}
    \begin{aligned}
        \mu(X_{[1,n]}, \tilde{Y}) 
        & \leq \mu(X_{[1,n]}, Y_{[1,n]})
        + \frac{\Delta}{\sigma_{\min}^2} \log(ep) \left( 1 \vee \abs*{\log\frac{\Delta}{\sigma_{\min}^2}} \right),
    \end{aligned}
    \end{equation}
    with probability at least $1 - \delta$, where
    \begin{equation*}
    \begin{aligned}
        \Delta & \equiv C \bar{L}_{\min\{4, q\}}^{1/2} \left(\frac{\bar{\nu}_{q}}{\delta} \right)^{\max\{(2/q)-(1/2),0\}}
        \left(\frac{\log(2p/\delta)}{n}\right)^{\min\{1-(2/q), 1/2\}} \\
        & \quad + C \left(\frac{\bar{\nu}_{q}}{\delta} \right)^{2/q} \left(\frac{\log(2p/\delta)}{n}\right)^{1-(2/q)}
    \end{aligned}
    \end{equation*}
    and $C > 0$ is a universal constant.
\end{theorem}
\cref{thm:1_dep_bootstrap} extends Corollary 3.1 of \citetalias{chernozhukov2020nearly} to the $m$-dependent case and recovers the optimal $1/\sqrt{n}$ rate whenever $q \geq 4$.

\medskip
To consider the general case of $m > 1$, we adopt a block construction akin to that of \citet{zhang2014bootstrapping}. In detail, we consider the sequence  $(X'_1,\ldots,X'_n)$ defined in \cref{eq:block_setting}, where $n' \equiv \floor{{n}/{m}}$. Then, 
by \cref{thm:1_to_m}, this shorter sequence is $1$-ring dependent and therefore we may apply the proposed bootstrap procedure for $1$-dependence to the block sums.
Combining \cref{thm:1_to_m,thm:1_dep_bootstrap}, we obtain the following result as a corollary.
\begin{corollary} \label{thm:m_dep_bootstrap}
    Suppose that Assumptions \eqref{assmp:min_var}, \eqref{assmp:min_ev} and \eqref{assmp:var_ev} hold. Then, for $m > 1$, $q \geq 3$ and $\delta \in (0, 1)$, 
    \begin{equation*}
    \begin{aligned}
        \mu(X_{[1,n]}, \tilde{Y}) 
        & \leq \mu(X_{[1,n]}, Y_{[1,n]})
        + \frac{\Delta}{\sigma_{\min}^2} \log (ep) \left( 1 \vee \abs*{\log\frac{\Delta}{\sigma_{\min}^2}} \right),
    \end{aligned}
    \end{equation*}
    with probability at least $1 - \delta$, where
    \begin{equation*}
    \begin{aligned}
        \Delta & \equiv C m^{\min\{2-(2/q), 3/2\}} \bar{L}_{\min\{4, q\}}^{1/2} \left(\frac{\bar{\nu}_{q}}{\delta} \right)^{\max\{(2/q)-(1/2),0\}}
        \left(\frac{\log(2p/\delta)}{n}\right)^{\min\{1-(2/q), 1/2\}} \\
        & \quad + C m^{2-(2/q)} \left(\frac{\bar{\nu}_{q}}{\delta} \right)^{2/q} \left(\frac{\log(2p/\delta)}{n}\right)^{1-(2/q)}
    \end{aligned}
    \end{equation*}
    and $C > 0$ is a universal constant.
\end{corollary}
If $q \geq 4$ and $n \geq m$, the term $\Delta$ is of order $m^{3/2} / \sqrt{n}$ up to logarithmic factors in $p$, which is  smaller than the Berry--Esseen bound for $\mu(X_{[1,n]},Y_{[1,n]})$ in Theorem~\ref{thm:m_dep_berry_esseen}. Hence, the bootstrap error is dominated by the Gaussian approximation error, and the bootstrap procedure attains the same convergence rate as that of the infeasible Gaussian approximation, paralleling the analogous bootstrap guarantees for the independent case recently obtained by \citet{lopes2020central, lopes2020bootstrapping}; \citetalias{chernozhukov2020nearly}.

\medskip
\textbf{Solving the optimization problem in \cref{eq:rho_tilde}. }
We observe that the optimization problem in \cref{eq:rho_tilde} is convex and can be solved numerically with projected subgradient descent. Starting from $\tilde\Sigma^{(0)} \equiv \diag(\hat\Sigma_{11}, \dots, \hat\Sigma_{pp})$, we iterate
\begin{equation} \label{eq:psg_update}
    \tilde\Sigma^{(k+1)} \equiv \Pi_{\mathbb{S}_{+}^p}(\tilde\Sigma^{(k)} - \eta g^{(k)}),
\end{equation}
where $\eta > 0$ is a fixed step size, $g^{(k)} \in \partial\norm{\Sigma' - \hat\Sigma}_\infty |_{\Sigma' = \tilde\Sigma^{(k)}}$, and $\Pi_{\mathbb{S}_{+}^p}$ denotes projection onto the cone of $p \times p$ positive semidefinite matrices under the Frobenius norm (implemented by eigenvalue thresholding).  A convenient subgradient is obtained by choosing one index that attains the current maximal entry-wise deviation.  Let  
\begin{equation*}
    \mathcal{M}(k)
    \equiv {\arg\max}_{(i,j) \in [p] \times [p]} \norm*{\tilde\Sigma^{(k)}_{ij}-\hat\Sigma_{ij}}
\end{equation*}
and, for any $(i^*,j^*) \in \mathcal M(k)$, let
\begin{equation} \label{eq:subgradient}
    g^{(k)}_{ij}
    = \begin{cases}
      \mathrm{sgn}(\tilde\Sigma^{(k)}_{i^*j^*} - \hat\Sigma_{i^*j^*}),
      & (i,j)=(i^*,j^*) ~\text{or}~ (i,j)=(j^*,i^*),\\
      0, & \text{otherwise}.
    \end{cases}
\end{equation}
The approximation by projected subgradient descent satisfies the following finite sample bound. 

\begin{theorem} \label{thm:psg_error}
    Let $\{\tilde\Sigma^{(k)} : k \in [K]\}$ be the sequence generated by \cref{eq:psg_update} with $K \in \nats$ and $\eta = p \norm{\hat\Sigma}_\infty/\sqrt{K}$, and define $k^* \equiv {\arg\min}_{0 \leq k \leq K} \norm{\tilde\Sigma^{(k)}-\hat\Sigma}_\infty$.
    Then, with probability at least $1 - \delta'$,
    \begin{equation*}
        \norm{\tilde\Sigma^{(k^*)} - \hat\Sigma}_\infty - \norm{\tilde\Sigma - \hat\Sigma}_\infty 
        \leq C \frac{p \norm{\Sigma}_\infty}{\sqrt{K}},
    \end{equation*}
    for some universal constant $C>0$, where
    \begin{equation*}
        \delta' = C p \cdot \exp\left( - \frac{n \norm{\Sigma}_\infty^2}{C \bar{L}_4} \right) 
        + \left( \frac{C}{\norm{\Sigma}_\infty} \right)^{q/2} \left( \frac{\log(2p/\delta_\circ)}{n} \right)^{q/2 - 1} \bar\nu_q,
    \end{equation*}
    and $\delta_\circ = \left( \frac{C}{\norm{\Sigma}_\infty} \right)^{q/2} \left( \frac{1}{n} \right)^{q/2 - 1} \bar\nu_q$.
    The computational cost of each iteration is $O(p^{2})$, so the overall computational complexity is $O(p^{2} K)$.
\end{theorem}

We note that the condition for the convergence of $\delta'$ in \cref{thm:psg_error} is weaker than those required for convergence to zero of the bootstrap error bound in (the second term on the right hand side of) \cref{eq:bootstrap_error_bound}. The approximation error of $\tilde{\Sigma}^{(k^*)}$ converges to zero as $K$ increases to infinity, and for $K \geq p^2 n$, $\tilde\Sigma^{(k^*)}$
inherits the same accuracy as $\tilde\Sigma$. As a result, the bootstrap error bound in
\cref{eq:bootstrap_error_bound} remains valid (with possibly a different constant).

\begin{corollary}
    Conditional on $\mathscr{X}_{[1,n]}$, let $\tilde{Y} \sim N(\vec 0, \tilde\Sigma^{(k^{*})})$. Then with probability at least $1 - \delta'$, \cref{eq:bootstrap_error_bound} continues to hold with an adjusted universal constant, where $\delta'$ is defined as in \cref{thm:psg_error}.
\end{corollary}

\section{Discussion} \label{sec:discussion}

We derived a Berry--Esseen bound for high-dimensional $m$-dependent random vectors over hyper-rectangles under the weakest possible settings of finite third moments, achieving the optimal rate (in $n$) of order $1/\sqrt{n}$. For fixed $p$, the rate implied by our result is $m^{(q-1)/(q-2)} / \sqrt{n}$, matching the classical optimal rate for $m$-dependent (univariate) random variables. 
For applicability in statistical inference, we further propose a new bootstrap scheme to estimate the ``limiting'' Gaussian distribution. Our bootstrap scheme is new in that it does not use the classical Newey-West Heteroscedastic and autocorrelation (HAC) consistent estimator of the covariance matrix, but a projection of the naive estimator of the covariance matrix. This choice played an important role in getting faster rates for the bootstrap approximation.

Our results about the Gaussian approximations of $m$-dependent samples could be applied in existing analyses of physical dependence settings. \citet{zhang2018gaussian} introduced the $m$-approximation technique to study the Gaussian approximation of weakly dependent time series under physical dependence. The technique extends Berry--Esseen bounds for $m$-dependent samples to weaker temporal dependencies; see Theorem 2.1 and the end of Section 2.2 therein. Similarly, \citet{chang2021central} derived an $1/n^{1/6}$ rate under $m$-dependence to samples with physical dependence. The resulting rate in Theorem 3 improved the rates obtained by \citet{zhang2017gaussian}.

Another important future direction is extending our technique to samples with generalized graph dependency. Random vectors $X_1, X_2, \dots, X_n \in \reals^p$ are said to have dependency structure defined by graph $G = ([n], E)$ if $X_i \indep X_j$ if $(i,j) \in E$. Graph dependency generalizes $m$-dependence as a special case by taking $E = \{(i,j): \abs{i-j} \leq m\}$. To the best of our knowledge, the only CLT result is that of \citet{chen2004normal} for $1$-dimensional samples with graph dependency. Extending their result to high-dimensional samples has wide potential applicability in statistical network analysis.

\section{Proof techniques} \label{sec:pf_sketch}

The proofs of \cref{thm:1_dep_berry_esseen_3P,thm:1_dep_berry_esseen_4P} are  somewhat involved and contains multiple steps. For the reader's convenience, below we sketch the proof of a simplified and weaker version of our bounds, assuming  $3 \leq q$ and $1$-dependence, instead of $1$-ring dependence. In this particular case, we have $X_1 \indep X_n$. This results in a  Berry--Esseen bound similar to the one presented \cref{thm:1_dep_berry_esseen_3P}, but with
\begin{equation*}
    L_{3,\max} \equiv \max_{i \in [1,n]} L_{3,i} 
    \textand
    \nu_{q,\max} \equiv \max_{i \in [1,n]} \nu_{q,i}
\end{equation*}
in place of $\bar{L}_3$ and $\bar{\nu}_q$, respectively.
In the proof, we use the inductive relationship between anti-concentration bounds and Kolmogorov-Smirnov statistics. Anti-concentration refers to the probability of a random variable to be contained in a small subset (typically an annulus). 
An anti-concentration probability bound for the Gaussian distribution that is commonly used in the recent literature on high-dimensional central limit theorem is that established by 
\citet{nazarov2003maximal} \citep[see also][]{chernozhukov2017detailed}, who derived an upper bound for the probability that a Gaussian random vector is contained in
$A_{r,\delta} \equiv \{x \in \reals^p: x_k \leq r_k + \delta, \forall k \in [p]\}
\setminus \{x \in \reals^p: x_k \leq r_k - \delta, \forall k \in [p]\}$
for $r \in \reals^p$ and $\delta \in [0, \infty)$.

\begin{lemma}[Gaussian anti-concentration inequality; \citealp{nazarov2003maximal,chernozhukov2017detailed}]
\label{thm:G_anti_conc}
For a random vector $Y \dist N(0,\Sigma)$ in $\reals^p$, $r \in \reals^p$, and $\delta \in [0, \infty)$,
\begin{equation*}
    \Pr[Y \in A_{r,\delta}] \leq C \delta \sqrt{\frac{\log(ep)}{\min_{k \in [p]} \Sigma_{kk}}}
\end{equation*}
for an absolute constant $C > 0$.
\end{lemma} 

The relevant anti-concentration probability for our problem involves the conditional probability of all the partial sums of the data $X_{I}$ given $\mathscr{X}_{I'}$ for $I, I' \subset [1,n]$. We denote the supremum of these probabilities over all possible choices of the parameter $r \in \mathbb{R}$ and the condition $\mathscr{X}_{I'}$ by
\begin{equation*}
    \kappa_{I|I'}(\delta) \equiv \sup_{r\in\reals^p} \underset{\mathscr{X}_{I'}}{\mathrm{esssup}} ~\Pr[X_{I} \in A_{r,\delta} | \mathscr{X}_{I'}].
\end{equation*}
Throughout the proof, we use various  monotonicity properties of $\kappa$, detailed in the next result. 
\begin{lemma} \label{thm:kappa_comparison} 
    Suppose that $0 \leq i_1 < i_2 \leq n$ and that $0 \leq \delta \leq \delta'$. Then under $1$-dependence or $1$-ring dependence,
    \begin{enumerate}
        \item $\kappa_{(i_1,i_2)}(\delta) \leq \kappa_{(i_1,i_2)|\{i_2\}}(\delta)$
        ~\text{and}~
        $\kappa_{(i_1,i_2)}(\delta) \leq \kappa_{(i_1,i_2)|\{i_1\}}(\delta)$;
        \item $\kappa_{(i_1,i_2)|\{i_2\}}(\delta) \leq \kappa_{(i_1,i_2)|\{i_1,i_2\}}(\delta)$
        ~\text{and}~
        $\kappa_{(i_1,i_2)|\{i_1\}}(\delta) \leq \kappa_{(i_1,i_2)|\{i_1,i_2\}}(\delta)$;
        \item $\kappa_{(i_1,i_2)|\{i_2\}}(\delta) \leq \kappa_{(i_1,i_2-1)|\{i_2-1\}}$ 
        ~\text{and}~
        $\kappa_{(i_1,i_2)|\{i_1\}}(\delta) \leq \kappa_{(i_1+1,i_2)|\{i_1+1\}}$;
        \item $\kappa_{(i_1,i_2)|\{i_1,i_2\}}(\delta) \leq \kappa_{(i_1,i_2-1)|\{i_1,i_2-1\}}$ 
        ~\text{and}~
        $\kappa_{(i_1,i_2)|\{i_1,i_2\}}(\delta) \leq \kappa_{(i_1+1,i_2)|\{i_1+1,i_2\}}$;
        \item $\kappa_{(i_1,i_2)}(\delta) \leq \kappa_{(i_1,i_2)}(\delta')$, 
        $\kappa_{(i_1,i_2)|\{i_2\}}(\delta) \leq \kappa_{(i_1,i_2)|\{i_2\}}(\delta')$, 
        and $\kappa_{(i_1,i_2)|\{i_1,i_2\}}(\delta) \leq \kappa_{(i_1,i_2)|\{i_1,i_2\}}(\delta')$.
    \end{enumerate}
\end{lemma}

We also denote the Kolmogorov-Smirnov statistics of interest by
\begin{equation*}
    \mu_{I} \equiv \mu({X}_{I}, {Y}_{I}),
\end{equation*}
keeping track explicitly in our notation of the range of summation.
We start the proof by establishing the inductive relationship from $\kappa$ to $\mu$: for $\delta > \sigma_{\min}$,
\begin{equation} \label{eq:mu_by_kappa}
\begin{aligned}
    \sqrt{n} \mu_{[1,n]}
    & \leq C(\nu_q, \sigma_{\min}, \underline{\sigma}) \delta \log(ep) 
    + C(\nu_q, \underline{\sigma}) \frac{ \log(en) (\log(ep))^{3/2}}{\delta} \sup_{i \in [1,n)} \sqrt{i} \kappa_{[1,i]|\{i+1\}}\left(\delta\right).
\end{aligned}
\end{equation}
To derive this key bound, we use the Lindeberg swapping technique, using the approach of \citetalias{kuchibhotla2020high}. Our key contribution in this aspect lies in addressing the added complexity due to the dependence structure. For details,  refer to \cref{sec:mu_by_kappa_3N}.

Using \cref{eq:mu_by_kappa}, our next step is to establish an upper bound on the anti-concentration probabilities $\kappa_{[1,i]|\{i+1\}}(\delta)$. Under independence, $\kappa_{[1,i]|\{i+1\}}(\delta)$ represents the marginal anti-concentration probability since conditioning on $X_{j+1}$ in the definition of $\kappa_{[1,i]|\{i+1\}}(\delta)$ can be omitted. Consequently, an upper bound for $\kappa_{[1,i]|\{i+1\}}(\delta)$ can be obtained in a straightforward way (see \cref{thm:G_anti_conc}); that is, for any $i \in [1,n)$,
\begin{equation} \label{eq:kappa_by_mu_indep}
    \kappa_{[1,i]|\{i+1\}}(\delta) = \kappa_{[1,i]}(\delta) \leq \mu_{[1,i]} + C \delta \sqrt{\frac{\log(ep)}{\sigma_{\min} \cdot i}}.
\end{equation} \citetalias{kuchibhotla2020high} implicitly employed a dual induction approach using \cref{eq:mu_by_kappa,eq:kappa_by_mu_indep} to establish a high-dimensional Berry--Esseen bound with the desired $1/\sqrt{n}$ rate, up to logarithms. However, \cref{eq:kappa_by_mu_indep} falls short when dealing with $1$-dependence, as $X_{[1,j]}$ becomes dependent on $X_{i+1}$ in the definition of $\kappa_{[1,i]|\{i+1\}}(\delta)$.
In the case of univariate dependent $X_i$, \citet{chen2004normal} derived a non-inductive upper bound for the conditional anti-concentration probability using a telescoping method (see Proposition 3.2 therein). Extending this method to the high-dimensional case is non-trivial. In our work, we adopt a similar intuition, but instead of aiming for a non-inductive bound, we establish an inductive relationship from $\mu$ to $\kappa$; specifically, for $i \in [1,n)$ and $\delta > \sigma_{\min}$, we show in \cref{thm:induction_lemma_AC_N} below that
\begin{equation} \label{eq:kappa_by_mu}
    \sqrt{i}\kappa_{[1,i]|\{i+1\}}(\delta) \leq C \left( \frac{\delta + \nu_1}{\sigma_{\min}} \sqrt{\log(ep)} + \max_{j \in [1,i)} \sqrt{j} \mu_{[1,j]} \right).
\end{equation}
Finally, using \cref{eq:mu_by_kappa,eq:kappa_by_mu}, we perform a dual induction to complete the proof.


The proofs of \cref{thm:1_dep_berry_esseen_3P,thm:1_dep_berry_esseen_4P} share similar steps but incorporate additional technical steps, explained in the last two subsections. In \cref{sec:pf_sketch_4N}, we present the iterated Lindeberg swapping method, which helps improve the dimension complexity when a finite fourth moment condition is satisfied. Additionally, in \cref{sec:permutation_argument}, we introduce permutation arguments to enhance the Berry--Esseen bounds by replacing the maximal moments with the average moments. For the complete proofs, please refer to \cref{sec:pf_1_dep_3P,sec:pf_1_dep_4P}.

\subsection{Induction from \texorpdfstring{$\kappa$}{κ} to \texorpdfstring{$\mu$}{μ} for \texorpdfstring{$3 \leq q$}{3 ≤ q}}
\label{sec:mu_by_kappa_3N}

Let $i$ be fixed in $[1,n]$. The quantity we want to control concerns expectations of indicator functions, which are not smooth. For this reason, most proofs of CLTs apply a smoothing to replace indicator functions by smooth functions. We use the mixed smoothing proposed by \citet{chernozhukov2020nearly}. Namely, for $r \in \reals^p$ and $\delta, \phi > 0$, let 
\begin{equation*}
    \rho_{r,\phi}^{\delta}(x) \equiv \Exp[ f_{r,\phi}(x + \delta Z)],
\end{equation*}
where 
\begin{equation*}
    f_{r,\phi}(x) \equiv \begin{cases}
        1, & \text{if} ~ \max\{x_k - r_k: k \in [p]\} < 0, \\
        1 - \phi \max\{ x_k - r_k: k \in [p]\},
        & \text {if} ~ 0 \leq \max\{ x_k - r_k: k \in [p]\} < 1/\phi, \\
        0, & \text{if} ~ 1/\phi \leq \max\{x_k - r_k: k \in [p]\}.
    \end{cases} 
\end{equation*}
This smoothing leads to a bias term , which can be controlled using  Lemma 1 of \citetalias{kuchibhotla2020high} and Lemma 2.1 of \citetalias{chernozhukov2020nearly}. For convenience, we report these two results in the next lemma.

\begin{lemma} \label{thm:smoothing}
Suppose that $X$ is a $p$-dimensional random vector, and $Y \dist N(0,\Sigma)$ is a $p$-dimensional Gaussian random vector. Then, for any $\delta > 0$,
\begin{equation*}
    \mu(X, Y) 
    \leq C \frac{\delta\log(ep) + \sqrt{\log(ep)}/\phi}{\sqrt{\min_{k \in [p]} \Sigma_{kk}}}
    + C \sup_{r \in \reals^p} \abs*{ \Exp\left[\rho_{r,\phi}^\delta(X) - \rho_{r,\phi}^\delta(Y)\right]},
\end{equation*}
where $C > 0$ is a universal constant.
\end{lemma}

Because $\min_{k \in [p]} \Sigma_{kk} \geq n \sigma_{\min}^2$, \cref{thm:smoothing} implies
\begin{equation*}
    \mu_{[1,n]} 
    \leq \frac{C}{\sqrt{n}} \frac{\delta \log(ep)}{\sigma_{\min}}  
    + \frac{C}{\sqrt{n}} \frac{\sqrt{\log(ep)}}{\phi \sigma_{\min}}
    + C \sup_{r \in \reals^p} \abs*{ \Exp\left[\rho_{r,\phi}^\delta(X_{[1,n]}) - \rho_{r,\phi}^\delta(Y_{[1,n]})\right]}.
\end{equation*}

\medskip
\noindent{\bf Lindeberg swapping.} 
In the standard Lindeberg swapping approach, one seeks to upper bound the quantity \begin{equation*}
    \sup_{r \in \reals^p} \abs*{ \Exp\left[\rho_{r,\phi}^\delta(X_{[1,n]}) - \rho_{r,\phi}^\delta(Y_{[1,n]})\right]}
\end{equation*}
by expressing it as
\begin{equation} \label{eq:first_lindeberg_swapping_N}
\begin{aligned}
    & \sup_{r \in \reals^p} \abs*{ \Exp\left[\rho_{r,\phi}^\delta(X_{[1,n]}) - \rho_{r,\phi}^\delta(Y_{[1,n]})\right]} \\
    & = \sup_{r \in \reals^p} \abs*{ \sum_{j=1}^{n} 
    \Exp \left[
        \rho_{r,\phi}^\delta(W^\cmpl_{[j,j]} + X_j)
        - \rho_{r,\phi}^\delta(W^\cmpl_{[j,j]} + Y_j)
    \right]},
\end{aligned}
\end{equation}
where $W_{[j_1,j_2]}^\cmpl \equiv X_{[1,j_1)} + Y_{(j_2,n]}$, and then by further bounding each term in the summation via a third-order remainder terms of Taylor expansions up to order $3$, leveraging the second moment matching between $X_j$ and $Y_j$. Below, we will use the symbol $W$ as a wildcard, representing either $X$ or $Y$ depending on the context. 
\begin{equation} \label{eq:lindeberg_swapping_indep}
\begin{aligned}
    & \Exp\left[\rho_{r,\phi}^\delta(W^\cmpl_{[j,j]} + X_j)\right] \\
    & = \Exp\left[\rho_{r,\phi}^\delta(W^\cmpl_{[j,j]})\right]
    + \Exp\left[ \inner*{ \nabla \rho_{r,\phi}^\delta(W^\cmpl_{[j,j]}), X_j} \right]
    + \Exp\left[ \frac{1}{2} \inner*{ 
    \nabla^2 \rho_{r,\phi}^\delta(W^\cmpl_{[j,j]}), X_j^{\otimes 2}} \right]  \\
    & \quad + \Exp\left[ \frac{1}{2} \int_0^1 (1-t)^2 \inner*{
    \nabla^3 \rho_{r,\phi}^\delta(W^\cmpl_{[j,j]} + t X_j), X_j^{\otimes 3}} ~dt \right] \\
    & = \Exp\left[ \rho_{r,\phi}^\delta(W^\cmpl_{[j,j]}) \right]
    + \frac{1}{2} \inner*{ 
    \Exp\left[ \nabla^2 \rho_{r,\phi}^\delta(W^\cmpl_{[j,j]}) \right], \Exp[ X_j^{\otimes 2} ]} \\
    & \quad + \frac{1}{2} \int_0^1 (1-t)^2 \Exp\left[ 
    \inner*{ \nabla^3 \rho_{r,\phi}^\delta(W^\cmpl_{[j,j]} + t X_j), X_j^{\otimes 3}} \right] ~dt,
\end{aligned}
\end{equation}
so by $\Exp[X_j^{\otimes 2}] = \Exp[Y_j^{\otimes 2}]$,
\begin{equation*} 
\begin{aligned}
    & \Exp\left[\rho_{r,\phi}^\delta(W^\cmpl_{[j,j]} + X_j) - \rho_{r,\phi}^\delta(W^\cmpl_{[j,j]} + Y_j)\right]\\
    & = \frac{1}{2} \int_0^1 (1-t)^2 \Exp\left[ 
    \inner*{ \nabla^3 \rho_{r,\phi}^\delta(W^\cmpl_{[j,j]} + t X_j), X_j^{\otimes 3}} \right] ~dt \\
    & \quad - \frac{1}{2} \int_0^1 (1-t)^2 \Exp\left[ 
    \inner*{ \nabla^3 \rho_{r,\phi}^\delta(W^\cmpl_{[j,j]} + t X_j), Y_j^{\otimes 3}} \right] ~dt.
\end{aligned}
\end{equation*}
However, in the case of $1$-dependence, the second equality of \cref{eq:lindeberg_swapping_indep} no longer holds due to the dependency between $W_{[j,j]}^\cmpl$ and $X_j$. To address this issue, we introduce Taylor expansions on $\nabla\rho_{r,\phi}^\delta(W^\cmpl_{[j,j]})$ and $\nabla^2 \rho_{r,\phi}^\delta(W^\cmpl_{[j,j]})$ to break the dependency before proceeding with the second-order moment matching. This additional step involves lengthy calculations and unwieldily specifications of remainder terms. We provide the full details in \cref{sec:first_lindeberg_swapping}.
As a result,
\begin{equation} \label{eq:first_lindeberg_result_3N}
\begin{aligned}
    & \sum_{j=1}^{n} \Exp \left[
        \rho_{r,\phi}^\delta(W^\cmpl_{[j,j]} + X_j)
        - \rho_{r,\phi}^\delta(W^\cmpl_{[j,j]} + Y_j) \right] 
    = \sum_{j=1}^{n} \Exp \left[ \mathfrak{R}_{X_j}^{(3,1)} 
        - \mathfrak{R}_{Y_j}^{(3,1)} \right],
\end{aligned}
\end{equation}
where $\mathfrak{R}_{X_j}^{(3,1)}$ and $\mathfrak{R}_{Y_j}^{(3,1)}$ are remainder terms of the Taylor expansions specified in \cref{sec:first_lindeberg_swapping}.

\medskip
\noindent{\bf Remainder lemma.} Then, we upper bound the remainder terms using the upper bounds for the differentials of $\rho_{r,\phi}^{\delta}$. In particular, \citetalias{chernozhukov2020nearly} showed in Lemmas 6.1 and 6.2 that
\begin{equation*}
\begin{aligned}
    \sup_{w \in \reals^p} \sum_{k_1,\dots,k_\alpha} \sup_{\norm{z}_\infty \leq \frac{2\delta}{\sqrt{\log(ep)}}}
        \abs*{ \nabla^{(k_1,\dots,k_\alpha)} \rho_{r,\phi}^{\delta}
        ( w + z ) } 
    \leq C \frac{\phi^{\gamma} (\log(ep))^{(\alpha-\gamma)/2}}{{\delta}^{\alpha-\gamma}}
\end{aligned}
\end{equation*}
for any $\gamma \in [0,1]$. For the event that $W_{[j,j]}^\cmpl$ is in the anulus $A_{r,\delta'}$ for a $\delta'$ to be specified shortly, we use the above inequality to bound the remainder term. Out of the event, the differential is sufficiently small. Hence, the upper bounds involve with anti-concentration probabilities of $W_{[j,j]}^\cmpl$.
The following lemma is derived by applying \cref{thm:generalized_remainder_lemma} to the term $\mathfrak{R}_{W_j}^{(3,1)}$ with parameters $\alpha = 3$, $\gamma_1 = 0$, $\gamma_2 = \min\{1, q-3\}$, and $\eta = q-3$. This particular choice of $\gamma_2$ and $\eta$ ensures that the condition $\gamma_2 \in [0,1]$ is satisfied and guarantees that the resulting exponent of $\delta_{n-j}$ is at least $3$. See \cref{sec:pf_remainder_lemma}.

\begin{lemma} \label{thm:remainder_lemma_3N}
    There exist universal constants $C > 0$ and $\alpha > 0$ such that for any $n$, $1$-dependent sequence $(X_1, \dots, X_n)$ satisfying Assumption~\eqref{assmp:min_ev}, $j \in [1,n]$, $q \geq 3$, $\delta \geq \sigma_{\min}$ and $\phi \geq \frac{1}{\delta \log(ep)}$,
    \begin{equation*}
    \begin{aligned}
        \abs*{\Exp \left[ \mathfrak{R}_{W_j}^{(3,1)} \right]} 
        & \leq C \frac{(\log(ep))^{3/2}}{\delta_{n-j}^3} \left[ 
            {L}_{3,\max} 
            + \phi^{\min\{1, q-3\}} \nu_{q,\max} \frac{ {(\log(ep))}^{\max\{0, q-4\}/2}}{(\delta_{n-j}/\alpha)^{\max\{0, q-4\}}}
        \right] \\
        & \quad \times \min\{ 1, \kappa_{[1,j-4]|\{j-3\}}(\delta_{n-j}^\circ) + \kappa^\circ_{j}(\delta_{n-j}) \},
    \end{aligned}
    \end{equation*}
    where $W$ represents either $X$ or $Y$, $\delta_{n-j}^2 \equiv \delta^2 + \underline{\sigma}^2 \max\{n-j,0\}$, $\delta^\circ_{n-j} \equiv 12 \delta_{n-j} \sqrt{\log(pn)}$ and $\kappa^\circ_{j}(\delta) \equiv \frac{\delta \log(ep)}{\sigma_{\min} \sqrt{\max\{j, 1\}}}$.
\end{lemma} 
Plugging back into $\mu_{[1,n]}$, we get
\begin{equation*}
\begin{aligned}
    \mu_{[1,n]}
    & \leq\frac{C}{\sqrt{n}} \frac{\delta \log(ep)}{\sigma_{\min}} 
    + \frac{C}{\sqrt{n}} \frac{\sqrt{\log(ep)}}{\phi \sigma_{\min}}
    + C \sum_{j=1}^{n} \abs*{ \Exp\left[ \mathfrak{R}_{X_j}^{(3,1)} 
        - \mathfrak{R}_{Y_j}^{(3,1)} \right]} \\
    & \leq \frac{C}{\sqrt{n}} \frac{\delta \log(ep)}{\sigma_{\min}}  
    + \frac{C}{\sqrt{n}} \frac{\sqrt{\log(ep)}}{\phi \sigma_{\min}} \\
    & \quad + C \sum_{j=1}^i \frac{(\log(ep))^{3/2}}{\delta_{n-j}^3} \left[ 
        {L}_{3,\max} 
        + \phi^{\min\{1, q-3\}} \nu_{q,\max} \frac{ {(\log(ep))}^{\max\{0, q-4\}/2}}{(\delta_{n-j}/\alpha)^{\max\{0, q-4\}}}
    \right] \\
    & \quad \times \min\{ 1, \kappa_{[1,j-4]|\{j-3\}}(\delta_{n-j}^\circ) + \kappa^\circ_{j}(\delta_{n-j}) \}.
\end{aligned}
\end{equation*}


\medskip
\noindent{\bf Partitioning the sum.} 
%
We upperbound the last line of the previous equation by breaking down the summation over $j$ into the two parts and analyze them separately. First, for $j < n/2$, 
\begin{equation*}
\begin{aligned}
    & \sum_{j < n/2} \frac{(\log(ep))^{3/2}}{\delta_{n-j}^3} \left[ 
        {L}_{3,\max} 
        + \phi^{\min\{1, q-3\}} \nu_{q,\max} \frac{ {(\log(ep))}^{\max\{0, q-4\}/2}}{(\delta_{n-j}/\alpha)^{\max\{0, q-4\}}}
    \right] \\
    & \quad \times \min\{ 1, \kappa_{[1,j-4]|\{j-3\}}(\delta_{n-j}^\circ) + \kappa^\circ_{j}(\delta_{n-j}) \} \\
    & \leq \sum_{j < n/2} \frac{(\log(ep))^{3/2}}{\delta_{n-j}^3} \left[ 
        {L}_{3,\max} 
        + \phi^{\min\{1, q-3\}} \nu_{q,\max} \frac{ {(\log(ep))}^{\max\{0, q-4\}/2}}{(\delta_{n-j}/\alpha)^{\max\{0, q-4\}}}
    \right] \\         
    & \leq \frac{C}{\sqrt{n}} \frac{(\log(ep))^{5/2}}{\underline{\sigma}^2\sigma_{\min}} 
    \left[ 
        {L}_{3,\max} 
        + \phi^{\min\{1, q-3\}} \nu_{q,\max} \frac{ {(\log(ep))}^{\max\{0, q-4\}/2}}{(\delta_{n-j}/\alpha)^{\max\{0, q-4\}}}
    \right],
\end{aligned}
\end{equation*}
because
\begin{equation} \label{eq:sum_delta_3}
\begin{aligned}
    \sum_{j=1}^{\floor{n/2}} \frac{1}{\delta_{n-j}^3} 
    & \le \int_{n-\floor{n/2}}^{n} \frac{1}{(\delta^2 + t\underline{\sigma}^2)^{3/2}}dt 
    \le -\frac{2}{\underline{\sigma}^2\delta_{n}} 
    + \frac{2}{\underline{\sigma}^2\delta_{n-\floor{n/2}}} \\
    & \le \frac{C}{\underline{\sigma}^3\sqrt{n-n/2}} 
    \overset{(*)}{\le} \frac{C \log(ep)}{\underline{\sigma}^2\sigma_{\min}\sqrt{n}},
\end{aligned}
\end{equation}
where $(*)$ follows Assumption~\eqref{assmp:var_ev}.

\medskip
For $j \geq n/2$,
\begin{equation*}
\begin{aligned}
    & \sum_{j \geq n/2} \frac{(\log(ep))^{3/2}}{\delta_{n-j}^3} \left[ 
        {L}_{3,\max} 
        + \phi^{\min\{1, q-3\}} \nu_{q,\max} \frac{ {(\log(ep))}^{\max\{0, q-4\}/2}}{(\delta_{n-j}/\alpha)^{\max\{0, q-4\}}}
    \right] \\
    & \quad \times \min\{ 1, \kappa_{[1,j-4]|\{j-3\}}(\delta_{n-j}^\circ) + \kappa^\circ_{j}(\delta_{n-j}) \} \\
    & \leq \sum_{j \geq n/2} \frac{(\log(ep))^{3/2}}{\delta_{n-j}^3} \left[ 
        {L}_{3,\max} 
        + \phi^{\min\{1, q-3\}} \nu_{q,\max} \frac{ {(\log(ep))}^{\max\{0, q-4\}/2}}{(\delta_{n-j}/\alpha)^{\max\{0, q-4\}}}
    \right]
    \kappa_{[1,j-4]|\{j-3\}}(\delta_{n-j}^\circ) \\
    & \quad + \sum_{j \geq n/2} \frac{(\log(ep))^{3/2}}{\delta_{n-j}^3} \left[ 
        {L}_{3,\max} 
        + \phi^{\min\{1, q-3\}} \nu_{q,\max} \frac{ {(\log(ep))}^{\max\{0, q-4\}/2}}{(\delta_{n-j}/\alpha)^{\max\{0, q-4\}}}
    \right]
    \kappa^\circ_{j}(\delta_{n-j}).
\end{aligned}
\end{equation*}
The last term is upper bounded by
\begin{equation*}
\begin{aligned}
    & \sum_{j \geq n/2} \frac{(\log(ep))^{3/2}}{\delta_{n-j}^3} \left[ 
        {L}_{3,\max} 
        + \phi^{\min\{1, q-3\}} \nu_{q,\max} \frac{ {(\log(ep))}^{\max\{0, q-4\}/2}}{(\delta_{n-j}/\alpha)^{\max\{0, q-4\}}}
    \right] \kappa^\circ_{j}(\delta_{n-j}) \\
    & \leq \frac{C}{\sqrt{n}} \frac{(\log(ep))^{5/2}}{\underline{\sigma}^2\sigma_{\min}}
    \left[ 
        {L}_{3,\max} + \phi^{\min\{1, q-3\}} \nu_{q,\max} \frac{ {(\log(ep))}^{\max\{0, q-4\}/2}}{(\delta/\alpha)^{\max\{0, q-4\}}}
    \right]
    \log\left(1 + \frac{\sqrt{n}\underline{\sigma}}{\delta}\right).
\end{aligned}
\end{equation*}
because
\begin{equation} \label{eq:sum_delta_3_kappa}
    \sum_{j=\ceil{n/2}}^{n} \frac{\kappa^\circ_{j}(\delta_{n-j})}{\delta_{n-j}^3}
    = \sum_{j=\ceil{n/2}}^{n} \frac{\log(ep)}{\delta_{n-j}^2 \sigma_{\min} \sqrt{\max\{j,1\}}}
    \leq \frac{C\log(ep)}{\underline{\sigma}^2\sigma_{\min}\sqrt{n}}
    \log\left(1 + \frac{\sqrt{n}\underline{\sigma}}{\delta}\right).
\end{equation}
Putting everything together, we have established a convenient relationship between $\mu_{[1,i]} \equiv \mu(X_{[1,i]}, Y_{[1,i]})$ and the conditional anti-concentration probability $\kappa_{[1,j-4]|\{j-3\}}(\delta)$, as illustrated in the next result.

\begin{lemma} \label{thm:induction_lemma_BE_3N}
    There exist universal constants $C > 0$ and $\alpha > 0$ such that for any $n$, $1$-dependent sequence $(X_i \in \reals^p: i \in [1,n])$ satisfying Assumptions \eqref{assmp:min_var}, \eqref{assmp:min_ev} and \eqref{assmp:var_ev}, $q \geq 3$, $\delta \geq \sigma_{\min}$ and $\phi \geq \frac{1}{\delta \log(ep)}$,
    \begin{equation*}
    \begin{aligned}
        \mu_{[1,n]}
        & \leq \frac{C}{\sqrt{n}} \frac{\delta \log(ep)}{\sigma_{\min}}  
        + \frac{C}{\sqrt{n}} \frac{\sqrt{\log(ep)}}{\phi \sigma_{\min}} \\
        & \quad + \frac{C}{\sqrt{n}} \frac{(\log(ep))^{5/2}}{\underline{\sigma}^2\sigma_{\min}}
        \left[ 
            {L}_{3,\max} + \phi^{\min\{1, q-3\}} \nu_{q,\max} \frac{ {(\log(ep))}^{\max\{0, q-4\}/2}}{(\delta/\alpha)^{\max\{0, q-4\}}}
        \right] \\
        & \quad \times \log\left(1 + \frac{\sqrt{n}\underline{\sigma}}{\delta}\right) \\
        & \quad + C \sum_{j \geq n/2} \frac{(\log(ep))^{3/2}}{\delta_{n-j}^3} \left[ 
            {L}_{3,\max} + \phi^{\min\{1, q-3\}} \nu_{q,\max} \frac{ {(\log(ep))}^{\max\{0, q-4\}/2}}{(\delta_{n-j}/\alpha)^{\max\{0, q-4\}}}
        \right] \\
        & \quad \times \kappa_{[1,j-4]|\{j-3\}}(\delta_{n-j}^\circ),
    \end{aligned}
    \end{equation*}
    where $\delta_{n-j}^2 \equiv \delta^2 + \underline{\sigma}^2 \max\{n-j,0\}$ and $\delta^\circ_{n-j} \equiv 12 \delta_{n-j} \sqrt{\log(pn)}$.
\end{lemma}

\medskip
\subsection{\texorpdfstring{$\mu$}{μ} to \texorpdfstring{$\kappa$}{κ}}
\label{sec:kappa_by_mu_N}
Having obtained the induction from $\kappa$ to $\mu$ in the previous subsection, we now proceed to obtain an induction from $\mu$ to $\kappa$. This step has been  implicitly used in the proofs of high-dimensional CLTs for independent observations; see, e.g., \citetalias{kuchibhotla2020high}. However, as mentioned in \cref{sec:pf_sketch}, 
the dependence between $X_{[1,i]}$ and $X_{i+1}$ in $\kappa_{[1,i]|\{i+1\}}(\delta) \equiv \sup_{r\in\reals^p} {\mathrm{esssup}}_{X_{i+1}} \Pr[{X}_{[1,i]} \in A_{r,\delta} | X_{i+1}]$ makes the step non-trivial. We make a breakthrough using a similar approach described in \cref{sec:mu_by_kappa_3N}, where we used the Taylor expansion to eliminate the stochastic dependence. However, once again, the conditional anti-concentration probability involves a conditional expectation of an indicator function, which lacks smoothness. So we first apply a smoothing technique to the indicator function and leverage the Taylor expansion on the smoothed indicator, subsequently bounding the resulting remainder terms.

\medskip
\noindent{\bf Smoothing.} 
For the conditional anti-concentration probability, we use a standard smoothing, rather than the mixed smoothing we used in \cref{sec:mu_by_kappa_3N}.
For $r \in \reals^p$ and $\delta \in [0, \infty)$, let
\begin{equation*}
    \varphi_{r,\delta}^\vareps(x) \equiv \Exp[ \Ind\{x + \vareps Z \in A_{r,\delta}\} ],
\end{equation*}
where $Z$ is the $p$-dimensional standard Gaussian random vector.
For some $h > 0$,
\begin{equation*}
\begin{aligned}
    \varphi_{r,\delta}^\vareps(x) - \Ind\{x \in A_{r,\delta}\}
    & = \int (\Ind\{x + \vareps z \in A_{r,\delta}\} - \Ind\{x \in A_{r,\delta}\}) \phi(z) dz \\
    & = \int_{\norm{z}_\infty \leq 10\sqrt{\log(ph)}} (\Ind\{x + \vareps z \in A_{r,\delta}\} - \Ind\{x \in A_{r,\delta}\}) \phi(z) dz \\ 
    & \quad + \int_{\norm{z}_\infty > 10\sqrt{\log(ph)}} (\Ind\{x + \vareps z \in A_{r,\delta}\} - \Ind\{x \in A_{r,\delta}\}) \phi(z) dz \\
    & \geq - \Ind\{\norm{x - \partial A_{r,\delta}}_\infty \leq 10 \vareps \sqrt{\log(ph)}\}
    \Ind\{x \in A_{r,\delta}\} \\
    & \quad - \Pr[\norm{Z}_\infty > 10 \sqrt{\log(ph)}],
\end{aligned}
\end{equation*}
where $\partial A_{r,\delta}$ is the boundary of $A_{r,\delta}$, $Z$ is the $p$-dimensional standard Gaussian random vector and $\phi(z)$ is the density function of $Z$. Hence,
\begin{equation}\label{eq:varphi_lowerbound}
\begin{aligned}
    \varphi_{r,\delta}^\vareps(x)  
    & \geq \Ind\{x \in A_{r,\delta}\} 
    - \Ind\{\norm{x - \partial A_{r,\delta}}_\infty \leq 10 \vareps \sqrt{\log(ph)}\}
    \Ind\{x \in A_{r,\delta}\} \\
    & \quad - \Pr[\norm{Z}_\infty > 10 \sqrt{\log(ph)}] \\
    & = \Ind\{x \in A_{r,\delta-\vareps^\circ}\} - \frac{1}{h^4},
\end{aligned}
\end{equation}
where $\vareps^\circ = 10 \vareps \sqrt{\log(ph)}$. On the other hand,
\begin{equation}\label{eq:varphi_upperbound}
\begin{aligned}
    \varphi_{r,\delta}^\vareps(x) 
    & \leq \Ind\{x \in A_{r,\delta}\} + \Ind\{\norm{x - \partial A_{r,\delta}}_\infty \leq 10 \vareps \sqrt{\log(ph)}\}
    \Ind\{x \notin A_{r,\delta}\} \\
    & \quad + \Pr[\norm{Z}_\infty > 10 \sqrt{\log(ph)}]\\
    & = \Ind\{x \in A_{r,\delta+\vareps^\circ}\} + \frac{1}{h^4}.
\end{aligned}
\end{equation}
As a result, for any $h > 0$,
\begin{equation} \label{eq:anti_concentration_smoothing}
\begin{aligned}
    & \Pr[X_{[1,i]} \in A_{r,\delta} | X_{i+1}] \\
    & \leq \Exp[\varphi_{r,\delta+\vareps^\circ}^\vareps(X_{[1,i]}) 
    | X_{i+1}] 
    + \frac{1}{h^4}.
\end{aligned}
\end{equation}

\medskip
\noindent{\bf Taylor expansion.}
Applying the Taylor expansion to $\Exp[\varphi_{r,\delta+\vareps^\circ}^\vareps(X_{[1,i]}) 
    | X_{i+1}] $,
\begin{equation*}
\begin{aligned}
    & \Exp[\Exp[\varphi_{r,\delta+\vareps^\circ}^\vareps(X_{[1,i]}) 
    | X_{i+1}]] \\
    & \leq \Exp\left[ 
        \varphi_{r,\delta+\vareps^\circ}^\vareps(X_{[1,i]} - X_{i-1}) | X_{i+1}
    \right]
    + \Exp[ \mathfrak{R}_{X_{i-1}}^{(1)} | X_{i+1} ] \\
\end{aligned}
\end{equation*}
where 
\begin{equation*}
    \mathfrak{R}_{X_{i-1}}^{(1)}
    \equiv \int_0^1 \inner*{
        \nabla \varphi_{r,\delta+\vareps^\circ}^\vareps(X_{[1,i]} - t X_{i-1}),
        X_{i-1}
    } dt
\end{equation*}
First, using \cref{eq:varphi_upperbound}, 
\begin{equation*}
\begin{aligned}
    & \Exp[\varphi_{r,\delta+\vareps^\circ}^\vareps(X_{[1,i]} - X_{i-1}) | X_{i+1}] \\
    & \leq \Exp[\Ind\{X_{[1,i]} - X_{i-1} \in A_{r,\delta+2\vareps^\circ}\} 
    + \frac{1}{h^4} | X_{i+1}] \\
    & \leq \Exp[\Pr[X_{[1,i-2]} \in A_{r_1, \delta + 2\vareps^\circ} | \mathscr{X}_{[i, i+1]}] | X_{i+1}] + \frac{1}{h^4},
\end{aligned}
\end{equation*}
where $r_1 = r - X_{i}$ is a Borel measurable function with respect to $\mathscr{X}_{[i,i+1]}$. Because $X_{[1,i-2]} \indep \mathscr{X}_{[i, i+1]}$, 
\begin{equation} \label{eq:anti_concentration_H1}
\begin{aligned}
    & \Pr[X_{[1,i-2]} \in A_{r_1, \delta+2\vareps^\circ} | \mathscr{X}_{[i,i+1]}] \\
    & \leq \Pr[Y_{[1,i-2]} \in A_{r_1, \delta+2\vareps^\circ} ] 
    + 2\mu(X_{[1,i-2]}, Y_{[1,i-2]}) \\
    & \leq C\frac{\delta + 20\vareps\sqrt{\log(ph)}}{\sigma_{\min}} \sqrt{\frac{\log(ep)}{i-2}}
    + 2\mu_{[1,i-2]},
\end{aligned}
\end{equation}
almost surely due to the Gaussian anti-concentration inequality (\cref{thm:G_anti_conc}).
In sum, we get 
\begin{equation} \label{eq:anti_concentration_Taylor_N}
\begin{aligned}
    & \Pr[X_{[1,i]} \in A_{r,\delta} | X_{i+1}] \\
    & \leq C\frac{\delta + 20\vareps\sqrt{\log(ph)}}{\sigma_{\min}} \sqrt{\frac{\log(ep)}{i-2}}
    + 2\mu_{[1,i-2]} 
    + \frac{1}{h^4} \\
    & \quad + \Exp[ \mathfrak{R}_{X_{i-1}}^{(1)} | X_{i+1} ].
\end{aligned}
\end{equation}

\medskip
\noindent{\bf Bounding the remainder.}
Bounding the remainder term $\Exp[ \mathfrak{R}_{X_{i-1}}^{(1)} | X_{i+1} ]$ proceeds similarly to the proof of the remainder lemma (\cref{thm:remainder_lemma_3N}), resulting into an upper bound with 
a conditional anti-concentration probability bound $\kappa_{[1,i-2]|\{i-1\}}(\vareps^\circ)$. We relegate the bounding details to \cref{sec:pf_anti_concentration}. Putting the upperbound back to the previous inequality, we obtain the following lemma.

\begin{lemma} 
\label{thm:induction_lemma_AC_N}
    There exists a universal constant $C > 0$ such that for any $n$, $1$-dependent sequence $(X_i \in \reals^p: i \in [1,n])$ satisfying Assumption \eqref{assmp:min_var}, $i \in [1,n)$, $\delta > 0$ and $\vareps \geq \sigma_{\min}$,
    \begin{equation*} 
    \begin{aligned}
        & \kappa_{[1,i]|\{i+1\}}(\delta) \\
        & \leq C \left( 
            \frac{\sqrt{\log(ep)}}{\vareps} 
            ~ \nu_{1,i-1} 
            ~ \kappa_{[1,i-2]|\{i-1\}}(\vareps^\circ)
            + \mu_{[1,i-2]} 
        \right)\\
        & \quad + \min\left\{1, C\frac{\delta + 2\vareps^\circ}{\sigma_{\min}} \sqrt{\frac{\log(ep)}{i-2}} \right\} + \frac{C}{\sigma_{\min}} \nu_{1,i-1} \frac{\log(ep)}{\sqrt{i-2}},
    \end{aligned}
    \end{equation*}
    where $\vareps^\circ \equiv 20\vareps \sqrt{\log(p(i-2))}$. 
\end{lemma}

The above inequality resembles the relationship described in Eq. (3.16) of \citet{chen2004normal}. 
In their work on univariate observations, however, the $\kappa$ in the right hand side was $\kappa_{[1,I]}(\vareps^\circ)$ instead of with the reduced index set $[1, i-2]$. Hence, they could plug-in $\delta = \vareps^\circ$ and upper bound $\kappa_{[1,i]}(\vareps^\circ)$ for some suitable $\vareps$, which we referred to `a telescoping method' in the introductory part of \cref{sec:pf_sketch} 
This resulted into a non-inductive upper bound for $\kappa_{[1,i]}(\delta)$ (see Eqs. (3.17) and (3.18) therein). In our setting, the reduced index set $[1,i-2]$ makes the technique ineffective. Alternatively, we proceed to the dual induction levering on the reduced index set.

\medskip

\subsection{Dual Induction} \label{sec:dual_indcution_3N} 
In this part we use the dual induction to prove the following lemma.
\begin{lemma}
    There exist positive universal constants $\mathfrak{C}_{1,\kappa}$, $\mathfrak{C}_{2,\kappa}$, $\mathfrak{C}_{3,\kappa}$, $\mathfrak{C}_{4,\kappa}$, $\mathfrak{C}_{1,\mu}$ and $\mathfrak{C}_{2,\mu}$ such that for any $n$, $1$-dependent sequence $(X_i \in \reals^p: i \in [1,n])$ satisfying Assumptions \eqref{assmp:min_var}, \eqref{assmp:min_ev} and \eqref{assmp:var_ev} and $\delta \geq 0$,
    \begin{equation} \label{eq:hypothesis_AC_3N} \tag{HYP-AC-1}
    \begin{aligned}
        & \sqrt{i} \kappa_{[1,i]|\{i+1\}}(\delta)
        \leq 
        \tilde\kappa_{1,i} {L}_{3,\max}
        + \tilde\kappa_{2,i} \nu_{q,\max}^{1/(q-2)}
        + \tilde\kappa_{3,i} \nu_{1,\max}
        + \tilde\kappa_{4} \delta, \\
        & \forall i \in [1,n),
    \end{aligned}
    \end{equation}
    \begin{equation} \label{eq:hypothesis_BE_3N} \tag{HYP-BE-1}
        \sqrt{n} \mu_{[1,n]}
        \leq \tilde\mu_{1,n} {L}_{3,\max}
        + \tilde\mu_{2,n} \nu_{q,\max}^{1/(q-2)},
    \end{equation}
    where $\tilde\kappa_{1,i} = \mathfrak{C}_{1,\kappa} \tilde\mu_{1,i}$, 
    $\tilde\kappa_{2,i} = \mathfrak{C}_{2,\kappa} \tilde\mu_{2,i}$, 
    $\tilde\kappa_{3,i} = \mathfrak{C}_{3,\kappa} \frac{\log(ep)\sqrt{\log(pi)}}{\sigma_{\min}}$,
    $\tilde\kappa_{4} = \mathfrak{C}_{4,\kappa} \frac{\sqrt{\log(ep)}}{\sigma_{\min}}$,
    \begin{equation*}
    \begin{aligned}
        \tilde\mu_{1,n} 
        & = \mathfrak{C}_{1,\mu} \frac{(\log(ep))^{3/2}\sqrt{\log(pn)}}{\underline\sigma^2\sigma_{\min}} \log\left(e n\right), \\
        \tilde\mu_{2,n}
        & = \mathfrak{C}_{2,\mu} \frac{\log(ep)\sqrt{\log(pn)}}{\underline\sigma^{2/(q-2)} \sigma_{\min}} \log\left(e n\right).
    \end{aligned}
    \end{equation*}
\end{lemma}

If $\mathfrak{C}_{1,\kappa}$, $\mathfrak{C}_{2,\kappa}$, $\mathfrak{C}_{3,\kappa}$, $\mathfrak{C}_{4,\kappa}$, $\mathfrak{C}_{1,\mu}$, $\mathfrak{C}_{2,\mu} \geq 2$, then \eqref{eq:hypothesis_AC_3N} and \eqref{eq:hypothesis_BE_3N}, requiring $\kappa_{[1,i]|\{i+1\}}(\delta) \leq 1$ almost surely for all $i \in [1, n)$ and $\mu_{[1,n]} \leq 1$ only, trivially holds for $n \leq 16$. Now we consider the case of $n > 16$. Suppose that the induction hypotheses hold for all smaller $n$.

\medskip
We first derive \eqref{eq:hypothesis_AC_3N} for any $i \in [1,n)$. We note that $(X_1, \dots, X_{n-1})$ is a $1$-dependent sequence satisfying Assumptions \eqref{assmp:min_var}, \eqref{assmp:min_ev} and \eqref{assmp:var_ev} with the same $\sigma_{\min}$ and $\underline{\sigma}$ as the original data $(X_1, \dots, X_n)$. We formalize this fact into the following lemma:
\begin{lemma} \label{thm:assmp_induction_N}
    Suppose that $(X_1, \dots, X_n)$ is a $1$-dependent sequence satisfying Assumptions \eqref{assmp:min_var}, \eqref{assmp:min_ev} and \eqref{assmp:var_ev}. Then for any $i \in [1,n)$, $(X_{1}, \dots, X_{i})$ is $1$-ring dependent and satisfies Assumptions \eqref{assmp:min_var}, \eqref{assmp:min_ev} and \eqref{assmp:var_ev} with the same $\sigma_{\min}$ and $\underline{\sigma}$ as the original data.
\end{lemma}
As a result, \eqref{eq:hypothesis_AC_3N} applied to $(X_1, \dots, X_{n-1})$ verifies that the same conditional anti-concentration inequality holds for $X_{[1,i]}$ given $X_{i+1}$ for $ i \in [1,n-1)$. For $i = n-1$, by \cref{thm:induction_lemma_AC_N}, 
for any $\vareps \geq \sigma_{\min}$ and $\delta > 0$,
\begin{equation*} 
\begin{aligned}
    & \kappa_{[1,i]|\{i+1\}}(\delta) \\
    & \leq C
    \frac{\sqrt{\log(ep)}}{\vareps} \nu_{1,\max} 
    \min\{1, \kappa_{[1,i-2]|\{i-1\}}(\vareps^\circ)\} \\
    & \quad + C \mu_{[1,i-2]}
    + C\frac{\delta + 2\vareps^\circ}{\sigma_{\min}} \sqrt{\frac{\log(ep)}{i-2}}
    + C\frac{\nu_{1,\max}}{\sigma_{\min}} \frac{\log(ep)}{\sqrt{i-2}},
\end{aligned}
\end{equation*}
where $\vareps^\circ = 20\vareps \sqrt{\log(p(i-2))}$ and $C > 0$ is a universal constant. Following \eqref{eq:hypothesis_AC_3N}, we upperbound $\kappa_{[1,i-2]|\{i-1\}}(\vareps^\circ)$. Furthermore, since $(X_1, \dots, X_{i-2})$ is a $1$-dependent sequence satisfying Assumptions~\eqref{assmp:min_var}, \eqref{assmp:min_ev} and \eqref{assmp:var_ev} with the same $\sigma_{\min}$ and $\underline{\sigma}$ as the original data (see \cref{thm:assmp_induction_N}), \eqref{eq:hypothesis_BE_3N} holds for $\mu_{[1,i-2]}$, and we obtain that
\begin{equation*} 
\begin{aligned}
    & \kappa_{[1,i]|\{i+1\}}(\delta) \\
    & \leq \frac{C}{\sqrt{i-2}} 
    \frac{\sqrt{\log(ep)}}{\vareps} \nu_{1,\max}
    \left[ 
        \tilde\kappa_{1,i-2} {L}_{3,\max}
        + \tilde\kappa_{2,i-2} {\nu}_{q,\max}^{1/(q-2)}
        + \tilde\kappa_{3,i-2} \nu_{1,\max}
        + \tilde\kappa_{4} {\vareps^\circ}
    \right]
    \\
    & \quad + \frac{C}{\sqrt{i-2}} \tilde\mu_{1,i-2} {{L}_{3,\max}}
    + \frac{C}{\sqrt{i-2}} \tilde\mu_{2,i-2} {{\nu}_{q,\max}^{1/(q-2)}} \\
    & \quad + C\frac{\delta + 2\vareps^\circ}{\sigma_{\min}} \sqrt{\frac{\log(ep)}{i-2}}
    + C\frac{\nu_{1,\max}}{\sigma_{\min}} \frac{\log(ep)}{\sqrt{i-2}}.
\end{aligned}
\end{equation*}
As a result, the $\tilde\kappa$'s satisfy the recursive inequality
\begin{equation} \label{eq:recursion_kappa_3N}
\begin{aligned}
    & \sqrt{i} \kappa_{[1,i]|\{i+1\}}(\delta) \\
    & \leq \mathfrak{C}' \frac{\sqrt{\log(ep)}}{\vareps} \nu_{1,\max}
    \left[ 
        \tilde\kappa_{1,i-2} {L}_{3,\max}
        + \tilde\kappa_{2,i-2} \nu_{q,\max}^{1/(q-2)}
        + \tilde\kappa_{3,i-2} \nu_{1,\max}
        + \tilde\kappa_{4} \vareps^\circ
    \right]
    \\
    & \quad + \mathfrak{C}'  \left[ 
        \tilde\mu_{1,i-2} {{L}_{3,\max}}
        + \tilde\mu_{2,i-2} {{\nu}_{q,\max}^{1/(q-2)}} 
        + \frac{\delta + 2\vareps^\circ}{\sigma_{\min}} \sqrt{\log(ep)}
        + \frac{\nu_{1,\max}}{\sigma_{\min}} \log(ep)
    \right],
\end{aligned}
\end{equation}
for some universal constant $\mathfrak{C}'$, whose value does not change in this subsection.
Plugging in $\vareps = \max\{2 \mathfrak{C}', 1\} \sqrt{\log(ep)} {\nu}_{1,\max} \geq \sigma_{\min}$,
\begin{equation*}
\begin{aligned}
    & \sqrt{i} \kappa_{[1,i]|\{i+1\}}(\delta) \\
    & \leq \frac{1}{2}
    \left[ 
        \tilde\kappa_{1,i-2} {{L}_{3,\max}}
        + \tilde\kappa_{2,i-2} {{\nu}_{q,\max}^{1/(q-2)}}
        + \tilde\kappa_{3,i-2} {\nu_{1,\max}}
    \right] 
    + 20 \mathfrak{C}' \tilde\kappa_{4} \sqrt{\log(ep)\log(pi)} \nu_{1,\max} \\
    & \quad  
    + \mathfrak{C}' \left[ 
        \tilde\mu_{1,i-2} {L}_{3,\max}
        + \tilde\mu_{2,i-2} \nu_{q,\max}^{1/(q-2)}  
        + \frac{\log(ep)(1+40 \sqrt{\log(p i)})}{\sigma_{\min}} {\nu}_{1,\max}
        + \frac{\sqrt{\log(ep)}}{\sigma_{\min}} \delta
    \right]
    \\
    & \leq \tilde\kappa_{1,i} {L}_{3,\max}
    + \tilde\kappa_{2,i} \nu_{q,\max}^{1/(q-2)}
    + \tilde\kappa_{3,i} \nu_{1,\max}
    + \tilde\kappa_{4} \delta,
\end{aligned}
\end{equation*}
where $\tilde\kappa_{1,i} = \mathfrak{C}_{1,\kappa} \tilde\mu_{1,i}$, 
$\tilde\kappa_{2,i} = \mathfrak{C}_{2,\kappa} \tilde\mu_{2,i}$, 
$\tilde\kappa_{3,i} = \mathfrak{C}_{3,\kappa} \frac{\log(ep)\sqrt{\log(pi)}}{\sigma_{\min}}$, and
$\tilde\kappa_{4} = \mathfrak{C}_{4,\kappa} \frac{\sqrt{\log(ep)}}{\sigma_{\min}}$, provided by $\mathfrak{C}_{1,\kappa} = \mathfrak{C}_{2,\kappa} = \max\{2 \mathfrak{C}', 2\}$, $\mathfrak{C}_{3,\kappa} = \max\{82 \mathfrak{C}' + 20 \mathfrak{C}' \mathfrak{C}_{4,\kappa}, 2\}$ and $\mathfrak{C}_{4,\kappa} = \max\{ \mathfrak{C}', 2\}$. 
The inequality $\max\{2 \mathfrak{C}', 1\} \sqrt{\log(ep)} {\nu}_{1,\max} \geq \sigma_{\min}$ holds because 
\begin{equation} \label{eq:nu_1_vs_sigma_min}
\begin{aligned}
    \nu_{1,\max} 
    & \geq \max_{i \in [n]} \Exp[\norm{Y_i}_\infty] 
    \geq \max_{i \in [n]} \max_{k \in [p]} \Exp[\abs{Y_{ik}}] \\
    & \overset{(*)}{\geq} \max_{i \in [n]} \max_{k \in [p]} \frac{2}{\sqrt{\pi}} \sqrt{\Exp[\abs{Y_{ik}}^2]} 
    \geq \frac{2}{\sqrt{\pi}} \sigma_{\min}, 
\end{aligned}
\end{equation}
where the inequality $(*)$ follows that $Y_{ik}$ is (marginally) a Gaussian random variable for all $i$ and $k$.
This proves \eqref{eq:hypothesis_AC_3N} at $n$.

\medskip
Now we prove \eqref{eq:hypothesis_BE_3N} at $n$. We first upper bound the last term in \cref{thm:induction_lemma_BE_3N}. For $\delta \geq \sigma_{\min}$, we have that
\begin{equation*}
\begin{aligned}
    & C \sum_{j \geq n/2} \frac{(\log(ep))^{3/2}}{\delta_{n-j}^3} \left[ 
        {L}_{3,\max} + \phi^{\min\{1, q-3\}} \nu_{q,\max} \frac{ {(\log(ep))}^{\max\{0, q-4\}/2}}{(\delta_{n-j}/\alpha)^{\max\{0, q-4\}}}
    \right]
    \kappa_{[1,j-4]|\{j-3\}}(\delta_{n-j}^\circ).
\end{aligned}
\end{equation*}
Applying \eqref{eq:hypothesis_AC_3N} to $\kappa_{[1,j-3)|\{j-3\}}(\delta_{n-j}^\circ)$ in $\mathfrak{T}_{2,1}$,
\begin{equation*}
\begin{aligned}
    & C \sum_{j \geq n/2} \frac{(\log(ep))^{3/2}}{\delta_{n-j}^3} 
    \left[ 
        {L}_{3,\max} + \phi^{\min\{1, q-3\}} \nu_{q,\max} \frac{ {(\log(ep))}^{\max\{0, q-4\}/2}}{(\delta_{n-j}/\alpha)^{\max\{0, q-4\}}}
    \right] 
    \kappa_{[1,j-4]|\{j-3\}}(\delta_{n-j}^\circ) \\
    & \leq C \sum_{j \geq n/2} 
    \frac{(\log(ep))^{3/2}}{\delta_{n-j}^3 \sqrt{j-4}} 
    \left[ 
        {L}_{3,\max} + \phi^{\min\{1, q-3\}} \nu_{q,\max} \frac{ {(\log(ep))}^{\max\{0, q-4\}/2}}{(\delta_{n-j}/\alpha)^{\max\{0, q-4\}}}
    \right] \\
    & \quad \times \left[ \begin{aligned}
        & \tilde\kappa_{1,j-4} {L}_{3,\max}
        + \tilde\kappa_{2,j-4} \nu_{q,\max}^{1/(q-2)}
        + \tilde\kappa_{3,j-4} \nu_{2,\max}^{1/2}
        + \tilde\kappa_{4} \delta_{n-j}^\circ \\
    \end{aligned} \right] \\
\end{aligned}
\end{equation*}
Recall that $\tilde\kappa_{1,j-4} = \mathfrak{C}_{1,\kappa} \tilde\mu_{1,j-4}$, 
$\tilde\kappa_{2,j-4} = \mathfrak{C}_{2,\kappa} \tilde\mu_{2,j-4}$, 
$\tilde\kappa_{3,j-4} = \mathfrak{C}_{3,\kappa} \frac{\log(ep)\sqrt{\log(p(j-4))}}{\sigma_{\min}}$ and 
$\tilde\kappa_{4} = \mathfrak{C}_{4,\kappa} \frac{\sqrt{\log(ep)}}{\sigma_{\min}}$. 
Thus, using Eq. (15) in \citetalias{kuchibhotla2020high}, which gives that
\begin{equation} \label{eq:sum_delta_nolog}
\begin{aligned}
    \sum_{j=\ceil{n/2}}^{n} \frac{1}{\delta_{n-j}^2} 
    & \leq \frac{C}{\underline\sigma^2} \log\left(1 + \frac{\sqrt{n}\underline{\sigma}}{\delta}\right), \\
    \sum_{j=\ceil{n/2}}^{n} \frac{1}{\delta_{n-j}^3} 
    & \leq \frac{C}{\delta\underline\sigma^2},
\end{aligned}
\end{equation}
we obtain that
\begin{equation*}
\begin{aligned}
    & C \sum_{j \geq n/2} \frac{(\log(ep))^{3/2}}{\delta_{n-j}^3}
    \left[ 
        {L}_{3,\max} + \phi^{\min\{1, q-3\}} \nu_{q,\max} \frac{ {(\log(ep))}^{\max\{0, q-4\}/2}}{(\delta_{n-j}/\alpha)^{\max\{0, q-4\}}}
    \right]
    \kappa_{[1,j-4]|\{j-3\}}(\delta_{n-j}^\circ) \\ 
    & \leq \frac{C}{\sqrt{n}} \frac{(\log(ep))^{3/2}}{\underline{\sigma}^2} 
    \left[ 
        {L}_{3,\max} + \phi^{\min\{1, q-3\}} \nu_{q,\max} \frac{ {(\log(ep))}^{\max\{0, q-4\}/2}}{(\delta/\alpha)^{\max\{0, q-4\}}}
    \right] \\
    & \quad \times \left[ \begin{aligned}
        & (\tilde\mu_{1,n-4} {L}_{3,\max} 
        +\tilde\mu_{2,n-4} \nu_{q,\max}^{1/(q-2)})
        \frac{1}{\delta} 
        +\frac{\log(ep)\sqrt{\log(pn)}}{\sigma_{\min}} \frac{\nu_{1,\max}}{\delta} \\
        & + \frac{\sqrt{\log(ep)\log(pn)}}{\sigma_{\min}} 
        \log\left(1 + \frac{\sqrt{n}\underline{\sigma}}{\delta}\right)
    \end{aligned} \right].
\end{aligned}
\end{equation*}
Thus, as long as $\delta \geq \sqrt{\log(ep)} \nu_{1,\max}$ and $\phi > 0$, we arrive at the recursive inequality on $\tilde\mu$'s 
\begin{equation} \label{eq:recursion_mu}
\begin{aligned}
    \sqrt{n} \mu_{[1,n]} 
    & \leq \mathfrak{C}'' \frac{(\log(ep))^{3/2}}{\underline{\sigma}^2 \delta} 
    \left[ 
        {L}_{3,\max} + \phi^{\min\{1, q-3\}} \nu_{q,\max} \frac{ {(\log(ep))}^{\max\{0, q-4\}/2}}{(\delta/\alpha)^{\max\{0, q-4\}}}
    \right] \\
    & \quad \times \left[ \begin{aligned}
        \tilde\mu_{1,n-4} {L}_{3,\max} 
        + \tilde\mu_{2,n-4} \nu_{q,\max}^{1/(q-2)}
    \end{aligned} \right] 
    + \mathfrak{C}'' \left[ \frac{\delta \log(ep)}{\sigma_{\min}}  
    + \frac{\sqrt{\log(ep)}}{\phi \sigma_{\min}} \right]\\
    & \quad + \mathfrak{C}'' \frac{(\log(ep))^{5/2}}{\underline{\sigma}^2\sigma_{\min}}
    \log\left(1 + \frac{\sqrt{n}\underline{\sigma}}{\delta}\right)
    \left[ 
        {L}_{3,\max} + \phi^{\min\{1, q-3\}} \nu_{q,\max} \frac{ {(\log(ep))}^{\max\{0, q-4\}/2}}{(\delta/\alpha)^{\max\{0, q-4\}}}
    \right] \\
    & \quad + \mathfrak{C}'' \frac{(\log(ep))^{3/2}}{\underline{\sigma}^2} \log\left(1 + \frac{\sqrt{n}\underline{\sigma}}{\delta}\right) 
    \left[ 
        {L}_{3,\max} + \phi^{\min\{1, q-3\}} \nu_{q,\max} \frac{ {(\log(ep))}^{\max\{0, q-4\}/2}}{(\delta/\alpha)^{\max\{0, q-4\}}}
    \right] \\
    & \quad \times \frac{\sqrt{\log(ep)\log(pn)}}{\sigma_{\min}},
\end{aligned}
\end{equation}
where $\mathfrak{C}''$ is a universal constant whose value does not change in this subsection.
Taking $\delta = \frac{\max\{4 \mathfrak{C}'', \alpha\}}{\sqrt{\log(ep)}} \left( 
    \frac{{L}_{3,\max}}{\underline{\sigma}^2} (\log(ep))^2
    + \left(\frac{\nu_{q,\max}}{\underline{\sigma}^2} (\log(ep))^{\max\{2, q-2\}} \right)^{\frac{1}{q-2}} 
\right) \geq \sqrt{\log(ep)} \nu_{1,\max} \overset{ \text{\cref{eq:nu_1_vs_sigma_min}}}{\geq} \sigma_{\min} $ and $\phi = \frac{1}{\delta \sqrt{\log(ep)}}$ we arrive at the bound
\begin{equation*}
\begin{aligned}
    \sqrt{n} \mu_{[1,n]}
    & \leq \frac{1}{2} \max_{j<n}\tilde\mu_{1,j} {L}_{3,\max}
    + \frac{1}{2} \max_{j<n}\tilde\mu_{2,j} \nu_{q,\max}^{1/(q-2)} \\
    & \quad + \mathfrak{C}^{(3)} \left(
        {L}_{3,\max} \frac{(\log(ep))^2}{\underline{\sigma}^2}
        + \nu_{q,\max}^{1/(q-2)} \frac{(\log(ep))^{\max\{2/(q-2),1\}}}{\underline{\sigma}^{2/(q-2)}}
    \right)
    \frac{\sqrt{\log(p n)}}{\sigma_{\min}} \log\left(e n\right) \\
\end{aligned}
\end{equation*}
for another universal constant $\mathfrak{C}^{(3)}$, whose value only depends on $\mathfrak{C}''$.
Taking $\mathfrak{C}_1 = \mathfrak{C}_2 = \max\{2 \mathfrak{C}^{(3)}, 2\}$ the quantities
\begin{equation*}
\begin{aligned}
    \tilde\mu_{1,n} 
    & = \mathfrak{C}_1 \frac{(\log(ep))^{2}\sqrt{\log(pn)}}{\underline\sigma^2\sigma_{\min}} \log\left(e n\right), \\
    \tilde\mu_{2,n}
    & = \mathfrak{C}_2 \frac{(\log(ep))^{\max\{2/(q-2),1\}}\sqrt{\log(pn)}}{\underline\sigma^{2/(q-2)} \sigma_{\min}} \log\left(e n\right)
\end{aligned}
\end{equation*}
satisfy
\begin{equation*}
\begin{aligned}
    \sqrt{n} \mu_{[1,n]}
    \leq \tilde\mu_{1,n} {L}_{3,\max}
    + \tilde\mu_{2,n} \nu_{q,\max}^{1/(q-2)},
\end{aligned}
\end{equation*}
which proves \eqref{eq:hypothesis_BE_3N} for a given $n$.
Finally, a mathematical induction over $n$ proves our theorem.

\subsection{\texorpdfstring{The case $q \geq 4$}{4 ≤ q}} \label{sec:pf_sketch_4N}

When the fourth moment exists, we  obtain a better sample complexity by further decomposing the third order remainder $\mathfrak{R}_{W_j}^{(3,1)}$. Based on the Taylor expansions up to order $4$, 
\begin{equation} \label{eq:decompose_third_remainder_N}
\begin{aligned}
    & \sum_{j=1}^{n} \Exp \left[  
        \mathfrak{R}_{X_j}^{(3,1)} 
        - \mathfrak{R}_{Y_j}^{(3,1)} 
    \right] \\
    & = \frac{1}{6} \sum_{j=1}^{n}  
        \inner*{ \Exp\left[ \nabla^3 \rho_{r,\phi}^\delta(X_{[1,j-1)} + Y_{(j+1,n]}) \right], \Exp[ X_j^{\otimes 3} ]} \\
    & \quad + \frac{1}{2} \sum_{j=1}^{n-1} \inner*{ 
        \Exp\left[ \nabla^3 \rho_{r,\phi}^\delta(X_{[1,j-1)} + Y_{(j+2,n]}) \right], 
        \Exp[ X_{j} \otimes X_{j+1} \otimes (X_j + X_{j+1})] 
    } \\
    & \quad + \sum_{j=1}^{n-2} \inner*{ 
        \Exp\left[ \nabla^3 \rho_{r,\phi}^\delta(X_{[1,j-1)} + Y_{(j+3,n]}) \right], 
        \Exp\left[ X_{j} \otimes X_{j+1} \otimes X_{j+2} \right]} \\
    & \quad + \sum_{j=1}^{n} \Exp \left[ \mathfrak{R}_{X_j}^{(4,1)} 
        - \mathfrak{R}_{Y_j}^{(4,1)} \right],
\end{aligned}
\end{equation}
where $\mathfrak{R}_{X_j}^{(4,1)}$ and $\mathfrak{R}_{Y_j}^{(4,1)}$ are remainder terms of the Taylor expansions specified in \cref{sec:first_lindeberg_swapping}.
We re-apply the Lindeberg swapping. For brevity, we only look at the first term, but similar arguments apply to the other third order moment terms. 
We observe that
\begin{equation} \label{eq:decompose_third_moment_N}
\begin{aligned}
    & \inner*{ 
        \Exp[ \nabla^3 \rho_{r,\phi}^\delta(X_{[1,j-1)} + Y_{(j+1,n]}) ], 
        \Exp[ X_j^{\otimes 3} ]
    } \\
    & = \inner*{ 
        \Exp\left[ \nabla^3 \rho_{r,\phi}^\delta(Y_{[1,j-1)} + Y_{(j+1,n]}) \right], 
        \Exp[ X_j^{\otimes 3} ]
    } \\
    & \quad + \Big\langle 
        \Exp\left[ \nabla^3 \rho_{r,\phi}^\delta(X_{[1,j-1)} + Y_{(j+1,n]}) 
        - \nabla^3 \rho_{r,\phi}^\delta(Y_{[1,j-1)} + Y_{(j+1,n]}) \right], 
        \Exp[ X_j^{\otimes 3} ]\Big\rangle.
\end{aligned}
\end{equation}
For the first term, 
because $Y_{[1,j-1)} + Y_{(j+1,n]}$ is Gaussian, and  the smallest eigenvalue of its covariance matrix is at least $(n-5) \underline{\sigma}^2$ by \cref{assmp:min_ev}, 
\begin{equation*}
    \Exp\left[ 
        \nabla^3 \rho_{r,\phi}^\delta(Y_{[1,j-1)} + Y_{(j+1,n]}) 
    \right]
    = \Exp\left[ 
        \nabla^3 \rho_{r,\phi}^{\delta_{i-5}}(Y^\circ_{[1,j-1)\cup(j+1,n]}) 
    \right],
\end{equation*}
where $\delta_{n-5}^2 \equiv \delta^2 + \underline{\sigma}^2 \max\{n-5,0\}$ and $Y^\circ_{[1,j-1)\cup(j+1,n]}$ is the centered Gaussian random vector with covariance $\Var[Y_{[1,j-1)} + Y_{(j+1,n]}] - (n-5) \underline{\sigma}^2 I_p$. 
Hence,
\begin{equation} \label{eq:third_moment_Gaussian_bound}
\begin{aligned}
    & \abs*{ \inner*{ 
        \Exp\left[ \nabla^3 \rho_{r,\phi}^\delta(Y_{[1,j-1)} + Y_{(j+1,n]}) \right], 
        \Exp[X_j^{\otimes 3}] } } \\
    & = \abs*{ \inner*{ 
        \Exp\left[ 
            \nabla^3 \rho_{r,\phi}^{\delta_{n-5}}(Y^\circ_{[1,j-1)\cup(j+1,n]}) 
        \right], 
        \Exp[X_j^{\otimes 3}] } } \\
    & \leq L_{3,j} \sup_{r, w \in \reals^p} \sum_{k_1,k_2,k_3 = 1}^p \abs{\nabla^{(k_1,k_2,k_3)} \rho_{r,\phi}^{\delta_{n-5}}(w)} \\
    & \overset{\mathrm{(i)}}{\leq} C L_{3,j} \frac{(\log(ep))^{3/2}}{\delta_{n-5}^3}
    \leq \frac{C}{n^{3/2}} L_{3,j} \frac{(\log(ep))^{3/2}}{\underline{\sigma}^3} 
    \overset{\mathrm{(ii)}}{\leq} \frac{C}{n^{3/2}} L_{3,j} \frac{(\log(ep))^2}{\underline{\sigma}^2 \sigma_{\min}},
\end{aligned}
\end{equation} 
where $\mathrm{(i)}$ follows Lemma 6.2 of \citetalias{chernozhukov2020nearly} and $\mathrm{(ii)}$ follows from Assumption~\eqref{assmp:var_ev}.
For the second term, we re-apply Lindeberg swapping and obtain, for $j \in [3, n-1]$, that
\begin{equation} \label{eq:third_lindeberg_swapping_N}
\begin{aligned}
    & \Big\langle 
        \Exp\left[ \nabla^3 \rho_{r,\phi}^\delta(X_{[1,j-1)} + Y_{(j+1,n]}) \right] 
    - \Exp\left[ \nabla^3  \rho_{r,\phi}^\delta(Y_{[1,j-1)} + Y_{(j+1,n]}) \right], \Exp[ X_j^{\otimes 3} ]\Big\rangle. \\
    & = \sum_{k=1}^{j-2} \Big\langle
        \Exp\left[ \nabla^3 \rho_{r,\phi}^\delta(X_{[1,k)} + X_k + Y_{(k,j-1) \cup (j+1,n]}) \right] \\ 
    & \omit{\hfill $- \Exp\left[ \nabla^3   \rho_{r,\phi}^\delta(X_{[1,k)} + Y_k + Y_{(k,j-1) \cup (j+1,n]}) \right], 
        \Exp[ X_j^{\otimes 3} ]
    \Big\rangle.$} \\
\end{aligned}
\end{equation}
By a Taylor series expansion around $X_{[1,k)} + Y_{(k,j-1) \cup (j+1,n]}$, this difference can be rewritten as 
\begin{equation} \label{eq:third_lindeberg_result_N}
\begin{aligned}
    & \sum_{j=1}^{n} \Big\langle 
        \Exp\left[ \nabla^3 \rho_{r,\phi}^\delta(X_{[1,j-1)} + Y_{(j+1,n]}) 
        - \nabla^3 \rho_{r,\phi}^\delta(Y_{[1,j-1)} + Y_{(j+1,n]}) \right], 
        \Exp[ X_j^{\otimes 3} ]
    \Big\rangle \\
    & = \sum_{j=3}^{n} \sum_{k=1}^{j-2} \Exp \left[ 
        \mathfrak{R}_{X_j,X_k}^{(6,1,1)}
        - \mathfrak{R}_{X_j,Y_k}^{(6,1,1)}
    \right], \\
\end{aligned}
\end{equation}
where $\mathfrak{R}_{X_j,W_k}^{(6,1,1)}$ is a sixth order remainder term. The details of this expansion and the specification of the remainder terms 
are given in \cref{sec:third_lindeberg_swapping}. 
Hence, 
\begin{equation*}
\begin{aligned}
    & \sum_{j=1}^{n}  
    \inner*{ \Exp\left[ \nabla^3 \rho_{r,\phi}^\delta(X_{[1,j-1)} + Y_{(j+1,n]}) \right], \Exp[ X_j^{\otimes 3} ]} \\
    & \leq \frac{C}{\sqrt{n}} L_{3,\max} \frac{(\log(ep))^2}{\underline\sigma^2 \sigma_{\min}}
    + \sum_{j=3}^{n} \sum_{k=1}^{j-2} \Exp \left[ 
        \mathfrak{R}_{X_j,X_k}^{(6,1,1)}
        - \mathfrak{R}_{X_j,Y_k}^{(6,1,1)}
    \right].
\end{aligned}
\end{equation*}
Similarly,
\begin{equation*}
\begin{aligned}
    & \sum_{j=1}^{n-1} \inner*{ 
        \Exp\left[ \nabla^3 \rho_{r,\phi}^\delta(X_{[1,j-1)} + Y_{(j+2,n]}) \right], 
        \Exp[ X_{j} \otimes X_{j+1} \otimes (X_j + X_{j+1})] 
    } \\
    & \leq \frac{C}{\sqrt{n}} L_{3,\max} \frac{(\log(ep))^2}{\underline\sigma^2 \sigma_{\min}}
    + \sum_{j=3}^{n-1} \sum_{k=1}^{j-2} \Exp \left[ 
        \mathfrak{R}_{X_j,X_k}^{(6,1,2)}
        - \mathfrak{R}_{X_j,Y_k}^{(6,1,2)}
    \right], \textand
\end{aligned}
\end{equation*}
\begin{equation*}
\begin{aligned}
    & \sum_{j=1}^{n-2} \inner*{ 
        \Exp\left[ \nabla^3 \rho_{r,\phi}^\delta(X_{[1,j-1)} + Y_{(j+3,n]}) \right], 
        \Exp\left[ X_{j} \otimes X_{j+1} \otimes X_{j+2} \right]
    } \\
    & \leq \frac{C}{\sqrt{n}} L_{3,\max} \frac{(\log(ep))^2}{\underline\sigma^2 \sigma_{\min}}
    + \sum_{j=3}^{n-2} \sum_{k=1}^{j-2} \Exp \left[ 
        \mathfrak{R}_{X_j,X_k}^{(6,1,3)}
        - \mathfrak{R}_{X_j,Y_k}^{(6,1,3)}
    \right],
\end{aligned}
\end{equation*}
where $\mathfrak{R}_{X_j,W_k}^{(6,1,2)}$ and $\mathfrak{R}_{X_j,W_k}^{(6,1,3)}$ are analogous sixth-order remainder terms.
Putting everything together, we get that
\begin{equation} \label{eq:decompose_third_result_N}
\begin{aligned}
    \sum_{j=1}^{n} \Exp \left[  
        \mathfrak{R}_{X_j}^{(3,1)} 
        - \mathfrak{R}_{Y_j}^{(3,1)} 
    \right]
    & \leq \frac{C}{\sqrt{n}} L_{3,\max} \frac{(\log(ep))^2}{\underline\sigma^2 \sigma_{\min}} \\
    & \quad + \sum_{j=1}^{n} \abs*{ \Exp \left[ \mathfrak{R}_{X_j}^{(4,1)} 
        - \mathfrak{R}_{Y_j}^{(4,1)} \right] }
    + \sum_{j=3}^{n} \sum_{k=1}^{j-2} \abs*{ \Exp \left[ 
        \mathfrak{R}_{X_j,X_k}^{(6,1)}
        - \mathfrak{R}_{X_j,Y_k}^{(6,1)}
    \right] },
\end{aligned}
\end{equation}
where $\mathfrak{R}_{X_j,W_k}^{(6,1)} = \frac{1}{6} \mathfrak{R}_{X_j,W_k}^{(6,1,1)} + \frac{1}{2} \mathfrak{R}_{X_j,W_k}^{(6,1,2)} + \mathfrak{R}_{X_j,W_k}^{(6,1,3)}$. Thus,
\begin{equation*}
\begin{aligned}
    \mu_{[1,n]}
    & \leq \frac{C}{\sqrt{n}} \frac{\delta \log(ep)}{\sigma_{\min}}  
    + \frac{C}{\sqrt{n}} \frac{\sqrt{\log(ep)}}{\phi \sigma_{\min}} 
    + \frac{C}{\sqrt{n}} L_{3,\max} \frac{(\log(ep))^2}{\underline\sigma^2 \sigma_{\min}} \\
    & \quad + \sum_{j=1}^{n} \abs*{ \Exp \left[ \mathfrak{R}_{X_j}^{(4,1)} 
        - \mathfrak{R}_{Y_j}^{(4,1)} \right] }
    + \sum_{j=3}^{n} \sum_{k=1}^{j-2} \abs*{ \Exp \left[ 
        \mathfrak{R}_{X_j,X_k}^{(6,1)}
        - \mathfrak{R}_{X_j,Y_k}^{(6,1)}
    \right] }.
\end{aligned}
\end{equation*}
A remainder lemma similar to \cref{thm:remainder_lemma_3N} can be derived for $\mathfrak{R}_{W_j}^{(4,1)}$ and $\mathfrak{R}_{X_j,W_k}^{(6,1)}$ -- see \cref{thm:remainder_lemma_4N}). Summing up the upper bounds iteratively over $k$ and $j$ results in a finite fourth moment version of \cref{thm:induction_lemma_BE_3N} (see \cref{thm:induction_lemma_BE_4N}). Finally, a dual induction argument delivers the desired Berry--Esseen bound with finite fourth moments. The details are provided in \cref{sec:pf_1_dep_4N}.


\begin{remark}
As we discussed in \cref{sec:comparison_indep}, the bottleneck of our Berry--Esseen bound is often the first term with the third moment (i.e., $L_3$). A significant improvement by the iterated Lindeberg swapping is reducing the term's order of $\log(ep)$ from $\sqrt{\log(pn)} \log^2(ep)$ to $\sqrt{\log(pn)} \log^{3/2}(ep)$. One may repeat the Lindeberg swapping to further improve the order. For example, the first term of $\mathfrak{R}^{(6,1)}_{X_j, W_k}$ is 
\begin{equation*}
    \frac{1}{2} \int_0^1 (1-t)^2 \inner*{
        \nabla^6 \rho_{r,\phi}^\delta(X_{[1,k)} + t W_k + Y_{(k,j-1) \cup (j+1,n]}),
        X_j^{\otimes 3} \otimes W_k^{\otimes 3}
    } ~dt.
\end{equation*}
Based on the Taylor expansion up to order $7$, 
\begin{equation*}
\begin{aligned}
    & \frac{1}{2} \int_0^1 (1-t)^2 \inner*{
        \nabla^6 \rho_{r,\phi}^\delta(X_{[1,k)} + t W_k + Y_{(k,j-1) \cup (j+1,n]}),
        X_j^{\otimes 3} \otimes W_k^{\otimes 3}
    } ~dt \\
    & = \frac{1}{6} \inner*{\nabla^6 \rho_{r,\phi}^\delta(X_{[1,k)} + Y_{(k,j-1) \cup (j+1,n]}),
        X_j^{\otimes 3} \otimes W_k^{\otimes 3}
    } \\
    & \quad + \frac{1}{6} \int_0^1 (1-t)^3 \inner*{
        \nabla^7 \rho_{r,\phi}^\delta(X_{[1,k)} + t W_k + Y_{(k,j-1) \cup (j+1,n]}),
        X_j^{\otimes 3} \otimes W_k^{\otimes 4}
    } ~dt \\
    & = \frac{1}{6} \inner*{\nabla^6 \rho_{r,\phi}^\delta(X_{[1,k-1)} + Y_{(k+1,j-1) \cup (j+1,n]}),
        X_j^{\otimes 3} \otimes W_k^{\otimes 3}
    } \\
    & \quad + \frac{1}{6} \int_0^1 \Big\langle
        \nabla^7 \rho_{r,\phi}^\delta(X_{[1,k-1)} + t (X_{k-1} + Y_{k+1}) + Y_{(k+1,j-1) \cup (j+1,n]}), \\
    & \hspace{3in} X_j^{\otimes 3} \otimes W_k^{\otimes 3} \otimes (X_{k-1} + Y_{k+1})
    \Big\rangle ~dt \\
    & \quad + \frac{1}{6} \int_0^1 (1-t)^3 \inner*{
        \nabla^7 \rho_{r,\phi}^\delta(X_{[1,k)} + t W_k + Y_{(k,j-1) \cup (j+1,n]}),
        X_j^{\otimes 3} \otimes W_k^{\otimes 4}
    } ~dt. \\
\end{aligned}
\end{equation*}
\end{remark}
Like \cref{eq:decompose_third_moment_N}, one may decompose the first term and re-apply the Lindeberg swapping:
\begin{equation*}
\begin{aligned}
    & \frac{1}{6} \inner*{\nabla^6 \rho_{r,\phi}^\delta(X_{[1,k-1)} + Y_{(k+1,j-1) \cup (j+1,n]}),
        X_j^{\otimes 3} \otimes W_k^{\otimes 3}
    } \\
    & = \frac{1}{6} \inner*{\nabla^6 \rho_{r,\phi}^\delta(Y_{[1,k-1)} + Y_{(k+1,j-1) \cup (j+1,n]}),
        X_j^{\otimes 3} \otimes W_k^{\otimes 3}
    } \\
    & \quad + \sum_{l=1}^{k-2} \frac{1}{6} \Big\langle
        \nabla^6 \rho_{r,\phi}^\delta(X_{[1,l)} +  X_l + Y_{(l,k-1)} + Y_{(k+1,j-1) \cup (j+1,n]}) \\
    & \hspace{1in} - \nabla^6 \rho_{r,\phi}^\delta(X_{[1,l)} +  Y_l + Y_{(l,k-1)} + Y_{(k+1,j-1) \cup (j+1,n]}), 
    X_j^{\otimes 3} \otimes W_k^{\otimes 3}
    \Big\rangle.
\end{aligned}
\end{equation*}
To make a successful improvement, we recommend using piece-wise quadratic $f_{r,\phi}$, instead of the piece-wise linear one defined in \cref{sec:mu_by_kappa_3N}. This choice of $f_{r,\phi}$ allows improved remainder lemmas for the sixth, seventh and ninth-order remainder terms. At the end, infinitely repeating the Lindeberg swapping may improve the order asymptotically to $\sqrt{\log(pn)} \log(ep)$. 

In this paper, we do not pursue further refining \cref{thm:1_dep_berry_esseen_4P}. Because $Y_j$ is Gaussian, the dimension complexity cannot be improved from $\log^4(ep)$ due to the last term with the $q$-th moment (i.e., $\nu_q$). Hence, further Lindeberg swappings do not help match the dimension complexity of \citetalias{chernozhukov2020nearly} under bounded $X_j$.

\subsection{Removing Assumption \texorpdfstring{\eqref{assmp:var_ev}}{Assumption (VAR-EV)}} \label{sec:pf_without_var_ev}

Assumption \eqref{assmp:var_ev} was invoked in the proof of the simplified result in in two key steps: first, in obtaining the upper bound of the "Partitioning the sum" step in \cref{sec:mu_by_kappa_3N} for $j < n/2$,
and, secondly, in obtaining the upper bound in \cref{eq:third_moment_Gaussian_bound}.

Without \eqref{assmp:var_ev},
the rightmost upper bound in \cref{eq:sum_delta_3} is $\frac{C}{\underline{\sigma}^3\sqrt{n}}$ rather than $\frac{C \log(ep)}{\underline{\sigma}^2\sigma_{\min}\sqrt{n}}$. With this new upper bound, summing over $j$, we obtain
\begin{equation*}
\begin{aligned}
    & \sum_{j=1}^{i} \abs*{ \Exp \left[ \mathfrak{R}_{X_j}^{(3,1)} 
        - \mathfrak{R}_{Y_j}^{(3,1)} \right] } \\
    & \leq \frac{C}{\sqrt{n}} \frac{(\log(ep))^{5/2}}{\underline{\sigma}^2 \min\{\sigma_{\min}, \underline{\sigma} \sqrt{\log(ep)}\}}
    \log\left(1 + \frac{\sqrt{n}\underline{\sigma}}{\delta}\right)
    \\
    & \quad \times \left[ 
        {L}_{3,\max} + \phi^{\min\{1, q-3\}} \nu_{q,\max} \frac{ {(\log(ep))}^{\max\{0, q-4\}/2}}{(\delta/\alpha)^{\max\{0, q-4\}}}
    \right]\\
    & \quad + C \sum_{j \geq {n/2}} \frac{(\log(ep))^{3/2}}{\delta_{n-j}^3} \left[ 
        {L}_{3,\max} + \phi^{\min\{1, q-3\}} \nu_{q,\max} \frac{ {(\log(ep))}^{\max\{0, q-4\}/2}}{\delta_{n-j}^{\max\{0, q-4\}}}
    \right] \kappa_{[1,j-4]|\{j-3\}}(\delta_{n-j}^\circ)
\end{aligned}
\end{equation*}
The resulting induction lemma from $\kappa$ to $\mu$  (or equivalently, the modified form of \cref{thm:induction_lemma_BE_3N}) becomes
\begin{equation*}
\begin{aligned}
    \mu_{[1,i]}
    & \leq \frac{C}{\sqrt{n}} \frac{\delta \log(ep)}{\sigma_{\min}}  
    + \frac{C}{\sqrt{n}} \frac{\sqrt{\log(ep)}}{\phi \sigma_{\min}} \\
    & \quad + \frac{C}{\sqrt{n}} \frac{(\log(ep))^{5/2}}{\underline{\sigma}^2 \min\{\sigma_{\min}, \underline{\sigma} \sqrt{\log(ep)}\}}
    \log\left(1 + \frac{\sqrt{n}\underline{\sigma}}{\delta}\right) \\
    & \quad \times \left[ 
        {L}_{3,\max} + \phi^{\min\{1, q-3\}} \nu_{q,\max} \frac{ {(\log(ep))}^{\max\{0, q-4\}/2}}{(\delta/\alpha)^{\max\{0, q-4\}}}
    \right] \\
    & \quad + C \sum_{j \geq {n/2}} \frac{(\log(ep))^{3/2}}{\delta_{n-j}^3} \left[ 
        {L}_{3,\max} + \phi^{\min\{1, q-3\}} \nu_{q,\max} \frac{ {(\log(ep))}^{\max\{0, q-4\}/2}}{\delta_{n-j}^{\max\{0, q-4\}}}
    \right] \kappa_{[1,j-4]|\{j-3\}}(\delta_{n-j}^\circ),
\end{aligned}
\end{equation*} 
and the same dual induction derives the desired conclusion. The same changes apply to the summation of $\mathfrak{R}^{(4,1)}_{X_j}$ and $\mathfrak{R}^{(6,1)}_{X_j,X_k}$.

For \cref{eq:third_moment_Gaussian_bound}, without \eqref{assmp:var_ev}, the rightmost upper bound is $\frac{C}{n^{3/2}} L_{3,j} \frac{(\log(ep))^{3/2}}{\underline{\sigma}^3}$ rather than $\frac{C}{n^{3/2}} L_{3,j} \frac{(\log(ep))^2}{\underline{\sigma}^2 \sigma_{\min}}$. That is,
\begin{equation*}
\begin{aligned}
    \abs*{ \inner*{ 
        \Exp\left[ \nabla^3 \rho_{r,\phi}^\delta(Y_{[1,j-1)} + Y_{(j+1,n]}) \right], 
        \Exp[X_j^{\otimes 3}] } }
    & \leq \frac{C}{n^{3/2}} L_{3,j} \frac{(\log(ep))^{3/2}}{\underline{\sigma}^3}. \\
\end{aligned}
\end{equation*} 
As a result, 
\begin{equation*}
\begin{aligned}
    \mu_{[1,n]}
    & \leq \frac{C}{\sqrt{n}} \frac{\delta \log(ep)}{\sigma_{\min}}  
    + \frac{C}{\sqrt{n}} \frac{\sqrt{\log(ep)}}{\phi \sigma_{\min}} 
    + \frac{C}{\sqrt{n}} L_{3,\max} \frac{(\log(ep))^{2}}{\underline{\sigma}^2 \min\{\sigma_{\min}, \underline{\sigma} \sqrt{\log(ep)}\}} \\
    & + \sum_{j=1}^{n} \abs*{ \Exp \left[ \mathfrak{R}_{X_j}^{(4,1)} 
        - \mathfrak{R}_{Y_j}^{(4,1)} \right] }
    + \sum_{j=3}^{n} \sum_{k=1}^{j-2} \abs*{ \Exp \left[ 
        \mathfrak{R}_{X_j,X_k}^{(6,1)}
        - \mathfrak{R}_{X_j,Y_k}^{(6,1)}
    \right] }.
\end{aligned}
\end{equation*}

\subsection{\texorpdfstring{$1$}{1}-ring dependence and a permutation argument}
\label{sec:permutation_argument}

We note that the Lindeberg swapping in \cref{eq:first_lindeberg_swapping_N} is not symmetric with respect to the indices. The asymmetry resulted in a worse rates in \cref{thm:induction_lemma_BE_3N,thm:induction_lemma_AC_N}, by having the maximal moment terms, $L_{3,\max}$ and $\nu_{q,\max}$. 
To obtain an improved Berry--Esseen bounds with averaged moment terms, $\bar{L}_3$ and $\bar{\nu}_q$, as in \cref{thm:1_dep_berry_esseen_3P,thm:1_dep_berry_esseen_4P}, it is desired to relax the asymmetry in the Lindeberg swappings. One such way is to take the average of the upper bounds over permutations of the indices as done in \citet{deng2020slightly,deng2020beyond}. However, because $1$-dependence is specific to the index ordering, the only permutation preserving the dependence structure is the flipping of the indices (i.e., $X_1 \mapsto X_n, X_2 \mapsto X_{n-1}, \dots, X_n \mapsto X_1$), which is not sufficient for our purpose. We allow more permutations by weakening the dependence structure to $1$-ring dependence. By allowing $X_1$ and $X_n$ dependent on each other, index rotations (i.e., $X_1 \mapsto X_{j^\circ}, X_2 \mapsto X_{j^\circ+1}, \dots, X_n \mapsto X_{j^\circ-1}$ with some $j^\circ \in [1,n]$) are added to the catalog of available permutations. By averaging the upper bound in \cref{eq:first_lindeberg_swapping_N} over the permutations, we obtain
\begin{equation} \label{eq:first_lindeberg_swapping_P}
\begin{aligned}
    & \abs*{ \Exp[\rho_{r,\phi}^\delta(X_{[1,n]})] - \Exp[\rho_{r,\phi}^\delta(Y_{[1,n]})]} \\
    & \leq \frac{1}{n} \sum_{j^\circ=1}^n \sum_{j=1}^{n-1} 
    \left| \Exp \left[
        \rho_{r,\phi}^\delta(X_{(j^\circ,j^\circ+j)_n} + X_{[j^\circ+j]_n} + Y_{(j^\circ+j,j^\circ+n]_n})
    \right. \right. \\
    & \hspace{2in} \left. \left.
        - \rho_{r,\phi}^\delta(X_{(j^\circ,j^\circ+j)_n} + Y_{[j^\circ+j]_n} + Y_{(j^\circ+j,j^\circ+n]_n})
    \right] \right|, \\
\end{aligned}
\end{equation}
where $[j^\circ+j]_n$ is $j^\circ+j$ modulo $n$, and
\begin{equation*}
    [i,j]_n \equiv \{[i]_n, [i+1]_n, \dots, [j-1]_n, [j]_n\}.
\end{equation*}
The subscript $n$ notates that the interval is defined modulo $n$. If the ambient modulo is obvious, we omit the subscript. The other types of intervals, $(i,j]_n, [i,j)_n$ and $(i,j)_n$, are similarly defined. For full notation details, please refer to \cref{sec:Z_n}. A similar permutation argument also applies to \cref{thm:induction_lemma_AC_N}; see \cref{thm:induction_lemma_AC_P}. The dual induction on the resulting induction lemmas proves \cref{thm:1_dep_berry_esseen_3P} for $3 \leq q < 4$. We relegate the proof details to \cref{sec:pf_1_dep_3P}.

For $4 \leq q$, there is the second Lindeberg swapping during the decomposition of the third order remainder terms (e.g., \cref{eq:third_lindeberg_swapping_N}). The same permutation argument as in \cref{eq:first_lindeberg_swapping_P} provides the following averaged version: for $3 \leq j \leq n$,
\begin{equation} \label{eq:third_lindeberg_swapping_P}
\begin{aligned}
    & \inner*{ 
        \Exp[\nabla^3 \rho_{r,\phi}^{\delta}(X_{[1,j-1)} + Y_{(j+1,n]})
        - \nabla^3 \rho_{r,\phi}^{\delta}(Y_{[1,j-1)} + Y_{(j+1,n]})], 
        \Exp[X_j^{\otimes 3}] } \\
    & \leq \frac{1}{j-2} \sum_{k^\circ=1}^{j-2} \sum_{k=1}^{j-2} \Big\langle
        \Exp[\nabla^3 \rho_{r,\phi}^{\delta}(X_{[k^\circ,k^\circ+k)_{j-2}} + X_{[k^\circ+k]_{j-2}} + Y_{(k^\circ+k,k^\circ+j-1)_{j-2}\cup(j,n]}) \\
    & \hspace{0.5in} - \nabla^3 \rho_{r,\phi}^{\delta}(X_{[k^\circ,k^\circ+k)_{j-2}} + Y_{[k^\circ+k]_{j-2}} + Y_{(k^\circ+k,k^\circ+j-1)_{j-2}\cup(j,n]})], 
        \Exp[X_j^{\otimes 3}] \Big\rangle. \\
\end{aligned}
\end{equation}
Then, the dual induction with \cref{thm:induction_lemma_AC_P} proves \cref{thm:1_dep_berry_esseen_4P} for $4 \leq q$. We relegate the proof details to \cref{sec:pf_1_dep_4P}.

%
%

\begin{acks}[Acknowledgments]
We thank the editor and the reviewer for their numerous insightful comments and suggestions, especially regarding the correct rate for the sub-Weibull cases, which have led to significant improvements in our presentation and results.
\end{acks}
\begin{funding}
All the authors were supported in part by NSF grant DMS-2113611.
\end{funding}

\begin{supplement}
\stitle{Supplement to "Dual Induction CLT for High-dimensional \texorpdfstring{$m$}{m}-dependent Data"}
\sdescription{A PDF manuscript providing supplemental descriptions of proofs and extensive details regarding arguments conveyed in the main text.}
\end{supplement}


\bibliographystyle{imsart-nameyear} 
\bibliography{2_ref-main}       


\newpage

\counterwithin{figure}{section}
\counterwithin{equation}{section}

\begin{appendix}

\section{Proof of Theorems} \label{sec:pf_thm}

\subsection{Proof details of Section~\ref{sec:pf_sketch_4N}} \label{sec:pf_1_dep_4N}

\medskip
We recall from \cref{sec:pf_sketch_4N} that
\begin{equation} \label{eq:first_lindeberg_result_4N}
\begin{aligned}
    & \sup_{r \in \reals^p} \abs*{ \Exp[\rho_{r,\phi}^\delta(X_{[1,n]})] - \Exp[\rho_{r,\phi}^\delta(Y_{[1,n]})]} 
    = \sum_{j=1}^{n} \abs*{ \Exp \left[ 
        \mathfrak{R}_{X_j}^{(3,1)} 
        - \mathfrak{R}_{Y_j}^{(3,1)} 
    \right] }\\
    & \leq \frac{C}{\sqrt{n}} L_{3,\max} \frac{(\log(ep))^2}{\underline\sigma^2 \sigma_{\min}}
    + \sum_{j=1}^{n} \abs*{ \Exp \left[ \mathfrak{R}_{X_j}^{(4,1)} 
        - \mathfrak{R}_{Y_j}^{(4,1)} \right] }
    + \sum_{j=3}^{n} \sum_{k=1}^{j-2} \abs*{ \Exp \left[ 
        \mathfrak{R}_{X_j,X_k}^{(6,1)}
        - \mathfrak{R}_{X_j,Y_k}^{(6,1)}
    \right] },
\end{aligned}
\end{equation}
where $\mathfrak{R}_{X_j,W_k}^{(6,1)} = \frac{1}{6} \mathfrak{R}_{X_j,W_k}^{(6,1,1)} + \frac{1}{2} \mathfrak{R}_{X_j,W_k}^{(6,1,2)} + \mathfrak{R}_{X_j,W_k}^{(6,1,3)}$.The upper bounds of the remainder terms are given as the following lemma.

\begin{lemma} \label{thm:remainder_lemma_4N}
    There exist universal constants $C > 0$ and $\alpha > 0$ such that for any $n$, $1$-dependent sequence $(X_1, \dots, X_n)$ satisfying Assumption~\eqref{assmp:min_ev}, $j \in [1,n]$, $k \in [1,j-2]$, $q \geq 4$, $\delta \geq \sigma_{\min}$ and $\phi \geq \frac{1}{\delta \log(ep)}$,
    \begin{equation*}
    \begin{aligned}
        \abs*{\Exp \left[ \mathfrak{R}_{W_j}^{(4,1)} \right]} 
        & \leq C \phi \left[ 
            {L}_{4,\max} \frac{(\log(ep))^{3/2}}{\delta_{n-j}^3}
            + \nu_{q,\max} \frac{ {(\log(ep))}^{(q-1)/2}}{(\delta_{n-j}/\alpha)^{q-1}}
        \right] \\
        & \quad \times \min\{ 1, \kappa_{[1,j-5]|\{j-4\}}(\delta_{n-j}^\circ) + \kappa^\circ_{j}(\delta_{n-j}) \},
    \end{aligned}
    \end{equation*}
    \begin{equation*}
    \begin{aligned}
        \abs*{\Exp \left[ \mathfrak{R}_{X_j,W_k}^{(6,1)} \right]} 
        & \leq C \phi L_{3,\max} \left[ \begin{aligned}
            L_{3,\max} \frac{(\log(ep))^{5/2}}{\delta_{n-k}^5}
            + \nu_{q,\max} \frac{ {(\log(ep))}^{(q+2)/2}}{(\delta_{n-k}/\alpha)^{q+2}}
        \end{aligned} \right] \\
        & \quad \times \min\{ 1, \kappa_{[1,k-4]|\{k-3\}}(\delta_{n-k}^\circ) + \kappa^\circ_{k}(\delta_{n-k}) \}.
    \end{aligned}
    \end{equation*}
    where $W$ represents either $X$ or $Y$, $\delta_{n-j}^2 \equiv \delta^2 + \underline{\sigma}^2 \max\{n-j,0\}$, $\delta^\circ_{n-j} \equiv 12 \delta_{n-j} \sqrt{\log(pn)}$ and $\kappa^\circ_{j}(\delta) \equiv \frac{\delta \log(ep)}{\sigma_{\min} \sqrt{\max\{j, 1\}}}$.
\end{lemma}

Back to \cref{eq:first_lindeberg_result_4N}, we get
\begin{equation*}
\begin{aligned}
    & \abs*{ \Exp[\rho_{r,\phi}^\delta(X_{[1,n]})] - \Exp[\rho_{r,\phi}^\delta(Y_{[1,n]})]} \\
    & \leq \frac{C}{\sqrt{n}} L_{3,\max} \frac{(\log(ep))^2}{\underline{\sigma}^2 \sigma_{\min}} \\
    & \quad + C \phi \sum_{j=1}^{n} \left[ 
        {L}_{4,\max} \frac{(\log(ep))^{3/2}}{\delta_{n-j}^3}
        + \nu_{q,\max} \frac{{(\log(ep))}^{(q-1)/2}}{(\delta_{n-j}/\alpha)^{q-1}}
    \right] \\
    & \quad \hspace{0.65in} \times \min\{1, \kappa_{[1,j-5]|\{j-4\}}(\delta_{n-j}^\circ) + \kappa^\circ_{j}(\delta_{n-j})\} \\
    & \quad + C \phi \sum_{j=3}^{n} 
    L_{3,\max} \sum_{k=1}^{j-2} \left[ 
        {L}_{3,\max} \frac{(\log(ep))^{5/2}}{\delta_{n-k}^5}
        + \nu_{q,\max} \frac{{(\log(ep))}^{(q+2)/2}}{(\delta_{n-k}/\alpha)^{q+2}}
    \right] \\
    & \quad \hspace{1.3in} \times \min\{1, \kappa_{[1,k-4]|\{k-3\}}(\delta_{n-k}^\circ) + \kappa^\circ_{k}(\delta_{n-k})\},
\end{aligned}
\end{equation*}

\medskip
\noindent{\bf Partitioning the sum.} 
Again, we partition the summations at $j = n/2$. For the first summation, similar calculations with the finite third moment cases lead to
\begin{equation*}
\begin{aligned}
    & C \phi \sum_{j=1}^n \left[ 
        {L}_{4,\max} \frac{(\log(ep))^{3/2}}{\delta_{n-j}^3}
        + \nu_{q,\max} \frac{{(\log(ep))}^{(q-1)/2}}{(\delta_{n-j}/\alpha)^{q-1}}
    \right] \\
    & \qquad \quad \times \min\{1, \kappa_{[1,j-5]|\{j-4\}}(\delta_{n-j}^\circ) + \kappa^\circ_{j}(\delta_{n-j})\} \\
    & \leq \frac{C \phi}{\sqrt{n}} \left[
        {L}_{4,\max} \frac{(\log(ep))^{5/2}}{\underline{\sigma}^2\sigma_{\min}} 
        + {\nu}_{q,\max} \frac{{(\log(ep))}^{(q+1)/2}}{\underline{\sigma}^{2}\sigma_{\min} (\delta/\alpha)^{q-4}} 
    \right] 
    \log\left(1 + \frac{\sqrt{n}\underline{\sigma}}{\delta}\right) \\
    & \quad + C \phi \sum_{j \geq n/2} \left[ 
        {L}_{4,\max} \frac{(\log(ep))^{3/2}}{\delta_{n-j}^3}
        + \nu_{q,\max} \frac{{(\log(ep))}^{(q-1)/2}}{(\delta_{n-j}/\alpha)^{q-1}}
    \right]
    \kappa_{[1,j-5]|\{j-4\}}(\delta_{n-j}^\circ) \\
\end{aligned}
\end{equation*}
by noting \cref{eq:sum_delta_3,eq:sum_delta_3_kappa} and that
\begin{equation}\label{eq:sum_delta_q}
\begin{aligned}
    \sum_{j=1}^{\floor{n/2}} \frac{1}{\delta_{n-j}^{q-1}} 
    & \leq \sum_{j=1}^{\floor{n/2}} \frac{1}{\delta^{q-4} \delta_{n-j}^{3}} 
    \leq \frac{C \log(ep)}{\underline{\sigma}^{2}\sigma_{\min} \delta^{q-4} \sqrt{n}},
\end{aligned}         
\end{equation}
\begin{equation} \label{eq:sum_delta_q_kappa}
\begin{aligned}
    \sum_{j=\ceil{n/2}}^{n} \frac{\kappa^\circ_{j}(\delta_{n-j})}{\delta_{n-j}^{q-1}}
    & \leq \sum_{j=\ceil{n/2}}^{n} \frac{\log(ep)}{\delta^{q-4} \delta_{n-j}^{2} \sigma_{\min} \sqrt{\max\{j,1\}}} \\
    & \leq \frac{C\log(ep)}{\underline{\sigma}^{2}\sigma_{\min} \delta^{q-4} \sqrt{n}} \log\left(1 + \frac{\sqrt{n}\underline{\sigma}}{\delta}\right),
\end{aligned}
\end{equation}
for some universal constant $C > 0$. On the other hand, for the second summation,
\begin{equation*}
\begin{aligned}
    & C \phi \sum_{j=3}^{n} {L}_{3,\max} \sum_{k=1}^{\min\{j-2,\floor{n/2}\}} \left[ 
        {L}_{3,\max} \frac{(\log(ep))^{5/2}}{\delta_{n-k}^5}
        + \nu_{q,\max} \frac{{(\log(ep))}^{(q+2)/2}}{(\delta_{n-k}/\alpha)^{q+2}}
    \right] \\
    & \leq \frac{C \phi}{\sqrt{n}} \sum_{j=3}^{n} {L}_{3,\max} 
    \left[ 
        {L}_{3,\max} \frac{(\log(ep))^{7/2}}
        {\underline{\sigma}^{2}\sigma_{\min} \delta_{n-\floor{n/2}}^2}
        + {\nu}_{q,\max} \frac{{(\log(ep))}^{(q+4)/2}}
        {\underline{\sigma}^{2}\sigma_{\min} (\delta_{n-\floor{n/2}}/\alpha)^{q-1}}
    \right] \\
    & \leq \frac{C \phi}{\sqrt{n}} {L}_{3,\max} 
    \left[ 
        {L}_{3,\max} \frac{(\log(ep))^{7/2}}
        {\underline{\sigma}^{4}\sigma_{\min}} 
        + {\nu}_{q,\max} \frac{{(\log(ep))}^{(q+4)/2}}
        {\underline{\sigma}^{4}\sigma_{\min} (\delta/\alpha)^{q-3}}
    \right],
\end{aligned}
\end{equation*}
where the last inequality comes from $\sum_{j=3}^{n} \frac{1}{\delta^2_{n-\floor{n/2}}} \leq \sum_{j=3}^{n} \frac{1}{(n-\floor{n/2}) \underline{\sigma}^2} \leq \frac{C}{\underline{\sigma}^2}$, and
\begin{equation*}
\begin{aligned}
    & C \phi \sum_{j=\ceil{n/2}}^{n} L_{3,\max} \sum_{k=\ceil{n/2}}^{j-2} \left[ 
        L_{3,\max} \frac{(\log(ep))^{5/2}}{\delta_{n-k}^5}
        + \nu_{q,\max} \frac{{(\log(ep))}^{(q+2)/2}}{(\delta_{n-k}/\alpha)^{q+2}}
    \right] \kappa^\circ_{k}(\delta_{n-j}) \\
    & \leq \frac{C \phi}{\sqrt{n}} \sum_{j=\ceil{n/2}}^{n} L_{3,\max} 
    \left[ 
        {L}_{3,\max} \frac{(\log(ep))^{7/2}}{ \underline{\sigma}^{2}\sigma_{\min} \delta_{n-j}^2}
        + {\nu}_{q,\max} \frac{{(\log(ep))}^{(q+4)/2}}{\underline{\sigma}^{2}\sigma_{\min} (\delta_{n-j}/\alpha)^{q-1}}
    \right] 
    \\
    & \leq \frac{C \phi}{\sqrt{n}} {L}_{3,\max} 
    \left[ 
        {L}_{3,\max} \frac{(\log(ep))^{7/2}}{ \underline{\sigma}^{4}\sigma_{\min}} 
        + {\nu}_{q,\max} \frac{{(\log(ep))}^{(q+4)/2}}{\underline{\sigma}^{4}\sigma_{\min} (\delta/\alpha)^{q-3}}
    \right] \log\left(1 + \frac{\sqrt{n}\underline{\sigma}}{\delta}\right),
\end{aligned}
\end{equation*}
where the inequliaties follow \cref{eq:sum_delta_q,eq:sum_delta_q_kappa}, respectively.
In sum, we obtain a finite fourth moment version of \cref{thm:induction_lemma_BE_3N}:
\begin{lemma} \label{thm:induction_lemma_BE_4N}
    There exist universal constants $C > 0$ and $\alpha > 0$ such that for any $n$, $1$-dependent sequence $(X_i \in \reals^p: i \in [1,n])$ satisfying Assumptions \eqref{assmp:min_var}, \eqref{assmp:min_ev} and \eqref{assmp:var_ev}, $q \geq 4$, $\delta \geq \sigma_{\min}$ and $\phi \geq \frac{1}{\delta \log(ep)}$,
    for any $\delta \geq \sigma_{\min}$,
    \begin{equation*}
    \begin{aligned}
        & \mu_{[1,n]} \\
        & \leq \frac{C}{\sqrt{n}} \left[
            \frac{\delta \log(ep)}{\sigma_{\min}}  
            + \frac{\sqrt{\log(ep)}}{\phi \sigma_{\min}}
            + {L}_{3,\max} \frac{(\log(ep))^{2}}{\underline{\sigma}^2 \sigma_{\min}} 
        \right]
        \\
        & \quad + \frac{C\phi}{\sqrt{n}} \left[
            {L}_{4,\max} \frac{(\log(ep))^{5/2}}{\underline{\sigma}^2\sigma_{\min}}
            + {\nu}_{q,\max} \frac{(\log(ep))^{(q+1)/2}}{\underline{\sigma}^{2}\sigma_{\min} (\delta/\alpha)^{q-4}}
        \right] \log\left(1 + \frac{\sqrt{n}\underline{\sigma}}{\delta}\right) 
        \\
        & \quad + \frac{C \phi}{\sqrt{n}} {L}_{3,\max} 
        \left[ 
            {L}_{3,\max} \frac{(\log(ep))^{7/2}}{ \underline{\sigma}^{4}\sigma_{\min}} 
            + {\nu}_{q,\max} \frac{{(\log(ep))}^{(q+4)/2}}{\underline{\sigma}^{4}\sigma_{\min} (\delta/\alpha)^{q-3}}
        \right] \log\left(1 + \frac{\sqrt{n}\underline{\sigma}}{\delta}\right) 
        \\
        & \quad + C \phi \sum_{j=\ceil{n/2}}^{n} \left[ 
            L_{4,\max} \frac{(\log(ep))^{3/2}}{\delta_{n-j}^3}
            + \nu_{q,\max} \frac{{(\log(ep))}^{(q-1)/2}}{(\delta_{n-j}/\alpha)^{q-1}}
        \right]
        \kappa_{[1,j-5]|\{j-4\}}(\delta_{n-j}^\circ) 
        \\
        & \quad + C \phi \sum_{j=\ceil{n/2}}^{n} L_{3,\max}
        \sum_{k=\ceil{n/2}}^{j-2} \left[ 
            L_{3\max} \frac{(\log(ep))^{5/2}}{\delta_{n-k}^5}
            + \nu_{q,\max} \frac{(\log(ep))^{(q+2)/2}}{(\delta_{n-k}/\alpha)^{q+2}}
        \right] 
        \kappa_{[1,k-4]|\{k-3\}}(\delta_{n-k}^\circ),
    \end{aligned}
    \end{equation*}
    where $\delta_{n-j}^2 \equiv \delta^2 + \underline{\sigma}^2 \max\{n-j,0\}$ and $\delta^\circ_{n-j} \equiv 12 \delta_{n-j} \sqrt{\log(pn)}$.
\end{lemma}

\medskip
\noindent{\bf Dual Induction.}
In this case, our induction hypotheses are as follows.
\begin{lemma}
    There exist positive universal constants $\mathfrak{C}_{1,\kappa}$, $\mathfrak{C}_{2,\kappa}$, $\mathfrak{C}_{3,\kappa}$, $\mathfrak{C}_{4,\kappa}$, $\mathfrak{C}_{5,\kappa}$, $\mathfrak{C}_{1,\mu}$, $\mathfrak{C}_{2,\mu}$ and $\mathfrak{C}_{3,\mu}$ such that for any $n$, $1$-dependent sequence $(X_i \in \reals^p: i \in [1,n])$ satisfying Assumptions \eqref{assmp:min_var}, \eqref{assmp:min_ev} and \eqref{assmp:var_ev} and $\delta \geq 0$,
    \begin{equation} \label{eq:hypothesis_AC_4N} \tag{HYP-AC-2}
    \begin{aligned}
        & \sqrt{i} \kappa_{[1,i]|\{i+1\}}(\delta)
        \leq 
        \tilde\kappa_{1,i} {L}_{3,\max}
        + \tilde\kappa_{2,i} {L}_{4,\max}^{1/2}
        + \tilde\kappa_{3,i} \nu_{q,\max}^{1/(q-2)}
        + \tilde\kappa_{4,i} \nu_{1,\max}
        + \tilde\kappa_{5} \delta, \\
        & \forall i \in [1,n),
    \end{aligned}
    \end{equation}
    \begin{equation} \label{eq:hypothesis_BE_4N} \tag{HYP-BE-2}
        \sqrt{n} \mu_{[1,n]}
        \leq \tilde\mu_{1,n} {L}_{3,\max}
        + \tilde\mu_{2,n} {L}_{4,\max}^{1/2}
        + \tilde\mu_{3,n} \nu_{q,\max}^{1/(q-2)},
    \end{equation}
    where $\tilde\kappa_{1,i} = \mathfrak{C}_{1,\kappa} \tilde\mu_{1,i}$, 
    $\tilde\kappa_{2,i} = \mathfrak{C}_{2,\kappa} \tilde\mu_{2,i}$, 
    $\tilde\kappa_{3,i} = \mathfrak{C}_{3,\kappa} \tilde\mu_{3,i}$,
    $\tilde\kappa_{4,i} = \mathfrak{C}_{4,\kappa} \frac{\log(ep)\sqrt{\log(pi)}}{\sigma_{\min}}$,
    $\tilde\kappa_{5} = \mathfrak{C}_{5,\kappa} \frac{\sqrt{\log(ep)}}{\sigma_{\min}}$,
    \begin{equation*}
    \begin{aligned}
        \tilde\mu_{1,n} 
        & = \mathfrak{C}_{1,\mu} \frac{(\log(ep))^{3/2}\sqrt{\log(pn)}}{\underline\sigma^2\sigma_{\min}} \log\left(e n\right), \\
        \tilde\mu_{2,n}
        & = \mathfrak{C}_{2,\mu} \frac{\log(ep)\sqrt{\log(pn)}}{\underline\sigma \sigma_{\min}} \log\left(e n\right) \\
        \tilde\mu_{3,n}
        & = \mathfrak{C}_{3,\mu} \frac{\log(ep)\sqrt{\log(pn)}}{\underline\sigma^{2/(q-2)} \sigma_{\min}} \log\left(e n\right)
    \end{aligned}
    \end{equation*}
\end{lemma}

If $\mathfrak{C}_{1,\kappa}$, $\mathfrak{C}_{2,\kappa}$, $\mathfrak{C}_{3,\kappa}$, $\mathfrak{C}_{4,\kappa}$, $\mathfrak{C}_{5,\kappa}$, $\mathfrak{C}_{1,\mu}$, $\mathfrak{C}_{1,\mu}$, $\mathfrak{C}_{3,\mu} \geq 2$, then \eqref{eq:hypothesis_AC_4N} and \eqref{eq:hypothesis_BE_4N}, requiring $\kappa_{[1,i]|\{i+1\}}(\delta) \leq 1$ almost surely for all $i \in [1, n)$ and $\mu_{[1,n]} \leq 1$ only, trivially holds for $n \leq 36$. Now we consider the case of $n > 36$. Suppose that the induction hypotheses hold for all smaller $n$.

\medskip
We first derive \eqref{eq:hypothesis_AC_4N} for any $i \in [1,n)$. By \cref{thm:assmp_induction_N}, $(X_1, \dots, X_{n-1})$ is a $1$-dependent sequence satisfying Assumptions \eqref{assmp:min_var}, \eqref{assmp:min_ev} and \eqref{assmp:var_ev} with the same $\sigma_{\min}$ and $\underline{\sigma}$ as the original data $(X_1, \dots, X_n)$. As a result, \eqref{eq:hypothesis_AC_4N} applied to $(X_1, \dots, X_{n-1})$ verifies that the same conditional anti-concentration inequality holds for $X_{[1,i]}$ given $X_{i+1}$ for $ i \in [1,n-1)$. For $i = n-1$, by \cref{thm:induction_lemma_AC_N}, 
for any $\vareps \geq \sigma_{\min}$ and $\delta > 0$,
\begin{equation*} 
\begin{aligned}
    & \kappa_{[1,i]|\{i+1\}}(\delta) \\
    & \leq C
    \frac{\sqrt{\log(ep)}}{\vareps} \nu_{1,\max} 
    \min\{1, \kappa_{[1,i-2]|\{i-1\}}(\vareps^\circ)\} \\
    & \quad + C \mu_{[1,i-2]}
    + C\frac{\delta + 2\vareps^\circ}{\sigma_{\min}} \sqrt{\frac{\log(ep)}{i-2}}
    + C\frac{\nu_{1,\max}}{\sigma_{\min}} \frac{\log(ep)}{\sqrt{i-2}},
\end{aligned}
\end{equation*}
where $\vareps^\circ = 20\vareps \sqrt{\log(p(i-2))}$ and $C > 0$ is a universal constant.
Following \eqref{eq:hypothesis_AC_4N}, we upperbound $\kappa_{[1,i-2]|\{i-1\}}(\vareps^\circ)$. Furthermore, since $(X_1, \dots, X_{i-2})$ is a $1$-dependent sequence satisfying Assumptions~\eqref{assmp:min_var}, \eqref{assmp:min_ev} and \eqref{assmp:var_ev} with the same $\sigma_{\min}$ and $\underline{\sigma}$ as the original data (see \cref{thm:assmp_induction_N}), \eqref{eq:hypothesis_BE_4N} holds for $\mu_{[1,i-2]}$, and we obtain
\begin{equation*} 
\begin{aligned}
    & \kappa_{[1,i]|\{i+1\}}(\delta) \\
    & \leq \frac{C}{\sqrt{i-2}} 
    \frac{\sqrt{\log(ep)}}{\vareps} \nu_{1,\max}
    \left[ 
        \tilde\kappa_{1,i-2} {L}_{3,\max}
        + \tilde\kappa_{2,i-2} {L}_{4,\max}^{1/2}
        + \tilde\kappa_{3,i-2} {\nu}_{q,\max}^{1/(q-2)}
        + \tilde\kappa_{4,i-2} \nu_{1,\max}
        + \tilde\kappa_{5} {\vareps^\circ}
    \right]
    \\
    & \quad + \frac{C}{\sqrt{i-2}} \tilde\mu_{1,i-2} {{L}_{3,\max}}
    + \frac{C}{\sqrt{i-2}} \tilde\mu_{2,i-2} {{L}_{4,\max}^{1/2}}
    + \frac{C}{\sqrt{i-2}} \tilde\mu_{3,i-2} {\nu_{q,\max}^{1/(q-2)}} \\
    & \quad + C\frac{\delta + 2\vareps^\circ}{\sigma_{\min}} \sqrt{\frac{\log(ep)}{i-2}}
    + C\frac{\nu_{1,\max}}{\sigma_{\min}} \frac{\log(ep)}{\sqrt{i-2}}.
\end{aligned}
\end{equation*}
As a result, we obtain a recursive inequality on $\tilde\kappa$'s that
\begin{equation} \label{eq:recursion_kappa_4N}
\begin{aligned}
    & \sqrt{i} \kappa_{[1,i]|\{i+1\}}(\delta) \\
    & \leq \mathfrak{C}' \frac{\sqrt{\log(ep)}}{\vareps} \nu_{1,\max}
    \left[ 
        \tilde\kappa_{1,i-2} {L}_{3,\max}
        + \tilde\kappa_{2,i-2} {L}_{4,\max}^{1/2}
        + \tilde\kappa_{3,i-2} \nu_{q,\max}^{1/(q-2)}
        + \tilde\kappa_{4,i-2} \nu_{1,\max}
        + \tilde\kappa_{5} \vareps^\circ
    \right]
    \\
    & \quad + \mathfrak{C}'  \left[ 
        \tilde\mu_{1,i-2} {L}_{3,\max}
        + \tilde\mu_{2,i-2} {L}_{4,\max}^{1/2}
        + \tilde\mu_{3,i-2} \nu_{q,\max}^{1/(q-2)}
        + \frac{\delta + 2\vareps^\circ}{\sigma_{\min}} \sqrt{\log(ep)}
        + \frac{\nu_{1,\max}}{\sigma_{\min}} \log(ep)
    \right],
\end{aligned}
\end{equation}
for some universal constant $\mathfrak{C}'$, whose value does not change in this subsection.
Plugging in $\vareps = \max\{2 \mathfrak{C}', 1\} \sqrt{\log(ep)} {\nu}_{1,\max} \overset{\text{\cref{eq:nu_1_vs_sigma_min}}}{\geq} \sigma_{\min}$,
\begin{equation*}
\begin{aligned}
    & \sqrt{i} \kappa_{[1,i]|\{i+1\}}(\delta) \\
    & \leq \frac{1}{2}
    \left[ 
        \tilde\kappa_{1,i-2} {L}_{3,\max}
        + \tilde\kappa_{2,i-2} {L}_{4,\max}^{1/2}
        + \tilde\kappa_{3,i-2} \nu_{q,\max}^{1/(q-2)}
        + \tilde\kappa_{4,i-2} \nu_{1,\max}
    \right] \\
    & \quad + 20 \mathfrak{C}' \tilde\kappa_{5} \sqrt{\log(ep)\log(pi)} \nu_{1,\max} \\
    & \quad  
    + \mathfrak{C}' \left[ 
        \tilde\mu_{1,i-2} {L}_{3,\max}
        + \tilde\mu_{2,i-2} \nu_{q,\max}^{1/(q-2)}  
        + \frac{\log(ep)(1+40 \sqrt{\log(p i)})}{\sigma_{\min}} {\nu}_{1,\max}
        + \frac{\sqrt{\log(ep)}}{\sigma_{\min}} \delta
    \right]
    \\
    & \leq \tilde\kappa_{1,i} {L}_{3,\max}
    + \tilde\kappa_{2,i} \nu_{q,\max}^{1/(q-2)}
    + \tilde\kappa_{3,i} \nu_{1,\max}
    + \tilde\kappa_{4} \delta,
\end{aligned}
\end{equation*}
where $\tilde\kappa_{1,i} = \mathfrak{C}_{1,\kappa} \tilde\mu_{1,i}$, 
$\tilde\kappa_{2,i} = \mathfrak{C}_{2,\kappa} \tilde\mu_{2,i}$, 
$\tilde\kappa_{3,i} = \mathfrak{C}_{3,\kappa} \tilde\mu_{3,i}$, 
$\tilde\kappa_{4,i} = \mathfrak{C}_{4,\kappa} \frac{\log(ep)\sqrt{\log(pi)}}{\sigma_{\min}}$, and
$\tilde\kappa_{5} = \mathfrak{C}_{5,\kappa} \frac{\sqrt{\log(ep)}}{\sigma_{\min}}$, provided by $\mathfrak{C}_{1,\kappa} = \mathfrak{C}_{2,\kappa} = \mathfrak{C}_{3,\kappa} = \max\{2 \mathfrak{C}', 2\}$, $\mathfrak{C}_{4,\kappa} = \max\{82 \mathfrak{C}' + 20 \mathfrak{C}' \mathfrak{C}_{5,\kappa}, 2\}$ and $\mathfrak{C}_{5,\kappa} = \max\{ \mathfrak{C}', 2\}$. 
This proves \eqref{eq:hypothesis_AC_4N} at $n$.

For \eqref{eq:hypothesis_BE_4N}, we first upper bound the last two terms in \cref{thm:induction_lemma_BE_4N}: 
\begin{equation*}
\begin{aligned}
    & C \phi \sum_{j=\ceil{n/2}}^{n} \left[ 
        L_{4,\max} \frac{(\log(ep))^{3/2}}{\delta_{n-j}^3}
        + \nu_{q,\max} \frac{{(\log(ep))}^{(q-1)/2}}{(\delta_{n-j}/\alpha)^{q-1}}
    \right]
    \kappa_{[1,j-5]|\{j-4\}}(\delta_{n-j}^\circ) 
    \\
    & + C \phi \sum_{j=\ceil{n/2}}^{n} L_{3,\max}
    \sum_{k=\ceil{n/2}}^{j-2} \left[ 
        L_{3\max} \frac{(\log(ep))^{5/2}}{\delta_{n-k}^5}
        + \nu_{q,\max} \frac{(\log(ep))^{(q+2)/2}}{(\delta_{n-k}/\alpha)^{q+2}}
    \right] 
    \kappa_{[1,k-4]|\{k-3\}}(\delta_{n-k}^\circ). \\
\end{aligned}
\end{equation*}
Applying \eqref{eq:hypothesis_AC_4N} to $\kappa_{[1,j-5]|\{j-4\}}(\delta_{n-j}^\circ)$,
\begin{equation*}
\begin{aligned}
    & C \phi \sum_{j=\ceil{n/2}}^{n} \left[ 
        L_{4,\max} \frac{(\log(ep))^{3/2}}{\delta_{n-j}^3}
        + \nu_{q,\max} \frac{{(\log(ep))}^{(q-1)/2}}{(\delta_{n-j}/\alpha)^{q-1}}
    \right]
    \kappa_{[1,j-5]|\{j-4\}}(\delta_{n-j}^\circ)  
    \\
    & \leq C \phi \sum_{j=\ceil{n/2}}^{n} \left[ 
        L_{4,\max} \frac{(\log(ep))^{3/2}}{\delta_{n-j}^3}
        + \nu_{q,\max} \frac{{(\log(ep))}^{(q-1)/2}}{(\delta_{n-j}/\alpha)^{q-1}}
    \right] \\
    & \quad \times \left[ \begin{aligned}
        & \tilde\kappa_{1,j-5} ~{L}_{3,\max}
        + \tilde\kappa_{2,j-5} ~{L}_{4,\max}^{1/2}
        + \tilde\kappa_{3,j-5} ~\nu_{q,\max}^{1/(q-2)}
        + \tilde\kappa_{4,j-5} ~\nu_{1,\max}
        + \tilde\kappa_{5} \delta_{n-j}^\circ \\
    \end{aligned} \right] 
    \\
    & \leq \frac{C\phi}{\sqrt{n}} \left[ 
        L_{4,\max} (\log(ep))^{3/2}
        + \nu_{q,\max} \frac{{(\log(ep))}^{(q-1)/2}}{(\delta/\alpha)^{q-4}}
    \right]\\
    & \quad \times \left[ \begin{aligned}
        & (\tilde\kappa_{1,n-5} ~{L}_{3,\max}
        + \tilde\kappa_{2,n-5} ~{L}_{4,\max}^{1/2}
        + \tilde\kappa_{3,n-5} ~\nu_{q,\max}^{1/(q-2)}
        + \tilde\kappa_{4,n-5} ~\nu_{1,\max}) 
        \sum_{j=\ceil{n/2}}^n \frac{1}{\delta_{n-j}^3} \\
        & + \tilde\kappa_5 \sqrt{\log(p n)} 
        \sum_{j=\ceil{n/2}}^{n} \frac{1}{\delta_{n-j}^2}
    \end{aligned} \right] 
    \\
    & \leq \frac{C\phi}{\sqrt{n}} \left[ 
        L_{4,\max} (\log(ep))^{3/2}
        + \nu_{q,\max} \frac{{(\log(ep))}^{(q-1)/2}}{(\delta/\alpha)^{q-4}}
    \right]\\
    & \quad \times \left[ \begin{aligned}
        & (\tilde\mu_{1,n-5} ~{L}_{3,\max} 
         +\tilde\mu_{2,n-5} ~{L}_{4,\max}^{1/2}
         +\tilde\mu_{3,n-5} ~\nu_{q,\max}^{1/(q-2)})
        \frac{1}{\delta} 
        +\frac{\log(ep)\sqrt{\log(pn)}}{\sigma_{\min}} \frac{\nu_{1,\max}}{\delta} \\
        & + \frac{\sqrt{\log(ep)\log(pn)}}{\sigma_{\min}} 
        \log\left(1 + \frac{\sqrt{n}\underline{\sigma}}{\delta}\right)
    \end{aligned} \right].
\end{aligned}
\end{equation*}
by \cref{eq:sum_delta_nolog} and that $\tilde\kappa_{1,n-5} = \mathfrak{C}_{1,\kappa} \tilde\mu_{1,n-5}$, $\tilde\kappa_{2,n-5} = \mathfrak{C}_{2,\kappa} \tilde\mu_{2,n-5}$, $\tilde\kappa_{3,n-5} = \mathfrak{C}_{3,\kappa} \tilde\mu_{3,n-5}$, $\tilde\kappa_{4,n-5} = \mathfrak{C}_{4,\kappa} \frac{\log(ep)\sqrt{\log(p(n-5))}}{\sigma_{\min}}$ and $\tilde\kappa_{5} = \mathfrak{C}_{5,\kappa} \frac{\sqrt{\log(ep)}}{\sigma_{\min}}$. 
Similarly,
\begin{equation*}
\begin{aligned}
    & C \phi \sum_{j=\ceil{n/2}}^{n} L_{3,\max}
    \sum_{k=\ceil{n/2}}^{j-2} \left[ 
        L_{3\max} \frac{(\log(ep))^{5/2}}{\delta_{n-k}^5}
        + \nu_{q,\max} \frac{(\log(ep))^{(q+2)/2}}{(\delta_{n-k}/\alpha)^{q+2}}
    \right] 
    \kappa_{[1,k-4]|\{k-3\}}(\delta_{n-k}^\circ) 
    \\
    & \leq C \phi \sum_{j=\ceil{n/2}}^{n} L_{3,\max}
    \sum_{k=\ceil{n/2}}^{j-2} \left[ 
        L_{3\max} \frac{(\log(ep))^{5/2}}{\delta_{n-k}^5}
        + \nu_{q,\max} \frac{(\log(ep))^{(q+2)/2}}{(\delta_{n-k}/\alpha)^{q+2}}
    \right] \\
    & \quad \times 
    \left[ \begin{aligned}
        & \tilde\kappa_{1,k-4} ~{L}_{3,\max}
        + \tilde\kappa_{2,k-4} ~{L}_{4,\max}^{1/2}
        + \tilde\kappa_{3,k-4} ~\nu_{q,\max}^{1/(q-2)}
        + \tilde\kappa_{4,k-4} ~\nu_{1,\max}
        + \tilde\kappa_{5} \delta_{n-k}^\circ \\
    \end{aligned} \right]  
    \\
    & \leq C \phi \sum_{j=\ceil{n/2}}^{n} L_{3,\max} \left[ 
        L_{3\max} \frac{(\log(ep))^{5/2}}{\delta_{n-j}^3}
        + \nu_{q,\max} \frac{(\log(ep))^{(q+2)/2}}{(\delta_{n-j}/\alpha)^{q}}
    \right] \\
    & \quad \times 
    \left[ \begin{aligned}
        & \tilde\kappa_{1,j-6} ~{L}_{3,\max}
        + \tilde\kappa_{2,j-6} ~{L}_{4,\max}^{1/2}
        + \tilde\kappa_{3,j-6} ~\nu_{q,\max}^{1/(q-2)}
        + \tilde\kappa_{4,j-6} ~\nu_{1,\max}
        + \tilde\kappa_{5} \delta_{n-j}^\circ \\
    \end{aligned} \right]  
    \\
    & \leq C \phi L_{3,\max} \left[ 
        L_{3\max} (\log(ep))^{5/2}
        + \nu_{q,\max} \frac{(\log(ep))^{(q+2)/2}}{(\delta/\alpha)^{q-3}}
    \right] \\
    & \quad \times \left[ \begin{aligned}
        & (\tilde\mu_{1,n-6} ~{L}_{3,\max} 
         +\tilde\mu_{2,n-6} ~{L}_{4,\max}^{1/2}
         +\tilde\mu_{3,n-6} ~\nu_{q,\max}^{1/(q-2)})
        \frac{1}{\delta} 
        +\frac{\log(ep)\sqrt{\log(pn)}}{\sigma_{\min}} \frac{\nu_{1,\max}}{\delta} \\
        & + \frac{\sqrt{\log(ep)\log(pn)}}{\sigma_{\min}} 
        \log\left(1 + \frac{\sqrt{n}\underline{\sigma}}{\delta}\right)
    \end{aligned} \right].
    \\
\end{aligned}
\end{equation*}
In sum, as long as $\delta \geq \nu_{1,\max} \sqrt{\log(ep)}$ and $\phi > 0$,
\begin{equation*}
\begin{aligned}
    & \sqrt{n} \mu_{[1,n]} \\
    & \leq \mathfrak{C}^{(4)} \phi \left[ \begin{aligned} 
        & {L}_{4,\max} \frac{(\log(ep))^{3/2}}{\underline{\sigma}^2 \delta} 
        + \nu_{q,\max} \frac{(\log(ep))^{(q-1)/2}}{\underline\sigma^2 (\delta/\alpha)^{q-3}} \\
        & + {L}_{3,\max} \left[ 
            {L}_{3,\max} \frac{(\log(ep))^{5/2}}{\underline{\sigma}^4 \delta}
            + \nu_{q,\max} \frac{(\log(ep))^{(q+2)/2}}{\underline{\sigma}^4 (\delta/\alpha)^{q-2}}
        \right] 
    \end{aligned} \right] \\
    & \quad \times \left[ \begin{aligned}
        \tilde\mu_{1,n-6} {L}_{3,\max}
        + \tilde\mu_{2,n-6} {L}_{4,\max}^{1/2}
        + \tilde\mu_{3,n-6} \nu_{q,\max}^{1/(q-2)}
    \end{aligned} \right] \\
    & \quad + \mathfrak{C}^{(4)} \left[
        \frac{\delta \log(ep)}{\sigma_{\min}}  
        + \frac{\sqrt{\log(ep)}}{\phi \sigma_{\min}}
        + {L}_{3,\max} \frac{(\log(ep))^{2}}{\underline{\sigma}^2 \sigma_{\min}} 
    \right]
    \\
    & \quad + \mathfrak{C}^{(4)} \phi \left[ \begin{aligned} 
        & {L}_{4,\max} \frac{(\log(ep))^{3/2}}{\underline{\sigma}^2} 
        + 2^{q-4} {\nu}_{q,\max} \frac{(\log(ep))^{(q-1)/2}}{\underline\sigma^2 (\delta/\alpha)^{q-4}} \\
        & + {L}_{3,\max} \left[ 
            {L}_{3,\max} \frac{(\log(ep))^{5/2}}{\underline{\sigma}^4}
            + 2^{q-3} {\nu}_{q,\max} \frac{(\log(ep))^{(q+2)/2}}{\underline{\sigma}^4 (\delta/\alpha)^{q-3}}
        \right] 
    \end{aligned} \right] \log\left(1 + \frac{\sqrt{n}\underline{\sigma}}{\delta}\right) 
    \\
\end{aligned}
\end{equation*}
where $\mathfrak{C}^{(4)}$ is a universal constant whose value does not change in this subsection.
Taking $\delta = \max\{ 8 \mathfrak{C}^{(4)}, \alpha \} \left( 
    \frac{{L}_{3,\max}}{\underline{\sigma}^2} \sqrt{\log(ep)}
    + \left(\frac{{L}_{4,\max}}{\underline{\sigma}^2}\right)^{\frac{1}{2}} 
    + \left(\frac{\nu_{q,\max}}{\underline{\sigma}^2}\right)^{\frac{1}{q-2}} 
\right) \sqrt{\log(ep)} \geq \nu_{1,\max} \sqrt{\log(ep)}$ and $\phi = \frac{1}{\delta \sqrt{\log(ep)}}$, 
\begin{equation*}
\begin{aligned}
    & \sqrt{n} \mu_{[1,n]} \\
    & \leq \frac{1}{2} \max_{j<n}\tilde\mu_{1,j} {L}_{3,\max}
    + \frac{1}{2} \max_{j<n}\tilde\mu_{2,j} {L}_{4,\max}^{1/2}
    + \frac{1}{2} \max_{j<n}\tilde\mu_{3,j} \nu_{q,\max}^{1/(q-2)} \\
    & \quad + \mathfrak{C}^{(5)} \left(
        {L}_{3,\max} \frac{(\log(ep))^{3/2}}{\underline{\sigma}^2}
        + {L}_{4,\max}^{1/2} \frac{\log(ep)}{\underline{\sigma}}
        + \nu_{q,\max}^{1/(q-2)} \frac{\log(ep)}{\underline{\sigma}^{2/(q-2)}}
    \right)
    \frac{\sqrt{\log(pn)}}{\sigma_{\min}} \log\left(e n\right) \\
\end{aligned}
\end{equation*}
for another universal constant $\mathfrak{C}^{(5)}$, whose value only depends on $\mathfrak{C}''$.
Taking $\mathfrak{C}_1 = \mathfrak{C}_2 = \mathfrak{C}_3 = \max\{2 \mathfrak{C}^{(5)}, 2\}$,
\begin{equation*}
\begin{aligned}
    \tilde\mu_{1,i} 
    & = \mathfrak{C}_1 \frac{(\log(ep))^{3/2}\sqrt{\log(pn)}}{\underline\sigma^2\sigma_{\min}} \log\left(e n\right), \\
    \tilde\mu_{2,i}
    & = \mathfrak{C}_2 \frac{\log(ep)\sqrt{\log(pn)}}{\underline\sigma \sigma_{\min}} \log\left(e n\right) \\
    \tilde\mu_{3,i}
    & = \mathfrak{C}_3 \frac{\log(ep)\sqrt{\log(pn)}}{\underline\sigma^{2/(q-2)} \sigma_{\min}} \log\left(e n\right)
\end{aligned}
\end{equation*}
satisfies
\begin{equation*}
\begin{aligned}
    \sqrt{n} \mu_{[1,n]}
    \leq \tilde\mu_{1,n} {L}_{3,\max}
    + \tilde\mu_{2,n} {L}_{4,\max}^{1/2}
    + \tilde\mu_{3,n} \nu_{q,\max}^{1/(q-2)},
\end{aligned}
\end{equation*}
which proves \eqref{eq:hypothesis_BE_4N} at $n$.
Finally, a mathematical induction over $n$ proves our theorem.

\subsection{Notation of modulo \texorpdfstring{$n$}{n}} \label{sec:Z_n}

To facilitate the notations under permutation arguments, we introduce modulo notations.
We denote $i$ modulo $n$ by $[i]_n$ and
%
define
\begin{equation*}
    [i,j]_n \equiv \{[i]_n, [i+1]_n, \dots, [j-1]_n, [j]_n\}.
\end{equation*}
We note $[i,j]_n$ is well-defined even if $i$ and/or $j$ are/is smaller than $1$ or larger than $n$. In \cref{sec:pf_1_dep_3P}, \cref{sec:pf_1_dep_4P} and relevant sections, by intervals in $[1,n]$, we indicate $[i,j]_n$ for any $i,j \in \ints$ satisfying $i \leq j$.
The other types of intervals, $(i,j]_n, [i,j)_n$ and $(i,j)_n$, are defined similarly.

For a $1$-ring dependent sequence $(X_1, \dots, X_n)$, we allow a slight notational conflict so that $X_{i} \equiv X_{[i]_n}$. It means that the next element of $n$ is $1$, which is the same as $n+1$ modulo $n$. Similarly, $X_{[i,j]} \equiv X_{[i,j]_n} = \sum_{k=i}^n X_{[k]_n}$. Let this notation extend to the indices of $L_{q,i}$ and $\nu_{q,i}$.
Given the modulo notation, the monotonicity of $\kappa$ described in \cref{thm:kappa_comparison} naturally extends to the cases where $i_1$ and/or $i_2$ are/is smaller than $1$ or larger than $n$, as long as $i_1 < i_2 < i_1+n$.

\subsection{Proof of Theorem \ref{thm:1_dep_berry_esseen_3P}} \label{sec:pf_1_dep_3P}

The proof for $1$-ring dependent cases with finite third moments is similar with the $1$-dependent cases in \cref{sec:pf_sketch}. In $1$-ring depent cases, we need to address the additional dependence between $X_1$ and $X_n$ and the average across the permutations in \cref{sec:permutation_argument}. 

\medskip
\noindent{\bf Breaking the ring.} First, we note that the Berry-Esseen bound under $1$-ring dependence can be reduced to the bound under $1$-dependence:
\begin{equation} \label{eq:breaking_ring}
\begin{aligned}
    & \abs*{ \Exp\left[\rho_{r,\phi}^\delta(X_{[1,n]}) - \rho_{r,\phi}^\delta(Y_{[1,n]})\right]} \\
    & \leq \abs*{ \Exp\left[\rho_{r,\phi}^\delta(X_{[1,n)}) - \rho_{r,\phi}^\delta(Y_{[1,n)})\right]} \\
    & \quad + \abs*{ \Exp\left[ \rho_{r,\phi}^\delta(X_{[1,n]}) - \rho_{r,\phi}^\delta(X_{[1,n)}) + \rho_{r,\phi}^\delta(Y_{[1,n]}) - \rho_{r,\phi}^\delta(Y_{[1,n)})\right]}.
\end{aligned}
\end{equation}
Note that we removed $X_n$ and $Y_n$ from $X_{[1,n]}$ and $Y_{[1,n]}$, respectively, to break the $1$-ring dependence. By the Taylor expansion centered at $X_{[1,n)}$,
\begin{equation} \label{eq:zeroth_taylor_expansion_3P}
\begin{aligned}
    & \rho_{r,\phi}^{\delta}(X_{[1,n]}) - \rho_{r,\phi}^\delta(X_{[1,n)}) \\
    & = \frac{1}{2} \inner*{\nabla^2 \rho_{r,\phi}^\delta(X_{(1,n-1)}), X_n^{\otimes 2}}
    + \inner*{\nabla^2 \rho_{r,\phi}^\delta(X_{(2,n-1)}), X_n \otimes X_1} \\
    & \quad + \inner*{\nabla^2 \rho_{r,\phi}^\delta(X_{(1,n-2)}), X_{n-1} \otimes X_n} 
    + \mathfrak{R}_X^{(3)},
\end{aligned}
\end{equation}
where $\mathfrak{R}_X^{(3)}$ is specified in \cref{sec:Taylor_expansion}. This is the same for $\rho_{r,\phi}^{\delta}(Y_{[1,n]}) - \rho_{r,\phi}^\delta(Y_{[1,n)})$ but with $Y$ in place of $X$.

\medskip
\noindent{\bf First Lindeberg swapping.} We bound $\abs*{ \Exp[\rho_{r,\phi}^\delta(X_{[1,n)})] - \Exp[\rho_{r,\phi}^\delta(Y_{[1,n)})]}$ by the Lindeberg swapping as in \cref{sec:mu_by_kappa_3N}. Here we define 
\begin{equation*}
    W_{[i,j]}^\cmpl \equiv X_{[1,i)} + Y_{(j,n)}.
\end{equation*}
Note that unlike \cref{sec:mu_by_kappa_3N}, the $n$-th element is removed. Then,
\begin{equation} \label{eq:first_lindeberg_result_3P}
\begin{aligned}
    & \sum_{j=1}^{n-1} \Exp \left[
        \rho_{r,\phi}^\delta(W^\cmpl_{[j,j]} + X_j)
        - \rho_{r,\phi}^\delta(W^\cmpl_{[j,j]} + Y_j) \right] 
    = \sum_{j=1}^{n-1} \Exp \left[ \mathfrak{R}_{X_j}^{(3,1)} 
        - \mathfrak{R}_{Y_j}^{(3,1)} \right],
\end{aligned}
\end{equation}
where $\mathfrak{R}_{X_j}^{(3,1)}$ and $\mathfrak{R}_{Y_j}^{(3,1)}$ are remainder terms of the Taylor expansions specified in \cref{sec:first_lindeberg_swapping}.

\medskip
\noindent{\bf Second moment decomposition and second Lindeberg swapping.}
To bound the second order terms in \cref{eq:zeroth_taylor_expansion_3P}, we re-apply the Lindeberg swapping. For simplicity, we only look at $\inner*{\nabla^2 \rho_{r,\phi}^\delta(X_{(1,n-1)}), X_n^{\otimes 2}}$, but similar arguments work for the other second moment terms. Because $\Exp[X_n^{\otimes 2}] = \Exp[Y_n^{\otimes 2}]$,
\begin{equation*}
\begin{aligned}
    & \inner*{ 
        \Exp\left[ \nabla^2 \rho_{r,\phi}^\delta(X_{(1,n-1)}) \right], 
        \Exp\left[ X_n^{\otimes 2} \right]
    }
    - \inner*{ 
        \Exp\left[ \nabla^2 \rho_{r,\phi}^\delta(Y_{(1,n-1)}) \right], 
        \Exp\left[ Y_n^{\otimes 2} \right]
    } \\
    & = \sum_{j=2}^{n-2} \inner*{ 
        \Exp\left[ 
            \nabla^2 \rho_{r,\phi}^\delta(X_{(1,j)} + X_j + Y_{(j,n-1)}) 
            -\nabla^2 \rho_{r,\phi}^\delta(X_{(1,j)} + Y_j + Y_{(j,n-1)})
        \right], 
        \Exp\left[ X_n^{\otimes 2} \right]
    },
\end{aligned}
\end{equation*}
By the Taylor expansion up to order $3$,
\begin{equation*}
\begin{aligned}
    & \inner*{ 
        \Exp\left[ \nabla^2 \rho_{r,\phi}^\delta(X_{(1,n-1)}) \right], 
        \Exp\left[ X_n^{\otimes 2} \right]}
    - \inner*{ 
        \Exp\left[ \nabla^2 \rho_{r,\phi}^\delta(Y_{(1,n-1)}) \right], 
        \Exp\left[ X_n^{\otimes 2} \right]} \\
    & = \sum_{j=2}^{n-2} \Exp\left[ \mathfrak{R}_{X_j}^{(3,2,1)} - \mathfrak{R}_{Y_j}^{(3,2,1)} \right],
\end{aligned}
\end{equation*}
where $\mathfrak{R}_{X_j}^{(3,2,1)}$ is the third-order remainder, specified in \cref{sec:second_lindeberg_swapping}. Similarly, 
\begin{equation*}
\begin{aligned}
    & \inner*{ 
        \Exp\left[ \nabla^2 \rho_{r,\phi}^\delta(X_{(2,n-1)}) \right], 
        \Exp\left[ X_n \otimes X_1 \right]}
    - \inner*{ 
        \Exp\left[ \nabla^2 \rho_{r,\phi}^\delta(Y_{(2,n-1)}) \right], 
        \Exp\left[ X_n \otimes X_1 \right]} \\
    & = \sum_{j=3}^{n-2} \Exp\left[ \mathfrak{R}_{X_j}^{(3,2,2)} - \mathfrak{R}_{Y_j}^{(3,2,2)} \right], \textand
\end{aligned}
\end{equation*}
\begin{equation*}
\begin{aligned}
    & \inner*{ 
        \Exp\left[ \nabla^2 \rho_{r,\phi}^\delta(X_{(1,n-2)}) \right], 
        \Exp\left[ X_n \otimes X_{n-1} \right]}
    - \inner*{ 
        \Exp\left[ \nabla^2 \rho_{r,\phi}^\delta(Y_{(1,n-2)}) \right], 
        \Exp\left[ X_n \otimes X_{n-1} \right]} \\
    & = \sum_{j=2}^{n-3} \Exp\left[ \mathfrak{R}_{X_j}^{(3,2,3)} - \mathfrak{R}_{Y_j}^{(3,2,3)} \right],
\end{aligned}
\end{equation*}
where $\mathfrak{R}_{W_j}^{(3,2,2)}$ and $\mathfrak{R}_{W_j}^{(3,2,3)}$ are simlarly derived third-order remainder terms.
Putting all the above terms together, 
\begin{equation*}
\begin{aligned}
    & \Exp\left[ \rho_{r,\phi}^{\delta}(X_{[1,n]}) - \rho_{r,\phi}^\delta(X_{[1,n)}) 
    - \rho_{r,\phi}^{\delta}(Y_{[1,n]}) + \rho_{r,\phi}^\delta(Y_{[1,n)}) \right] \\
    & = \Exp\left[ \mathfrak{R}_X^{(3)} - \mathfrak{R}_Y^{(3)} \right]
    + \sum_{j=2}^{n-2} \Exp\left[ \mathfrak{R}_{X_j}^{(3,2)} - \mathfrak{R}_{Y_j}^{(3,2)} \right],
\end{aligned}
\end{equation*}
where $\mathfrak{R}_{W_j}^{(3,2)} = \frac{1}{2} \mathfrak{R}_{W_j}^{(3,2,1)} + \mathfrak{R}_{W_j}^{(3,2,2)} + \mathfrak{R}_{W_j}^{(3,2,3)}$, and
\begin{equation} \label{eq:second_lindeberg_result_3P}
\begin{aligned}
    & \abs*{ \Exp\left[\rho_{r,\phi}^\delta(X_{[1,n]}) - \rho_{r,\phi}^\delta(Y_{[1,n]})\right]} \\
    & \leq \abs*{\Exp\left[\rho_{r,\phi}^\delta(X_{[1,n)}) - \rho_{r,\phi}^\delta(Y_{[1,n)})\right]} \\
    & \quad + \abs*{\Exp\left[\rho_{r,\phi}^{\delta}(X_{[1,n]}) - \rho_{r,\phi}^\delta(X_{[1,n)}) 
    - \rho_{r,\phi}^{\delta}(Y_{[1,n]}) + \rho_{r,\phi}^\delta(Y_{[1,n)})\right]} \\
    & \leq \abs*{\Exp\left[ \mathfrak{R}_X^{(3)} - \mathfrak{R}_Y^{(3)} \right]}
    + \sum_{j=1}^{n-1} \abs*{ \Exp \left[ 
        \mathfrak{R}_{X_j}^{(3,1)} 
        - \mathfrak{R}_{Y_j}^{(3,1)} 
    \right] }
    + \sum_{j=2}^{n-2} \abs*{ \Exp\left[ \mathfrak{R}_{X_j}^{(3,2)} - \mathfrak{R}_{Y_j}^{(3,2)} \right] }.
\end{aligned}
\end{equation}

\medskip
\noindent{\bf Remainder lemma.} Similar to \cref{sec:mu_by_kappa_3N}, the remainder terms $\mathfrak{R}_{W}^{(3)}$, $\mathfrak{R}_{W_j}^{(3,1)}$ and $\mathfrak{R}_{W_j}^{(3,2)}$ are upper bounded by conditional anti-concentration probability bounds. For $q > 0$, let
\begin{equation*}
    \tilde{L}_{q,j} \equiv \sum_{j'=j-3}^{j+3} L_{q,j'} \textand 
    \tilde{L}_{q,[k]_{j-2}} \equiv \sum_{k'=k-3}^{k+3} L_{q,[k']_{j-2}},
\end{equation*}
where $[k']_{j-2}$ is $k'$ modulo $j-2$, and $\tilde{\nu}_{q,j}$ and $\tilde{\nu}_{q,[k]_{j-2}}$ are similarly defined.

\begin{lemma} \label{thm:remainder_lemma_3P}
    There exist universal constants $C > 0$ and $\alpha > 0$ such that for any $n$, $1$-dependent sequence $(X_1, \dots, X_n)$ satisfying Assumption~\eqref{assmp:min_ev}, $j \in [1,n]$, $q \geq 3$, $\delta \geq \sigma_{\min}$ and $\phi \geq \frac{1}{\delta \log(ep)}$,
    \begin{equation*}
    \begin{aligned}
        \abs*{\Exp \left[ \mathfrak{R}_{W}^{(3)} \right]} 
        & \leq C \frac{(\log(ep))^{3/2}}{\delta^3}
        \left[ 
            \tilde{L}_{3,n} 
            + \phi^{\min\{1, q-3\}} \tilde\nu_{q,n} \frac{ {(\log(ep))}^{\max\{0, q-4\}/2}}{(\delta/\alpha)^{\max\{0, q-4\}}}
        \right] \\
        & \quad \times \min\{\kappa_{(3,n-3)|\{3,n-3\}}(\delta^\circ) + \kappa^\circ_{n}(\delta), 1\},
    \end{aligned}
    \end{equation*}
    \begin{equation*}
    \begin{aligned}
        \abs*{\Exp \left[ \mathfrak{R}_{W_j}^{(3,1)} \right]} 
        & \leq C \frac{(\log(ep))^{3/2}}{\delta_{n-j}^3}
        \left[ 
            \tilde{L}_{3,j} 
            + \phi^{\min\{1, q-3\}} \tilde\nu_{q,j} \frac{ {(\log(ep))}^{\max\{0, q-4\}/2}}{(\delta_{n-j}/\alpha)^{\max\{0, q-4\}}}
        \right] \\
        & \quad \times \min\{\kappa_{(3,j-3)|\{3,j-3\}}(\delta_{n-j}^\circ) + \kappa^\circ_{j}(\delta_{n-j}), 1\},
    \end{aligned}
    \end{equation*}
    \begin{equation*}
    \begin{aligned}
        \abs*{\Exp \left[ \mathfrak{R}_{W_j}^{(3,2)} \right]}
        & \leq C \frac{(\log(ep))^{3/2}}{\delta_{n-j}^3} \left[ \begin{aligned}
            \tilde{L}_{3,j}
            + \phi^{\min\{1, q-3\}} (\tilde\nu_{q,j} + \nu_{q,n}) \frac{ {(\log(ep))}^{\max\{0, q-4\}/2}}{(\delta_{n-j}/\alpha)^{\max\{0, q-4\}}}
        \end{aligned} \right] \\
        & \quad \times \min\{\kappa_{(3,j-3)|\{3,j-3\}}(\delta_{n-j}^\circ) + \kappa^\circ_{j}(\delta_{n-j}), 1\},
    \end{aligned}
    \end{equation*}
    where $W$ represents either $X$ or $Y$, $\delta_{n-j}^2 \equiv \delta^2 + \underline{\sigma}^2 \max\{n-j,0\}$, $\delta^\circ_{n-j} \equiv 12 \delta_{n-j} \sqrt{\log(pn)}$ and $\kappa^\circ_{j}(\delta) \equiv \frac{\delta \log(ep)}{\sigma_{\min} \sqrt{\max\{j, 1\}}}$.
\end{lemma}

\medskip
\noindent{\bf Permutation argument.} 
We apply the permutation argument to \cref{eq:breaking_ring} as in \cref{eq:first_lindeberg_swapping_P}:
\begin{equation*}
\begin{aligned}
    & \abs*{ \Exp\left[\rho_{r,\phi}^\delta(X_{[1,n]}) - \rho_{r,\phi}^\delta(Y_{[1,n]})\right]} \\
    & \leq \frac{1}{n} \sum_{j^\circ=1}^n \abs*{ \Exp\left[
        \rho_{r,\phi}^\delta(X_{j^\circ+[1,n)}) - \rho_{r,\phi}^\delta(Y_{j^\circ+[1,n)})
    \right]} \\
    & \quad + \frac{1}{n} \sum_{j^\circ=1}^n \abs*{ \Exp\left[ \rho_{r,\phi}^\delta(X_{[1,n]}) - \rho_{r,\phi}^\delta(X_{j^\circ+[1,n)}) + \rho_{r,\phi}^\delta(Y_{[1,n]}) - \rho_{r,\phi}^\delta(Y_{j^\circ+[1,n)})\right]}.
\end{aligned}
\end{equation*}
where $j^\circ+[1,n)$ is the shifted interval of $[1, n)$ by $j^\circ$, namely, $\{j^\circ+1, \dots, j^\circ+n-1\}$. 
Together with the results in \cref{thm:remainder_lemma_3P},
\begin{equation*}
\begin{aligned}
    & \abs*{ \Exp\left[\rho_{r,\phi}^\delta(X_{[1,n]}) - \rho_{r,\phi}^\delta(Y_{[1,n]})\right]} \\
    & \leq \frac{C}{n} \sum_{j^\circ=1}^n \frac{(\log(ep))^{3/2}}{\delta^3} 
    \left[ 
        \tilde{L}_{3,j^\circ} 
        + \phi^{\min\{1, q-3\}} \tilde\nu_{q,j^\circ} \frac{ {(\log(ep))}^{\max\{0, q-4\}/2}}{(\delta/\alpha)^{\max\{0, q-4\}}}
    \right]  \\
    & \quad \times \min\{1, \kappa_{j^\circ+(3,n-3)|\{j^\circ+3,j^\circ+n-3\}}(\delta^\circ) + \kappa^\circ_{n}(\delta)\} \\
    & \quad + \frac{C}{n} \sum_{j^\circ=1}^n \sum_{j=1}^{n-1} \frac{(\log(ep))^{3/2}}{\delta_{n-j}^3} 
    \left[ 
        \tilde{L}_{3,j^\circ} 
        + \phi^{\min\{1, q-3\}} \tilde\nu_{q,j^\circ} \frac{ {(\log(ep))}^{\max\{0, q-4\}/2}}{(\delta_{n-j}/\alpha)^{\max\{0, q-4\}}}
    \right] \\ 
    & \quad \times \min\{1, \kappa_{j^\circ+(3,j-3)|\{j^\circ+3,j^\circ+j-3\}}(\delta_{n-j}^\circ) + \kappa^\circ_{j}(\delta_{n-j})\} \\
    & \quad + \frac{C}{n} \sum_{j^\circ=1}^n \sum_{j=1}^{n-1} \frac{(\log(ep))^{3/2}}{\delta_{n-j}^3}  
    \left[ 
        \tilde{L}_{3,j^\circ+j}
        + \phi^{\min\{1, q-3\}} \tilde\nu_{q,j^\circ+j} \frac{ {(\log(ep))}^{\max\{0, q-4\}/2}}{(\delta_{n-j}/\alpha)^{\max\{0, q-4\}}}
    \right] \\
    & \quad \times \min\{1, \kappa_{j^\circ+(3,j-3)|\{j^\circ+3,j^\circ+j-3\}}(\delta_{n-j}^\circ) + \kappa^\circ_{j}(\delta_{n-j})\}. \\
\end{aligned}
\end{equation*}

\medskip
\noindent{\bf Partitioning the sum.} 
We partition the summation over $j$ at 
$n/2$. A notable distinction from \cref{sec:mu_by_kappa_3N} is that we should also take averages over $j^\circ$ alongside the summations over $j$.
For $j < n/2$, 
\begin{equation*}
\begin{aligned}
    & \frac{C}{n} \sum_{j^\circ=1}^n \sum_{j < n/2} \frac{(\log(ep))^{3/2}}{\delta_{n-j}^3}  
    \left[ 
        \tilde{L}_{3,j^\circ+j}
        + \phi^{\min\{1, q-3\}} \tilde\nu_{q,j^\circ+j} \frac{ {(\log(ep))}^{\max\{0, q-4\}/2}}{(\delta_{n-j}/\alpha)^{\max\{0, q-4\}}}
    \right] \\
    & \quad \times \min\{1, \kappa_{j^\circ+(3,j-3)|\{j^\circ+3,j^\circ+j-3\}}(\delta_{n-j}^\circ) + \kappa^\circ_{j}(\delta_{n-j})\} \\
    & \leq \frac{C}{n} \sum_{j^\circ=1}^n \sum_{j < n/2} \frac{(\log(ep))^{3/2}}{\delta_{n-j}^3}  
    \left[ 
        \tilde{L}_{3,j^\circ+j}
        + \phi^{\min\{1, q-3\}} \tilde\nu_{q,j^\circ+j} \frac{ {(\log(ep))}^{\max\{0, q-4\}/2}}{(\delta_{n-j}/\alpha)^{\max\{0, q-4\}}}
    \right] \\
    & \leq C \sum_{j < n/2} \frac{(\log(ep))^{3/2}}{\delta_{n-j}^3} 
    \left[ 
        \bar{L}_{3}
        + \phi^{\min\{1, q-3\}} \bar\nu_{q} \frac{ {(\log(ep))}^{\max\{0, q-4\}/2}}{(\delta_{n-j}/\alpha)^{\max\{0, q-4\}}}
    \right] \\
    & \leq \frac{C}{\sqrt{n}} \frac{(\log(ep))^{5/2}}{\underline{\sigma}^2\sigma_{\min}}
    \left[ 
        \bar{L}_{3}
        + \phi^{\min\{1, q-3\}} \bar\nu_{q} \frac{ {(\log(ep))}^{\max\{0, q-4\}/2}}{(\delta/\alpha)^{\max\{0, q-4\}}}
    \right],
\end{aligned}
\end{equation*}
because of \cref{eq:sum_delta_3}.
For $j \geq n/2$,
\begin{equation*}
\begin{aligned}
    & \frac{C}{n} \sum_{j^\circ=1}^n \sum_{j \geq n/2} \frac{(\log(ep))^{3/2}}{\delta_{n-j}^3}  
    \left[ 
        \tilde{L}_{3,j^\circ+j}
        + \phi^{\min\{1, q-3\}} \tilde\nu_{q,j^\circ+j} \frac{ {(\log(ep))}^{\max\{0, q-4\}/2}}{(\delta_{n-j}/\alpha)^{\max\{0, q-4\}}}
    \right] \\
    & \quad \times \min\{1, \kappa_{j^\circ+(3,j-3)|\{j^\circ+3,j^\circ+j-3\}}(\delta_{n-j}^\circ) + \kappa^\circ_{j}(\delta_{n-j})\}. \\
    & \leq \frac{C}{n} \sum_{j^\circ=1}^n \sum_{j \geq n/2}  \frac{(\log(ep))^{3/2}}{\delta_{n-j}^3}
    \left[ 
        \tilde{L}_{3,j^\circ+j}
        + \phi^{\min\{1, q-3\}} \tilde\nu_{q,j^\circ+j} \frac{ {(\log(ep))}^{\max\{0, q-4\}/2}}{(\delta_{n-j}/\alpha)^{\max\{0, q-4\}}}
    \right] \\
    & \quad \times \kappa_{j^\circ+(3,j-3)|\{j^\circ+3,j^\circ+j-3\}}(\delta_{n-j}^\circ) \\
    & \quad + C \sum_{j \geq n/2} \frac{(\log(ep))^{3/2}}{\delta_{n-j}^3}
    \left[ 
        \bar{L}_{3}
        + \phi^{\min\{1, q-3\}} \bar\nu_{q} \frac{ {(\log(ep))}^{\max\{0, q-4\}/2}}{(\delta_{n-j}/\alpha)^{\max\{0, q-4\}}}
    \right]
    \kappa^\circ_{j}(\delta_{n-j}).
\end{aligned}
\end{equation*}
The last term is upper bounded by
\begin{equation*}
\begin{aligned}
    & C \sum_{j \geq n/2} \frac{(\log(ep))^{3/2}}{\delta_{n-j}^3}
    \left[ 
        \bar{L}_{3}
        + \phi^{\min\{1, q-3\}} \bar\nu_{q} \frac{ {(\log(ep))}^{\max\{0, q-4\}/2}}{(\delta_{n-j}/\alpha)^{\max\{0, q-4\}}}
    \right]
    \kappa^\circ_{j}(\delta_{n-j}) \\
    & \leq \frac{C}{\sqrt{n}} \frac{(\log(ep))^{5/2}}{\underline{\sigma}^2\sigma_{\min}}
    \left[ 
        \bar{L}_{3}
        + \phi^{\min\{1, q-3\}} \bar\nu_{q} \frac{ {(\log(ep))}^{\max\{0, q-4\}/2}}{(\delta/\alpha)^{\max\{0, q-4\}}}
    \right]
    \log\left(1 + \frac{\sqrt{n}\underline{\sigma}}{\delta}\right)
\end{aligned}
\end{equation*}
because of \cref{eq:sum_delta_3_kappa}.
%
%
In sum, we obtain the following induction from $\kappa_{I|I'}(\delta)$ for $I \subsetneq [1,n]$ to
\begin{equation*}
    \mu_{[1,n]} = \mu(X_{[1,n]}, Y_{[1,n]}).
\end{equation*}
\begin{lemma} \label{thm:induction_lemma_BE_3P}
    There exists a universal constant $C > 0$ such that for any $n$, $1$-ring dependent sequence $(X_i \in \reals^p: i \in [1,n])$ satisfying Assumptions \eqref{assmp:min_var}, \eqref{assmp:min_ev} and \eqref{assmp:var_ev}, $q \geq 3$, $\delta \geq \sigma_{\min}$ and $\phi \geq \frac{1}{\delta \log(ep)}$,
    \begin{equation*}
    \begin{aligned}
        & \mu_{[1,n]} \\
        & \leq \frac{C}{\sqrt{n}} \frac{\delta \log(ep)}{\sigma_{\min}}  
        + \frac{C}{\sqrt{n}} \frac{\sqrt{\log(ep)}}{\phi \sigma_{\min}} \\
        & \quad + \frac{C}{\sqrt{n}} \frac{(\log(ep))^{5/2}}{\underline{\sigma}^2\sigma_{\min}} 
        \left[ 
            \bar{L}_{3}
            + \phi^{\min\{1, q-3\}} \bar\nu_{q} \frac{ {(\log(ep))}^{\max\{0, q-4\}/2}}{(\delta/\alpha)^{\max\{0, q-4\}}}
        \right]
        \log\left(1 + \frac{\sqrt{n}\underline{\sigma}}{\delta}\right)\\
        & \quad + \frac{C}{n} \sum_{j^\circ=1}^n \sum_{j=\ceil{n/2}}^{n} \frac{(\log(ep))^{3/2}}{\delta_{n-j}^3}
        \left[ 
            \tilde{L}_{3,j^\circ}
            + \phi^{\min\{1, q-3\}} \tilde\nu_{q,j^\circ} \frac{ {(\log(ep))}^{\max\{0, q-4\}/2}}{(\delta_{n-j}/\alpha)^{\max\{0, q-4\}}}
        \right] \\
        & \quad \times \kappa_{j^\circ+(3,j-3)|\{j^\circ+3,j^\circ+j-3\}}(\delta_{n-j}^\circ) \\
        & \quad + \frac{C}{n} \sum_{j^\circ=1}^n \sum_{j=\ceil{n/2}}^{n-1} \frac{(\log(ep))^{3/2}}{\delta_{n-j}^3}
        \left[ 
            \tilde{L}_{3,j^\circ+j}
            + \phi^{\min\{1, q-3\}} \tilde\nu_{q,j^\circ+j} \frac{ {(\log(ep))}^{\max\{0, q-4\}/2}}{(\delta_{n-j}/\alpha)^{\max\{0, q-4\}}}
        \right] \\
        & \quad \times \kappa_{j^\circ+(3,j-3)|\{j^\circ+3,j^\circ+j-3\}}(\delta_{n-j}^\circ),
    \end{aligned}
    \end{equation*}
    for some absolute constant $C > 0$.
\end{lemma}

\medskip
\noindent{\bf Anti-concentration inequality.} 
%
We now proceed to obtain an induction from $\mu$ to $\kappa$. For the sake of simplicity, we argue the upperbound of $\kappa_{[1,i]|\{0, i+1\}}(\delta)$ for $i \in [6, n)$, where $(X_1, \dots, X_n)$ is $1$-ring dependent and $\delta > 0$, but the arguments easily extend to arbitrary $i_1, i_2 \in \ints$ satisfying $i_1 < i_2 < i_1+n-1$. 

First, using the monotonicity of $\kappa$ (see \cref{thm:kappa_comparison}), we obtain
\begin{equation*}
    \kappa_{[1, i]|\{0, i+1\}}(\delta) 
    \leq \frac{1}{i^\circ-3} \sum_{j=0}^{i^\circ-4} \kappa_{[j+1, j+i^\circ]|\{j, j+i^\circ+1\}}(\delta),
\end{equation*}
where $i^\circ \equiv \floor{\frac{i+4}{2}} \geq 5$.
This step is necessary for the permutation argument in the subsequent induction step. Then using a similar smoothing and Taylor expansion to those in \cref{sec:kappa_by_mu_N}, we obtain for any $j \in [0, i^\circ-4]$, $r \in \reals^p$ and $\delta, \epsilon, h > 0$, 
\begin{equation} \label{eq:anti_concentration_Taylor_P}
\begin{aligned}
    & \Pr[X_{[j+1, j+i^\circ]} \in A_{r,\delta} | \mathscr{X}_{\{j, j+i^\circ+1\}}] \\
    & \leq \Exp[\varphi_{r,\delta+\vareps^o}^\vareps(X_{[j+1, j+i^\circ]}) 
    | \mathscr{X}_{\{j, j+i^\circ+1\}}] 
    + \frac{1}{h^4} \\
    & \leq \Exp\left[ 
        \varphi_{r,\delta+\vareps^o}^\vareps(X_{[j+1, j+i^\circ]} - X_{\{j+2,j+i^\circ-1\}}) | \mathscr{X}_{\{j, j+i^\circ+1\}} 
    \right]
    + \Exp[ \mathfrak{R}_{X_{\{j+2, j+i^\circ-1\}}}^{(1)} | \mathscr{X}_{\{j, j+i^\circ+1\}} ]
    + \frac{1}{h^4} \\
    & \leq \Exp[\Pr[
        X_{[j+3,j+i^\circ-2]} \in A_{r_1, \delta + 2\vareps^o} 
        | \mathscr{X}_{\{j+1, j+i^\circ\}}
    ] | \mathscr{X}_{\{j, j+i^\circ+1\}} ] 
    + \Exp[ \mathfrak{R}_{X_{\{j+2, j+i^\circ-1\}}}^{(1)} | \mathscr{X}_{\{j, j+i^\circ+1\}} ]
    + \frac{2}{h^4} \\
    & \leq C\frac{\delta + 20\vareps\sqrt{\log(ph)}}{\sigma_{\min}} \sqrt{\frac{\log(ep)}{i^\circ-4}}
    + 2\mu_{[j+3,j+i^\circ-2]} + \Exp[ \mathfrak{R}_{X_{\{j+2, j+i^\circ-1\}}}^{(1)} | \mathscr{X}_{\{j, j+i^\circ+1\}} ] + \frac{2}{h^4},
\end{aligned}
\end{equation}
where $\epsilon^\circ \equiv 10 \epsilon \sqrt{\log(ph)}$, and $r_1 \equiv r - X_{\{j+1, j+i^\circ\}}$. Bounding the remainder term $\Exp[ \mathfrak{R}_{X_{\{j+2, j+i^\circ-1\}}}^{(1)} | \mathscr{X}_{\{j, j+i^\circ+1\}} ]$ similarly to \cref{sec:kappa_by_mu_N}, we obtain the following lemma.

\begin{lemma} 
\label{thm:induction_lemma_AC_P}
    There exists a universal constant $C > 0$ such that for any $n$, $1$-ring dependent sequence $(X_i \in \reals^p: i \in [1,n])$ satisfying Assumption \eqref{assmp:min_var}, $i_1, i_2 \in \ints$ satisfying $i_1+6 \leq i_2 < i_1+n-1$, $\delta > 0$ and $\vareps \geq \sigma_{\min}$,
    \begin{equation*} 
    \begin{aligned}
        & \kappa_{[i_1, i_2]|\{i_1-1, i_2+1\}}(\delta) \\
        & \leq \frac{C}{i^\circ-3} \sum_{j=0}^{i^\circ-4} 
        \frac{\sqrt{\log(ep)}}{\vareps} 
        ~ (\nu_{1,j+i_1+1} + \nu_{1,j+i_1+i^\circ-2}) 
        ~ \kappa_{j+i_1+[2,i^\circ-3]|\{j+i_1+1, j+i_1+i^\circ-2\}}(\vareps^\circ)\\
        & \quad + \frac{C}{i^\circ-3} \sum_{j=0}^{i^\circ-4} \mu_{j+i_1+[2,i^\circ-3]} 
        + \min\left\{1, C\frac{\delta + 2\vareps^\circ}{\sigma_{\min}} \sqrt{\frac{\log(ep)}{i^\circ-4}} \right\} \\
        & \quad + \frac{C}{\sigma_{\min}} \bar{\nu}_{1,(i_1,i_2)} \frac{\log(ep)}{\sqrt{i^\circ-4}},
    \end{aligned}
    \end{equation*}
    where $i^\circ \equiv \floor{\frac{\abs{[i_1,i_2]}+4}{2}}$ and $\vareps^\circ \equiv 20\vareps \sqrt{\log(p(i^\circ-4))}$. 
\end{lemma}

\medskip
\noindent{\bf Dual Induction.}  In this part we use the dual induction to prove the following lemma.
\begin{lemma}
    There exist positive universal constants $\mathfrak{C}_{1,\kappa}$, $\mathfrak{C}_{2,\kappa}$, $\mathfrak{C}_{3,\kappa}$, $\mathfrak{C}_{4,\kappa}$, $\mathfrak{C}_{1,\mu}$ and $\mathfrak{C}_{2,\mu}$ such that for any $n$, $1$-ring dependent sequence $(X_i \in \reals^p: i \in [1,n])$ satisfying Assumptions \eqref{assmp:min_var}, \eqref{assmp:min_ev} and \eqref{assmp:var_ev} and $\delta \geq 0$,
    \begin{equation} \label{eq:hypothesis_AC_3P} \tag{HYP-AC-3}
    \begin{aligned}
        & \sqrt{\abs{I}} \kappa_{I|\{i_1-1, i_2+1\}}(\delta)
        \leq 
        \tilde\kappa_{1,\abs{I}} \bar{L}_{3,I}
        + \tilde\kappa_{2,\abs{I}} \bar\nu_{q,I}^{1/(q-2)}
        + \tilde\kappa_{3,\abs{I}} \bar\nu_{1,I}
        + \tilde\kappa_{4} \delta, \\
        & \forall I = [i_1, i_2] ~~\text{s.t.}~~ i_1 < i_2 < i_1+n,
    \end{aligned}
    \end{equation}
    \begin{equation} \label{eq:hypothesis_BE_3P} \tag{HYP-BE-3}
        \sqrt{n} \mu_{[1,n]}
        \leq \tilde\mu_{1,n} \bar{L}_{3}
        + \tilde\mu_{2,n} \bar\nu_{q}^{1/(q-2)},
    \end{equation}
    where $\tilde\kappa_{1,i} = \mathfrak{C}_{1,\kappa} \tilde\mu_{1,i}$, 
    $\tilde\kappa_{2,i} = \mathfrak{C}_{2,\kappa} \tilde\mu_{2,i}$, 
    $\tilde\kappa_{3,i} = \mathfrak{C}_{3,\kappa} \frac{\log(ep)\sqrt{\log(pi)}}{\sigma_{\min}}$,
    $\tilde\kappa_{4} = \mathfrak{C}_{4,\kappa} \frac{\sqrt{\log(ep)}}{\sigma_{\min}}$,
    \begin{equation*}
    \begin{aligned}
        \tilde\mu_{1,n} 
        & = \mathfrak{C}_{1,\mu} \frac{(\log(ep))^{3/2}\sqrt{\log(pn)}}{\underline\sigma^2\sigma_{\min}} \log\left(e n\right), \\
        \tilde\mu_{2,n}
        & = \mathfrak{C}_{2,\mu} \frac{\log(ep)\sqrt{\log(pn)}}{\underline\sigma^{2/(q-2)} \sigma_{\min}} \log\left(e n\right).
    \end{aligned}
    \end{equation*}
\end{lemma}
If $\mathfrak{C}_{1,\kappa}$, $\mathfrak{C}_{2,\kappa}$, $\mathfrak{C}_{3,\kappa}$, $\mathfrak{C}_{4,\kappa}$, $\mathfrak{C}_{1,\mu}$, $\mathfrak{C}_{2,\mu} \geq 2$, then \eqref{eq:hypothesis_BE_3P} and \eqref{eq:hypothesis_AC_3P}, requiring $\mu_{[1,n]} \leq 1$ and $\kappa_{I|\{i_1-1,i_2+1\}}(\delta) \leq 1$ almost surely for all $I \subsetneq [1,n]$ only, trivially holds for $n \leq 16$. Now we consider the case of $n > 16$. Suppose that the induction hypotheses hold for all smaller $n$. 

We first derive \eqref{eq:hypothesis_AC_3P} for any $I = [i_1, i_2]$ satisfying $i_1 < i_2 < i_1+n$. For the case of $i \leq 16$, \eqref{eq:hypothesis_AC_3P} trivially holds given $\mathfrak{C}_{1,\kappa}, \mathfrak{C}_{2,\kappa}, \mathfrak{C}_{3,\kappa}, \mathfrak{C}_{4,\kappa} \geq 2$. For the case of $i > 16$, we first assume that \eqref{eq:hypothesis_AC_3P} holds for all $I$ satisfying $\abs{I} < i$ and then show this extends to all $I$ with $\abs{I} = i$. Without loss of generality, we only show for the interval $[1,i]$.
By \cref{thm:induction_lemma_AC_P}, 
for any $\vareps \geq \sigma_{\min}$ and $\delta > 0$,
\begin{equation*} 
\begin{aligned}
    & \kappa_{[1,i]|\{0, i+1\}}(\delta) \\
    & \leq \frac{C}{i^\circ-3} \sum_{j=0}^{i^\circ-4}  
    \frac{\sqrt{\log(ep)}}{\vareps} (\nu_{1,j+2} + \nu_{1,j+i^\circ-1}) 
    \min\{1, \kappa_{j+[3,i^\circ-2]|\{j+2, j+i^\circ-1\}}(\vareps^\circ)\} \\
    & \quad + \frac{C}{i^\circ-3} \sum_{j=0}^{i^\circ-4} \mu_{j+[3,i^\circ-2]}
    + C\frac{\delta + 2\vareps^\circ}{\sigma_{\min}} \sqrt{\frac{\log(ep)}{i^\circ-4}}
    + C\frac{\bar\nu_{1,(1,i)}}{\sigma_{\min}} \frac{\log(ep)}{\sqrt{i^\circ-4}},
\end{aligned}
\end{equation*}
where $i^\circ = \floor{\frac{i+4}{2}}$, $\vareps^\circ = 20\vareps \sqrt{\log(p(i^\circ-4))}$. 
Since $(X_1, \dots, X_n)$ is a $1$-ring dependent sequence satisfying Assumptions \eqref{assmp:min_var}, \eqref{assmp:min_ev} and \eqref{assmp:var_ev} and the interval $j+[3,i^\circ-2]$ is a proper subset of $[1,n]$, \eqref{eq:hypothesis_AC_3P} holds for $\kappa_{j+[3,i^\circ-2]|\{j+2, j+i^\circ-1\}}(\vareps^\circ)$. Furthermore, $(X_{j+2}, \dots, X_{j+i^\circ-1})$ is a $1$-ring dependent sequence satisfying Assumptions \eqref{assmp:min_var}, \eqref{assmp:min_ev} and \eqref{assmp:var_ev} with the same $\sigma_{\min}$ and $\underline{\sigma}$ as the original data $(X_1, \dots, X_n)$. We formalize this fact into the following lemma:
\begin{lemma} \label{thm:assmp_induction_P}
    Suppose that $(X_1, \dots, X_n)$ is a $1$-ring dependent sequence satisfying Assumptions \eqref{assmp:min_var}, \eqref{assmp:min_ev} and \eqref{assmp:var_ev}. Then for any $i_1$ and $i_2$ satisfying $i_1 \leq i_2 \leq i_1+n$, $(X_{i_1}, X_{i_1+1}, \dots, X_{i_2})$ is $1$-ring dependent and satisfies Assumptions \eqref{assmp:min_var}, \eqref{assmp:min_ev} and \eqref{assmp:var_ev} with the same $\sigma_{\min}$ and $\underline{\sigma}$ as the original data.
\end{lemma}
As a result, \eqref{eq:hypothesis_BE_3P} holds for $\mu_{j+[3,i^\circ-2]}$. Putting the resulting upperbounds for $\kappa_{j+[3,i^\circ-2]|\{j+2, j+i^\circ-1\}}(\vareps^\circ)$ and $\mu_{j+[3,i^\circ-2]}$ back to the previous upperbound for $\kappa_{[1,i]|\{0,i+1\}}$, we obtain
\begin{equation*} 
\begin{aligned}
    & \kappa_{[1,i]|\{0,i+1\}}(\delta) \\
    & \leq \frac{C}{(i^\circ-4)^{3/2}} \sum_{j=0}^{i^\circ-4}  
    \frac{\sqrt{\log(ep)}}{\vareps} (\nu_{1,j+2} + \nu_{1,j+i^\circ-1})
    \\
    & \quad \times \left[ 
        \tilde\kappa_{1,i^\circ-4} \bar{L}_{3,j+[3,i^\circ-2]}
        + \tilde\kappa_{2,i^\circ-4} \bar{\nu}_{q,j+[3,i^\circ-2]}^{1/(q-2)}
        + \tilde\kappa_{3,i^\circ-4} \bar\nu_{1,j+[3,i^\circ-2]}
        + \tilde\kappa_{4} {\vareps^\circ}
    \right]
    \\
    & \quad + \frac{C}{(i^\circ-4)^{3/2}} \sum_{j=0}^{i^\circ-4} \tilde\mu_{1,i^\circ-4} {\bar{L}_{3,j+[3,i^\circ-2]}}
    + \frac{C}{(i^\circ-4)^{3/2}} \sum_{j=0}^{i^\circ-4} \tilde\mu_{2,i^\circ-4} {\bar{L}_{4,j+[3,i^\circ-2]}^{1/2}} \\
    &\quad + \frac{C}{(i^\circ-4)^{3/2}} \sum_{j=0}^{i^\circ-4} \tilde\mu_{3,i^\circ-4} {\bar{\nu}_{q,j+[3,i^\circ-2]}^{1/(q-2)}}
    + C\frac{\delta + 2\vareps^\circ}{\sigma_{\min}} \sqrt{\frac{\log(ep)}{i^\circ-4}}
    + C\frac{\bar\nu_{1,(1,i)}}{\sigma_{\min}} \frac{\log(ep)}{\sqrt{i^\circ-4}}.
\end{aligned}
\end{equation*}
To provide an upper bound in terms of $\bar{L}_3$ and $\bar{\nu}_q$, we use the following lemma based on Jensen's inequality.
\begin{lemma} \label{thm:jensen} 
    Suppose that $j, k \geq \frac{n}{2}$. For any $q_1, q_2, q_3 > 0$ and $\alpha \leq 1$,
    \begin{equation*}
    \begin{aligned}
        \frac{1}{n} \sum_{j^\circ=1}^{n} L_{q_1,j^\circ+j} \bar{L}_{q_2,(j^\circ,j^\circ+j)}^\alpha
        \leq C \bar{L}_{q_1} \bar{L}_{q_2}^\alpha 
    \end{aligned}
    \end{equation*}
    and
    \begin{equation*}
    \begin{aligned}
        \frac{1}{n} \sum_{j^\circ=1}^{n} \frac{L_{q_1,j^\circ+j}}{j-1} 
        \sum_{k^\circ=1}^{j-1} L_{q_2,j^\circ+[k^\circ+k]_{j-1}} \bar{L}_{q_3, j^\circ+(l,k+l)_{j-1}}^\alpha
        \leq C \bar{L}_{q_1} \bar{L}_{q_2} \bar{L}_{q_3}^\alpha, \\
    \end{aligned}
    \end{equation*}
    where $j^\circ+(k^\circ,k^\circ+k)_{j-1}$ is the shifted interval of $(k^\circ,k^\circ+k)_{j-1}$ by $j^\circ$, namely, $\{j^\circ+[k^\circ+1]_{j-1}, \dots, j^\circ+[k^\circ+k-1]_{j-1}\}$. The same inequality holds when some $L$ is replaced with $\nu$.
\end{lemma}

According to the lemma,
\begin{equation*}
\begin{aligned}
    \frac{1}{i^\circ-3} \sum_{j=0}^{i^\circ-4} \bar{L}_{3,j+[3,i^\circ-2]}
    & \leq C \bar{L}_{3,(2,i-1)},
    \\
    \frac{1}{i^\circ-3} \sum_{j=0}^{i^\circ-4}  (\nu_{1,j+2} + \nu_{1,j+i^\circ-1})
    \bar{L}_{3,j+[3,i^\circ-2]}
    & \leq C \bar\nu_{1,(1,i)} \bar{L}_{3,(2,i-1)}, 
    \\
    \frac{1}{i^\circ-3} \sum_{j=0}^{i^\circ-4}  (\nu_{1,j+2} + \nu_{1,j+i^\circ-1})
    \bar\nu_{1,j+[3,i^\circ-2]}
    & \leq C \bar\nu_{1,(1,i)} \bar\nu_{1,(2,i-1)}.
\end{aligned}
\end{equation*}
Similar inequalities hold with $\bar{\nu}_{q}^{1/(q-2)}$ in place of $\bar{L}_3$.
As a result, we obtain a similar recursive inequality on $\tilde\kappa$'s with \cref{eq:recursion_kappa_3N} except that $L_{3,\max}$, $\nu_{q,\max}$ and $\nu_{1,\max}$ are replaced with $\bar{L}_{q,[1,i]}$, $\bar{\nu}_{q,[1,i]}$ and $\bar{\nu}_{1,[1,i]}$, respectively. 
\begin{equation*} 
\begin{aligned}
    & \sqrt{i} \kappa_{[1,i]|\{0,i+1\}}(\delta) \\
    & \leq \mathfrak{C}' \frac{\sqrt{\log(ep)}}{\vareps} \bar\nu_{1,(1,i)}
    \left[ 
        \tilde\kappa_{1,i^\circ-4} \bar{L}_{3,(2,i-1)}
        + \tilde\kappa_{2,i^\circ-4} \bar{\nu}_{q,(2,i-1)}^{1/(q-2)}
        + \tilde\kappa_{3,i^\circ-4} \bar\nu_{1,(2,i-1)}
        + \tilde\kappa_{4} \vareps^\circ
    \right]
    \\
    & \quad + \mathfrak{C}'  \left[ 
        \tilde\mu_{1,i^\circ-4} {\bar{L}_{3,(2,i-1)}}
        + \tilde\mu_{2,i^\circ-4} {\bar{\nu}_{q,(2,i-1)}^{1/(q-2)}} 
        + \frac{\delta + 2\vareps^\circ}{\sigma_{\min}} \sqrt{\log(ep)}
        + \frac{\bar\nu_{1,(1,i)}}{\sigma_{\min}} \log(ep)
    \right],
\end{aligned}
\end{equation*}
for some universal constant $\mathfrak{C}'$, whose value does not change over lines.
Plugging in $\vareps = \max\{2 \mathfrak{C}', 1\} \sqrt{\log(ep)} \bar{\nu}_{1,(1,i)} \overset{\text{\cref{eq:nu_1_vs_sigma_min}}}{\geq} \sigma_{\min}$,
\begin{equation*}
\begin{aligned}
    & \sqrt{i} \kappa_{[1,i]|\{0,i+1\}}(\delta) \\
    & \leq \frac{1}{2}
    \left[ 
        \tilde\kappa_{1,i^\circ-4} {\bar{L}_{3,(2,i-1)}}
        + \tilde\kappa_{2,i^\circ-4} {\bar{\nu}_{q,(2,i-1)}^{1/(q-2)}}
        + \tilde\kappa_{3} {\bar\nu_{1,(2,i-1)}}
    \right] 
    + 20 \mathfrak{C}' \tilde\kappa_4 \sqrt{\log(ep)\log(pi^\circ)} \bar\nu_{(1,i)} \\
    & \quad  
    + \mathfrak{C}' \left[ \begin{aligned}
        \tilde\mu_{1,i^\circ-4} \bar{L}_{3,(2,i-1)}
        + \tilde\mu_{2,i^\circ-4} \bar\nu_{q,(2,i-1)}^{1/(q-2)} 
        + \frac{\log(ep)(1+40\sqrt{\log(pi^\circ)})}{\sigma_{\min}} \bar{\nu}_{1,(1,i)}
        + \frac{\sqrt{\log(ep)}}{\sigma_{\min}} \delta
    \end{aligned} \right] \\
    & \leq \tilde\kappa_{1,i} \bar{L}_{3,[1,i]}
    + \tilde\kappa_{2,i} \bar\nu_{q,[1,i]}^{1/(q-2)}
    + \tilde\kappa_{3,i} \bar\nu_{2,[1,i]}^{1/2}
    + \tilde\kappa_{4} \delta,
\end{aligned}
\end{equation*}
where $\tilde\kappa_{1,i} = \mathfrak{C}_{1,\kappa} \tilde\mu_{1,i}$, 
$\tilde\kappa_{2,i} = \mathfrak{C}_{2,\kappa} \tilde\mu_{2,i}$, 
$\tilde\kappa_{3,i} = \mathfrak{C}_{3,\kappa} \frac{\log(ep)\sqrt{\log(pi)}}{\sigma_{\min}}$, and
$\tilde\kappa_{4} = \mathfrak{C}_{4,\kappa} \frac{\sqrt{\log(ep)}}{\sigma_{\min}}$, provided by $\mathfrak{C}_{1,\kappa} = \mathfrak{C}_{2,\kappa} = \max\{2 \mathfrak{C}', 2\}$, $\mathfrak{C}_{3,\kappa} = \max\{82 \mathfrak{C}' + 40 \mathfrak{C}'^2, 2\}$ and $\mathfrak{C}_{4,\kappa} = \max\{ \mathfrak{C}', 2\}$. 
A generalization the above argument to any interval with length $i$ and a mathematical induction for $i < n$ proves \eqref{eq:hypothesis_AC_3P} at $n$.

\medskip

The proof of \eqref{eq:hypothesis_BE_3P} at $n$ also proceeds similarly using \cref{thm:jensen}. 
%
We first upper bound the last two terms in \cref{thm:induction_lemma_BE_3P}: for $\delta \geq \sigma_{\min}$,
\begin{equation*}
\begin{aligned}
    & \frac{C}{n} \sum_{j^\circ=1}^n \sum_{j=\ceil{n/2}}^{n} \frac{(\log(ep))^{3/2}}{\delta_{n-j}^3}
    \left[ 
        \tilde{L}_{3,j^\circ}
        + \phi^{\min\{1, q-3\}} \tilde\nu_{q,j^\circ} \frac{ {(\log(ep))}^{\max\{0, q-4\}/2}}{(\delta_{n-j}/\alpha)^{\max\{0, q-4\}}}
    \right] \\
    & \times \kappa_{j^\circ+(3,j-3)|\{j^\circ+3,j^\circ+j-3\}}(\delta_{n-j}^\circ) \\
    & + \frac{C}{n} \sum_{j^\circ=1}^n \sum_{j=\ceil{n/2}}^{n-1} \frac{(\log(ep))^{3/2}}{\delta_{n-j}^3}
    \left[ 
        \tilde{L}_{3,j^\circ+j}
        + \phi^{\min\{1, q-3\}} \tilde\nu_{q,j^\circ+j} \frac{ {(\log(ep))}^{\max\{0, q-4\}/2}}{(\delta_{n-j}/\alpha)^{\max\{0, q-4\}}}
    \right] \\
    &\times \kappa_{j^\circ+(3,j-3)|\{j^\circ+3,j^\circ+j-3\}}(\delta_{n-j}^\circ) \\
    & \equiv \left[ \mathfrak{T}_{1,1} + \mathfrak{T}_{1,2} \right] 
    + \left[ \mathfrak{T}_{2,1} + \mathfrak{T}_{2,2} \right].
\end{aligned}
\end{equation*}
Since $(X_1, \dots, X_n)$ is a $1$-ring dependent sequence satisfying Assumptions \eqref{assmp:min_var}, \eqref{assmp:min_ev} and \eqref{assmp:var_ev} and the interval $j^\circ+(3,j-3)$ (modulo $n$) is a proper subset of $[1,n]$, \eqref{eq:hypothesis_AC_3P} holds for $\kappa_{j^\circ+(3,j-3)|\{j^\circ+3, j^\circ+j-3\}}(\delta_{n-j}^\circ)$.
Hence
\begin{equation*}
\begin{aligned}
    \mathfrak{T}_{2,1} 
    & = \frac{C}{n} \sum_{j^\circ=1}^{n} \sum_{j=\ceil{n/2}}^{n-1} \tilde L_{3,j^\circ+j} \frac{(\log(ep))^{3/2}}{\delta_{n-j}^3} 
    \kappa_{j^\circ+(3,j-3)|\{j^\circ+3,j^\circ+j-3\}}(\delta_{n-j}^\circ)
    \\
    & \leq \frac{C}{n} \sum_{j^\circ=1}^{n} \sum_{j=\ceil{n/2}}^{n-1} \tilde L_{3,j^\circ+j} \frac{(\log(ep))^{3/2}}{\delta_{n-j}^3 \sqrt{j-7}} \\
    & \quad \times \left[ \begin{aligned}
        & \tilde\kappa_{1,j-7} \bar{L}_{3,j^\circ+(3,j-3)}
        + \tilde\kappa_{2,j-7} \bar\nu_{q,j^\circ+(3,j-3)}^{1/(q-2)}
        + \tilde\kappa_{3,j-7} \bar\nu_{1,j^\circ+(3,j-3)}
        + \tilde\kappa_{4} \delta_{n-j}^\circ \\
    \end{aligned} \right] \\
\end{aligned}
\end{equation*}
Based on \cref{thm:jensen}, we obtain
\begin{equation*}
\begin{aligned}
    & \frac{C}{n} \sum_{j^\circ=1}^{n} \sum_{j=\ceil{n/2}}^{n-1} \tilde L_{3,j^\circ+j} \frac{(\log(ep))^{3/2}}{\delta_{n-j}^3 \sqrt{j-7}} \\
    & \quad \times \left[ \begin{aligned}
        & \tilde\kappa_{1,j-7} \bar{L}_{3,j^\circ+(3,j-3)}
        + \tilde\kappa_{2,j-7} \bar\nu_{q,j^\circ+(3,j-3)}^{1/(q-2)}
        + \tilde\kappa_{3,j-7} \bar\nu_{1,j^\circ+(3,j-3)}
        + \tilde\kappa_{4} \delta_{n-j}^\circ \\
    \end{aligned} \right]
    \\
    & \leq C \sum_{j=\ceil{n/2}}^{n} \frac{(\log(ep))^{3/2}}{\delta_{n-j}^3 \sqrt{j}} \bar{L}_{3} \\
    & \quad \times \left[ \begin{aligned}
        \tilde\kappa_{1,j-7} \bar{L}_{3}
        + \tilde\kappa_{2,j-7} \bar\nu_{q}^{1/(q-2)}
        + \tilde\kappa_{3,j-7} \bar\nu_{1}
        + \tilde\kappa_{4} \delta_{n-j}^\circ \\
    \end{aligned} \right]
    \\
    & \leq \frac{C}{\sqrt{n}} (\log(ep))^{3/2} \bar{L}_{3} \\
    & \quad \times \left[ \begin{aligned}
        (\tilde\kappa_{1,n-1} \bar{L}_{3} 
        + \tilde\kappa_{2,n-1} \bar\nu_{q}^{1/(q-2)}
        + \tilde\kappa_{3,n-1} \bar\nu_{1}) 
        \sum_{j=\ceil{n/2}}^n \frac{1}{\delta_{n-j}^3} 
        + \tilde\kappa_4 \sqrt{\log(p n)} 
        \sum_{j=\ceil{n/2}}^{n} \frac{1}{\delta_{n-j}^2}
    \end{aligned} \right].
\end{aligned}
\end{equation*}
We recall that $\tilde\kappa_{1,n-1} = \mathfrak{C}_{1,\kappa} \tilde\mu_{1,n-1}$, $\tilde\kappa_{2,n-1} = \mathfrak{C}_{2,\kappa} \tilde\mu_{2,n-1}$, $\tilde\kappa_{3,n-1} = \mathfrak{C}_{3,\kappa} \tilde\mu_{3,n-1}$, $\tilde\kappa_{4,n-1} = \mathfrak{C}_{4,\kappa} \frac{\log(ep)\sqrt{\log(p(n-1))}}{\sigma_{\min}}$ and $\tilde\kappa_{5} = \mathfrak{C}_{5,\kappa} \frac{\sqrt{\log(ep)}}{\sigma_{\min}}$. 
Based on \cref{eq:sum_delta_nolog}, we obtain
\begin{equation*}
\begin{aligned}
    \mathfrak{T}_{2,1} 
    \leq \frac{C}{\sqrt{n}} \frac{(\log(ep))^{3/2}}{\underline{\sigma}^2} \bar{L}_{3}
    \left[ \begin{aligned}
        & (\tilde\mu_{1,n-1} \bar{L}_{3} 
         +\tilde\mu_{2,n-1} \bar\nu_{q}^{1/(q-2)})
        \frac{1}{\delta} 
        +\frac{\log(ep)\sqrt{\log(pn)}}{\sigma_{\min}} \frac{\bar\nu_{1}}{\delta} \\
        & + \frac{\sqrt{\log(ep)\log(pn)}}{\sigma_{\min}} 
        \log\left(1 + \frac{\sqrt{n}\underline{\sigma}}{\delta}\right)
    \end{aligned} \right].
\end{aligned}
\end{equation*}
Similarly,
\begin{equation*}
\begin{aligned}
    \mathfrak{T}_{2,2} 
    & \leq \frac{C}{\sqrt{n}} \frac{(\log(ep))^{\max\{3, q-1\}/2}}{\underline\sigma^2 (\delta/\alpha)^{\max\{0, q-4\}}}
    \phi^{\min\{1, q-3\}} \bar\nu_{q}  
    \left[ \begin{aligned}
        & (\tilde\mu_{1,n-1} \bar{L}_{3} 
         +\tilde\mu_{2,n-1} \bar\nu_{q}^{1/(q-2)})
        \frac{1}{\delta} 
        +\frac{\log(ep)\sqrt{\log(pn)}}{\sigma_{\min}} \frac{\bar\nu_{1}}{\delta} \\
        & + \frac{\sqrt{\log(ep)\log(pn)}}{\sigma_{\min}} 
        \log\left(1 + \frac{\sqrt{n}\underline{\sigma}}{\delta}\right)
    \end{aligned} \right].
\end{aligned}
\end{equation*}
The same arguments and bounds apply to $\mathfrak{T}_{1,1}$ and $\mathfrak{T}_{1,2}$. 
In sum, as long as $\delta \geq \bar\nu_1 \sqrt{\log(ep)}  \overset{\text{\cref{eq:nu_1_vs_sigma_min}}}{\geq} \sigma_{\min} $ and $\phi > 0$,
\begin{equation*}
\begin{aligned}
    & \sqrt{n} \mu_{[1,n]} \\
    & \leq \mathfrak{C}'' \frac{(\log(ep))^{3/2}}{\underline\sigma^2 \delta} \left[ 
        \bar{L}_{3} + \phi^{\min\{1, q-3\}} \bar\nu_{q} \frac{ {(\log(ep))}^{\max\{0, q-4\}/2}}{(\delta/\alpha)^{\max\{0, q-4\}}}
    \right] \\
    & \quad \times \left[ \begin{aligned}
        \tilde\mu_{1,n-1} \bar{L}_{3}
        + \tilde\mu_{2,n-1} \bar\nu_{q}^{1/(q-2)}
    \end{aligned} \right] 
    + \mathfrak{C}'' \left[ 
        \frac{\delta \log(ep)}{\sigma_{\min}}  
        + \frac{\sqrt{\log(ep)}}{\phi \sigma_{\min}}
    \right] \\
    & \quad + \mathfrak{C}'' \frac{(\log(ep))^{3/2}}{\underline{\sigma}^2} \log\left(en\right)
    \left[ 
        \bar{L}_{3} + \phi^{\min\{1, q-3\}} \bar\nu_{q} \frac{ {(\log(ep))}^{\max\{0, q-4\}/2}}{(\delta/\alpha)^{\max\{0, q-4\}}}
    \right] 
    \frac{\sqrt{\log(ep)\log(pn)}}{\sigma_{\min}}
\end{aligned}
\end{equation*}
where $\mathfrak{C}''$ is a universal constant whose value does not change over lines.
Taking $\delta = \frac{\max\{4 \mathfrak{C}'', 2\}}{\sqrt{\log(ep)}} \left( 
    \frac{\bar{L}_{3}}{\underline{\sigma}^2} (\log(ep))^2
    + \left(\frac{\bar\nu_{q}}{\underline{\sigma}^2} (\log(ep))^{\max\{2, q-2\}} \right)^{\frac{1}{q-2}} 
\right) \geq \bar\nu_{1} \sqrt{\log(ep)}$
and $\phi = \frac{1}{\delta \sqrt{\log(ep)}}$, 
\begin{equation*}
\begin{aligned}
    \sqrt{n} \mu_{[1,n]}
    & \leq \frac{1}{2} \max_{j<n}\tilde\mu_{1,j} \bar{L}_{3}
    + \frac{1}{2} \max_{j<n}\tilde\mu_{2,j} \bar\nu_{q}^{1/(q-2)} \\
    & \quad + \mathfrak{C}^{(3)} \left(
        \bar{L}_{3} \frac{(\log(ep))^2}{\underline{\sigma}^2}
        + \bar\nu_{q}^{1/(q-2)} \frac{(\log(ep))^{\max\{2/(q-2),1\}}}{\underline{\sigma}^{2/(q-2)}}
    \right)
    \frac{\sqrt{\log(p n)}}{\sigma_{\min}} \log\left(e n\right) \\
\end{aligned}
\end{equation*}
for another universal constant $\mathfrak{C}^{(3)}$,whose value only depends on $\mathfrak{C}''$.
Taking $\mathfrak{C}_1 = \mathfrak{C}_2 = \max\{2 \mathfrak{C}^{(3)}, 1\}$,
\begin{equation*}
\begin{aligned}
    \tilde\mu_{1,n} 
    & = \mathfrak{C}_1 \frac{(\log(ep))^{2}\sqrt{\log(pn)}}{\underline\sigma^2\sigma_{\min}} \log\left(e n\right), \\
    \tilde\mu_{2,n}
    & = \mathfrak{C}_2 \frac{(\log(ep))^{\max\{2/(q-2),1\}}\sqrt{\log(pn)}}{\underline\sigma^{2/(q-2)} \sigma_{\min}} \log\left(e n\right)
\end{aligned}
\end{equation*}
satisfies
\begin{equation*}
\begin{aligned}
    \sqrt{n} \mu_{[1,n]}
    \leq \tilde\mu_{1,n} \bar{L}_{3}
    + \tilde\mu_{2,n} \bar\nu_{q}^{1/(q-2)},
\end{aligned}
\end{equation*}
which proves \eqref{eq:hypothesis_BE_3P} at $n$. By a mathematical induction, the induction hypotheses hold for all $n$, and it concludes our proof.

\subsection{Proof of Theorem \ref{thm:1_dep_berry_esseen_4P}} \label{sec:pf_1_dep_4P}

We recall \cref{eq:second_lindeberg_result_3P} from \cref{sec:pf_1_dep_3P}:
\begin{equation*}
\begin{aligned}
    & \abs*{ \Exp\left[\rho_{r,\phi}^\delta(X_{[1,n]}) - \rho_{r,\phi}^\delta(Y_{[1,n]})\right]} \\
    & \leq \abs*{\Exp\left[ \mathfrak{R}_X^{(3)} - \mathfrak{R}_Y^{(3)} \right]}
    + \sum_{j=1}^{n-1} \abs*{ \Exp \left[ 
        \mathfrak{R}_{X_j}^{(3,1)} 
        - \mathfrak{R}_{Y_j}^{(3,1)} 
    \right] }
    + \sum_{j=2}^{n-2} \abs*{ \Exp\left[ \mathfrak{R}_{X_j}^{(3,2)} - \mathfrak{R}_{Y_j}^{(3,2)} \right] }.
\end{aligned}
\end{equation*}
In \cref{sec:pf_sketch_4N} we improved the rate by decomposing the third order remainder $\mathfrak{R}^{(3,1)}_{W_j}$ when the fourth moments were finite. Namely,
\begin{equation*}
\begin{aligned}
    & \sum_{j=1}^{n-1} \abs*{ \Exp \left[ 
        \mathfrak{R}_{X_j}^{(3,1)} 
        - \mathfrak{R}_{Y_j}^{(3,1)} 
    \right] }\\
    & \leq \frac{C}{\sqrt{n}} \bar{L}_{3} \frac{(\log(ep))^2}{\underline\sigma^2 \sigma_{\min}}
    + \sum_{j=1}^{n-1} \abs*{ \Exp \left[ \mathfrak{R}_{X_j}^{(4,1)} 
        - \mathfrak{R}_{Y_j}^{(4,1)} \right] }
    + \sum_{j=3}^{n-1} \sum_{k=1}^{j-2} \abs*{ \Exp \left[ 
        \mathfrak{R}_{X_j,X_k}^{(6,1)}
        - \mathfrak{R}_{X_j,Y_k}^{(6,1)}
    \right] }.
\end{aligned}
\end{equation*}
In $1$-ring dependence cases, there exists additional third-order remainder terms $\mathfrak{R}^{(3)}_{W}$ and $\mathfrak{R}^{(3,2)}_{X_j,W_k}$, which came up while breaking the $1$-ring dependence in \cref{sec:pf_1_dep_3P}. Based on Tyalor expansions up to order $4$ and the second moment matching between $X_j$ and $Y_j$, we decompose those additional remainder terms. First,
\begin{equation} \label{eq:decompose_third_remainder_P}
\begin{aligned}
    & \mathfrak{R}_X^{(3)} \\
    & = \frac{1}{6} \inner*{\nabla^3 \rho_{r,\phi}^\delta(X_{(1,n-1)}), X_n^{\otimes 3}}
    + \inner*{\nabla^3 \rho_{r,\phi}^\delta(X_{(1,n-3)}), X_{n-2} \otimes X_{n-1} \otimes X_n} \\
    & \quad + \inner*{\nabla^3 \rho_{r,\phi}^\delta(X_{(2,n-2)}),  X_{n-1} \otimes X_{n} \otimes X_1} 
    + \inner*{\nabla^3 \rho_{r,\phi}^\delta(X_{(3,n-1)}),  X_{n} \otimes X_{1} \otimes X_2} \\
    & \quad + \frac{1}{2} \inner*{
        \nabla^3 \rho_{r,\phi}^\delta(X_{(1,n-2)}), 
        X_{n-1} \otimes X_n \otimes (X_{n-1} + X_n)
    } \\
    & \quad + \frac{1}{2} \inner*{
        \nabla^3 \rho_{r,\phi}^\delta(X_{(2,n-1)}), 
        X_{n} \otimes X_1 \otimes (X_1 + X_n)
    } \\
    & \quad + \mathfrak{R}_X^{(4)},
\end{aligned}
\end{equation}
where $\mathfrak{R}_X^{(4)}$ is specified in \cref{sec:Taylor_expansion}. This is the same for $\rho_{r,\phi}^{\delta}(Y_{[1,n]}) - \rho_{r,\phi}^\delta(Y_{[1,n)})$ but with $Y$ in place of $X$. 
To bound the third-order moment terms, we re-apply the Lindeberg swapping and the Taylor expansion up to the sixth order as we did to $ \inner*{ 
    \Exp[ \nabla^3 \rho_{r,\phi}^\vareps(X_{[1,j-1)} + Y_{(j+1,n)}) ], 
    \Exp[ X_j^{\otimes 3} ]
}$ in \cref{sec:pf_sketch_4N}. As a result,
\begin{equation*}
\begin{aligned}
    & \abs*{ \Exp\left[ \mathfrak{R}_X^{(3)} - \mathfrak{R}_Y^{(3)} \right] } 
    \leq \frac{C}{\sqrt{n}} \bar{L}_{3} \frac{(\log(ep))^2}{\underline\sigma^2 \sigma_{\min}}
    + \abs*{ \Exp\left[ \mathfrak{R}_X^{(4)} - \mathfrak{R}_Y^{(4)} \right] }
    + \sum_{k=2}^{n-2} \abs*{\Exp \left[ 
        \mathfrak{R}_{X,X_k}^{(6)}
        - \mathfrak{R}_{X,Y_k}^{(6)}
    \right] },
\end{aligned}
\end{equation*}
where $\mathfrak{R}_{X,W_k}^{(6)}$ is the sixth-order remainder term specified in \cref{sec:Taylor_expansion}.

For $\sum_{j=2}^{n-2} \abs*{ \Exp\left[ \mathfrak{R}_{X_j}^{(3,2)} - \mathfrak{R}_{Y_j}^{(3,2)} \right] }$, by the Taylor expansion up to the fourth order,
\begin{equation*}
    \sum_{j=2}^{n-2} \Exp\left[ \mathfrak{R}_{X_j}^{(3,2)} - \mathfrak{R}_{Y_j}^{(3,2)} \right]
    = \sum_{j=2}^{n-2} \Exp\left[ \mathfrak{R}_{X_j}^{(4,2)} - \mathfrak{R}_{Y_j}^{(4,2)} \right],
\end{equation*}
where $\mathfrak{R}_{X_j}^{(4,2)}$ is the fourth-order remainder, specified in \cref{sec:second_lindeberg_swapping}. 
Putting all the above results together, we get
\begin{equation} \label{eq:decompose_third_result_P}
\begin{aligned}
    & \abs*{ \Exp[\rho_{r,\phi}^\delta(X_{[1,n]})] - \Exp[\rho_{r,\phi}^\delta(Y_{[1,n]})]} \\
    & \leq \frac{C}{\sqrt{n}} \bar{L}_{3} \frac{(\log(ep))^2}{\underline\sigma^2 \sigma_{\min}}
    + \abs*{ \mathfrak{R}_X^{(4)} - \mathfrak{R}_Y^{(4)} }
    + \sum_{j=1}^{n-1} \abs*{ \Exp \left[ \mathfrak{R}_{X_j}^{(4,1)} 
        - \mathfrak{R}_{Y_j}^{(4,1)} \right] }
    + \sum_{j=2}^{n-2} \abs*{ \Exp\left[ \mathfrak{R}_{X_j}^{(4,2)} - \mathfrak{R}_{Y_j}^{(4,2)} \right] }\\
    & \quad + \sum_{k=2}^{n-2} \abs*{ \Exp \left[ 
        \mathfrak{R}_{X,X_k}^{(6)}
        - \mathfrak{R}_{X,Y_k}^{(6)}
    \right] }
    + \sum_{j=3}^{n-1} \sum_{k=1}^{j-2} \abs*{ \Exp \left[ 
        \mathfrak{R}_{X_j,X_k}^{(6,1)}
        - \mathfrak{R}_{X_j,Y_k}^{(6,1)}
    \right] }.
\end{aligned}
\end{equation}

\medskip
\noindent{\bf Remainder lemma.} Similar to \cref{sec:mu_by_kappa_3N}, the remainder terms are upper bounded by conditional anti-concentration probability bounds. 
For $q > 0$, let
\begin{equation*}
    \tilde{L}_{q,j} \equiv \sum_{j'=j-4}^{j+4} L_{q,j'} \textand 
    \tilde{L}_{q,[k]_{j-2}} \equiv \sum_{k'=k-4}^{k+4} L_{q,[k']_{j-2}},
\end{equation*}
where $[k']_{j-2}$ is $k'$ modulo $j-2$, and $\tilde{\nu}_{q,j}$ and $\tilde{\nu}_{q,[k]_{j-2}}$ are similarly defined.

\begin{lemma} \label{thm:remainder_lemma_4P}
    Suppose that Assumption \eqref{assmp:min_ev} holds. For $W$ representing either $X$ or $Y$ and $j, k \in \ints_n$ such that $k \leq j-2$,
    \begin{equation*}
    \begin{aligned}
        \abs*{\Exp \left[ \mathfrak{R}_{W}^{(4)} \right]} 
        & \leq C \phi \left[ 
            \tilde{L}_{4,n} \frac{(\log(ep))^{3/2}}{\delta^3}
            + \tilde\nu_{q,n} \frac{ {(\log(ep))}^{(q-1)/2}}{\delta^{q-1}}
        \right] \\
        & \quad \times \min\{\kappa_{(4,n-4)|\{4,n-4\}}(\delta^\circ) + \kappa^\circ_n(\delta), 1\},
    \end{aligned}
    \end{equation*}
    \begin{equation*}
    \begin{aligned}
        \abs*{\Exp \left[ \mathfrak{R}_{W_j}^{(4,1)} \right]} 
        & \leq C \phi \left[ 
            \tilde{L}_{4,j} \frac{(\log(ep))^{3/2}}{\delta_{n-j}^3}
            + \tilde\nu_{q,j} \frac{ {(\log(ep))}^{(q-1)/2}}{(\delta_{n-j}/\alpha)^{q-1}}
        \right] \\
        & \quad \times \min\{\kappa_{(4,j-4)|\{4,j-4\}}(\delta_{n-j}^\circ) + \kappa^\circ_{j}(\delta_{n-j}), 1\},
    \end{aligned}
    \end{equation*}
    \begin{equation*}
    \begin{aligned}
        \abs*{\Exp \left[ \mathfrak{R}_{W_j}^{(4,2)} \right]}
        & \leq C \phi \left[ \begin{aligned}
            \left(\tilde{L}_{4,j} + \tilde L_{4,n} \right) \frac{(\log(ep))^{3/2}}{\delta_{n-j}^3}
            + \left(\tilde\nu_{q,j} + \tilde\nu_{q,n} \right) \frac{ {(\log(ep))}^{(q-1)/2}}{(\delta_{n-j}/\alpha)^{q-1}}
        \end{aligned} \right] \\
        & \quad \times \min\{\kappa_{(4,j-4)|\{4,j-4\}}(\delta_{n-j}^\circ) + \kappa^\circ_{j}(\delta_{n-j}), 1\},
    \end{aligned}
    \end{equation*}
    \begin{equation*}
    \begin{aligned}
        \abs*{\Exp \left[ \mathfrak{R}_{X,W_k}^{(6)} \right]} 
        & \leq C \phi \tilde L_{3,n} \left[ \begin{aligned}
            \tilde L_{3,[k]_{n-2}} \frac{(\log(ep))^{5/2}}{\delta_{n-k}^5}
            + \tilde\nu_{q,[k]_{n-2}} \frac{ {(\log(ep))}^{(q+2)/2}}{(\delta_{n-k}/\alpha)^{q+2}}
        \end{aligned} \right] \\
        & \quad \times \min\{\kappa_{(4,k-4)|\{4,k-4\}}(\delta_{n-k}^\circ) + \kappa^\circ_{k}(\delta_{n-k}), 1\}.
    \end{aligned}
    \end{equation*}
    \begin{equation*}
    \begin{aligned}
        \abs*{\Exp \left[ \mathfrak{R}_{X_j,W_k}^{(6,1)} \right]} 
        & \leq C \phi \tilde L_{3,j} \left[ \begin{aligned}
            \tilde L_{3,[k]_{j-2}} \frac{(\log(ep))^{5/2}}{\delta_{n-k}^5}
            + \tilde\nu_{q,[k]_{j-2}} \frac{ {(\log(ep))}^{(q+2)/2}}{(\delta_{n-k}/\alpha)^{q+2}}
        \end{aligned} \right] \\
        & \quad \times \min\{\kappa_{(4,k-4)|\{4,k-4\}}(\delta_{n-k}^\circ) + \kappa^\circ_{k}(\delta_{n-k}), 1\}.
    \end{aligned}
    \end{equation*}
    where $\delta_{n-j}^2 \equiv \delta^2 + \underline{\sigma}^2 \max\{n-j,0\}$, $\delta^\circ_{n-j} \equiv 12 \delta_{n-j} \sqrt{\log(pn)}$ and $\kappa^\circ_{j}(\delta) \equiv \frac{\delta \log(ep)}{\sigma_{\min} \sqrt{\max\{j, 1\}}}$, as long as $\delta \geq \sigma_{\min}$ and $\phi\delta \geq \frac{1}{\log(ep)}$.
\end{lemma}

\medskip
\noindent{\bf Permutation argument.} Recall the third Lindeberg swapping under $1$-dependence in \cref{eq:third_lindeberg_swapping_N,eq:third_lindeberg_result_N}: 
\begin{equation*}
\begin{aligned}
    & \Big\langle 
        \Exp\left[ \nabla^3 \rho_{r,\phi}^\delta(X_{[1,j-1)} + Y_{(j+1,n]}) \right] 
    - \Exp\left[ \nabla^3  \rho_{r,\phi}^\delta(Y_{[1,j-1)} + Y_{(j+1,n]}) \right], \Exp[ X_j^{\otimes 3} ]\Big\rangle. \\
    & = \sum_{k=1}^{j-2} \Big\langle
        \Exp\left[ \nabla^3 \rho_{r,\phi}^\delta(X_{[1,k)} + X_k + Y_{(k,j-1) \cup (j+1,n]}) \right] \\ 
    & \omit{\hfill $- \Exp\left[ \nabla^3   \rho_{r,\phi}^\delta(X_{[1,k)} + Y_k + Y_{(k,j-1) \cup (j+1,n]}) \right], 
        \Exp[ X_j^{\otimes 3} ]
    \Big\rangle.$} \\
    & = \sum_{j=3}^{n} \sum_{k=1}^{j-2} \Exp \left[ 
        \mathfrak{R}_{X_j,X_k}^{(6,1,1)}
        - \mathfrak{R}_{X_j,Y_k}^{(6,1,1)}
    \right]. \\
\end{aligned}
\end{equation*}
Under $1$-ring dependence, we apply the third Lindeberg swapping to $\rho_{r,\phi}^{\delta}(X_{[1,j-1)} + Y_{(j+1,n)})$ after breaking the ring, rather than $\rho_{r,\phi}^{\delta}(X_{[1,j-1)} + Y_{(j+1,n]})$. Then we apply \cref{thm:remainder_lemma_4P}toon the resulting remainder terms and obtain
\begin{equation*}
\begin{aligned}
    & \abs*{\inner*{ 
        \Exp[\nabla^3 \rho_{r,\phi}^{\delta}(X_{[1,j-1)} + Y_{(j+1,n)})
        - \nabla^3 \rho_{r,\phi}^{\delta}(Y_{[1,j-1)} + Y_{(j+1,n)})], 
        \Exp[X_j^{\otimes 3}] }} \\
    & \leq \sum_{k=1}^{j-2} \Big| \Big\langle 
        \Exp[\nabla^3 \rho_{r,\phi}^{\delta}(X_{[1,k)} + X_k + Y_{(k,j-1)\cup(j+1,n)}) \\
    & \hspace{2in} - \nabla^3 \rho_{r,\phi}^{\delta}(X_{[1,k)} + Y_k + Y_{(k,j-1)\cup(j+1,n)})], 
        \Exp[X_j^{\otimes 3}] 
    \Big\rangle \Big|\\
    & \leq C \phi \tilde L_{3,j} \sum_{k=1}^{j-2} \left[ \begin{aligned}
        \tilde L_{3,[k]_{j-2}} \frac{(\log(ep))^{5/2}}{\delta_{n-k}^5}
        + \tilde\nu_{q,[k]_{j-2}} \frac{ {(\log(ep))}^{(q+2)/2}}{(\delta_{n-k}/\alpha)^{q+2}}
    \end{aligned} \right] \\
    & \quad \times \min\{\kappa_{(4,k-4)|\{4,k-4\}}(\delta_{n-k}^\circ) + \kappa^\circ_{k}(\delta_{n-k}), 1\}, \\
\end{aligned}
\end{equation*}
for $3 \leq j \leq n-1$.
Above, the Lindeberg swapping started at $k = 1$. Because the inequality holds regardless where the swapping starts, 
\begin{equation*}
\begin{aligned}
    & \inner*{ 
        \Exp[\nabla^3 \rho_{r,\phi}^{\delta}(X_{[1,j-1)} + Y_{(j+1,n]})
        - \nabla^3 \rho_{r,\phi}^{\delta}(Y_{[1,j-1)} + Y_{(j+1,n]})], 
        \Exp[X_j^{\otimes 3}] } \\
    & \leq \frac{1}{j-2} \sum_{k^\circ=1}^{j-2} \sum_{k=1}^{j-2} \Big\langle
        \Exp[\nabla^3 \rho_{r,\phi}^{\delta}(X_{[k^\circ,k^\circ+k)_{j-2}} + X_{[k^\circ+k]_{j-2}} + Y_{(k^\circ+k,k^\circ+j-1)_{j-2}\cup(j,n]}) \\
    & \hspace{0.5in} - \nabla^3 \rho_{r,\phi}^{\delta}(X_{[k^\circ,k^\circ+k)_{j-2}} + Y_{[k^\circ+k]_{j-2}} + Y_{(k^\circ+k,k^\circ+j-1)_{j-2}\cup(j,n]})], 
        \Exp[X_j^{\otimes 3}] \Big\rangle \\
    & \leq C \phi \frac{\tilde L_{3,j}}{j-2} \sum_{k^\circ=1}^{j-2} \sum_{k=1}^{j-2} \left[ 
        \tilde L_{3,[k^\circ+k']_{j-2}} \frac{(\log(ep))^{5/2}}{\delta_{n-k}^5}
        + \tilde \nu_{q,[k^\circ+k']_{j-2}} \frac{ {(\log(ep))}^{(q+2)/2}}{(\delta_{n-k}/2)^{q+2}}
    \right] \\
    & \quad \hspace{1.1in} \times \min\{\kappa_{(k^\circ+4,k^\circ+k-4)_{j-2}|\{[k^\circ+4]_{j-2},[k^\circ+k-4]_{j-2}\}}(\delta_{n-k}^\circ) + \kappa^\circ_{k}(\delta_{n-k}), 1\}. \\
\end{aligned}
\end{equation*}
The permutation argument also applies to the first Lindeberg swapping in \cref{eq:third_lindeberg_swapping_N}. Together with the results in \cref{thm:remainder_lemma_4P}, \cref{eq:decompose_third_result_P} becomes
\begin{equation*}
\begin{aligned}
    & \abs*{ \Exp[\rho_{r,\phi}^\delta(X_{[1,n]})] - \Exp[\rho_{r,\phi}^\delta(Y_{[1,n]})]} \\
    & \leq \frac{1}{n} \sum_{j^\circ=1}^n \abs*{ \Exp\left[
        \rho_{r,\phi}^\delta(X_{j^\circ+[1,n)}) - \rho_{r,\phi}^\delta(Y_{j^\circ+[1,n)})
    \right]} \\
    & \quad + \frac{1}{n} \sum_{j^\circ=1}^n \abs*{ \Exp\left[ \rho_{r,\phi}^\delta(X_{[1,n]}) - \rho_{r,\phi}^\delta(X_{j^\circ+[1,n)}) + \rho_{r,\phi}^\delta(Y_{[1,n]}) - \rho_{r,\phi}^\delta(Y_{j^\circ+[1,n)})\right]} \\
    & \leq \frac{C}{\sqrt{n}} \bar{L}_{3} \frac{(\log(ep))^2}{\underline{\sigma}^2 \sigma_{\min}} \\
    & \quad + \frac{C \phi}{n} \sum_{j^\circ=1}^n \sum_{j=2}^{n} \left[ 
        (\tilde{L}_{4,j^\circ} + \tilde{L}_{4,j^\circ+j}) \frac{(\log(ep))^{3/2}}{\delta_{n-j}^3}
        + (\tilde\nu_{q,j^\circ} + \tilde\nu_{q,j^\circ+j}) \frac{{(\log(ep))}^{(q-1)/2}}{(\delta_{n-j}/2)^{q-1}}
    \right] \\
    & \quad \hspace{1in} \times \min\{\kappa_{j^\circ+(4,n-4)|\{j^\circ+4,j^\circ+j-4\}}(\delta_{n-j}^\circ) + \kappa^\circ_n(\delta_{n-j}), 1\} \\
    & \quad + \frac{C \phi}{n} \sum_{j^\circ=1}^n \sum_{j=3}^{n} 
    \frac{\tilde{L}_{3,j^\circ+j}}{j-2} \sum_{k^\circ=1}^{j-2} \sum_{k=1}^{j-2} \left[ 
        \tilde{L}_{3,j^\circ+[k^\circ+k]_{j-2}} \frac{(\log(ep))^{5/2}}{\delta_{n-k}^5}
        + \tilde\nu_{q,j^\circ+[k^\circ+k]_{j-2}} \frac{{(\log(ep))}^{(q+2)/2}}{(\delta_{n-k}/2)^{q+2}}
    \right] \\
    & \quad \hspace{1in} \times \min\{\kappa_{j^\circ + (k^\circ+4,k^\circ+k-4)_{j-2}|\{j^\circ + [k^\circ+4]_{j-2},j^\circ + [k^\circ+k-4]_{j-2}\}}(\delta_{n-k}^\circ) + \kappa^\circ_{k}(\delta_{n-k}), 1\}, \\
\end{aligned}
\end{equation*}
where $\tilde{L}_{q,j+[k]_{j-2}} \equiv \sum_{k'=k-4}^{k+4} L_{q,j+[k']_{j-2}}$, $\tilde\nu_{q,j+[k]_{j-2}}$ is similarly defined, and $j^\circ+(k^\circ,k^\circ+k)_{j-1}$ is the shifted interval of $(k^\circ,k^\circ+k)_{j-1}$ by $j^\circ$, namely, $\{j^\circ+[k^\circ+1]_{j-1}, \dots, j^\circ+[k^\circ+k-1]_{j-1}\}$.

\medskip
\noindent{\bf Partitioning the sum.} 
Again, we partition the summations at $n/2 = n (1 - \frac{\sigma_{\min}^2}{\underline{\sigma}^2 \log^2(4ep)})$. The calculations of the first three summations are similar to those in \cref{sec:pf_1_dep_3P}. Here we only take a look at the last summation where the summation iterates.
For $k < n/2$,
\begin{equation*}
\begin{aligned}
    & \frac{1}{n} \sum_{j^\circ=1}^n \sum_{j=3}^{n} \frac{\tilde{L}_{3,j^\circ+j}}{j-2} \sum_{k^\circ=1}^{j-2} \sum_{k=1}^{(j-2) \wedge \floor{n/2}} \left[ 
        \tilde{L}_{3,j^\circ+[k^\circ+k]_{j-2}} \frac{(\log(ep))^{5/2}}{\delta_{n-k}^5}
        + \tilde\nu_{q,j^\circ+[k^\circ+k]_{j-2}} \frac{{(\log(ep))}^{(q+2)/2}}{(\delta_{n-k}/\alpha)^{q+2}}
    \right] \\
    & \leq \frac{C}{n} \sum_{j^\circ=1}^n \sum_{j=3}^{n} \tilde{L}_{3,j^\circ+j} \sum_{k=1}^{(j-2) \wedge \floor{n/2}} \left[ 
        \bar{L}_{3} \frac{n(\log(ep))^{5/2}}{(j-2)\delta_{n-k}^5}
        + \bar{\nu}_{q} \frac{n{(\log(ep))}^{(q+2)/2}}{(j-2)(\delta_{n-k}/\alpha)^{q+2}}
    \right] \\
    & \leq \frac{C}{\sqrt{n}} \sum_{j=3}^{n} \bar{L}_{3} 
    \left[ 
        \bar{L}_{3} \frac{n(\log(ep))^{7/2}}
        {(j-2)\underline{\sigma}^{2}\sigma_{\min} \delta_{n-\floor{n/2}}^2}
        + \bar{\nu}_{q} \frac{n{(\log(ep))}^{(q+4)/2}}
        {(j-2)\underline{\sigma}^{2}\sigma_{\min} \delta_{n-\floor{n/2}}^{q-1}}
    \right] \\
    & \leq \frac{C}{\sqrt{n}} \bar{L}_{3} 
    \left[ 
        \bar{L}_{3} \frac{(\log(ep))^{7/2}}
        {\underline{\sigma}^{4}\sigma_{\min}}
        + \bar{\nu}_{q} \frac{{(\log(ep))}^{(q+4)/2}}
        {\underline{\sigma}^{4}\sigma_{\min} \delta^{q-3}}
    \right] \log(en),
\end{aligned}
\end{equation*}
where the third inequality comes from Eq. (16) and the last inequality comes from the fact that $\sum_{j=3}^{n} \frac{n}{(j-2) \delta^2_{n-\floor{n/2}}} \leq \sum_{j=3}^{n} \frac{n}{(j-2) (n-\floor{n/2}) \underline{\sigma}^2} \leq \frac{C}{\underline{\sigma}^2} \log(en)$.
For $k > n/2$, 
\begin{equation*}
\begin{aligned}
    & \frac{1}{n} \sum_{j^\circ=1}^n \sum_{j=\ceil{n/2}}^{n} \frac{\tilde L_{3,j^\circ+j}}{j-2} \sum_{k^\circ=1}^{j-2} \sum_{k=\ceil{n/2}}^{j-2} \left[ 
        \tilde L_{3,j^\circ+[k^\circ+k]_{j-2}} \frac{(\log(ep))^{5/2}}{\delta_{n-k}^5}
        + \tilde\nu_{q,j^\circ+[k^\circ+k]_{j-2}} \frac{{(\log(ep))}^{(q+2)/2}}{(\delta_{n-k}/\alpha)^{q+2}}
    \right] \\
    & \quad \hspace{1in} \times \min\{\kappa_{j^\circ + (k^\circ+4,k^\circ+k-4)_{j-2}|\{j^\circ + [k^\circ+4]_{j-2},j^\circ + [k^\circ+k-4]_{j-2}\}}(\delta_{n-k}^\circ) + \kappa^\circ_{k}(\delta_{n-k}), 1\} 
    \\
    & \leq \frac{1}{n} \sum_{i=1}^n \sum_{j=\ceil{n/2}}^{n} \frac{\tilde L_{3,j^\circ+j}}{j-2} \sum_{k^\circ=1}^{j-2} \sum_{k=\ceil{n/2}}^{j-2} \left[ 
        \tilde L_{3,j^\circ+[k^\circ+k]_{j-2}} \frac{(\log(ep))^{5/2}}{\delta_{n-k}^5}
        + \tilde \nu_{q,j^\circ+[k^\circ+k]_{j-2}} \frac{{(\log(ep))}^{(q+2)/2}}{(\delta_{n-k}/\alpha)^{q+2}}
    \right] \\
    & \quad \hspace{1in} \times \kappa_{j^\circ + (k^\circ+4,k^\circ+k-4)_{j-2}|\{j^\circ + [k^\circ+4]_{j-2},j^\circ + [k^\circ+k-4]_{j-2}\}}(\delta_{n-k}^\circ)
    \\
    & \quad + \frac{C}{n} \sum_{j^\circ=1}^n \sum_{j=\ceil{n/2}}^{n} \tilde L_{3,j^\circ+j} \sum_{k=\ceil{n/2}}^{j-2} \left[ 
        \bar{L}_{3,j^\circ+(0,j)} \frac{(\log(ep))^{5/2}}{\delta_{n-k}^5}
        + \bar{\nu}_{q,j^\circ+(0,j)} \frac{{(\log(ep))}^{(q+2)/2}}{(\delta_{n-k}/\alpha)^{q+2}}
    \right] 
    \kappa^\circ_{k}(\delta_{n-k}).
\end{aligned}
\end{equation*}
Based on \cref{thm:jensen},
\begin{equation*}
\begin{aligned}
    & \frac{1}{n} \sum_{j^\circ=1}^n \sum_{j=\ceil{n/2}}^{n} \tilde L_{3,j^\circ+j} \sum_{k=\ceil{n/2}}^{j-1} \left[ 
        \bar{L}_{3,(j^\circ,j^\circ+j)} \frac{(\log(ep))^{5/2}}{\delta_{n-k}^5}
        + \bar{\nu}_{q,(j^\circ,j^\circ+j)} \frac{{(\log(ep))}^{(q+2)/2}}{(\delta_{n-k}/\alpha)^{q+2}}
    \right] 
    \kappa^\circ_{k}(\delta_{n-k}) \\
    & \leq \frac{C}{n^{3/2}} \sum_{j^\circ=1}^n \sum_{j=\ceil{n/2}}^{n} \tilde L_{3,j^\circ+j} 
    \left[ 
        \bar{L}_{3,(j^\circ,j^\circ+j)} \frac{(\log(ep))^{7/2}}{ \underline{\sigma}^{2}\sigma_{\min} \delta_{n-j}^2}
        + \bar{\nu}_{q,(j^\circ,j^\circ+j)} \frac{{(\log(ep))}^{(q+4)/2}}{\underline{\sigma}^{2}\sigma_{\min} (\delta_{n-j}/\alpha)^{q-1}}
    \right] \\
    & \leq \frac{C}{\sqrt{n}} \sum_{j=\ceil{n/2}}^{n} \bar{L}_{3} 
    \left[ 
        \bar{L}_{3} \frac{(\log(ep))^{7/2}}{ \underline{\sigma}^{2}\sigma_{\min} \delta_{n-j}^2}
        + \bar{\nu}_{q} \frac{{(\log(ep))}^{(q+4)/2}}{\underline{\sigma}^{2}\sigma_{\min} (\delta_{n-j}/\alpha)^{q-1}}
    \right] \\
    & \leq \frac{C}{\sqrt{n}} \bar{L}_{3} 
    \left[ 
        \bar{L}_{3} \frac{(\log(ep))^{7/2}}{ \underline{\sigma}^{4}\sigma_{\min}} 
        + \bar{\nu}_{q} \frac{{(\log(ep))}^{(q+4)/2}}{\underline{\sigma}^{4}\sigma_{\min} \delta^{q-3}}
    \right] \log\left(1 + \frac{\sqrt{n}\underline{\sigma}}{\delta}\right),
\end{aligned}
\end{equation*}
where the first and last inequalities follow \cref{eq:sum_delta_3,eq:sum_delta_nolog}, respectively.
In sum, we obtain the following induction lemma.
\begin{lemma} \label{thm:induction_lemma_BE_4P}
    If Assumptions \eqref{assmp:min_var}, \eqref{assmp:min_ev} and \eqref{assmp:var_ev} hold, then 
    for any $\delta \geq \sigma_{\min}$,
    \begin{equation*}
    \begin{aligned}
        & \mu_{[1,n]} \\
        & \leq \frac{C}{\sqrt{n}} \frac{\delta \log(ep)}{\sigma_{\min}}  
        + \frac{C}{\sqrt{n}} \frac{\sqrt{\log(ep)}}{\phi \sigma_{\min}}
        + \frac{C}{\sqrt{n}} \bar{L}_{3} \frac{(\log(ep))^{2}}{\underline{\sigma}^2 \sigma_{\min}} \\
        & \quad + \frac{C\phi}{\sqrt{n}} \left[
            \bar{L}_4 \frac{(\log(ep))^{5/2}}{\underline{\sigma}^2\sigma_{\min}}
            + \bar{\nu}_q \frac{(\log(ep))^{(q+1)/2}}{\underline{\sigma}^{2}\sigma_{\min} (\delta/2)^{q-4}}
        \right] \log(en) \\
        & \quad + \frac{C\phi}{\sqrt{n}} \bar{L}_3 \left[
            \bar{L}_{3} \frac{(\log(ep))^{7/2}}
            {\underline{\sigma}^{4}\sigma_{\min}}
            + \bar{\nu}_{q} \frac{{(\log(ep))}^{(q+4)/2}}
            {\underline{\sigma}^{4}\sigma_{\min} (\delta/2)^{q-3}}
        \right] \log(en) \\
        & \quad + \frac{C \phi}{n} \sum_{j^\circ=1}^n \sum_{j=\ceil{n/2}}^{n} \left[ 
            (\tilde{L}_{4,j^\circ} + \tilde{L}_{4,j^\circ+j}) \frac{(\log(ep))^{3/2}}{\delta_{n-j}^3}
            + (\tilde\nu_{q,j^\circ} + \tilde\nu_{q,j^\circ+j}) \frac{{(\log(ep))}^{(q-1)/2}}{(\delta_{n-j}/2)^{q-1}}
        \right] \\
        & \quad \hspace{1in} \times \kappa_{j^\circ+(4,n-4)|\{j^\circ+4,j^\circ+j-4\}}(\delta_{n-j}^\circ) \\
        & \quad + \frac{C \phi}{n} \sum_{j^\circ=1}^n \sum_{j=\ceil{n/2}}^{n} \frac{\tilde L_{3,j^\circ+j}}{j-2} \\ 
        & \quad \quad \quad \times \sum_{k^\circ=1}^{j-2} \sum_{k=\ceil{n/2}}^{j-2} \left[ 
            \tilde L_{3,j^\circ+[k^\circ+k]_{j-2}} \frac{(\log(ep))^{5/2}}{\delta_{n-k}^5}
            + \tilde \nu_{q,j^\circ+[k^\circ+k]_{j-2}} \frac{(\log(ep))^{(q+2)/2}}{(\delta_{n-k}/2)^{q+2}}
        \right] \\
        & \quad \hspace{1.2in} \times \kappa_{j^\circ + (k^\circ+4,k^\circ+k-4)_{j-2}|\{j^\circ + [k^\circ+4]_{j-2},j^\circ + [k^\circ+k-4]_{j-2}\}}(\delta_{n-k}^\circ),
    \end{aligned}
    \end{equation*}
    for some absolute constant $C > 0$.
\end{lemma}

\medskip
\noindent{\bf Dual Induction.} Based on \cref{thm:induction_lemma_BE_4P,thm:induction_lemma_AC_P}, we proceed to carry out a dual induction in the same way as in \cref{sec:pf_1_dep_4N}, but using instead \cref{thm:jensen} from \cref{sec:pf_1_dep_3P}. We summarize the final result in the next lemma.
\begin{lemma}
    There exist positive universal constants $\mathfrak{C}_{1,\kappa}$, $\mathfrak{C}_{2,\kappa}$, $\mathfrak{C}_{3,\kappa}$, $\mathfrak{C}_{4,\kappa}$, $\mathfrak{C}_{5,\kappa}$, $\mathfrak{C}_{1,\mu}$, $\mathfrak{C}_{2,\mu}$ and $\mathfrak{C}_{3,\mu}$ such that for any $n$, any $1$-ring dependent sequence $(X_i \in \reals^p: i \in [1,n])$ satisfying Assumptions \eqref{assmp:min_var}, \eqref{assmp:min_ev} and \eqref{assmp:var_ev} and any $\delta \geq 0$,
    \begin{equation} \label{eq:hypothesis_AC_4P} \tag{HYP-AC-4}
    \begin{aligned}
        & \sqrt{\abs{I}} \kappa_{I|\{i_1-1, i_2+1\}}(\delta)
        \leq 
        \tilde\kappa_{1,\abs{I}} \bar{L}_{3,I}
        + \tilde\kappa_{2,\abs{I}} \bar\nu_{q,I}^{1/(q-2)}
        + \tilde\kappa_{3,\abs{I}} \bar\nu_{q,I}^{1/(q-2)}
        + \tilde\kappa_{4,\abs{I}} \bar\nu_{1,I}
        + \tilde\kappa_{5} \delta, \\
        & \forall I = [i_1, i_2] ~~\text{s.t.}~~ i_1 < i_2 < i_1+n,
    \end{aligned}
    \end{equation}
    \begin{equation} \label{eq:hypothesis_BE_4P} \tag{HYP-BE-4}
        \sqrt{n} \mu_{[1,n]}
        \leq \tilde\mu_{1,n} \bar{L}_{3}
        + \tilde\mu_{2,n} \bar{L}_{4}^{1/2}
        + \tilde\mu_{3,n} \bar\nu_{q}^{1/(q-2)},
    \end{equation}
    where $\tilde\kappa_{1,i} = \mathfrak{C}_{1,\kappa} \tilde\mu_{1,i}$, 
    $\tilde\kappa_{2,i} = \mathfrak{C}_{2,\kappa} \tilde\mu_{2,i}$, 
    $\tilde\kappa_{3,i} = \mathfrak{C}_{3,\kappa} \tilde\mu_{3,i}$,
    $\tilde\kappa_{4,i} = \mathfrak{C}_{4,\kappa} \frac{\log(ep)\sqrt{\log(pi)}}{\sigma_{\min}}$,
    $\tilde\kappa_{5} = \mathfrak{C}_{5,\kappa} \frac{\sqrt{\log(ep)}}{\sigma_{\min}}$,
    \begin{equation*}
    \begin{aligned}
        \tilde\mu_{1,n} 
        & = \mathfrak{C}_1 \frac{(\log(ep))^{3/2}\sqrt{\log(pn)}}{\underline\sigma^2\sigma_{\min}} \log\left(e n\right), \\
        \tilde\mu_{2,n}
        & = \mathfrak{C}_2 \frac{\log(ep)\sqrt{\log(pn)}}{\underline\sigma \sigma_{\min}} \log\left(e n\right) \\
        \tilde\mu_{3,n}
        & = \mathfrak{C}_3 \frac{\log(ep)\sqrt{\log(pn)}}{\underline\sigma^{2/(q-2)} \sigma_{\min}} \log\left(e n\right)
    \end{aligned}
    \end{equation*}
\end{lemma}

If $\mathfrak{C}_{1,\kappa}$, $\mathfrak{C}_{2,\kappa}$, $\mathfrak{C}_{3,\kappa}$, $\mathfrak{C}_{4,\kappa}$, $\mathfrak{C}_{1,\mu}$, $\mathfrak{C}_{2,\mu} \geq 2$, then \eqref{eq:hypothesis_BE_4P} and \eqref{eq:hypothesis_AC_4P}, requiring $\mu_{[1,n]} \leq 1$ and $\kappa_{I|\{i_1-1,i_2+1\}}(\delta) \leq 1$ almost surely for all $I \subsetneq [1,n]$ only, trivially holds for $n \leq 36$. Now we consider the case of $n > 36$. Suppose that the induction hypotheses hold for all smaller $n$. 

\medskip
We first derive \eqref{eq:hypothesis_AC_4P} for any $I = [i_1, i_2]$ satisfying $i_1 < i_2 < i_1+n$. For the case of $i \leq 36$, \eqref{eq:hypothesis_AC_4P} trivially holds given $\mathfrak{C}_{1,\kappa}, \mathfrak{C}_{2,\kappa}, \mathfrak{C}_{3,\kappa}, \mathfrak{C}_{4,\kappa}, \mathfrak{C}_{5,\kappa} \geq 2$. For the case of $i > 36$, we first assume that \eqref{eq:hypothesis_AC_4P} holds for all $I$ satisfying $\abs{I} < i$ and then show this extends to all $I$ with $\abs{I} = i$. Without loss of generality, we only show for the interval $[1,i]$.
By \cref{thm:induction_lemma_AC_P}, 
for any $\vareps \geq \sigma_{\min}$ and $\delta > 0$,
\begin{equation*} 
\begin{aligned}
    & \kappa_{[1,i]|\{0, i+1\}}(\delta) \\
    & \leq \frac{C}{i^\circ-3} \sum_{j=0}^{i^\circ-4}  
    \frac{\sqrt{\log(ep)}}{\vareps} (\nu_{1,j+2} + \nu_{1,j+i^\circ-1}) 
    \min\{1, \kappa_{j+[3,i^\circ-2]|\{j+2, j+i^\circ-1\}}(\vareps^\circ)\} \\
    & \quad + \frac{C}{i^\circ-3} \sum_{j=0}^{i^\circ-4} \mu_{j+[3,i^\circ-2]}
    + C\frac{\delta + 2\vareps^\circ}{\sigma_{\min}} \sqrt{\frac{\log(ep)}{i^\circ-4}}
    + C\frac{\bar\nu_{1,(1,i)}}{\sigma_{\min}} \frac{\log(ep)}{\sqrt{i^\circ-4}},
\end{aligned}
\end{equation*}
where $i^\circ = \floor{\frac{i+4}{2}}$, $\vareps^\circ = 20\vareps \sqrt{\log(p(i^\circ-4))}$ and $C > 0$ is an absolute constant. 
Since $(X_1, \dots, X_n)$ is a $1$-ring dependent sequence satisfying Assumptions \eqref{assmp:min_var}, \eqref{assmp:min_ev} and \eqref{assmp:var_ev} and the interval $j+[3,i^\circ-2]$ is a proper subset of $[1,n]$, \eqref{eq:hypothesis_AC_4P} holds for $\kappa_{j+[3,i^\circ-2]|\{j+2, j+i^\circ-1\}}(\vareps^\circ)$. Furthermore, $(X_{j+2}, \dots, X_{j+i^\circ-1})$ is a $1$-ring dependent sequence satisfying Assumptions \eqref{assmp:min_var}, \eqref{assmp:min_ev} and \eqref{assmp:var_ev} with the same $\sigma_{\min}$ and $\underline{\sigma}$ as the original data $(X_1, \dots, X_n)$. (See \cref{thm:assmp_induction_P}.) 
As a result, \eqref{eq:hypothesis_BE_4P} holds for $\mu_{j+[3,i^\circ-2]}$. Plugging the resulting upper bounds on $\kappa_{j+[3,i^\circ-2]|\{j+2, j+i^\circ-1\}}(\vareps^\circ)$ and $\mu_{j+[3,i^\circ-2]}$ back into the previous upper bound for $\kappa_{[1,i]|\{0,i+1\}}$, we obtain that
\begin{equation*} 
\begin{aligned}
    & \kappa_{[1,i]|\{0,i+1\}}(\delta) \\
    & \leq \frac{C}{(i^\circ-4)^{3/2}} \sum_{j=0}^{i^\circ-4}  
    \frac{\sqrt{\log(ep)}}{\vareps} (\nu_{1,j+2} + \nu_{1,j+i^\circ-1})
    \\
    & \quad \times \left[ 
        \tilde\kappa_{1,i^\circ-4} \bar{L}_{3,j+[3,i^\circ-2]}
        + \tilde\kappa_{2,i^\circ-4} \bar{L}_{4,j+[3,i^\circ-2]}^{1/2}
        + \tilde\kappa_{3,i^\circ-4} \bar{\nu}_{q,j+[3,i^\circ-2]}^{1/(q-2)}
        + \tilde\kappa_{4,i^\circ-4} \bar\nu_{1,j+[3,i^\circ-2]}
        + \tilde\kappa_{5} {\vareps^\circ}
    \right]
    \\
    & \quad + \frac{C}{(i^\circ-4)^{3/2}} \sum_{j=0}^{i^\circ-1} \tilde\mu_{1,i^\circ-4} {\bar{L}_{3,j+[3,i^\circ-2]}}
    + \frac{C}{(i^\circ-4)^{3/2}} \sum_{j=0}^{i^\circ-4} \tilde\mu_{2,i^\circ-4} {\bar{L}_{4,j+[3,i^\circ-2]}^{1/2}} \\
    &\quad + \frac{C}{(i^\circ-4)^{3/2}} \sum_{j=1}^{i^\circ-4} \tilde\mu_{3,i^\circ-4} {\bar{\nu}_{q,j+[3,i^\circ-2]}^{1/(q-2)}}
    + C\frac{\delta + 2\vareps^\circ}{\sigma_{\min}} \sqrt{\frac{\log(ep)}{i^\circ-4}}
    + C\frac{\bar\nu_{1,(1,i)}}{\sigma_{\min}} \frac{\log(ep)}{\sqrt{i^\circ-4}}.
\end{aligned}
\end{equation*}

By \cref{thm:jensen},
\begin{equation*}
\begin{aligned}
    \frac{1}{i^\circ-4} \sum_{j=0}^{i^\circ-4} \bar{L}_{3,j+[3,i^\circ-2]}
    & \leq C \bar{L}_{3,(2,j-1)},
    \\
    \frac{1}{i^\circ-4} \sum_{j=0}^{i^\circ-4}  (\nu_{1,j+2} + \nu_{1,j+i^\circ-1})
    \bar{L}_{3,j+[3,i^\circ-2]}
    & \leq C \bar\nu_{1,(1,i)} \bar{L}_{3,(2,i-1)}, 
    \\
    \frac{1}{i^\circ-4} \sum_{j=0}^{i^\circ-4}  (\nu_{1,j+2} + \nu_{1,j+i^\circ-1})
    \bar\nu_{1,j+[3,i^\circ-2]}
    & \leq C \bar\nu_{1,(1,i)} \bar\nu_{1,(2,i-1)}.
\end{aligned}
\end{equation*}
Similar inequalities hold with $\bar{L}_{4}^{1/2}$ and $\bar{\nu}_{q}^{1/(q-2)}$ in place of $\bar{L}_3$.
As a result,
\begin{equation*} 
\begin{aligned}
    & \sqrt{i} \kappa_{[1,i]|\{0,i+1\}}(\delta) \\
    & \leq \mathfrak{C}' \frac{\sqrt{\log(ep)}}{\vareps} \bar\nu_{1,(1,i)}
    \\
    & \quad \times \left[ 
        \tilde\kappa_{1,i^\circ-4} \bar{L}_{3,(2,i-1)}
        + \tilde\kappa_{2,i^\circ-4} \bar{L}_{4,(2,i-1)}^{1/2}
        + \tilde\kappa_{3,i^\circ-4} \bar{\nu}_{q,(2,i-1)}^{1/(q-2)}
        + \tilde\kappa_{4,i^\circ-4} \bar\nu_{1,(2,i-1)}
        + \tilde\kappa_{5} \vareps^\circ
    \right]
    \\
    & \quad + \mathfrak{C}'  \left[ 
        \tilde\mu_{1,i^\circ-4} {\bar{L}_{3,(2,i-1)}}
        + \tilde\mu_{2,i^\circ-4} {\bar{L}_{4,(2,i-1)}^{1/2}}
        + \tilde\mu_{3,i^\circ-4} {\bar{\nu}_{q,(2,i-1)}^{1/(q-2)}} 
        \right] \\
    & \quad + \mathfrak{C}' \left[
            \frac{\delta + 2\vareps^\circ}{\sigma_{\min}} \sqrt{\log(ep)}
        + \frac{\bar\nu_{1,(1,i)}}{\sigma_{\min}} \log(ep)
    \right],
\end{aligned}
\end{equation*}
for some universal constant $\mathfrak{C}'$, whose value does not change over lines.
Plugging in $\vareps = \max\{2 \mathfrak{C}', 1\} \sqrt{\log(ep)} \bar{\nu}_{1,(1,i)} \overset{\text{\cref{eq:nu_1_vs_sigma_min}}}{\geq} \sigma_{\min}$,
\begin{equation*}
\begin{aligned}
    & \sqrt{i} \kappa_{[1,i]|\{0,i+1\}}(\delta) \\
    & \leq \frac{1}{2}
    \left[ 
        \tilde\kappa_{1,i^\circ-4} {\bar{L}_{3,(2,i-1)}}
        + \tilde\kappa_{2,i^\circ-4} {\bar{L}_{4,(2,i-1)}^{1/2}}
        + \tilde\kappa_{3,i^\circ-4} {\bar{\nu}_{q,(2,i-1)}^{1/(q-2)}}
        + \tilde\kappa_{4} {\bar\nu_{1,(2,i-1)}}
        + \tilde\kappa_{5} {\delta}
    \right] \\
    & \quad  
    + \mathfrak{C}' \left[ \begin{aligned}
        \tilde\mu_{1,i^\circ-4} \bar{L}_{3,(2,i-1)}
        + \tilde\mu_{2,i^\circ-4} {\bar{L}_{4,(2,i-1)}^{1/2}}
        + \tilde\mu_{3,i^\circ-4} \bar\nu_{q,(2,i-1)}^{1/(q-2)} 
    \end{aligned} \right] \\
    & \quad  
    + \mathfrak{C}' \frac{\log(ep)}{\sigma_{\min}} \bar{\nu}_{1,(1,i)}
    + 40 \mathfrak{C}' \frac{\log(ep)\sqrt{\log(p(i^\circ-2))}}{\sigma_{\min}} \bar{\nu}_{2,(1,i)}^{1/2}
    + \mathfrak{C}' \frac{\sqrt{\log(ep)}}{\sigma_{\min}} \delta
    \\
    & \leq \tilde\kappa_{1,i} \bar{L}_{3,[1,i]}
    + \tilde\kappa_{2,i} \bar{L}_{4,[1,i]}^{1/2}
    + \tilde\kappa_{3,i} \bar\nu_{q,[1,i]}^{1/(q-2)}
    + \tilde\kappa_{4,i} \bar\nu_{2,[1,i]}^{1/2}
    + \tilde\kappa_{5} \delta,
\end{aligned}
\end{equation*}
where $\tilde\kappa_{1,i} = \mathfrak{C}_{1,\kappa} \tilde\mu_{1,i}$, 
$\tilde\kappa_{2,i} = \mathfrak{C}_{2,\kappa} \tilde\mu_{2,i}$, 
$\tilde\kappa_{3,i} = \mathfrak{C}_{3,\kappa} \tilde\mu_{3,i}$, 
$\tilde\kappa_{4,i} = \mathfrak{C}_{4,\kappa} \frac{\log(ep)\sqrt{\log(pi)}}{\sigma_{\min}}$, and
$\tilde\kappa_{5} = \mathfrak{C}_{5,\kappa} \frac{\sqrt{\log(ep)}}{\sigma_{\min}}$, provided by $\mathfrak{C}_{1,\kappa} = \mathfrak{C}_{2,\kappa} = \mathfrak{C}_{3,\kappa} = \max\{2 \mathfrak{C}', 2\}$, $\mathfrak{C}_{4,\kappa} = \max\{82 \mathfrak{C}' + 40 \mathfrak{C}'^2, 2\}$ and
$\mathfrak{C}_{5,\kappa} = \max\{ \mathfrak{C}', 2\}$. 
A generalization the above argument to any interval with length $i$ and a mathematical induction for $i < n$ proves \eqref{eq:hypothesis_AC_4P} at $n$.

\medskip
The proof of \eqref{eq:hypothesis_BE_4P} at $n$ also proceeds similarly using \cref{thm:jensen}. We first upper bound the last two terms in \cref{thm:induction_lemma_BE_4P}:
\begin{equation*}
\begin{aligned}
    & \frac{C \phi}{n} \sum_{j^\circ=1}^n \sum_{j=\ceil{n/2}}^{n} \left[ 
        (\tilde{L}_{4,j^\circ} + \tilde{L}_{4,j^\circ+j}) \frac{(\log(ep))^{3/2}}{\delta_{n-j}^3}
        + (\tilde\nu_{q,j^\circ} + \tilde\nu_{q,j^\circ+j}) \frac{{(\log(ep))}^{(q-1)/2}}{(\delta_{n-j}/2)^{q-1}}
    \right] \\
    & \quad \hspace{1in} \times \kappa_{j^\circ+(4,n-4)|\{j^\circ+4,j^\circ+j-4\}}(\delta_{n-j}^\circ) \\
    & \quad + \frac{C \phi}{n} \sum_{j^\circ=1}^n \sum_{j=\ceil{n/2}}^{n} \frac{\tilde L_{3,j^\circ+j}}{j-2} \\ 
    & \quad \quad \quad \times \sum_{k^\circ=1}^{j-2} \sum_{k=\ceil{n/2}}^{j-2} \left[ 
        \tilde L_{3,j^\circ+[k^\circ+k]_{j-2}} \frac{(\log(ep))^{5/2}}{\delta_{n-k}^5}
        + \tilde \nu_{q,j^\circ+[k^\circ+k]_{j-2}} \frac{(\log(ep))^{(q+2)/2}}{(\delta_{n-k}/2)^{q+2}}
    \right] \\
    & \quad \hspace{1.2in} \times \kappa_{j^\circ + (k^\circ+4,k^\circ+k-4)_{j-2}|\{j^\circ + [k^\circ+4]_{j-2},j^\circ + [k^\circ+k-4]_{j-2}\}}(\delta_{n-k}^\circ) \\
    & \equiv \left[ \mathfrak{T}_{1,1} + \mathfrak{T}_{1,2} \right] 
    + \left[ \mathfrak{T}_{2,1} + \mathfrak{T}_{2,2} \right].
\end{aligned}
\end{equation*}
Since $(X_{j^\circ}, \dots, X_{j^\circ+j-2})$ is a $1$-ring dependent sequence satisfying Assumptions \eqref{assmp:min_var}, \eqref{assmp:min_ev} and \eqref{assmp:var_ev} (see \cref{thm:assmp_induction_P}) and the interval $j^\circ+(k^\circ+4, k^\circ+k-4)_{j-2}$ is a proper subset of $[j^\circ, j^\circ+j-2]$, \eqref{eq:hypothesis_AC_4P} holds for $\kappa_{j^\circ + (k^\circ+4,k^\circ+k-4)_{j-2}|\{\cdots\}}(\delta_{n-k}^\circ)$.
Hence
based on \cref{thm:jensen} and \cref{eq:sum_delta_q,eq:sum_delta_q_kappa} with appropriate $q$, we obtain
\begin{equation*}
\begin{aligned}
    & \mathfrak{T}_{2,1} \\
    & \leq \frac{C\phi}{\sqrt{n}} \sum_{j=\ceil{n/2}}^{n-1} \sum_{k=\ceil{n/2}}^{j-2} \bar{L}_{3} \left[ 
        \bar{L}_{3} \frac{(\log(ep))^{5/2}}{\delta_{n-k}^5}
    \right] 
    \left[ \begin{aligned}
        & \tilde\kappa_{1,n-1} \bar{L}_{3}
        + \tilde\kappa_{2,n-1} \bar{L}_{4}^{1/2} \\
        & + \tilde\kappa_{3,n-1} \bar\nu_{q}^{1/(q-2)}
        + \tilde\kappa_{4,n-1} \bar\nu_{1}
        + \tilde\kappa_{5} \delta_{n-j}^\circ \\
    \end{aligned} \right] \\
    & \leq \frac{C\phi}{\sqrt{n}} \sum_{j=\ceil{n/2}}^{n-1} \bar{L}_{3} \left[ 
        \bar{L}_{3} \frac{(\log(ep))^{5/2}}{\underline{\sigma}^2 \delta_{n-j}^3}
    \right] 
    \left[ \begin{aligned}
        & \tilde\kappa_{1,n-1} \bar{L}_{3}
        + \tilde\kappa_{2,n-1} \bar{L}_{4}^{1/2} \\
        & + \tilde\kappa_{3,n-1} \bar\nu_{q}^{1/(q-2)}
        + \tilde\kappa_{4,n-1} \bar\nu_{1}
        + \tilde\kappa_{5} \delta_{n-j}^\circ \\
    \end{aligned} \right] \\
    & \leq \frac{C\phi}{\sqrt{n}} \bar{L}_{3} \left[ 
        \bar{L}_{3} \frac{(\log(ep))^{5/2}}{\underline{\sigma}^4}
    \right] 
    \left[ \begin{aligned}
        & (\tilde\mu_{1,n-1} \bar{L}_{3} 
         +\tilde\mu_{2,n-1} \bar{L}_{4}^{1/2}
         +\tilde\mu_{3,n-1} \bar\nu_{q}^{1/(q-2)})
        \frac{1}{\delta} \\
        & +\frac{\log(ep)\sqrt{\log(pn)}}{\sigma_{\min}} \frac{\bar\nu_{1}}{\delta} \\
        & + \frac{\sqrt{\log(ep)\log(pn)}}{\sigma_{\min}} 
        \log\left(1 + \frac{\sqrt{n}\underline{\sigma}}{\delta}\right)
    \end{aligned} \right], \\
\end{aligned}
\end{equation*}
and a similar upper bound for $\mathfrak{T}_{3,2}$
\begin{equation*}
\begin{aligned}
    & \mathfrak{T}_{3,2} \\
    & \leq \frac{C\phi}{\sqrt{n}} \sum_{j=\ceil{n/2}}^{n-1} \sum_{k=\ceil{n/2}}^{j-2} \bar{L}_{3} \left[ 
        \bar{\nu}_{q} \frac{(\log(ep))^{(q+2)/2}}{(\delta_{n-k}/2)^{q+2}}
    \right] 
    \left[ \begin{aligned}
        & \tilde\kappa_{1,n-1} \bar{L}_{3}
        + \tilde\kappa_{2,n-1} \bar{L}_{4}^{1/2} \\
        & + \tilde\kappa_{3,n-1} \bar\nu_{q}^{1/(q-2)}
        + \tilde\kappa_{4,n-1} \bar\nu_{1}
        + \tilde\kappa_{5} \delta_{n-j}^\circ \\
    \end{aligned} \right] \\
    & \leq \frac{C\phi}{\sqrt{n}} \sum_{j=\ceil{n/2}}^{n-1} \bar{L}_{3} \left[ 
        \bar{\nu}_{q} \frac{(\log(ep))^{(q+2)/2}}{\underline{\sigma}^2 (\delta_{n-j}/2)^{q}}
    \right] 
    \left[ \begin{aligned}
        & \tilde\kappa_{1,n-1} \bar{L}_{3}
        + \tilde\kappa_{2,n-1} \bar{L}_{4}^{1/2} \\
        & + \tilde\kappa_{3,n-1} \bar\nu_{q}^{1/(q-2)}
        + \tilde\kappa_{4,n-1} \bar\nu_{1}
        + \tilde\kappa_{5} \delta_{n-j}^\circ \\
    \end{aligned} \right] \\
    & \leq \frac{C\phi}{\sqrt{n}} \bar{L}_{3} \left[ 
        \bar{\nu}_{q} \frac{(\log(ep))^{(q+2)/2}}{\underline{\sigma}^4 (\delta/2)^{q-3}}
    \right] 
    \left[ \begin{aligned}
        & (\tilde\mu_{1,n-1} \bar{L}_{3} 
         +\tilde\mu_{2,n-1} \bar{L}_{4}^{1/2}
         +\tilde\mu_{3,n-1} \bar\nu_{q}^{1/(q-2)})
        \frac{1}{\delta} \\
        & +\frac{\log(ep)\sqrt{\log(pn)}}{\sigma_{\min}} \frac{\bar\nu_{1}}{\delta} \\
        & + \frac{\sqrt{\log(ep)\log(pn)}}{\sigma_{\min}}
    \end{aligned} \right], \\
\end{aligned}
\end{equation*}
The upperbounds for $\mathfrak{T}_{1,1}$ and $\mathfrak{T}_{1,2}$ are similar to those for $\mathfrak{T}_{2,1}$ and $\mathfrak{T}_{2,2}$ in \cref{sec:pf_1_dep_3P}. 
In sum, as long as $\delta \geq \bar\nu_1 \sqrt{\log(ep)} \overset{\text{\cref{eq:nu_1_vs_sigma_min}}}{\geq} \sigma_{\min} $ and $\phi > 0$,
\begin{equation*}
\begin{aligned}
    & \sqrt{n} \mu_{[1,n]} \\
    & \leq \mathfrak{C}'' \phi \left( 
        \bar{L}_{4} \frac{(\log(ep))^{3/2}}{\underline{\sigma}^2 \delta} 
        + \bar{\nu}_{q} \frac{(\log(ep))^{(q-1)/2}}{\underline\sigma^2(\delta/2)^{q-3}} 
        + \bar{L}_{3} \left[ 
            \bar{L}_{3} \frac{(\log(ep))^{5/2}}{\underline{\sigma}^4 \delta}
            + \bar{\nu}_{q} \frac{(\log(ep))^{(q+2)/2}}{\underline{\sigma}^4 (\delta/2)^{q-2}}
    \right] \right) \\
    & \quad \hspace{.5in} \times \left[ \begin{aligned}
        \tilde\mu_{1,n-1} \bar{L}_{3}
        + \tilde\mu_{2,n-1} \bar{L}_{4}^{1/2}
        + \tilde\mu_{3,n-1} \bar\nu_{q}^{1/(q-2)}
    \end{aligned} \right] \\
    & \quad + \mathfrak{C}'' \left[ 
        \frac{\delta \log(ep)}{\sigma_{\min}}  
        + \frac{\sqrt{\log(ep)}}{\phi \sigma_{\min}}
        + \bar{L}_{3} \frac{(\log(ep))^{2}}{\underline{\sigma}^2\sigma_{\min}} 
        + \phi \bar{\nu}_{q} \frac{(\log(ep))^{(q+1)/2}}{(q-4) \underline{\sigma}^{2}\sigma_{\min} (\delta/2)^{q-4}} 
    \right] \\
    & \quad + \mathfrak{C}'' \phi \left[
        \bar{L}_{4} \frac{(\log(ep))^{5/2}}{\underline{\sigma}^2\sigma_{\min}}
        + \bar{L}_{3} \bar{L}_{3} \frac{(\log(ep))^{7/2}}
        {\underline{\sigma}^{4}\sigma_{\min}}
        + \bar{L}_{3} \bar{\nu}_{q} \frac{{(\log(ep))}^{(q+4)/2}}
        {\underline{\sigma}^{4}\sigma_{\min} (\delta/2)^{q-3}}
    \right] \log(en) \\
    & \quad + \mathfrak{C}'' \phi \left[
        \bar{L}_{4} \frac{(\log(ep))^{3/2}}{\underline{\sigma}^2} \log\left(en\right) 
        + \bar{\nu}_{q} \frac{(\log(ep))^{(q-1)/2}}{(q-4) \underline\sigma^2 (\delta/2)^{q-4}} 
    \right] 
    \frac{\sqrt{\log(ep)\log(pn)}}{\sigma_{\min}}\\
    & \quad + \mathfrak{C}'' \phi \bar{L}_{3} \left[ 
        \bar{L}_{3} \frac{(\log(ep))^{5/2}}{\underline{\sigma}^4} \log\left(en\right)
        + \bar{\nu}_{q} \frac{(\log(ep))^{(q+2)/2}}{\underline{\sigma}^4 (\delta/2)^{q-3}}
    \right]
    \frac{\sqrt{\log(ep)\log(pn)}}{\sigma_{\min}},
\end{aligned}
\end{equation*}
where $\mathfrak{C}''$ is a universal constant whose value does not change over lines.
Taking $\delta = \max\{8 \mathfrak{C}'', 2\} \left( 
    \frac{\bar{L}_{3}}{\underline{\sigma}^2} \sqrt{\log(ep)}
    + \left(\frac{\bar{L}_{4}}{\underline{\sigma}^2}\right)^{\frac{1}{2}} 
    + \left(\frac{\bar\nu_{q}}{\underline{\sigma}^2}\right)^{\frac{1}{q-2}} 
\right) \sqrt{\log(ep)} \geq \bar\nu_1 \sqrt{\log(ep)}$ and $\phi = \frac{1}{\delta \sqrt{\log(ep)}}$, 
\begin{equation*}
\begin{aligned}
    & \sqrt{n} \mu_{[1,n]} \\
    & \leq \frac{1}{2} \max_{j<n}\tilde\mu_{1,j} \bar{L}_{3}
    + \frac{1}{2} \max_{j<n}\tilde\mu_{2,j} \bar{L}_{4}^{1/2}
    + \frac{1}{2} \max_{j<n}\tilde\mu_{3,j} \bar\nu_{q}^{1/(q-2)} \\
    & \quad + \mathfrak{C}^{(3)} \left(
        \bar{L}_{3} \frac{(\log(ep))^{3/2}}{\underline{\sigma}^2}
        + \bar{L}_{4}^{1/2} \frac{\log(ep)}{\underline{\sigma}}
        + \bar\nu_{q}^{1/(q-2)} \frac{\log(ep)}{\underline{\sigma}^{2/(q-2)}}
    \right)
    \frac{\sqrt{\log(pn)}}{\sigma_{\min}} \log\left(e n\right) \\
\end{aligned}
\end{equation*}
for another universal constant $\mathfrak{C}^{(3)}$, whose value only depends on $\mathfrak{C}''$.
Taking $\mathfrak{C}_1 = \mathfrak{C}_2 = \mathfrak{C}_3 = \max\{2 \mathfrak{C}^{(3)}, 1\}$,
\begin{equation*}
\begin{aligned}
    \tilde\mu_{1,n} 
    & = \mathfrak{C}_1 \frac{(\log(ep))^{3/2}\sqrt{\log(pn)}}{\underline\sigma^2\sigma_{\min}} \log\left(e n\right), \\
    \tilde\mu_{2,n}
    & = \mathfrak{C}_2 \frac{\log(ep)\sqrt{\log(pn)}}{\underline\sigma \sigma_{\min}} \log\left(e n\right) \\
    \tilde\mu_{3,n}
    & = \mathfrak{C}_3 \frac{\log(ep)\sqrt{\log(pn)}}{\underline\sigma^{2/(q-2)} \sigma_{\min}} \log\left(e n\right)
\end{aligned}
\end{equation*}
satisfies
\begin{equation*}
\begin{aligned}
    \sqrt{n} \mu_{[1,n]}
    \leq \tilde\mu_{1,n} \bar{L}_{3}
    + \tilde\mu_{2,n} \bar{L}_{4}^{1/2}
    + \tilde\mu_{3,n} \bar\nu_{q}^{1/(q-2)},
\end{aligned}
\end{equation*}
which proves \eqref{eq:hypothesis_BE_4P} at $n$. By a mathematical induction, the induction hypotheses hold for all $n$, and it concludes our proof.

\subsection{Proof of Theorem 3.1} \label{sec:pf_1_dep_bootstrap}


Due to Theorem 1.1 in \cite{fang2020high} (in the form shown as in Lemma 2.1 in \citetalias{chernozhukov2020nearly}), 
\begin{equation*}
    \mu(X_{[1,n]}, \tilde{Y}) \leq \mu(X_{[1,n]}, Y_{[1,n]}) + \frac{\Delta}{\sigma_{\min}^2} \log(ep) \left( 1 \vee \abs*{\log\frac{\Delta}{\sigma_{\min}^2}} \right),
\end{equation*}
where $\Delta \equiv \norm{\Var[\frac{1}{\sqrt{n}} Y_{[1,n]}] - \Var[\tilde{Y} | \mathcal{X}_{[1,n]}]}_\infty = \norm{\Sigma - \tilde\Sigma}_\infty$. 
%
Due to the optimality of $\tilde\Sigma$, 
\begin{equation*}
    \norm{\Sigma - \tilde\Sigma}_\infty 
    \leq \norm{\Sigma - \hat\Sigma}_\infty
    + \norm{\tilde\Sigma - \hat\Sigma}_\infty
    \leq 2 \norm{\Sigma - \hat\Sigma}_\infty.
\end{equation*}
where, conditional on $\mathcal{X}_{[1,n]}$, 
\begin{equation*}
    \hat\Sigma
    \equiv \sum_{i=1}^n \left( X_i X_i^\top 
    + X_i X_{i+1}^\top + X_{i+1} X_i^\top \right),
\end{equation*}
while the target variance is $\Sigma \equiv \sum_{i=1}^n \Exp[  X_i X_i^\top + X_i X_{i+1}^\top + X_{i+1} X_i^\top ]$.
In other words, $\hat\Sigma$ is centered around $\Sigma$, and an upperbound for $\Delta$ is driven by the law of large numbers. Specifically, for $i \in [n]$, let 
\begin{equation*}
\begin{aligned}
    S_i 
    & \equiv  X_i X_i^\top + X_i X_{i+1}^\top + X_{i+1} X_i^\top
    - \Exp\left[
        X_i X_i^\top + X_i X_{i+1}^\top + X_{i+1} X_i^\top
    \right].
\end{aligned}
\end{equation*}
The summands $(S_i: i \in [n])$ are $3$-ring dependent $p \times p$ mean zero random matrices, and $\Delta = \norm*{\frac{1}{n} \sum_{i=1}^{n} S_i}_\infty$. We divide these summands into three groups: 
for $l \in [3]$ and $j \in [\floor{n/3}]$, let $T^{(l)}_j \equiv S_{3(j-1)+l}$, except for $T^{(3)}_{\floor{n/3}}$, which is defined as $T^{(3)}_{\floor{n/3}} \equiv S_{3\floor{n/3}} + \dots + S_n$. We note that each $(T^{(l)}_j: j \in [\floor{n/3}])$ for $l \in [3]$ is an independent sequence of mean zero Gaussian random matrices, whose entries have finite $q/2$-th moments which are bounded by $3 L_{q,j}$ and $3 \nu_{q,j}$. Based on \cref{thm:conc_ineq_q}, for each $l \in [3]$, we obtain
\begin{equation} \label{eq:conc_ineq_Sigma}
\begin{aligned}
    \norm*{\frac{1}{n} \tsum_{j=1}^{\floor{n/3}} T^{(l)}_j}_\infty 
    & \leq C \bar{L}_{\min\{4, q\}}^{1/2} \left(\frac{\bar{\nu}_{q}}{\delta} \right)^{\max\{2/q-1/2,0\}}
    \left(\frac{\log(2p/\delta)}{n}\right)^{\min\{1-2/q, 1/2\}} \\
    & \quad + C \left(\frac{\bar{\nu}_{q}}{\delta} \right)^{2/q} \left(\frac{\log(2p/\delta)}{n}\right)^{1-2/q},
\end{aligned}
\end{equation}
with probability at least $1-\delta$ for some universal constant $C > 0$. Summing this upperbound over $l \in [3]$, we obtain the desired result. The proof of \cref{thm:conc_ineq_q} is provided in \cref{sec:pf_conc_ineq_q}.



\begin{lemma} \label{thm:conc_ineq_q}
    Suppose that $X_1, \dots, X_n$ are independent $p$-dim random vectors with finite $q$-th moments for some $q > 1$. Specifically, we denote
    \begin{equation*}
        L_{q,i} \equiv \max_{k \in [p]} \Exp[\abs{X_{i,k}}^q],
        \quad
        \nu_{q,i} \equiv \Exp[\norm{X_{i}}_\infty^q],
    \end{equation*}
    $\bar{L}_{q} \equiv \frac{1}{n} \sum_{i \in [n]} L_{q,i}$,
    and $\bar\nu_{q} \equiv \frac{1}{n} \sum_{i \in [n]} \nu_{q,i}$. Then with probability at least $1 - \delta$,
    \begin{equation*}
    \begin{aligned}
        \norm*{\frac{1}{n}\sum_{i=1}^n X_i}_\infty
        & \leq C \bar{L}_{\min\{2, q\}}^{1/2} \left(\frac{\bar{\nu}_q}{\delta} \right)^{\max\{1/q-1/2,0\}}
        \left(\frac{\log(2p/\delta)}{n}\right)^{\min\{1-1/q, 1/2\}} \\
        & \quad + C \left(\frac{\bar{\nu}_q}{\delta} \right)^{1/q} \left(\frac{\log(2p/\delta)}{n}\right)^{1-1/q},
    \end{aligned}
    \end{equation*}
    where $C > 0$ is a universal constant.
\end{lemma}

\subsection{Proof of Theorem 3.3} \label{sec:pf_psg_error}

By well known-properties of sub-gradient methods \citep[see, e.g.][Section 2]{boyd2003subgradient}, 
\begin{equation} \label{eq:psg_basic}
    \norm{\tilde\Sigma^{(k^*)} - \hat\Sigma}_\infty - \norm{\tilde\Sigma - \hat\Sigma}_\infty 
    \leq \frac{R^2}{2 \eta K} + \frac{G^2 \eta}{2},
\end{equation}
where $R \equiv \norm{\tilde\Sigma^{(0)} - \tilde\Sigma}_F$, and $G \equiv \max_{k \in [K]} \norm{g^{(k)}}_F$. Due to the choice of $g^{(k)}$ in \cref{eq:subgradient}, $G \leq 2$ almost surely. 
%
We also observe that
\begin{equation*}
    \norm{\tilde\Sigma^{(0)} - \tilde\Sigma}_F
    \leq \norm{\tilde\Sigma^{(0)}}_F + \norm{\tilde\Sigma}_F
    \leq \sqrt{p} \norm{\tilde\Sigma^{(0)}}_\infty + p \norm{\tilde\Sigma}_\infty.
\end{equation*}
Conditional on $\norm{\hat\Sigma - \Sigma}_\infty \leq \frac{1}{2} \norm{\Sigma}_\infty$, 
\begin{equation*}
\begin{aligned}
    \frac{1}{2} \norm{\Sigma}_\infty
    \leq \norm{\Sigma}_\infty - \norm{\hat\Sigma - \Sigma}_\infty
    \leq \norm{\hat\Sigma}_\infty 
    \leq \norm{\Sigma}_\infty + \norm{\hat\Sigma - \Sigma}_\infty
    \leq \frac{3}{2} \norm{\Sigma}_\infty 
\end{aligned}
\end{equation*}
almost surely. At the same time, due to the optimality of $\tilde\Sigma$,
\begin{equation*}
\begin{aligned}
    \norm{\tilde\Sigma}_\infty 
    & \leq \norm{\hat\Sigma}_\infty + \norm{\tilde\Sigma - \hat\Sigma}_\infty \\
    & \leq \norm{\hat\Sigma}_\infty + 2 \norm{\hat\Sigma - \Sigma}_\infty \\
    & \leq \norm{\hat\Sigma}_\infty + \norm{\Sigma}_\infty
    \leq 3 \norm{\hat\Sigma}_\infty,
\end{aligned}
\end{equation*}
almost surely. Hence, under the same condition, plugging in the step size $\eta = p\norm{\hat\Sigma}_\infty / \sqrt{K}$ to \cref{eq:psg_basic},
\begin{equation*}
\begin{aligned}
    \norm{\tilde\Sigma^{(k^*)} - \hat\Sigma}_\infty - \norm{\tilde\Sigma - \hat\Sigma}_\infty 
    & \leq C \cdot \frac{p \norm{\hat\Sigma}_\infty}{\sqrt{K}} \\
    & \leq \frac{3C}{2} \frac{p \norm{\Sigma}_\infty}{\sqrt{K}},
\end{aligned}
\end{equation*}
almost surely for some universal constant $C > 0$.

\smallskip
Now we give an upper bound for probability of the conditioning event, $\norm{\hat\Sigma - \Sigma}_\infty \leq \frac{1}{2} \norm{\Sigma}_\infty$.
According to \cref{eq:conc_ineq_Sigma}, with probability at least $1 - \delta$,
\begin{equation*}
\begin{aligned}
    \norm{\hat\Sigma - \Sigma}_\infty 
    & \leq C \bar{L}_{\min\{4, q\}}^{1/2} \left(\frac{\bar{\nu}_{q}}{\delta} \right)^{\max\{2/q-1/2,0\}}
    \left(\frac{\log(2p/\delta)}{n}\right)^{\min\{1-2/q, 1/2\}} \\
    & \quad + C \left(\frac{\bar{\nu}_{q}}{\delta} \right)^{2/q} \left(\frac{\log(2p/\delta)}{n}\right)^{1-2/q}. \\
\end{aligned}
\end{equation*}
Hence solving 
\begin{equation*}
\begin{aligned}
    \frac{1}{2} \norm{\Sigma}_\infty 
    & \geq C \bar{L}_{\min\{4, q\}}^{1/2} \left(\frac{\bar{\nu}_{q}}{\delta} \right)^{\max\{2/q-1/2,0\}}
    \left(\frac{\log(2p/\delta)}{n}\right)^{\min\{1-2/q, 1/2\}} \\
    & \quad + C \left(\frac{\bar{\nu}_{q}}{\delta} \right)^{2/q} \left(\frac{\log(2p/\delta)}{n}\right)^{1-2/q}, \\
\end{aligned}
\end{equation*}
we obtain
\begin{equation*}
\begin{aligned}
    & \Pr\left[
        \norm{\hat\Sigma - \Sigma}_\infty \geq \frac{1}{2} \norm{\Sigma}_\infty
    \right] \\
    & \leq C p \cdot \exp\left( - \frac{n \norm{\Sigma}_\infty^2}{C \bar{L}_4} \right) 
    + \left( \frac{C}{\norm{\Sigma}_\infty} \right)^{q/2} \left( \frac{\log(2p/\delta_\circ)}{n} \right)^{q/2 - 1} \bar\nu_q,
\end{aligned}
\end{equation*}
where $\delta_\circ = \left( \frac{C}{\norm{\Sigma}_\infty} \right)^{q/2} \left( \frac{1}{n} \right)^{q/2 - 1} \bar\nu_q$ for some universal constant $C > 0$.


\smallskip
Lastly, each update in \cref{eq:psg_update} has rank at most three, so once the eigendecomposition of $\tilde\Sigma^{(k)}$ is known, the next eigendecomposition can be calculated in $O(p^{2})$ operations \citep[Section 5]{golub1973some}.

\section{Proofs of Lemmas}

\subsection{Proof of Lemma~\ref{thm:1_to_m}} \label{sec:pf_1_to_m}

For $j \in \ints_{n'}$, let $I_j \equiv ((j-1)m, jm]_n$ for $j \neq n'$ and $I_j \equiv ((n'-1)m, n]_n$ for $j = n'$, where $(\cdot,\cdot]_n$ are intervals in $\ints_n$ as defined in \cref{sec:Z_n}. Then $X'_j = X_{I_j}$.

First, we show that $X'_1, \dots, X'_{n'}$ are $1$-ring dependent. Suppose that $d_{n'}(j_1,j_2) > 1$, where $d_{n'}(j_1,j_2) \equiv \min\{\abs{j_1-j_2}, n'-\abs{j_1-j_2}\}$ is the distance between $j_1$ and $j_2$ in $\ints_{n'}$ as defined in \cref{sec:Z_n}. We show $X'_{j_1} \indep X'_{j_2}$ in the following two cases separately: (i) Both $j_1$ and $j_2$ are not $n'$; or (ii) $j_1 = n'$.
\begin{itemize}
    \item {\bf Case (i), both $j_1$ and $j_2$ are not $n'$:} Let $i_1$ and $i_2$ be arbitrary elements in $I_{j_1}$ and $I_{j_2}$, respectively. Then
    \begin{equation*}
    \begin{aligned}
        \abs{i_1 - i_2} 
        & = \max\{i_1 - i_2, i_2 - i_1\}
        > \max\{ (j_1-1)m - j_2m, (j_2-1)m - j_1m \} \\
        & = m (\max\{j_1 - j_2, j_2 - j_1\} - 1)
        = m (\abs{j_1 - j_2} - 1) \\
        & \geq m (d_{n'}(j_1, j_2) - 1) \geq m,
    \end{aligned}
    \end{equation*}
    and
    \begin{equation*}
    \begin{aligned}
        \abs{i_1 - i_2} 
        & = \max\{i_1 - i_2, i_2 - i_1\}
        < \max\{ j_1m - (j_2-1)m, j_2m - (j_1-1)m \} \\
        & = m (\max\{j_1 - j_2, j_2 - j_1\} + 1)
        = m (\abs{j_1 - j_2} + 1) \\
        & \leq m (n' - d_{n'}(j_1, j_2) + 1) \leq n - m.
    \end{aligned}
    \end{equation*}
    In sum,
    \(
        d_n(i_1, i_2) \equiv \min\{\abs{i_1-i_2}, n-\abs{i_1-i_2}\}
        > m,
    \)
    and $X_{i_1} \indep X_{i_2}$ as $X_1, \dots, X_n$ are $m$-ring dependent. Thus $X'_{j_1} = X_{I_{j_1}} \indep X_{I_{j_2}} = X'_{j_2}$.
    
    \item {\bf Case (ii), $j_1 = n'$:} In this case $j_2 \in [2, n'-2]_{n'}$. Let $i_1$ and $i_2$ be arbitrary elements in $I_{j_1}$ and $I_{j_2}$, respectively. Then
    \begin{equation*}
    \begin{aligned}
        \abs{i_1 - i_2} 
        & = \max\{i_1 - i_2, i_2 - i_1\}
        > \max\{ (n'-1)m - j_2m, (j_2-1)m - n \} \\
        & = \max\{(n' - j_2)m, j_2m - n\} - m \\
        & \geq \max\{2m, 2m - n\} - m \geq m,
    \end{aligned}
    \end{equation*}
    and
    \begin{equation*}
    \begin{aligned}
        \abs{i_1 - i_2} 
        & = \max\{i_1 - i_2, i_2 - i_1\}
        < \max\{ n - (j_2-1)m, j_2m - (n'-1)m \} \\
        & = \max\{n - j_2 m, j_2m - n'm\} + m \\
        & \leq \max\{n - 2m, -2m\} + m \leq n - m.
    \end{aligned}
    \end{equation*}
    In sum,
    \(
        d_n(i_1, i_2) \equiv \min\{\abs{i_1-i_2}, n-\abs{i_1-i_2}\}
        > m,
    \)
    and $X_{i_1} \indep X_{i_2}$ as $X_1, \dots, X_n$ are $m$-ring dependent. Thus $X'_{j_1} = X_{I_{j_1}} \indep X_{I_{j_2}} = X'_{j_2}$.
\end{itemize}
Summing up those two cases, we conclude that $X'_1, \dots, X'_{n'}$ are $1$-ring dependent.

Next, we prove the upperbound for $\bar{L}'_q$. For $q > 1$,
\begin{equation*}
\begin{aligned}
    \bar{L}'_q 
    & = \frac{1}{n'} \sum_{j=1}^{n'} \Exp[\abs{X'_{j}}^q]
    = \frac{1}{n'} \sum_{j=1}^{n'} \Exp[\abs{\tsum_{i \in I_j} X_i}^q] \\
    & \overset{\mathrm{(i)}}{\leq} \frac{1}{n'} \sum_{j=1}^{n'} (\tsum_{i \in I_j} (\Exp[\abs{X_i}^q])^{1/q} )^q 
    = \frac{1}{n'} \sum_{j=1}^{n'} \abs{I_j}^q \left( 
       \frac{1}{\abs{I_j}} \tsum_{i \in I_j} (\Exp[\abs{X_i}^q])^{1/q} 
    \right)^q \\
    & \overset{\mathrm{(ii)}}{\leq} \frac{1}{n'} \sum_{j=1}^{n'} \abs{I_j}^{q-1} \sum_{i \in I_j} \Exp[\abs{X_i}^q]
    \leq \frac{1}{n'} \max_j \abs{I_j}^{q-1} \sum_{i=1}^{n} \Exp[\abs{X_i}^q] \\
    & \leq \frac{m}{n} (2m)^{q-1} \sum_{i=1}^{n} \Exp[\abs{X_i}^q]
    = 2^{q-1} m^q \bar{L}_q,
\end{aligned}
\end{equation*}
where $\mathrm{(i)}$ follows the triangle inequality, and $\mathrm{(ii)}$ follows the Young's inequality as the map $x \mapsto x^q$ is convex for $q > 1$.
The upperbound for $\bar{\nu}'_q$ is derived similarly.   

Last, we prove Assumption~\eqref{assmp:min_ev} with $(\underline{\sigma}'_I, \underline{\sigma}')$ instead of $(\underline{\sigma}_I, \underline{\sigma})$. For any interval $I'$ in $\ints_{n'}$, 
\begin{equation*}
\begin{aligned}
    \underline{\sigma}'^2_{I'}
    & = \lambda_{\min}(\Var[Y'_{I'} | \mathscr{Y}'_{I'^\cmpl}])
    \overset{\mathrm{(i)}}{\geq} \lambda_{\min}(\Var[Y_{(\sqcup_{j \in I'} I_j)} | \mathscr{Y}_{(\sqcup_{j \in I'} I_j)^\cmpl}]) \\
    & \overset{\mathrm{(ii)}}{\geq} \underline{\sigma}^2 \cdot \max\{ \tsum_{j \in I'} \abs{I_j} - 2m, 0\}
    \geq \underline{\sigma}^2 \cdot \max\{ \abs{I'} m - 2m, 0 \}
    = \underline{\sigma}'^2 \cdot \max\{ \abs{I'} - 2, 0 \},
\end{aligned}
\end{equation*}
where $\mathrm{(i)}$ follows that $\mathscr{Y}'_{I'^\cmpl}$ is determined by $\mathscr{Y}_{(\sqcup_{j \in I'} I_j)^\cmpl}$, and $\mathrm{(ii)}$ follows the original Assumption~\eqref{assmp:min_ev} with $(\underline{\sigma}_I, \underline{\sigma})$. Assumption ~\eqref{assmp:min_var} with $(\sigma'_{\min,I}, \sigma'_{\min})$ instead of $(\sigma_{\min,I}, \sigma_{\min})$ is derived similarly, and Assumption~\eqref{assmp:var_ev} with $(\underline{\sigma}', \sigma'_{\min})$ instead of $(\underline{\sigma}, \sigma_{\min})$ is straightforward.

\subsection{Proof of Lemma~\ref{thm:kappa_comparison}}

For the first inequality in the first statement,
\begin{equation*}
\begin{aligned}
    \kappa_{(i_1,i_2)}(\delta) 
    & = \sup_{r\in\reals^p} \Pr[{X}_{(i_1,i_2)} \in A_{r,\delta}] \\
    & = \sup_{r\in\reals^p} \Exp\left[ \Pr[{X}_{(i_1,i_2)} \in A_{r,\delta} | \mathscr{X}_{\{i_2\}}] \right] \\
    & \leq \sup_{r\in\reals^p} \underset{\mathscr{X}_{\{i_2\}}}{\mathrm{esssup}}
    ~ \Pr[{X}_{(i_1,i_2)} \in A_{r,\delta} | \mathscr{X}_{\{i_2\}}] \\
    & = \kappa_{(i_1,i_2)|\{i_2\}}(\delta). \\
\end{aligned}
\end{equation*}
Similar arguments apply to the other inequalities in the first two statements.
For the first inequality in the third statement,
\begin{equation*}
\begin{aligned}
    \kappa_{(i_1,i_2)|\{i_2\}}(\delta) 
    & = \sup_{r\in\reals^p} \underset{\mathscr{X}_{\{i_2\}}}{\mathrm{esssup}}
    ~ \Pr[{X}_{(i_1,i_2)} \in A_{r,\delta} | \mathscr{X}_{\{i_2\}}] \\
    & = \sup_{r\in\reals^p} \underset{\mathscr{X}_{\{i_2\}}}{\mathrm{esssup}}
    ~ \Exp\left[ 
        \Pr[{X}_{(i_1,i_2-1)} \in A_{r - X_{i_2-1},\delta} 
        | \mathscr{X}_{\{i_2-1, i_2\}}] 
    \big| \mathscr{X}_{\{i_2\}} \right] \\
    & \overset{\text{(*)}}{=} \sup_{r\in\reals^p} \underset{\mathscr{X}_{\{i_2\}}}{\mathrm{esssup}}
    ~ \Exp\left[ 
        \Pr[{X}_{(i_1,i_2-1)} \in A_{r - X_{i_2-1},\delta} 
        | \mathscr{X}_{\{i_2-1\}}] 
    \big| \mathscr{X}_{\{i_2\}} \right] \\
    & \leq \sup_{r\in\reals^p} 
    \underset{\mathscr{X}_{\{i_2-1\}}}{\mathrm{esssup}}
    ~ \Pr[{X}_{(i_1,i_2-1)} \in A_{r - X_{i_2-1},\delta} 
    | \mathscr{X}_{\{i_2-1\}}]  \\
    & = \kappa_{(i_1,i_2-1)|\{i_2-1\}}(\delta), \\
\end{aligned}
\end{equation*}
where $\text{(*)}$ is due to $1$-dependence or $1$-ring dependence.
Similar arguments apply to the other inequalities in the third and fourth statements. For the first inequality in the last statement,
\begin{equation*}
\begin{aligned}
    \kappa_{(i_1,i_2)}(\delta) 
    & = \sup_{r\in\reals^p} \Pr[{X}_{(i_1,i_2)} \in A_{r,\delta}] \\
    & \leq \sup_{r\in\reals^p} \Pr[{X}_{(i_1,i_2)} \in A_{r,\delta'}] \\
    & = \kappa_{(i_1,i_2)}(\delta'). \\
\end{aligned}
\end{equation*}
Similar arguments apply to the other inequalities in the last statement.


\subsection{Proof of Lemma~\ref{thm:smoothing}} \label{sec:pf_smoothing} 

The smoothing lemma is the result of the serial application of the following two lemmas. \cref{thm:phi_smoothing} is a corollary of Theorem 2.1 in \citetalias{chernozhukov2020nearly}. We provide a standalone proof in \cref{sec:pf_phi_smoothing}.

\begin{lemma}[Lemma 1, \citetalias{kuchibhotla2020high}]
\label{thm:G_smoothing}
Suppose that $X$ is a $p$-dimensional random vector, and $Y \dist N(0,\Sigma)$ is a $p$-dimensional Gaussian random vector. Then, for any $\delta > 0$ and a standard Gaussian random vector $Z$, 
\begin{equation*}
    \mu(X, Y) \leq C \mu(X + \delta Z, Y + \delta Z) 
    + C \frac{\delta \log(ep)}{\sqrt{\min_{k \in [p]} \Sigma_{kk}}}.
\end{equation*}
\end{lemma}

\begin{lemma}
\label{thm:phi_smoothing}
Suppose that $X$ is a $p$-dimensional random vector, and $Y \dist N(0,\Sigma)$ is a $p$-dimensional Gaussian random vector. Then, for any $\phi > 0$, 
\begin{equation*}
\begin{aligned}
    \mu\left(X, Y\right) 
    \leq \sup_{r \in \reals^p} \abs*{ \Exp[f_{r,\phi}(X)] - \Exp[f_{r,\phi}(Y)]}
    + \frac{C}{\phi} \sqrt{\frac{\log(ep)}{\min_{k \in [p]} \Sigma_{kk}}}.
\end{aligned}
\end{equation*}
\end{lemma}

For any $\delta > 0$,
\begin{equation} \label{eq:G_smoothing}
    \mu\left(X, Y\right)
    \leq \mu\left(X + \delta Z, Y + \delta Z \right)
    + C \frac{\delta \log(ep)}{\sqrt{\min_{k \in [p]} \Sigma_{kk}}}.
\end{equation}
Then, for any $\delta > 0$ and $\phi > 0$, 
\begin{equation*}
\begin{aligned}
    & \mu\left(X + \delta Z, Y + \delta Z \right) \\
    & \leq C \sup_{r \in \reals^p} \abs*{ \Exp[f_{r,\phi}(X + \delta Z)] - \Exp[f_{r,\phi}(Y + \delta Z)]}
    + \frac{C}{\phi} \sqrt{\frac{\log(ep)}{\min_{k \in [p]} \Sigma_{kk}}} \\
    & \leq C \sup_{r \in \reals^p} \abs*{ \Exp[\rho_{r,\phi}^\delta(X)] - \Exp[\rho_{r,\phi}^\delta(Y)]}
    + \frac{C}{\phi} \sqrt{\frac{\log(ep)}{\min_{k \in [p]} \Sigma_{kk}}}. \\
\end{aligned}
\end{equation*}
In sum,
\begin{equation*}
    \mu(X, Y) 
    \leq C \sup_{r \in \reals^p} \abs*{ \Exp[\rho_{r,\phi}^\delta(X)] - \Exp[\rho_{r,\phi}^\delta(Y)]}
    + C \frac{\delta \log(ep) + \sqrt{\log(ep)}/\phi}{\sqrt{\min_{k \in [p]} \Sigma_{kk}}}.
\end{equation*}

\subsection{Proof of Lemmas~\ref{thm:assmp_induction_N} and \ref{thm:assmp_induction_P}}

Here we prove \cref{thm:assmp_induction_P}, which handles a more complicated dependency structure than that of \cref{thm:assmp_induction_N}. The same proof technique applies to the $1$-dependence cases.
To clarify the following arguments, we define $\tilde{X}_1 \equiv X_{i_1}$, $\tilde{X}_2 \equiv X_{i_1+1}$, \ldots, $\tilde{X}_{\tilde{n}} \equiv X_{i_2}$, where $\tilde{n} = i_2 - i_1 + 1$. Define $\tilde{Y}_{\tilde{i}}$, $\tilde\sigma^2_{\min,\tilde{I}}$ and $\underline{\tilde\sigma}_{\tilde{I}|\tilde{I}'}$ similarly for any $\tilde{i} \in [1,\tilde{n}]$ and $\tilde{I}, \tilde{I}' \subset [1, \tilde{n}]$. For notational ease, let $i$, $I$ and $I'$ be the counterpart in $[1, n]$ of $\tilde{i}$, $\tilde{I}$ and $\tilde{I}'$.

First, because $(X_1, \dots, X_n)$ is $1$-ring dependent, for any $i, i'$ such that $i_1 \leq i < i' \leq i_2 < i_1+n$, $i' - i > 1 \Longrightarrow X_{i} \indep X_{i'}$. Hence 
\begin{equation*}
    1 < \tilde{i}' - \tilde{i} < n-1 
    \Longrightarrow 1 < i' - i 
    \Longrightarrow X_i \indep X_{i'}
    \Longrightarrow \tilde{X}_{\tilde{i}} \indep \tilde{X}_{\tilde{i}'},
\end{equation*}
which proves the $1$-ring dependence of $(\tilde{X}_1, \dots, \tilde{X}_{\tilde{n}})$.
Next, for any $\tilde{I} \subset [1, \tilde{n}]$, 
\begin{equation*}
    \tilde{\sigma}^2_{\min, \tilde{I}} 
    = \min_{k \in [p]} \Var[\tilde{Y}_{\tilde{I}, k}]
    = \min_{k \in [p]} \Var[Y_{I,k}]
    = \sigma^2_{\min, I}
    \geq \sigma^2_{\min} \cdot \abs{I}
    = \sigma^2_{\min} \cdot \abs{\tilde{I}},
\end{equation*}
which proves Assumption~\eqref{assmp:min_var} for $(\tilde{X}_1, \dots, \tilde{X}_{\tilde{n}})$. Also, for any $\tilde{I}, \tilde{I}' \subset [1, \tilde{n}]$,
\begin{equation*}
    \underline{\tilde\sigma}^2_{\tilde{I} | \tilde{I}'} 
    = \lambda_{\min}(\Var[\tilde{Y}_{\tilde{I}} | \tilde{\mathscr{Y}}_{\tilde{I}'}])
    = \lambda_{\min}(\Var[Y_I | \mathscr{Y}_{I'}])
    \geq \underline{\sigma}^2 \cdot \abs{I \cap I'^{\indep}}
    = \underline{\sigma}^2 \cdot \abs{\tilde{I} \cap \tilde{I}'^{\indep}}.
\end{equation*}
Thus Assumption~\eqref{assmp:min_ev} extends to $(\tilde{X}_1, \dots, \tilde{X}_{\tilde{n}})$. Finally,
%
%
Assumption~\eqref{assmp:var_ev} is trivial for $(\tilde{X}_1, \dots, \tilde{X}_{\tilde{n}})$ because we use the same $\sigma_{\min}$ and $\underline{\sigma}$ as those for the original data.

\subsection{Proof of Lemma~\ref{thm:phi_smoothing}} \label{sec:pf_phi_smoothing}

By \cref{thm:G_smoothing},
\begin{equation*}
\begin{aligned}
    \Pr[X \in A_r]
    & \leq \Exp[f_{r,\phi}(X)]
    = \Exp[f_{r,\phi}(Y)] + \Exp[f_{r,\phi}(X)] - \Exp[f_{r,\phi}(Y)] \\
    & \leq \Pr[Y \in A_{r + \frac{1}{\phi}\mathbf{1}}]
    + \Exp[f_{r,\phi}(X)] - \Exp[f_{r,\phi}(Y)] \\
    & \leq \Pr[Y \in A_r] + \frac{C}{\phi}\sqrt{\frac{\log(ep)}{\min_{k \in [p]}\Sigma_{kk}}},
\end{aligned}
\end{equation*}
and similarly,
\begin{equation*}
\begin{aligned}
    \Pr[Y \in A_r]
    & \leq \Pr[Y \in A_{r - \frac{1}{\phi}\mathbf{1}}]
    + \frac{C}{\phi}\sqrt{\frac{\log(ep)}{\min_{k \in [p]}\Sigma_{kk}}} \\
    & \leq \Exp[f_{r - \frac{1}{\phi}\mathbf{1},\phi}(Y)]
    + \frac{C}{\phi}\sqrt{\frac{\log(ep)}{\min_{k \in [p]}\Sigma_{kk}}} \\
    & \leq \Exp[f_{r - \frac{1}{\phi}\mathbf{1},\phi}(X)]
    + \Exp[f_{r - \frac{1}{\phi}\mathbf{1},\phi}(Y)]
    - \Exp[f_{r - \frac{1}{\phi}\mathbf{1},\phi}(X)]
    + \frac{C}{\phi}\sqrt{\frac{\log(ep)}{\min_{k \in [p]}\Sigma_{kk}}} \\
    & \leq \Pr[X \in A_r]
    + \Exp[f_{r - \frac{1}{\phi}\mathbf{1},\phi}(Y)]
    - \Exp[f_{r - \frac{1}{\phi}\mathbf{1},\phi}(X)]
    + \frac{C}{\phi}\sqrt{\frac{\log(ep)}{\min_{k \in [p]}\Sigma_{kk}}}.
\end{aligned}
\end{equation*}
Hence,
\begin{equation*}
    \sup_{r \in \reals^p} \abs*{\Pr[X \in A_r] - \Pr[Y \in A_r]}
    \leq \sup_{r \in \reals^p} \abs*{\Exp[f_{r,\phi}(X)] - \Exp[f_{r,\phi}(Y)]} 
    + \frac{C}{\phi}\sqrt{\frac{\log(ep)}{\min_{k \in [p]}\Sigma_{kk}}}.
\end{equation*}

\subsection{Proof of Remainder Lemmas} \label{sec:pf_remainder_lemma}

We observe that all remainder terms are in forms of 
\begin{equation*}
\begin{aligned}
    \frac{1}{(\beta-1)!} \int_0^1 (1-t)^{\beta-1} \Exp\left[
        \inner*{\nabla^\alpha \rho_{r,\phi}^\delta(X_{J_1} + W^\noindep + Y_{J_2} + t W_J), (\otimes_{k=1}^{\alpha-\beta} W_{j_k}) \otimes W_J^{\otimes \beta}} 
    \right] dt,
\end{aligned}
\end{equation*}
where $1 \leq \beta \leq \alpha$ and $J_1, J_2$ and $J$ are subsets of $[1,n]$ satisfying $\mathscr{X}_{J_1}$, $\mathscr{Y}_{J_2}$ and $\mathscr{W}_{\{j_1,\dots,j_{\alpha-\beta}\} \cup J}$ are mutually independent. We also note that $\mathscr{X}_{J_1} \cup \mathscr{Y}_{J_2} \cup \{W^\noindep\}$ is independent from $\mathscr{W}_{\{j_1,\dots,j_{\alpha-\beta}\}}$. Let $J^{\noindep}$ be the index set corresponding to $W^\noindep$. Here we prove a succinct form of the remainder lemmas:

\begin{lemma} \label{thm:generalized_remainder_lemma}
There exists a universal constant $C > 0$ such that for any $n$, $(X_i \in \reals^p: i \in [1,n])$ with any dependence structure satisfying Assumption \eqref{assmp:min_ev},
$\gamma_1, \gamma_2 \in [0,1]$, $\eta > 0$, $\delta \geq \sigma_{\min}$ and $\phi \geq \frac{1}{\delta\log(ep)}$, remainder terms in the above form satisfy
\begin{equation*}
\begin{aligned}
    & \abs*{ \int_0^1 (1-t)^{\beta-1} \Exp\left[
        \inner*{\nabla^\alpha \rho_{r,\phi}^\delta(X_{J_1} + W^\noindep + Y_{J_2} + t W_J), (\otimes_{k=1}^{\alpha-\beta} W_{j_k}) \otimes W_J^{\otimes \beta}} 
    \right] dt } \\
    & \leq C \left[
        \frac{(\log(ep))^{\alpha/2}}{{\delta'}^{\alpha}}
        \min\left\{1, \kappa_{J_1|J^\noindep}\left(\delta^\circ\right) + \kappa^\circ_{\abs{J_1}}(\delta') \right\}
    \right] \\
    & \quad \times \left[ \begin{aligned}
        & \frac{\phi^{\gamma_1} {\delta'}^{\gamma_1}}{(\log(ep))^{\gamma_1/2} } 
        \norm*{\Exp\left[
            \abs*{(\otimes_{k=1}^{\alpha-\beta} W_{j_k}) \otimes W_J^{\otimes \beta}}
        \right]}_\infty \\
        & + \frac{\phi^{\gamma_2} {\delta'}^{\gamma_2-\eta}}{(\log(ep))^{(\gamma_2-\eta)/2}} 
        \Exp\left[ 
            \prod_{k=1}^{\alpha-\beta} \norm{W_{j_k}}_\infty \norm{W_J}_\infty^{\beta}
            (\norm{W_J}_\infty^{\eta} + \norm{W^\noindep}_\infty^{\eta}) 
        \right]
    \end{aligned} \right],
\end{aligned}
\end{equation*}
where the minimum eigenvalue of $\Var[Y_{J_2} | \mathscr{Y}_{J_2^\cmpl}]$ is at least $\underline{\sigma}_{J_2}^2$ for some $\underline{\sigma}^2_{J_2} > 0$, $\abs*{(\otimes_{k=1}^{\alpha-\beta} W_{j_k}) \otimes W_J^{\otimes \beta}}$ is the element-wise absolute operation, $\delta' \equiv \sqrt{\delta^2 + \underline{\sigma}_{J_2}^2}$, $\delta^\circ \equiv 12 \delta' \sqrt{\log(p\abs{J_1})}$, and $\kappa^\circ_{\abs{J_1}}(\delta') \equiv \frac{\delta' \log(ep)}{\sigma_{\min} \sqrt{\max\{\abs{J_1}, 1\}}}$.
\end{lemma}

We prove the above generalized remainder lemma here. Because $\Var[Y_{J_2} | \mathscr{Y}_{J_2^\cmpl}]$ has a minimum eigenvalue at least $\underline{\sigma}_{J_2}^2$, 
\begin{equation*}
    Y_{J_2} | \mathscr{Y}_{J_2^\cmpl} \overset{d}{=} Y^\circ_{J_2} + \underline{\sigma}_{J_2} \cdot Z, ~~\text{almost surely},
\end{equation*}
where $Y^\circ_{J_2}$ is the Gaussian random vector with mean $\Exp[Y_{J_2}|\mathscr{Y}_{J_2^\cmpl}]$ and variance $\Var[Y_{J_2} | \mathscr{Y}_{J_2^\cmpl}] - \underline{\sigma}_{J_2}^2 I_p$.
For brevity, let $W^\circ \equiv X_{J_1} + Y^\circ_{J_2}$.
Because $\delta' = \sqrt{\delta^2 + \underline{\sigma}_{J_2}^2}$, the remainder term is decomposed into
\begin{equation*}
\begin{aligned}
    & \int_0^1 (1-t)^{\beta-1} \Exp\left[
        \inner*{\nabla^\alpha \rho_{r,\phi}^\delta(X_{J_1} + W^\noindep + Y_{J_2} + t W_J), (\otimes_{k=1}^{\alpha-\beta} W_{j_k}) \otimes W_J^{\otimes \beta}} 
    \right] dt \\
    & = \int_0^1 (1-t)^{\beta-1} \Exp\left[ 
        \inner*{\nabla^\alpha \rho_{r,\phi}^{\delta'} (W^\circ + W^\noindep + t W_J), (\otimes_{k=1}^{\alpha-\beta} W_{j_k}) \otimes W_J^{\otimes \beta}} 
    \right] dt \\
    & = \int_0^1 (1-t)^{\beta-1} \Exp\left[
        \inner*{\nabla^\alpha \rho_{r,\phi}^{\delta'} (W^\circ + W^\noindep + t W_J), (\otimes_{k=1}^{\alpha-\beta} W_{j_k}) \otimes W_J^{\otimes \beta}} \Ind_{W_J, 1}
    \right] dt \\
    & \quad + \int_0^1 (1-t)^{\beta-1} \Exp\left[
        \inner*{\nabla^{\alpha-1} \rho_{r,\phi}^{\delta'} (W^\circ + W^\noindep + t W_J), (\otimes_{k=1}^{\alpha-\beta} W_{j_k}) \otimes W_J^{\otimes \beta}}
        \Ind_{W_J,2}
    \right] dt \\
    & =: \int_0^1 (1-t)^{\beta-1} T_{W_J,1}(t) dt 
    + \int_0^1 (1-t)^{\beta-1} T_{W_J,2}(t) dt, 
\end{aligned}
\end{equation*}
where
\begin{equation*}
\begin{aligned}
    \Ind_{W_J,1} & = \Ind_{\{\norm{W_J}_\infty \text{ and } \norm{W^{\noindep}}_\infty < \frac{\delta'}{\sqrt{\log(ep)}}\}}, \\
    \Ind_{W_J,2} & = \Ind_{\{\norm{W_J}_\infty \text{ or } \norm{W^{\noindep}}_\infty \geq \frac{\delta'}{\sqrt{\log(ep)}}\}}.
\end{aligned}
\end{equation*}
We upper-bound the terms $T_{W_J,1}(t)$ and $T_{W_J,2}(t)$ separately. First,
\begin{equation} \label{eq:decompose_T_1}
\begin{aligned}
    & \abs*{T_{W_J,1}(t)} \\
    & = \abs*{ \Exp\left[
        \inner*{\nabla^\alpha \rho_{r,\phi}^{\delta'} (W^\circ + W^\noindep + t W_J), 
        (\otimes_{k=1}^{\alpha-\beta} W_{j_k}) \otimes W_J^{\otimes \beta}} \Ind_{W_J, 1}
    \right] } \\
    & = \abs*{ \Exp\left[
        \sum_{i_1,\dots,i_\alpha} \nabla^{(i_1,\dots,i_\alpha)} \rho_{r,\phi}^{\delta'} (W^\circ + W^\noindep + t W_J)
        \cdot \prod_{k \leq \alpha-\beta} W_{j_k}^{(i_k)} 
        \prod_{k > \alpha-\beta} W_J^{(i_k)} \cdot \Ind_{W_J,1}\right] } \\
    & \leq \Exp\left[
        \sum_{i_1,\dots,i_\alpha} \sup_{z \in \mathcal{B}}\abs*{\nabla^{(i_1,\dots,i_\alpha)} \rho_{r,\phi}^{\delta'} (W^\circ + z)}
        \cdot \abs*{\prod_{k \leq \alpha-\beta} W_{j_k}^{(i_k)} 
        \prod_{k > \alpha-\beta} W_J^{(i_k)}} 
        \cdot \Ind_{W_J,1} 
    \right] \\
    & \leq \sum_{i_1,\dots,i_\alpha} \Exp\left[ 
        \sup_{z \in \mathcal{B}}\abs*{\nabla^{(i_1,\dots,i_\alpha)} \rho_{r,\phi}^{\delta'} (W^\circ + z)}
        \cdot \abs*{\prod_{k \leq \alpha-\beta} W_{j_k}^{(i_k)} 
        \prod_{k > \alpha-\beta} W_J^{(i_k)}} 
        \cdot \Ind_{W_J,1} 
    \right] \\
    & = \sum_{i_1,\dots,i_\alpha} \Exp\left[ 
        \sup_{z \in \mathcal{B}}\abs*{\nabla^{(i_1,\dots,i_\alpha)} \rho_{r,\phi}^{\delta'} (W^\circ + z)}
    \right]
    \Exp\left[
        \abs*{\prod_{k \leq \alpha-\beta} W_{j_k}^{(i_k)} 
        \prod_{k > \alpha-\beta} W_J^{(i_k)}} 
        \cdot \Ind_{W_J,1} 
    \right] \\
    & \leq \Exp\left[
        \sum_{i_1,\dots,i_\alpha} \sup_{z \in \mathcal{B}}
        \abs*{ \nabla^{(i_1,\dots,i_\alpha)} \rho_{r,\phi}^{\delta'}
        ( W^\circ_j + z ) } 
    \right]
    \norm*{\Exp\left[
        \abs*{(\otimes_{k=1}^{\alpha-\beta} W_{j_k}) \otimes W_J^{\otimes \beta}}
    \right]}_\infty,
\end{aligned}
\end{equation}
where $\mathcal{B} = \{z: \norm{z}_\infty \leq \frac{2\delta'}{\sqrt{\log(ep)}}\}$, and $\abs*{(\otimes_{k=1}^{\alpha-\beta} W_{j_k}) \otimes W_J^{\otimes \beta}}$ is the element-wise absolute operation. The fifth equality follows the independence between $W^\circ$ and $\mathscr{W}_{\{j_1,\dots,j_{\alpha-\beta}\} \cup J}$. We decompose the expectation term on the last line by
\begin{equation} \label{eq:decompose_third_derivative}
\begin{aligned}
     & \Exp\left[
        \sum_{i_1,\dots,i_\alpha} \sup_{z \in \mathcal{B}}
        \abs*{ \nabla^{(i_1,\dots,i_\alpha)} \rho_{r,\phi}^{\delta'}
        ( W^\circ_j + z ) } 
    \right] \\
    & \leq \Exp\left[ \begin{aligned}
        \sum_{i_1,\dots,i_\alpha} \sup_{z \in \mathcal{B}}
        \abs*{ \nabla^{(i_1,\dots,i_\alpha)} \rho_{r,\phi}^{\delta'}
        ( W^\circ_j + z ) } 
        \cdot \Ind_{W^\circ,1}
    \end{aligned} \right] \\
    & \quad + \Exp\left[ \begin{aligned}
        \sum_{i_1,\dots,i_\alpha} \sup_{z \in \mathcal{B}}
        \abs*{ \nabla^{(i_1,\dots,i_\alpha)} \rho_{r,\phi}^{\delta'}
        ( W^\circ_j + z ) } 
        \cdot \Ind_{W^\circ,2}
    \end{aligned} \right],
\end{aligned}
\end{equation}
where the value of $h > 0$ will be determined later, and
\begin{equation*}
\begin{aligned}
    \Ind_{W^\circ,1} & = \Ind_{\left\{
            \norm{W^\circ - \partial A_r}_\infty \leq 12 {\delta'} \sqrt{\log(ph)}
        \right\}}, \\
    \Ind_{W^\circ,2} & = \Ind_{\left\{
            \norm{W^\circ - \partial A_r}_\infty > 12 {\delta'} \sqrt{\log(ph)}
        \right\}}.
\end{aligned}
\end{equation*}
For the first term, based on Lemmas 6.1 and 6.2 of \citetalias{chernozhukov2020nearly},
\begin{equation*}
\begin{aligned}
    & \Exp\left[ \begin{aligned}
        \sum_{i_1,\dots,i_\alpha} \sup_{z \in \mathcal{B}}
        \abs*{ \nabla^{(i_1,\dots,i_\alpha)} \rho_{r,\phi}^{\delta'}
        ( W^\circ + z ) } 
        \cdot \Ind_{W^\circ,1}
    \end{aligned} \right]  \\
    & \leq \sup_{w \in \reals^p} \sum_{i_1,\dots,i_\alpha} \sup_{z \in \mathcal{B}}
        \abs*{ \nabla^{(i_1,\dots,i_\alpha)} \rho_{r,\phi}^{\delta'}
        ( w + z ) } 
    \cdot \Exp\left[ \Ind_{W^\circ,1} \right]  \\
    & \leq C \frac{\phi^{\gamma_1} (\log(ep))^{(\alpha-\gamma_1)/2}}{{\delta'}^{\alpha-\gamma_1}}
    \Pr\left[
        \norm{X_{J_1} - \partial A_{r-Y_{J_2}^\circ}}_\infty \leq 12 {\delta'} \sqrt{\log(ph)}
    \right] \\
    & \leq C \frac{\phi^{\gamma_1} (\log(ep))^{(\alpha-\gamma_1)/2}}{{\delta'}^{\alpha-\gamma_1}}
    \min\left\{1, \kappa_{J_1}\left(12 {\delta'} \sqrt{\log(ph)}\right) \right\},
\end{aligned}
\end{equation*}
for any $\gamma_1 \in [0,1]$.
For the last term, based on Lemma 10.5 of \citet{lopes2022central},
\begin{equation*}
\begin{aligned}
    \Exp\left[ \begin{aligned}
        \sum_{i_1,\dots,i_\alpha} \sup_{z \in \mathcal{B}}
        \abs*{ \nabla^{(i_1,\dots,i_\alpha)} \rho_{r,\phi}^{\delta'}
        ( W^\circ + z ) } 
        \cdot \Ind_{W^\circ,2}
    \end{aligned} \right]
    \leq C \frac{1}{{\delta'}^\alpha h}.
\end{aligned}
\end{equation*}
Plugging the last two results in \cref{eq:decompose_T_1,eq:decompose_third_derivative}, the resulting upper bound for $T_{W_J,1}(t)$ is
\begin{equation*}
\begin{aligned}
    \abs*{T_{W_J,1}(t)}
    & \leq C \left[
        \frac{\phi^{\gamma_1} (\log(ep))^{(\alpha-\gamma_1)/2}}{{\delta'}^{\alpha-\gamma_1}}
        \min\left\{1, \kappa_{J_1}\left(12 {\delta'} \sqrt{\log(ph)}\right) \right\}
        + \frac{1}{{\delta'}^\alpha h}
    \right] \\
    & \quad \times \norm*{\Exp\left[
        \abs*{(\otimes_{k=1}^{\alpha-\beta} W_{j_k}) \otimes W_J^{\otimes \beta}}
    \right]}_\infty,
\end{aligned}
\end{equation*}
for any $t \in [0,1]$, $\gamma_1 \in [0,1]$ and $h > 0$. 
Replacing the minimum with $1$ and minimizing over $h > 0$, we get
\begin{equation*}
\begin{aligned}
    \abs*{T_{W_J,1}(t)}
    \leq C \frac{\phi^{\gamma_1} (\log(ep))^{(\alpha-\gamma_1)/2}}{{\delta'}^{\alpha-\gamma_1}} 
    \norm*{\Exp\left[
        \abs*{(\otimes_{k=1}^{\alpha-\beta} W_{j_k}) \otimes W_J^{\otimes \beta}}
    \right]}_\infty.  
\end{aligned}
\end{equation*}
Moreover, plugging-in
\begin{equation*}
     h = \frac{\sigma_{\min}}{\phi^{\gamma_1} {\delta'}^{1+\gamma_1}} \sqrt{\frac{\abs{J_1}}{(\log(ep))^{\alpha-\gamma_1+2}}},
\end{equation*}
we obtain
\begin{equation*}
\begin{aligned}
    \abs*{T_{W_J,1}(t)} 
    \leq C \frac{\phi^{\gamma_1} (\log(ep))^{(\alpha-\gamma_1)/2}}{{\delta'}^{\alpha-\gamma_1}} 
    \norm*{\Exp\left[
        \abs*{(\otimes_{k=1}^{\alpha-\beta} W_{j_k}) \otimes W_J^{\otimes \beta}}
    \right]}_\infty 
    (\kappa_{J_1}(\delta^\circ) + \kappa^\circ_{\abs{J_1}}(\delta')),
\end{aligned}
\end{equation*}
where $\delta^\circ \equiv 12 \delta' \sqrt{\log(p\abs{J_1})}$ and $\kappa^\circ_{\abs{J_1}}(\delta') \equiv \frac{\delta'\log(ep)}{\sigma_{\min}\sqrt{\abs{J_1}}}$. In sum,
\begin{equation*}
\begin{aligned}
    & \abs*{T_{W_J,1}(t)} \\
    & \leq C \frac{\phi^{\gamma_1} (\log(ep))^{(\alpha-\gamma_1)/2}}{{\delta'}^{\alpha-\gamma_1}} 
    \norm*{\Exp\left[
        \abs*{(\otimes_{k=1}^{\alpha-\beta} W_{j_k}) \otimes W_J^{\otimes \beta}}
    \right]}_\infty 
    \min\{1, \kappa_{J_1}(\delta^\circ) + \kappa^\circ_{\abs{J_1}}(\delta')\},
\end{aligned}
\end{equation*}
for any $t \in [0,1]$, $\gamma_1 \in [0,1]$ as long as $\delta \geq \sigma_{\min}$ and $\phi\delta \geq \frac{1}{\log(ep)}$.
Now we bound
\begin{equation*}
\begin{aligned}
    \abs*{T_{W_J,2}(t)}
    = \abs*{ \Exp\left[
        \inner*{\nabla^\alpha \rho_{r,\phi}^{\delta'} (W^\circ + W^\noindep + t W_J), 
        (\otimes_{k=1}^{\alpha-\beta} W_{j_k}) \otimes W_J^{\otimes \beta}} \Ind_{W_J, 2}
    \right] }. \\
\end{aligned}
\end{equation*}
Conditional on $W^\noindep$ and $W_J$,
\begin{equation*}
\begin{aligned}
    & \Exp\left[\left.
        \norm*{\nabla^\alpha \rho_{r,\phi}^{\delta'}
        \left(
            W^\circ + W^{\noindep} + t W_J
        \right)}_1
        \right| W^{\noindep}, W_J
    \right] \\
    & \leq \Exp\left[\left.
        \norm*{\nabla^\alpha \rho_{r,\phi}^{\delta'}
        \left(
            W^\circ + W^{\noindep} + t W_J
        \right)}_1
        \Ind_{\{ \norm{W^\circ - \partial A_{r'}}_\infty \leq 10 \delta' \sqrt{\log(ph)} \}}
        \right| W^{\noindep}, W_J
    \right] \\
    & \quad + \Exp\left[\left.
        \norm*{\nabla^\alpha \rho_{r,\phi}^{\delta'}
        \left(
            W^\circ + W^{\noindep} + t W_J
        \right)}_1
        \Ind_{\{ \norm{W^\circ - \partial A_{r'}}_\infty > 10 \delta' \sqrt{\log(ph)} \}}
        \right| W^{\noindep}, W_J
    \right],
\end{aligned}
\end{equation*}
where $r' = r - W^\noindep - t W_J$ is deterministic given $W^\noindep$ and $W_J$. Applying Lemmas 6.1, 6.2 of \citetalias{chernozhukov2020nearly} and Lemma 10.5 of \citet{lopes2022central} to the two terms, respectively, 
\begin{equation*}
\begin{aligned}
    & \Exp\left[\left.
        \norm*{\nabla^\alpha \rho_{r,\phi}^{\delta'}
        \left(
            W^\circ + W^{\noindep} + t W_J
        \right)}_1
        \right| W^{\noindep}, W_J
    \right] \\
    & \leq C \frac{\phi^{\gamma_2} (\log(ep))^{(\alpha-\gamma_2)/2}}{{\delta'}^{\alpha-\gamma_2}}
    \Pr\left[\left. \norm{W^\circ - \partial A_{r'}}_\infty \leq 10 \delta' \sqrt{\log(ph)} \right| W^\noindep, W_J \right] + C \frac{1}{\delta'^\alpha h} \\
    & \leq C \frac{\phi^{\gamma_2} (\log(ep))^{(\alpha-\gamma_2)/2}}{{\delta'}^{\alpha-\gamma_2}}
    \min\left\{1, \kappa_{J_1|J^\noindep}\left(10\delta'\sqrt{\log(ph)}\right)\right\} + C \frac{1}{\delta'^\alpha h},
\end{aligned}
\end{equation*}
almost surely, for any $\gamma_2 \in [0,1]$ and $h > 0$. Putting the last two results together,
\begin{equation*}
\begin{aligned}
    & \abs*{T_{W_J,2}(t)} \\
    & \leq C \left(
        \frac{\phi^{\gamma_2} (\log(ep))^{(\alpha-\gamma_2)/2}}{{\delta'}^{\alpha-\gamma_2}}
        \min\left\{1, \kappa_{J_1|J^\noindep}\left(10\delta'\sqrt{\log(ph)}\right)\right\} 
        + \frac{1}{\delta'^\alpha h}
    \right) \\
    & \quad \times \Exp\left[\begin{aligned}
        \prod_{k=1}^{\alpha-\beta} \norm{W_{j_k}}_\infty \norm{W_J}_\infty^{\beta} \Ind_{W_J, 2}
    \end{aligned}\right], \\
\end{aligned}
\end{equation*}
for any $\gamma_2 \in [0,1]$ and $h > 0$.
Because $\Ind_{W_J,2} = \Ind_{\{\norm{W_J}_\infty \text{ or } \norm{W^{\noindep}}_\infty \geq \frac{\delta'}{\sqrt{\log(ep)}}\}} \geq \Ind_{\{\norm{W_J}_\infty > \frac{\delta'}{\sqrt{\log(ep)}}\}} + \Ind_{\{\norm{W^\noindep}_\infty > \frac{\delta'}{\sqrt{\log(ep)}}\}}$,
\begin{equation*}
\begin{aligned}
    & \Exp\left[ \prod_{k=1}^{\alpha-\beta} \norm{W_{j_k}}_\infty \norm{W_J}_\infty^{\beta} \Ind_{W_J, 2} \right] \\
    & \leq \Exp\left[ \prod_{k=1}^{\alpha-\beta} \norm{W_{j_k}}_\infty \norm{W_J}_\infty^{\beta} 
    \Ind_{\{\norm{W_J}_\infty > \frac{\delta'}{\sqrt{\log(ep)}}\}} \right]
    + \Exp\left[ \prod_{k=1}^{\alpha-\beta} \norm{W_{j_k}}_\infty \norm{W_J}_\infty^{\beta} 
    \Ind_{\{\norm{W^\noindep}_\infty > \frac{\delta'}{\sqrt{\log(ep)}}\}} \right] \\
    & \leq \frac{(\log(ep))^{\eta/2}}{\delta'^{\eta}} \Exp\left[ 
        \prod_{k=1}^{\alpha-\beta} \norm{W_{j_k}}_\infty \norm{W_J}_\infty^{\beta}
        (\norm{W_J}_\infty^{\eta} + \norm{W^\noindep}_\infty^{\eta}) 
    \right], \\
\end{aligned}
\end{equation*}
for any $\eta > 0$.
The resulting upper bound for $T_{W_J,2}(t)$ is
\begin{equation*}
\begin{aligned}
    & \abs*{T_{W_J,2}(t)} \\
    & \leq C \left[
        \frac{\phi^{\gamma_2} (\log(ep))^{(\alpha-\gamma_2)/2}}{{\delta'}^{\alpha-\gamma_2}}
        \min\left\{1, \kappa_{J_1|J^\noindep}\left(10\delta'\sqrt{\log(ph)}\right)\right\} + \frac{1}{\delta'^\alpha h}
    \right] \\
    & \quad \times \frac{(\log(ep))^{\eta/2}}{\delta'^{\eta}} \Exp\left[ 
            \prod_{k=1}^{\alpha-\beta} \norm{W_{j_k}}_\infty \norm{W_J}_\infty^{\beta}
            (\norm{W_J}_\infty^{\eta} + \norm{W^\noindep}_\infty^{\eta}) 
        \right],
\end{aligned}
\end{equation*}
for any $\gamma_2 \in [0,1]$, $\eta > 0$ and $h > 0$. By similar choices of $h$ with for $T_{W_J,1}(t)$, we obtain
\begin{equation*}
\begin{aligned}
    & \abs*{T_{W_J,2}(t)} \\
    & \leq C \left[
        \frac{\phi^{\gamma_2} (\log(ep))^{(\alpha-\gamma_2+\eta)/2}}{{\delta'}^{\alpha-\gamma_2+\eta}}
        \min\left\{1, \kappa_{J_1|J^\noindep}\left(\delta^\circ\right) + \kappa^\circ_{\abs{J_1}}(\delta') \right\}
    \right] \\
    & \quad \times \Exp\left[ 
        \prod_{k=1}^{\alpha-\beta} \norm{W_{j_k}}_\infty \norm{W_J}_\infty^{\beta}
        (\norm{W_J}_\infty^{\eta} + \norm{W^\noindep}_\infty^{\eta}) 
    \right],
\end{aligned}
\end{equation*}
for any $t \in [0,1]$ and $\eta > 0$ as long as $\delta \geq \sigma_{\min}$ and $\phi\delta \geq \frac{1}{\log(ep)}$.
Putting the upperbounds for $T_{W_J,1}(t)$ and $T_{W_J,2}(t)$ together, we obtain the following upper bound for the entire remainder term:
\begin{equation*}
\begin{aligned}
    & \abs*{ \int_0^1 (1-t)^{\beta-1} \Exp\left[
        \inner*{\nabla^\alpha \rho_{r,\phi}^\delta(X_{J_1} + W^\noindep + Y_{J_2} + t W_J), (\otimes_{k=1}^{\alpha-\beta} W_{j_k}) \otimes W_J^{\otimes \beta}} 
    \right] dt } \\
    & \leq C \left[
        \frac{(\log(ep))^{\alpha/2}}{{\delta'}^{\alpha}}
        \min\left\{1, \kappa_{J_1|J^\noindep}\left(\delta^\circ\right) + \kappa^\circ_{\abs{J_1}}(\delta') \right\}
    \right] \\
    & \quad \times \left[ \begin{aligned}
        & \frac{\phi^{\gamma_1} {\delta'}^{\gamma_1}}{(\log(ep))^{\gamma_1/2} } 
        \norm*{\Exp\left[
            \abs*{(\otimes_{k=1}^{\alpha-\beta} W_{j_k}) \otimes W_J^{\otimes \beta}}
        \right]}_\infty \\
        & + \frac{\phi^{\gamma_2} {\delta'}^{\gamma_2-\eta}}{(\log(ep))^{(\gamma_2-\eta)/2}} 
        \Exp\left[ 
            \prod_{k=1}^{\alpha-\beta} \norm{W_{j_k}}_\infty \norm{W_J}_\infty^{\beta}
            (\norm{W_J}_\infty^{\eta} + \norm{W^\noindep}_\infty^{\eta}) 
        \right]
    \end{aligned} \right],
\end{aligned}
\end{equation*}
for any $\gamma_1, \gamma_2 \in [0,1]$ and $\eta > 0$ as long as $\delta \geq \sigma_{\min}$ and $\phi\delta \geq \frac{1}{\log(ep)}$.
This proves \cref{thm:generalized_remainder_lemma}.

\cref{thm:generalized_remainder_lemma} implies all the remainder theorems in the main text and appendices of this paper. For example, in \cref{thm:remainder_lemma_3N}, one of the third-order remainder terms in $\mathfrak{R}^{(3,1)}_{W_j}$ is
\begin{equation*}
\begin{aligned}
    & \Exp\left[ \int_0^1 (1-t)^2 \inner*{
        \nabla^3 \varphi_r^\vareps\left(
            W^\cmpl_{[j,j]} + t W_j
        \right),
        X_j^{\otimes 3}} ~dt 
    \right] \\
    & = \Exp\left[ \int_0^1 (1-t)^2 \inner*{
        \nabla^3 \varphi_r^\vareps\left(
            X_{[1,j-1)} + W^\noindep + Y_{(j+1,n]} + t W_j
        \right),
        X_j^{\otimes 3}} ~dt
    \right],
\end{aligned}
\end{equation*}
where $W_{[j_1,j_2]}^\cmpl \equiv X_{[1,j_1)} + Y_{(j_2,n]}$ and $W^\noindep \equiv X_{j-1} + Y_{j+1}$. Because of Assumption \eqref{assmp:min_ev}, the minimum eigenvalue of $\Var[Y_{(j+1,n]} | \mathscr{Y}_{[1,j+1]}]$ is at least $\underline{\sigma}^2 \max\{n-j-2,0\}$.
Let $\underline{\sigma}_{j}^2 \equiv \underline{\sigma}^2 \max\{j, 0\}$ and $\delta_{j}^2 \equiv \sqrt{\delta^2 + \underline{\sigma}_{j}^2}$. 
Applying \cref{thm:generalized_remainder_lemma} to this term with $\gamma_1 = 0$, $\gamma_2 = \min\{1, q-3\}$ and $\eta = q-3$, we obtain for $q \geq 3$, $\delta \geq \sigma_{\min}$ and $\phi \geq \frac{1}{\delta\log(ep)}$,
\begin{equation*}
\begin{aligned}
    & \abs*{ \Exp\left[ \int_0^1 (1-t)^2 \inner*{
        \nabla^3 \varphi_r^\vareps\left(
            X_{[1,j-1)} + W^\noindep + Y_{(j+1,n]} + t W_j
        \right),
        W_j^{\otimes 3}} ~dt
    \right] } \\
    & \leq C \frac{(\log(ep))^{3/2}}{\delta_{n-j-2}^3} 
        \min\{1, \kappa_{[1,j-2]|\{j-1\}}(\delta_{n-j-2}^\circ) + \kappa^\circ_{j-2}(\delta_{n-j-2})\} \\
    & \quad \times \left[ 
        \norm*{\Exp\left[
            \abs*{W_j^{\otimes 3}}
        \right]}_\infty
        + \frac{\phi^{\min\{1,q-3\}} \delta_{n-j-2}^{\min\{0,4-q\}}}{(\log(ep))^{\min\{0,4-q\}/2}} \Exp\left[ 
            \norm{W_j}_\infty^3
            (\norm{W_j}_\infty^{q-3} + \norm{X_{j-1}+Y_{j+1}}_\infty^{q-3}) 
        \right]
    \right] \\
    & \leq C \frac{(\log(ep))^{3/2}}{\delta_{n-j-2}^3} 
    \min\{1, \kappa_{[1,j-2]|\{j-1\}}(\delta_{n-j-2}^\circ) + \kappa^\circ_{j-2}(\delta_{n-j-2})\} \\
    & \quad \times \left[ 
        L_{3,j}
        + 2^{q-3} \phi^{\min\{1,q-3\}} \frac{(\log(ep))^{\max\{0,q-4\}/2}}{\delta^{\max\{0,q-4\}}_{n-j-2}} (\nu_{q,j-1} + \nu_{q,j} + \nu_{q,j+1})
    \right],
\end{aligned}
\end{equation*}
where $\delta^\circ_{j} \equiv 12 \delta_{j} \sqrt{\log(pn)}$ and $\kappa^\circ_{j}(\delta) \equiv \frac{\delta \log(ep)}{\sigma_{\min} \sqrt{\max\{j, 1\}}}$.
Based on similar applications of \cref{thm:generalized_remainder_lemma} to the other terms of $\mathfrak{R}^{(3,1)}_{X_j}$, 
\begin{equation*}
\begin{aligned}
    \abs*{\Exp \left[ \mathfrak{R}_{W_j}^{(3,1)} \right]} 
    & \leq C \frac{(\log(ep))^{3/2}}{\delta_{n-j-4}^3} \left[ 
        \sum_{j'=j-3}^{j+3} L_{3,j'} 
        + 2^{q-3} \phi^{\min\{1,q-3\}} \sum_{j'=j-3}^{j+3} \nu_{q,j'} \frac{(\log(ep))^{\max\{0,q-4\}/2}}{\delta^{\max\{0,q-4\}}_{n-j-4}}
    \right] \\
    & \quad \times \min\{\kappa_{[1,j-4]|\{j-3\}}(\delta_{n-j}^\circ) + \kappa^\circ_{j-4}(\delta_{n-j}), 1\} \\
    & \leq C \frac{(\log(ep))^{3/2}}{\delta_{n-j}^3} \left[ 
        \sum_{j'=j-3}^{j+3} L_{3,j'} 
        + \phi^{\min\{1,q-3\}} \sum_{j'=j-3}^{j+3} \nu_{q,j'} \frac{(\log(ep))^{\max\{0,q-4\}/2}}{(\delta_{n-j}/\alpha)^{\max\{0,q-4\}}}
    \right] \\
    & \quad \times \min\{\kappa_{[1,j-4]|\{j-3\}}(\delta_{n-j}^\circ) + \kappa^\circ_{j}(\delta_{n-j}), 1\},
\end{aligned}
\end{equation*}
where $C > 0$ and $\alpha > 0$ are universal constants.
The last inequality follows the monotonocity of $\kappa$ (\cref{thm:kappa_comparison}) and the choice of $\alpha$ such that $\delta_{n-j-4}/2 \geq \delta_{n-j}/\alpha$.
Similar arguments and bounds apply to the third-order remainder terms under $1$-ring dependence.

For \cref{thm:remainder_lemma_4N}, we consider the first term of $\mathfrak{R}_{W_j}^{(4,1)}$:
\begin{equation*}
\begin{aligned}
    & \Exp\left[ \int_0^1 (1-t)^3 \inner*{\nabla^4 \rho_{r,\phi}^\delta(W^\cmpl_{[j,j]} + t X_j), X_j^{\otimes 4}} ~dt \right] \\
    & = \Exp\left[ \int_0^1 (1-t)^3 \inner*{\nabla^4 \rho_{r,\phi}^\delta(X_{[1,j-1)} + W^\noindep + Y_{(j+1,n]} + t X_j), X_j^{\otimes 4}} ~dt \right].
\end{aligned}
\end{equation*}
Applying \cref{thm:generalized_remainder_lemma} to this term with $\gamma_1 = \gamma_2 = 1$ and $\eta = q-4$, we obtain for $q \geq 4$, $\delta \geq \sigma_{\min}$ and $\phi\delta \geq \frac{1}{\log(ep)}$,
\begin{equation*}
\begin{aligned}
    & \abs*{ \Exp\left[ 
        \int_0^1 (1-t)^3 \inner*{
        \nabla^4 \varphi_r^\vareps\left(
            X_{[1,j-1)} + W^\noindep + Y_{(j+1,n]} + t W_j
        \right),
        W_j^{\otimes 4}} ~dt
    \right] } \\
    & \leq C \phi 
        \min\{1, \kappa_{[1,j-2]|\{j-1\}}(\delta_{n-j-2}^\circ) + \kappa^\circ_{j-2}(\delta_{n-j-2})\} \\
    & \quad \times \left[ 
        \frac{(\log(ep))^{3/2}}{\delta_{n-j-2}^3} 
        \norm*{\Exp\left[
            \abs*{W_j^{\otimes 4}}
        \right]}_\infty
        + \frac{(\log(ep))^{(q-1)/2}}{(\delta_{n-j-2}/2)^{q-1}} 
        \Exp\left[ 
            \norm{W_j}_\infty^4
            (\norm{W_j}_\infty^{q-4} + \norm{X_{j-1}+Y_{j+1}}_\infty^{q-4}) 
        \right]
    \right] \\
    & \leq C \phi 
    \min\{1, \kappa_{[1,j-2]|\{j-1\}}(\delta_{n-j-2}^\circ) + \kappa^\circ_{j-2}(\delta_{n-j-2})\} \\
    & \quad \times \left[ 
        \frac{(\log(ep))^{3/2}}{\delta_{n-j-2}^3} L_{4,j}
        + 2^{q-4} \frac{(\log(ep))^{(q-1)/2}}{\delta_{n-j-2}^{q-1}} (\nu_{q,j-1} + \nu_{q,j} + \nu_{q,j+1})
    \right].
\end{aligned}
\end{equation*}
Based on similar applications of \cref{thm:generalized_remainder_lemma} to the other terms of $\mathfrak{R}^{(4,1)}_{X_j}$, 
\begin{equation*}
\begin{aligned}
    \abs*{\Exp \left[ \mathfrak{R}_{W_j}^{(4,1)} \right]} 
    & \leq C \phi \left[ 
        \frac{(\log(ep))^{3/2}}{\delta_{n-j-5}^3} \sum_{j'=j-4}^{j+4} L_{4,j'} 
        + 2^{q-4} \frac{(\log(ep))^{(q-1)/2}}{\delta_{n-j-5}^{q-1}} \sum_{j'=j-4}^{j+4} \nu_{q,j'}
    \right] \\
    & \quad \times \min\{\kappa_{[1,j-5]|\{j-4\}}(\delta_{n-j}^\circ) + \kappa^\circ_{j-5}(\delta_{n-j}), 1\} \\
    & \leq C \phi \left[ 
        \frac{(\log(ep))^{3/2}}{\delta_{n-j}^3} \sum_{j'=j-4}^{j+4} L_{4,j'} 
        + \frac{(\log(ep))^{(q-1)/2}}{(\delta_{n-j}/\alpha)^{q-1}} \sum_{j'=j-4}^{j+4} \nu_{q,j'}
    \right] \\
    & \quad \times \min\{\kappa_{[1,j-5]|\{j-4\}}(\delta_{n-j}^\circ) + \kappa^\circ_{j-5}(\delta_{n-j}), 1\},
\end{aligned}
\end{equation*}
where $C > 0$ and $\alpha > 0$ are universal constants.
The last inequality follows the monotonocity of $\kappa$ (\cref{thm:kappa_comparison}) and the choice of $\alpha$ such that $\delta_{n-j-5}/2 \geq \delta_{n-j}/\alpha$.
Similar arguments and bounds apply to the other fourth-order terms under $1$-ring dependence.

Finally, for the sixth-order remainder terms, we consider the first term of $\mathfrak{R}_{X_j, W_k}^{(6,1,1)}$:
\begin{equation*}
\begin{aligned}
    & \Exp\left[ 
        \int_0^1 (1-t)^2 \inner*{
            \nabla^6 \rho_{r,\phi}^\vareps(X_{[1,k)} + t X_k + Y_{(k,j-1) \cup (j+1,n]}) ,
            \Exp[X_j^{\otimes 3}] \otimes W_k^{\otimes 3}
        } ~dt
    \right] \\
    & = \Exp\left[ 
        \int_0^1 (1-t)^2 \inner*{
            \nabla^6 \rho_{r,\phi}^\vareps(X_{[1,k-1)} + W^\noindep + Y_{(k+1,j-1) \cup (j+1,n]}) + t X_k,
            \Exp[X_j^{\otimes 3}] \otimes W_k^{\otimes 3}
        } ~dt
    \right], \\
\end{aligned}
\end{equation*}
where $W^\noindep \equiv X_{k-1} + Y_{k+1}$. Applying \cref{thm:generalized_remainder_lemma} to this term with $\gamma_1 = \gamma_2 = 1$ and $\eta = q-3$, we obtain for $q \geq 4$,
\begin{equation*}
\begin{aligned}
    & \abs*{ \Exp\left[ 
        \int_0^1 (1-t)^2 \inner*{
            \nabla^6 \rho_{r,\phi}^\vareps(X_{[1,k)} + t X_k + Y_{(k,j-1) \cup (j+1,n]}) ,
            \Exp[X_j^{\otimes 3}] \otimes W_k^{\otimes 3}
        } ~dt
    \right] } \\
    & \leq C \phi 
        \min\{1, \kappa_{[1,k-2]|\{k-1\}}(\delta_{n-k-5}^\circ) + \kappa^\circ_{k-2}(\delta_{n-k-5})\} \\
    & \quad \times \left[ \begin{aligned}
        & \frac{(\log(ep))^{3/2}}{\delta_{n-k-5}^3} 
        \norm*{\Exp\left[
            \abs*{X_j^{\otimes 3}}
        \right]
        \Exp\left[
            \abs*{W_k^{\otimes 3}}
        \right]}_\infty \\
        & + \frac{(\log(ep))^{(q-1)/2}}{(\delta_{n-k-5}/2)^{q-1}} 
        \norm*{\Exp\left[
            \abs*{X_j^{\otimes 3}}
        \right]}_\infty
        \Exp\left[ 
             \norm{W_k}_\infty^3
            (\norm{W_k}_\infty^{q-3} + \norm{X_{k-1}+Y_{k+1}}_\infty^{q-3}) 
        \right]
    \end{aligned} \right] \\
    & \leq C \phi L_{3,j} \left( 
        L_{3,k} \frac{(\log(ep))^{5/2}}{\delta_{n-k-5}^5}
        + 2^{q-3} (\nu_{q,k-1} + \nu_{q,k} + \nu_{q,k+1}) \frac{{(\log(ep))}^{(q+2)/2}}{\delta_{n-k-5}^{q+2}}
    \right) \\
    & \quad \times \min\{\kappa_{[1,k-2]|\{k-1\}}(\delta_{n-k-5}^\circ) + \kappa^\circ_{k-2}(\delta_{n-k-5}), 1\},
\end{aligned}
\end{equation*}
as long as $\delta \geq \sigma_{\min}$ and $\phi\delta \geq \frac{1}{\log(ep)}$.
Similarly, we get for $q \geq 4$, $\delta \geq \sigma_{\min}$ and $\phi\delta \geq \frac{1}{\log(ep)}$,
\begin{equation*}
\begin{aligned}
    \abs*{\Exp \left[ \mathfrak{R}_{X_j,W_k}^{(6,1,1)} \right]} 
    & \leq C \phi L_{3,j} \left( \begin{aligned}
        \sum_{k'=k-2}^{k+2} L_{3,k'} \frac{(\log(ep))^{5/2}}{\delta_{n-k-7}^5}
        + 2^{q-3} \sum_{k'=k-3}^{k+3} \nu_{q,k'} \frac{ {(\log(ep))}^{(q+2)/2}}{\delta_{n-k-7}^{q+2}}
    \end{aligned} \right) \\
    & \quad \times \min\{\kappa_{[1,k-4]|\{k-3\}}(\delta_{n-k}^\circ) + \kappa^\circ_{k-4}(\delta_{n-k}), 1\} \\
    & \leq C \phi L_{3,j} \left( \begin{aligned}
        \sum_{k'=k-2}^{k+2} L_{3,k'} \frac{(\log(ep))^{5/2}}{\delta_{n-k}^5}
        + \sum_{k'=k-3}^{k+3} \nu_{q,k'} \frac{ {(\log(ep))}^{(q+2)/2}}{(\delta_{n-k}/\alpha)^{q+2}}
    \end{aligned} \right) \\
    & \quad \times \min\{\kappa_{[1,k-4]|\{k-3\}}(\delta_{n-k}^\circ) + \kappa^\circ_{k-4}(\delta_{n-k}), 1\},
\end{aligned}
\end{equation*}
where $C > 0$ and $\alpha > 0$ are universal constants.
The last inequality follows the monotonocity of $\kappa$ (\cref{thm:kappa_comparison}) and the choice of $\alpha$ such that $\delta_{n-k-7}/2 \geq \delta_{n-k}/\alpha$.
Similar arguments and bounds apply to the other sixth-order remainder terms.

\subsection{Proof of Anti-concentration Lemma} \label{sec:pf_anti_concentration}

%
%

In \cref{sec:pf_sketch,sec:pf_1_dep_4N}, we proved the Berry--Esseen bound where the observations $(X_1, \dots, X_n)$ are $1$‑dependent. The proof requires only an anti‑concentration bound for $X_{[1,i]}$ conditional on $X_{i+1}$.
We begin by treating this simpler situation.  
Subsequently, we extend the argument to the general setting in which $X_{[1,n]}$ is $1$‑ring dependent. In that broader scenario we need an anti‑concentration inequality for $X_I$ conditional on $\mathscr{X}_{I^{\cmpl}}$, where $I\subseteq\ints_n$ is any interval.

\medskip
{\bf Case of $1$-dependence:}
We use a similar approach with \cref{sec:pf_remainder_lemma} to upper bound $\Exp[ \mathfrak{R}_{X_{i-1}}^{(1)} | X_{i+1} ]$. First, we decompose it by
\begin{equation*}
\begin{aligned}
    & \Exp[ \mathfrak{R}_{X_{i-1}}^{(1)} | X_{i+1} ] \\ 
    & = \Exp\left[ \left. \int_0^1 \inner*{
            \nabla \varphi_{r,\delta+\vareps^\circ}^\vareps(X_{[1,i]} - t X_{i-1}),
            X_{i-1}
        } dt \right| X_{i+1} \right] \\
    & = \Exp\left[ \left. \int_0^1 \inner*{
            \nabla \varphi_{r,\delta+\vareps^\circ}^\vareps(X_{[1,i]} - t X_{i-1}),
            X_{i-1}
        } \Ind_\leq(t) dt \right| X_{i+1} \right] \\
    & \quad + \Exp\left[ \left. \int_0^1 \inner*{
            \nabla \varphi_{r,\delta+\vareps^\circ}^\vareps(X_{[1,i]} - t X_{i-1}),
            X_{i-1}
        } \Ind_>(t) dt \right| X_{i+1} \right], 
\end{aligned}
\end{equation*}
where $\Ind_\leq(t) = \Ind\{\norm{X_{[1,i]} - t X_{i-1} - \partial A_{r,\delta+\vareps^\circ}}_\infty \leq \vareps^\circ\}$, $\Ind_>(t) = \Ind\{\norm{X_{[1,i]} - t X_{i-1} - \partial A_{r,\delta+\vareps^\circ}}_\infty > \vareps^\circ\}$.
Based on Lemma 2.3 of \citet{fang2020large}, 
\begin{equation*}
\begin{aligned}
    & \Exp\left[ \left. \int_0^1 
        \inner*{
            \nabla \varphi_{r,\delta+\vareps^\circ}^\vareps(X_{[1,i]} - t X_{i-1}),
            X_{i-1}
        } 
        \Ind_\leq(t) 
    dt \right| X_{i+1} \right] \\
    & \leq \Exp\left[ \left. \int_0^1 
        \norm{\nabla \varphi_{r,\delta+\vareps^\circ}^\vareps(X_{[1,i]} - t X_{i-1})}_1 \Ind_\leq(t)
        \norm{X_{i-1}}_\infty 
    dt \right| X_{i+1} \right] \\
    & \leq C \frac{\sqrt{\log(ep)}}{\vareps} \Exp\left[ \left. \int_0^1 
        \Ind_\leq(t)
        \norm{X_{i-1}}_\infty 
    dt \right| X_{i+1} \right] \\
\end{aligned}
\end{equation*}
We use Tonelli's theorem to switch the order between the integration and the expectation conditional on $\mathscr{X}_{[i-1, i+1]}$:
\begin{equation*}
\begin{aligned}
    & \Exp\left[ \left. \int_0^1 
        \Ind_\leq(t)
        \norm{X_{i-1}}_\infty 
    dt \right| X_{i+1} \right] \\
    & \leq \Exp\left[ \left. \int_0^1 \Exp\left[ \left. 
        \Ind_\leq(t)
        \norm{X_{i-1}}_\infty \right| | \mathscr{X}_{[i-1, i+1]} \right] 
    dt \right| X_{i+1} \right] \\
    & = \Exp\left[ \left. \int_0^1 \left(
    \begin{aligned}
        & \Pr[X_{[1,i-2]} \in A_{r_1, \vareps^\circ} | \mathscr{X}_{[i-1, i+1]}] \\
        & + \Pr[X_{[1,i-2]} \in A_{r_2, \vareps^\circ} | \mathscr{X}_{[i-1, i+1]}]
    \end{aligned} \right)
    \norm{X_{i-1}}_\infty dt \right| X_{i+1} \right] \\
    & \leq \kappa_{[1,i-2]|\{i-1\}}(\vareps^\circ)
    \Exp[\norm{X_{i-1}}_\infty | X_{i+1}] 
    = \kappa_{[1,i-2]|\{i-1\}}(\vareps^\circ) \cdot \nu_{1,i-1},
\end{aligned}
\end{equation*}
where $r_1 = r - (1-t) X_{i-1} - X_{i} - (\delta+\vareps^\circ)\mathbf{1}$ and $r_2 = r - (1-t) X_{i-1} - X_{i} + (\delta+\vareps^\circ)\mathbf{1}$ are Borel measurable functions with respect to $\mathscr{X}_{[i-1, i+1]}$. On the other hand, based on Lemma 10.5 of \citet{lopes2022central},
\begin{equation*}
\begin{aligned}
    &\Exp\left[ \left. \int_0^1 \inner*{
            \nabla \varphi_{r,\delta+\vareps^\circ}^\vareps(X_{[1,i]} - t X_{i-1}),
            X_{i-1}
        } \Ind_>(t) 
    dt \right| X_{i+1} \right] \\
    & \leq \Exp\left[ \left. \int_0^1 
        \norm{\nabla \varphi_{r,\delta+\vareps^\circ}^\vareps(X_{[1,i]} - t X_{i-1})}_1 \Ind_>(t)
        \norm{X_{i-1}}_\infty 
    dt \right| X_{i+1} \right] \\
    & = \frac{1}{\vareps h^4} \Exp\left[ \left. \int_0^1 \Exp\left[ \left. 
        \norm{X_{i-1}}_\infty \right| \mathscr{X}_{[i-1, i+1]} \right] 
    dt \right| X_{i+1} \right] \\
    & = \frac{1}{\vareps h^4} \Exp[\norm{X_{i-1}}_\infty | X_{i+1}] 
    = \frac{1}{\vareps h^4} \cdot \nu_{1,i-1}.
\end{aligned}
\end{equation*}
Combining the two terms, we get
\begin{equation*}
\begin{aligned}
    \Exp[ \mathfrak{R}_{X_{i-1}}^{(1)} | X_{i+1} ] 
    \leq \nu_{1,i-1} \left( 
        C \frac{\sqrt{\log(ep)}}{\vareps} 
        \max\left\{1, \kappa_{[1,i-2]|\{i-1\}}(\vareps^\circ) \right\}
        + \frac{1}{\vareps h^4}
    \right).
\end{aligned}    
\end{equation*}
Putting the above result back to \cref{eq:anti_concentration_Taylor_N}, we get for any $h > 0$,
\begin{equation*}
\begin{aligned}
    & \Pr[X_{[1,i]} \in A_{r,\delta} | X_{i+1}] \\
    & \leq \nu_{1,i-1} \left( 
        C \frac{\sqrt{\log(ep)}}{\vareps} 
        \max\left\{1, \kappa_{[1,i-2]|\{i-1\}}(\vareps^\circ) \right\}
        + \frac{1}{\vareps h^4}
    \right)
    + C\frac{\delta + 2\vareps^\circ}{\sigma_{\min}} \sqrt{\frac{\log(ep)}{i-2}} + 2\mu_{[1,i-2]}.
\end{aligned} 
\end{equation*}
Plugging in $h = \left(\frac{\sigma_{\min}}{\vareps} \frac{\sqrt{i-2}}{\log(ep)} \right)^{1/4}$,
\begin{equation*}
\begin{aligned}
    & \Pr[X_{[1,i]} \in A_{r,\delta} | X_{i+1}] \\
    & \leq C \nu_{1,i-1}
    \frac{\sqrt{\log(ep)}}{\vareps} 
    \max\left\{1, \kappa_{[1,i-2]|\{i-1\}}(\vareps^\circ) + \kappa^\circ_{i-2} \right\}
    + C\frac{\delta + 2\vareps^\circ}{\sigma_{\min}} \sqrt{\frac{\log(ep)}{i-2}} + 2\mu_{[1,i-2]},
\end{aligned}
\end{equation*}
where $\vareps^\circ = 10\vareps\sqrt{\log(p\max\{i_0-2,1\})}$ and $\kappa^\circ_j = \frac{\vareps}{\sigma_{\min}} \sqrt{\frac{\log(ep)}{\max\{j,1\}}}$, as long as $\vareps \geq \sigma_{\min}$.
%
Because the righthand side is not dependent on $r$, for any $\vareps \geq \sigma_{\min}, \delta > 0$, 
\begin{equation*} 
\begin{aligned}
    & \kappa_{[1,i]|\{i+1\}}(\delta) \\
    & \leq C \left( 
        \nu_{1,\max} \frac{\sqrt{\log(ep)}}{\vareps} \kappa_{[1,i-2]|\{i-1\}}(\vareps^\circ)
        + \mu_{[1,i-2]} 
    + \frac{\delta + 2\vareps^\circ}{\sigma_{\min}} \sqrt{\frac{\log(ep)}{i-2}}
    + \frac{\nu_{1,\max}}{\sigma_{\min}} \frac{\log(ep)}{\sqrt{i-2}}
    \right).
\end{aligned}
\end{equation*}

\medskip
{\bf Case of general settings:} We consider the case of $j = 0$, which is to bound $\Exp[\mathfrak{R}^{(1)}_{X_{\{2, i^\circ-1\}}} | \mathscr{X}_{\{0, i^\circ+1\}}]$.
First, we decompose the remainder term by
\begin{equation*}
\begin{aligned}
    & \Exp\left[ \mathfrak{R}^{(1)}_{X_{\{2,i^\circ-1\}}} | \mathscr{X}_{\{0, i^\circ+1\}} \right] \\ 
    & = \Exp\left[ \left. \int_0^1 \inner*{
            \nabla \varphi_{r,\delta+\vareps^\circ}^\vareps(X_{[1,i^\circ]} - t X_{\{2,i^\circ-1\}}),
            X_{\{2,i^\circ-1\}}
        } dt \right| \mathscr{X}_{\{0, i^\circ+1\}} \right] \\
    & = \Exp\left[ \left. \int_0^1 \inner*{
            \nabla \varphi_{r,\delta+\vareps^\circ}^\vareps(X_{[1,i^\circ]} - t X_{\{2,i^\circ-1\}}),
            X_{\{2,i^\circ-1\}}
        } \Ind_\leq(t) dt \right| \mathscr{X}_{\{0, i^\circ+1\}} \right] \\
    & \quad + \Exp\left[ \left. \int_0^1 \inner*{
            \nabla \varphi_{r,\delta+\vareps^\circ}^\vareps(X_{[1,i^\circ]} - t X_{\{2,i^\circ-1\}}),
            X_{\{2,i^\circ-1\}}
        } \Ind_>(t) dt \right| \mathscr{X}_{\{0, i^\circ+1\}} \right], 
\end{aligned}
\end{equation*}
where $\Ind_\leq(t) = \Ind\{\norm{X_{[1,i^\circ]} - t X_{\{2,i^\circ-1\}} - \partial A_{r,\delta+\vareps^\circ}}_\infty \leq \vareps^\circ\}$, $\Ind_>(t) = \Ind\{\norm{X_{[1,i^\circ]} - t X_{\{2,i^\circ-1\}} - \partial A_{r,\delta+\vareps^\circ}}_\infty > \vareps^\circ\}$.
Based on Lemma 2.3 of \citet{fang2020large}, 
\begin{equation*}
\begin{aligned}
    & \Exp\left[ \left. \int_0^1 
        \inner*{
            \nabla \varphi_{r,\delta+\vareps^\circ}^\vareps(X_{[1,i^\circ]} - t X_{\{2,i^\circ-1\}}),
            X_{\{2,i^\circ-1\}}
        } 
        \Ind_\leq(t) 
    dt \right| \mathscr{X}_{\{0, i^\circ+1\}} \right] \\
    & \leq \Exp\left[ \left. \int_0^1 
        \norm{\nabla \varphi_{r,\delta+\vareps^\circ}^\vareps(X_{[1,i^\circ]} - t X_{\{2,i^\circ-1\}})}_1 \Ind_\leq(t)
        \norm{X_{\{2,i^\circ-1\}}}_\infty 
    dt \right| \mathscr{X}_{\{0, i^\circ+1\}} \right] \\
    & \leq C \frac{\sqrt{\log(ep)}}{\vareps} \Exp\left[ \left. \int_0^1 
        \Ind_\leq(t)
        \norm{X_{\{2,i^\circ-1\}}}_\infty 
    dt \right| \mathscr{X}_{\{0, i^\circ+1\}} \right] \\
\end{aligned}
\end{equation*}
We use Tonelli's theorem to switch the order between the integration and the expectation conditional on $\mathscr{X}_{(i-2,3)}$:
\begin{equation*}
\begin{aligned}
    & \Exp\left[ \left. \int_0^1 
        \Ind_\leq(t)
        \norm{X_{\{2,i^\circ-1\}}}_\infty 
    dt \right| \mathscr{X}_{\{0, i^\circ+1\}} \right] \\
    & \leq \Exp\left[ \left. \int_0^1 \Exp\left[ \left. 
        \Ind_\leq(t)
        \norm{X_{\{2,i^\circ-1\}}}_\infty \right| | \mathscr{X}_{\{2,i^\circ-1\}} \right] 
    dt \right| \mathscr{X}_{\{0, i^\circ+1\}} \right] \\
    & = \Exp\left[ \left. \int_0^1 \left(
    \begin{aligned}
        & \Pr[X_{[3,i^\circ-2]} \in A_{r_1, \vareps^\circ} | \mathscr{X}_{\{2,i^\circ-1\}}] \\
        & + \Pr[X_{[3,i^\circ-2]} \in A_{r_2, \vareps^\circ} | \mathscr{X}_{\{2,i^\circ-1\}}]
    \end{aligned} \right)
    \norm{X_{\{2,i^\circ-1\}}}_\infty dt \right| \mathscr{X}_{\{0, i^\circ+1\}} \right] \\
    & \leq \kappa_{[3,i^\circ-2]|\{2,i^\circ-1\}}(\vareps^\circ)
    \Exp[\norm{X_{\{2,i^\circ-1\}}}_\infty | \mathscr{X}_{\{0, i^\circ+1\}}] 
    = \kappa_{[3,i^\circ-2]|\{2,i^\circ-1\}}(\vareps^\circ)
    (\nu_{1,2} + \nu_{1,i^\circ-1}),
\end{aligned}
\end{equation*}
where $r_1 = r - (1-t) X_{\{2,i^\circ-1\}} - X_{1} - X_{i} - (\delta+\vareps^\circ)\mathbf{1}$ and $r_2 = r - (1-t) X_{\{2,i^\circ-1\}} - X_{1} - X_{i} + (\delta+\vareps^\circ)\mathbf{1}$ are Borel measurable functions with respect to $X_{\{2,i^\circ-1\}}$. On the other hand, based on Lemma 10.5 of \citet{lopes2022central},
\begin{equation*}
\begin{aligned}
    &\Exp\left[ \left. \int_0^1 \inner*{
            \nabla \varphi_{r,\delta+\vareps^\circ}^\vareps(X_{[1,i^\circ]} - t X_{\{2,i^\circ-1\}}),
            X_{\{2,i^\circ-1\}}
        } \Ind_>(t) 
    dt \right| \mathscr{X}_{\{0, i^\circ+1\}} \right] \\
    & \leq \Exp\left[ \left. \int_0^1 
        \norm{\nabla \varphi_{r,\delta+\vareps^\circ}^\vareps(X_{[1,i^\circ]} - t X_{\{2,i^\circ-1\}})}_1 \Ind_>(t)
        \norm{X_{\{2,i^\circ-1\}}}_\infty 
    dt \right| \mathscr{X}_{\{0, i^\circ+1\}} \right] \\
    & = \frac{1}{\vareps h^4} \Exp\left[ \left. \int_0^1 \Exp\left[ \left. 
        \norm{X_{\{2,i^\circ-1\}}}_\infty \right| \mathscr{X}_{\{2,i^\circ-1\}} \right] 
    dt \right| \mathscr{X}_{\{0, i^\circ+1\}} \right] \\
    & = \frac{1}{\vareps h^4} \Exp[\norm{X_{\{2,i^\circ-1\}}}_\infty | \mathscr{X}_{\{0, i^\circ+1\}}] 
    = \frac{1}{\vareps h^4} (\nu_{1,2} + \nu_{1,i^\circ-1}).
\end{aligned}
\end{equation*}
Combining the two terms, we get
\begin{equation*}
\begin{aligned}
    \Exp[ \mathfrak{R}_{X_{\{2,i^\circ-1\}}}^{(1)} | \mathscr{X}_{\{0,i^\circ+1\}} ] 
    \leq (\nu_{1,2} + \nu_{1,i^\circ-1}) \left( 
        C \frac{\sqrt{\log(ep)}}{\vareps} 
        \max\left\{1, \kappa_{[3,i^\circ-2]|\{2,i^\circ-1\}}(\vareps^\circ) \right\}
        + \frac{1}{\vareps h^4}
    \right).
\end{aligned}    
\end{equation*}
Putting the above result back to \cref{eq:anti_concentration_Taylor_P}, we get for any $h > 0$,
\begin{equation*}
\begin{aligned}
    & \Pr[X_{[1,i^\circ]} \in A_{r,\delta} | \mathscr{X}_{\{0, i^\circ+1\}}] \\
    & \leq (\nu_{1,2} + \nu_{1,i^\circ-1}) \left( 
        C \frac{\sqrt{\log(ep)}}{\vareps} 
        \max\left\{1, \kappa_{[3,i^\circ-2]}(\vareps^\circ) \right\}
        + \frac{1}{\vareps h^4}
    \right)
    \\
    & \quad + C\frac{\delta + 2\vareps^\circ}{\sigma_{\min}} \sqrt{\frac{\log(ep)}{i^\circ-4}} + 2\mu_{[3,i^\circ-2]}.
\end{aligned}
\end{equation*}
Plugging in $h = \left(\frac{\sigma_{\min}}{\vareps} \frac{\sqrt{i^\circ-4}}{\log(ep)} \right)^{1/4}$,
\begin{equation*}
\begin{aligned}
    & \Pr[X_{[1,i^\circ]} \in A_{r,\delta} | \mathscr{X}_{\{0, i^\circ+1\}}] \\
    & \leq C (\nu_{1,2} + \nu_{1,i^\circ-1}) 
    \frac{\sqrt{\log(ep)}}{\vareps} 
    \max\left\{1, \kappa_{[3,i^\circ-2]}(\vareps^\circ) + \kappa^\circ_{i^\circ-4} \right\}
    \\
    & \quad + C\frac{\delta + 2\vareps^\circ}{\sigma_{\min}} \sqrt{\frac{\log(ep)}{i^\circ-4}} + 2\mu_{[3,i^\circ-2]},
\end{aligned}
\end{equation*}
where $\vareps^\circ = 10\vareps\sqrt{\log(p\max\{i_0-2,1\})}$ and $\kappa^\circ_j = \frac{\vareps}{\sigma_{\min}} \sqrt{\frac{\log(ep)}{\max\{j,1\}}}$, as long as $\vareps \geq \sigma_{\min}$.
%
Because the righthand side is not dependent on $r$, for any $\vareps \geq \sigma_{\min}, \delta > 0$, 
\begin{equation*} 
\begin{aligned}
    & \kappa_{[1,i^\circ]|\{0,i^\circ+1\}}(\delta) \\
    & \leq C \left( 
        (\nu_{1,2} + \nu_{1,i^\circ-1}) \frac{\sqrt{\log(ep)}}{\vareps} \kappa_{[3,i^\circ-2]|\{2,i^\circ-1\}}(\vareps^\circ)
        + \mu_{[3,i^\circ-2]} 
    \right)\\
    & \quad + C\frac{\delta + 2\vareps^\circ}{\sigma_{\min}} \sqrt{\frac{\log(ep)}{i^\circ-4}}
    + C\frac{\bar\nu_{1,(1,i)}}{\sigma_{\min}} \frac{\log(ep)}{\sqrt{i^\circ-4}}.
\end{aligned}
\end{equation*}

\subsection{Proof of Lemma~\ref{thm:jensen}}

Suppose that $j, k \geq \frac{n}{2}$ and $\alpha \leq 1$. Because $x \mapsto x^\alpha$ is concave, by Jensen's inequality, for any $q_1, q_2 > 0$,
\begin{equation*}   
\begin{aligned}
    \frac{1}{n} \sum_{i=1}^{n} L_{q_1,i+j} \bar{L}_{q_2,(i,i+j)}^\alpha
    & \leq \bar{L}_{q_1} \left( 
        \frac{\frac{1}{n} \sum_{i=1}^{n} L_{q_1,i+j} \bar{L}_{q_2,(i,i+j)}}{\bar{L}_{q_1}} 
    \right)^\alpha \\
    & \leq \bar{L}_{q_1} \left( 
        \frac{\frac{1}{n (j-1)} \sum_{i=1}^{n} L_{q_1,i} \sum_{i=1}^n L_{q_2,i}}{\bar{L}_{q_1}} 
    \right)^\alpha\\ 
    & \leq \bar{L}_{q_1} \left(
        \frac{n}{j-1} \bar{L}_{q_2}
    \right)^\alpha 
    \leq 2^\alpha \bar{L}_{q_1} \bar{L}_{q_2}.
\end{aligned}
\end{equation*}
Similarly, for any $q_1, q_2, q_3 > 0$,
\begin{equation*}
\begin{aligned}
    & \frac{1}{n} \sum_{i=1}^{n} \frac{L_{q_1,i+j}}{j-1} 
    \sum_{l=1}^{j-1} L_{q_2,i+[k+l]_{j-1}} \bar{L}_{q_3, i+(l,k+l)_{j-1}}^\alpha \\
    & \leq \frac{1}{n} \sum_{i=1}^{n} \frac{L_{q_1,i+j}}{j-1} \sum_{l=1}^{j-1} L_{q_2,i+[k+l]_{j-1}} \left(
        \frac{\frac{1}{n} \sum_{i=1}^{n} \frac{L_{q_1,i+j}}{j-1} 
        \sum_{l=1}^{j-1} L_{q_2,i+[k+l]_{j-1}} \bar{L}_{q_3, i+(l,k+l)_{j-1}}} 
        {\frac{1}{n} \sum_{i=1}^{n} \frac{L_{q_1,i+j}}{j-1} \sum_{l=1}^{j-1} L_{q_2,i+[k+l]_{j-1}}}
    \right)^\alpha \\
    & \leq \frac{1}{n} \sum_{i=1}^{n} L_{q_1,i+j} \bar{L}_{q_2,(i,i+j)} \left(
        \frac{\frac{1}{j-1} \sum_{i=1}^{n} \frac{L_{q_1,i+j}}{j-1} 
        \sum_{l=1}^{j-1} L_{q_2,i+[k+l]_{j-1}} \bar{L}_{q_3}} 
        {\frac{1}{n} \sum_{i=1}^{n} \frac{L_{q_1,i+j}}{j-1} \sum_{l=1}^{j-1} L_{q_2,i+[k+l]_{j-1}}}
    \right)^\alpha \\
    & \leq \frac{1}{n} \sum_{i=1}^{n} L_{q_1,i+j} \bar{L}_{q_2,(i,i+j)} \left(
        \frac{n}{j-1} \bar{L}_{q_3}
    \right)^\alpha \\
    & \leq \frac{2^\alpha}{n} \sum_{i=1}^{n} L_{q_1,i+j} \bar{L}_{q_2,(i,i+j)} \bar{L}_{q_3}^\alpha \\
    & \leq 2^{\alpha+1} \bar{L}_{q_1} \bar{L}_{q_2} \sum_{l=1}^{j-1} L_{q_2,i+[k+l]_{j-1}} \bar{L}_{q_3}^\alpha. \\
\end{aligned}
\end{equation*}

\subsection{Proof of Lemma~\ref{thm:conc_ineq_q}} \label{sec:pf_conc_ineq_q}

For any $B > 0$, we divide the target of the bound as follows
\begin{equation*}
\begin{aligned}
    \norm*{\frac{1}{n}\sum_{i=1}^n X_i}_\infty
    & \leq \norm*{ \frac{1}{n} \sum_{i=1}^n \Big( 
        X_i \mathbf{1}\{\norm{X_i}_\infty \leq B\} 
        - \Exp[X_i \mathbf{1}\{\norm{X_i}_\infty \leq B\} ] 
    \Big) }_\infty \\
    & \quad + \norm*{ \frac{1}{n} \sum_{i=1}^n \Big( 
        X_i \mathbf{1}\{\norm{X_i}_\infty > B\} 
        - \Exp[X_i \mathbf{1}\{\norm{X_i}_\infty > B\}] 
    \Big) }_\infty \\
    & \equiv \mathfrak{T}_1 + \mathfrak{T}_2.
\end{aligned}
\end{equation*}
With bounded $q$-norms, each element of $X_i \mathbf{1}\{\norm{X_i}_\infty \leq B\}$ has variance upperbounded by $L_{\min\{2, q\},i} B^{\max\{0, 2-q\}}$. By Bernstein's inequality, for any $\delta \in [0,1]$, with probability at least $1 - \delta/2$,  
\begin{equation*}
\begin{aligned}
    \mathfrak{T}_1 
    \leq \sqrt{ C \bar{L}_{\min\{2, q\}} B^{\max\{0, 2-q\}} \frac{\log(2p/\delta)}{n} }
    + C B \frac{\log(2p/\delta)}{n},
\end{aligned}
\end{equation*}
where $C$ is a universal constant. On the other hand,
\begin{equation*}
\begin{aligned}
    \Pr\left[ \mathfrak{T}_2 > \gamma \right]
    & \leq \frac{2 \Exp[ \frac{1}{n} \sum_{i=1}^n \norm{X_i}_\infty \mathbf{1}\{\norm{X_i}_\infty > B \}]}{\gamma} 
    \leq \frac{2 \bar{\nu}_{q}}{B^{q-1} \gamma}.
\end{aligned}
\end{equation*}
Combining the two bounds, we get that, with probability at least $1 - \delta$, 
\begin{equation*}
\begin{aligned}
    \norm*{\frac{1}{n}\sum_{i=1}^n X_i}_\infty
    \leq \sqrt{ C \bar{L}_{\min\{2, q\}} B^{\max\{0, 2-q\}} \frac{\log(2p/\delta)}{n} }
    + C B \frac{\log(2p/\delta)}{n}
    + \frac{4 \bar{\nu}_{q} }{B^{q-1} \delta}.
\end{aligned}
\end{equation*}
Minimizing the right hand side with respect to $B$ (i.e., $B = (\frac{n \bar{\nu}_q}{\delta \log(2p/\delta)})^{1/q}$), we obtain
\begin{equation*}
\begin{aligned}
    \norm*{\frac{1}{n}\sum_{i=1}^n X_i}_\infty
    & \leq C \bar{L}_{\min\{2, q\}}^{1/2} \left(\frac{\bar{\nu}_q}{\delta} \right)^{\max\{1/q-1/2,0\}}
    \left(\frac{\log(2p/\delta)}{n}\right)^{\min\{1-1/q, 1/2\}} \\
    & \quad + C \left(\frac{\bar{\nu}_q}{\delta} \right)^{1/q} \left(\frac{\log(2p/\delta)}{n}\right)^{1-1/q}.
\end{aligned}
\end{equation*}

\section{Details of the Taylor expansions for \texorpdfstring{$m=1$}{m=1}}
\label{sec:Taylor_expansion}

\subsection{Breaking the ring} \label{sec:breaking_ring}

We apply the Taylor expansion to $\rho_{r,\phi}^{\delta}(X_{[1,n]})$ centered at $\rho_{r,\phi}^\delta(X_{[1,n)})$ as follows:
\begin{equation*}
\begin{aligned}
    & \rho_{r,\phi}^{\delta}(X_{[1,n]}) - \rho_{r,\phi}^\delta(X_{[1,n)}) \\
    & = \inner*{\nabla \rho_{r,\phi}^\delta(X_{[1,n)}), X_n} 
    + \frac{1}{2} \inner*{\nabla^2 \rho_{r,\phi}^\delta(X_{[1,n)}), X_n^{\otimes 2}} \\
    & \quad + \frac{1}{2} \int_0^1 (1-t)^2 \inner*{
    \nabla^3 \rho_{r,\phi}^\delta(X_{[1,n)} + t X_n), X_n^{\otimes 3}} ~dt.
\end{aligned}
\end{equation*}
Because $X_{[1,n)}$ and $X_n$ are dependent via $X_1$ and $X_{n-1}$ due to $1$-dependency, we re-apply the Taylor expansion to $\nabla \rho_{r,\phi}^\delta(X_{[1,n)})$ and $\nabla^2 \rho_{r,\phi}^\delta(X_{[1,n)})$ centered at $X_{(1,n-1)}$:
\begin{equation*}
\begin{aligned}
    & \inner*{ \nabla \rho_{r,\phi}^\delta(X_{[1,n)}), X_n } \\
    & = \inner*{ \nabla \rho_{r,\phi}^\delta(X_{(1,n-1)}), X_n }
    + \inner*{\nabla^2 \rho_{r,\phi}^\delta(X_{(1,n-1)}), X_n \otimes (X_1 + X_{n-1})} \\
    & \quad + \int_0^1 (1-t) 
    \inner*{ \nabla^3 \rho_{r,\phi}^\delta(X_{(1,n-1)} + t (X_1 + X_{n-1})),
    X_n \otimes (X_1 + X_{n-1})^{\otimes 2}} ~dt, \textand
\end{aligned}
\end{equation*}
\begin{equation*}
\begin{aligned}
    & \inner*{ \nabla^2 \rho_{r,\phi}^\delta(X_{[1,n)}), X_n^{\otimes 2} } \\
    & = \inner*{ \nabla^2 \rho_{r,\phi}^\delta(X_{(1,n-1)}), X_n^{\otimes 2} } \\
    & \quad + \int_0^1 
    \inner*{ \nabla^3 \rho_{r,\phi}^\delta(X_{(1,n-1)} + t (X_1 + X_{n-1})),
    X_n^{\otimes 2} \otimes (X_1 + X_{n-1})} ~dt.
\end{aligned}
\end{equation*}
Similarly,
\begin{equation*}
\begin{aligned}
    & \inner*{\nabla^2 \rho_{r,\phi}^\delta(X_{(1,n-1)}), X_n \otimes X_1} \\
    & = \inner*{\nabla^2 \rho_{r,\phi}^\delta(X_{(2,n-1)}), X_n \otimes X_1} \\
    & \quad + \int_0^1 
    \inner*{ \nabla^3 \rho_{r,\phi}^\delta(X_{(2,n-1)} + t X_2),
    X_n \otimes X_1 \otimes X_2} ~dt, \textand
\end{aligned}
\end{equation*}
\begin{equation*}
\begin{aligned}
    & \inner*{\nabla^2 \rho_{r,\phi}^\delta(X_{(1,n-1)}), X_n \otimes X_{n-1}} \\
    & = \inner*{\nabla^2 \rho_{r,\phi}^\delta(X_{(1,n-2)}), X_n \otimes X_{n-1}} \\
    & \quad + \int_0^1 
    \inner*{ \nabla^3 \rho_{r,\phi}^\delta(X_{(1,n-2)} + t X_2),
    X_n \otimes X_{n-1} \otimes X_{n-2}} ~dt.
\end{aligned}
\end{equation*}
In sum, because $\Exp[X_j] = 0$,
\begin{equation*}
\begin{aligned}
    & \Exp\left[ \rho_{r,\phi}^{\delta}(X_{[1,n]}) - \rho_{r,\phi}^\delta(X_{[1,n)}) \right] \\
    & = \frac{1}{2} \inner*{
        \Exp[\nabla^2 \rho_{r,\phi}^\delta(X_{(1,n-1)})], 
        \Exp[X_n^{\otimes 2}]
    }
    + \inner*{
        \Exp[\nabla^2 \rho_{r,\phi}^\delta(X_{(2,n-1)})], 
        \Exp[X_n \otimes X_1]
    } \\
    & \quad + \inner*{
        \Exp[\nabla^2 \rho_{r,\phi}^\delta(X_{(1,n-2)})], 
        \Exp[X_{n-1} \otimes X_n]
    } 
    + \Exp\left[ \mathfrak{R}_X^{(3)} \right],
\end{aligned}
\end{equation*}
where $\mathfrak{R}_X^{(3)}$ is the summation of all the above third-order remainder terms. That is,
\begin{equation*}
\begin{aligned}
    & \mathfrak{R}_X^{(3)} \\
    & = \frac{1}{2} \int_0^1 (1-t)^2 \inner*{
    \nabla^3 \rho_{r,\phi}^\delta(X_{[1,n)} + t X_n), X_n^{\otimes 3}} ~dt \\
    & \quad + \int_0^1 (1-t) 
    \inner*{ \nabla^3 \rho_{r,\phi}^\delta(X_{(1,n-1)} + t (X_1 + X_{n-1})),
    X_n \otimes (X_1 + X_{n-1})^{\otimes 2}} ~dt \\
    & \quad + \frac{1}{2} \int_0^1 
    \inner*{ \nabla^3 \rho_{r,\phi}^\delta(X_{(1,n-1)} + t (X_1 + X_{n-1})),
    X_n^{\otimes 2} \otimes (X_1 + X_{n-1})} ~dt \\
    & \quad + \int_0^1 
    \inner*{ \nabla^3 \rho_{r,\phi}^\delta(X_{(2,n-1)} + t X_2),
    X_n \otimes X_1 \otimes X_2} ~dt \\
    & \quad + \int_0^1 
    \inner*{ \nabla^3 \rho_{r,\phi}^\delta(X_{(1,n-2)} + t X_2),
    X_n \otimes X_{n-1} \otimes X_{n-2}} ~dt    
\end{aligned}
\end{equation*}

To further decompose $\mathfrak{R}_X^{(3)}$, we apply the Taylor expansion up to order $4$. For example, 
\begin{equation*}
\begin{aligned}
    & \frac{1}{2} \int_0^1 (1-t)^2 \inner*{
    \nabla^3 \rho_{r,\phi}^\delta(X_{[1,n)} + t X_n), X_n^{\otimes 3}} ~dt \\
    & = \frac{1}{6} \inner*{ \nabla^3 \rho_{r,\phi}^\delta(X_{[1,n)}), X_n^{\otimes 3}}
    + \frac{1}{6} \int_0^1 (1-t)^3 \inner*{
    \nabla^4 \rho_{r,\phi}^\delta(X_{[1,n)} + t X_n), X_n^{\otimes 4}} ~dt.
\end{aligned}
\end{equation*}
Again, to break the dependency between $\nabla^3 \rho_{r,\phi}^\delta(X_{[1,n)})$ and $X_n^{\otimes 3}$, we re-apply the Taylor expansion centered at $X_1 + X_{n-1}$.
\begin{equation*}
\begin{aligned}
    & \inner*{ \nabla^3 \rho_{r,\phi}^\delta(X_{[1,n)}), X_n^{\otimes 3}} \\
    & = \inner*{ \nabla^3 \rho_{r,\phi}^\delta(X_{(1,n-1)}), X_n^{\otimes 3}} \\
    & \quad + \int_0^1 \inner*{
    \nabla^4 \rho_{r,\phi}^\delta(X_{(1,n-1)} + t (X_1+X_{n-1})), X_n^{\otimes 3} \otimes (X_1 + X_{n-1})} ~dt.
\end{aligned}
\end{equation*}
Repeating this to the other third-order remainder terms, we get
\begin{equation*}
\begin{aligned}
    & \mathfrak{R}_X^{(3)} \\
    & = \frac{1}{6} \inner*{\nabla^3 \rho_{r,\phi}^\delta(X_{(1,n-1)}), X_n^{\otimes 3}}
    + \inner*{\nabla^3 \rho_{r,\phi}^\delta(X_{(1,n-3)}), X_{n-2} \otimes X_{n-1} \otimes X_n} \\
    & \quad + \inner*{\nabla^3 \rho_{r,\phi}^\delta(X_{(2,n-2)}),  X_{n-1} \otimes X_{n} \otimes X_1} 
    + \inner*{\nabla^3 \rho_{r,\phi}^\delta(X_{(3,n-1)}),  X_{n} \otimes X_{1} \otimes X_2} \\
    & \quad + \frac{1}{2} \inner*{
        \nabla^3 \rho_{r,\phi}^\delta(X_{(1,n-2)}), 
        X_{n-1} \otimes X_n \otimes (X_{n-1} + X_n)
    } \\
    & \quad + \frac{1}{2} \inner*{
        \nabla^3 \rho_{r,\phi}^\delta(X_{(2,n-1)}), 
        X_{n} \otimes X_1 \otimes (X_1 + X_n)
    } \\
    & \quad + \mathfrak{R}_X^{(4)},
\end{aligned}
\end{equation*}
where 
\begin{equation*}
\begin{aligned}
    & \mathfrak{R}_X^{(4)} \\
    & = \frac{1}{6} \int_0^1 (1-t)^3 \inner*{
        \nabla^4 \rho_{r,\phi}^\delta(X_{[1,n)} + t X_n), 
        X_n^{\otimes 4}} ~dt \\
    & \quad + \frac{1}{2} \int_0^1 (1-t)^2 \inner*{
        \nabla^4 \rho_{r,\phi}^\delta(X_{(1,n-1)} + t (X_1+X_{n-1})), 
        X_n \otimes (X_1 + X_{n-1})^{\otimes 3}} ~dt \\
    & \quad + \frac{1}{2} \int_0^1 (1-t) \inner*{
        \nabla^4 \rho_{r,\phi}^\delta(X_{(1,n-1)} + t (X_1+X_{n-1})), 
        X_n^{\otimes 2} \otimes (X_1 + X_{n-1})^{\otimes 2}} ~dt \\
    & \quad + \frac{1}{6} \int_0^1 \inner*{
        \nabla^4 \rho_{r,\phi}^\delta(X_{(1,n-1)} + t (X_1+X_{n-1})), 
        X_n^{\otimes 3} \otimes (X_1 + X_{n-1})} ~dt \\
    & \quad + \int_0^1 (1-t) \inner*{
        \nabla^4 \rho_{r,\phi}^\delta(X_{(2,n-1)} + t X_2), 
        X_n \otimes X_1 \otimes X_2^{\otimes 2}} ~dt \\
    & \quad + \int_0^1 (1-t) \inner*{
        \nabla^4 \rho_{r,\phi}^\delta(X_{(1,n-2)} + t X_{n-2}), 
        X_n \otimes X_{n-1} \otimes X_{n-2}^{\otimes 2}} ~dt \\
    & \quad + \frac{1}{2} \int_0^1 \inner*{
        \nabla^4 \rho_{r,\phi}^\delta(X_{(2,n-1)} + t X_2), 
        X_n \otimes X_1^{\otimes 2} \otimes X_2} ~dt \\
    & \quad + \frac{1}{2} \int_0^1 \inner*{
        \nabla^4 \rho_{r,\phi}^\delta(X_{(1,n-2)} + t X_{n-2}), 
        X_n \otimes X_{n-1}^{\otimes 2} \otimes X_{n-2}} ~dt \\
    & \quad + \frac{1}{2} \int_0^1 \inner*{
        \nabla^4 \rho_{r,\phi}^\delta(X_{(2,n-1)} + t X_2), 
        X_n^{\otimes 2} \otimes X_1 \otimes X_2} ~dt \\
    & \quad + \frac{1}{2} \int_0^1 \inner*{
        \nabla^4 \rho_{r,\phi}^\delta(X_{(1,n-2)} + t X_{n-2}), 
        X_n^{\otimes 2} \otimes X_{n-1} \otimes X_{n-2}} ~dt \\
    & \quad + \int_0^1 \inner*{
        \nabla^4 \rho_{r,\phi}^\delta(X_{(3,n-1)} + t X_3), 
        X_n \otimes X_1 \otimes X_2 \otimes X_3} ~dt \\
    & \quad + \int_0^1 \inner*{
        \nabla^4 \rho_{r,\phi}^\delta(X_{(2,n-2)} + t (X_2 + X_{n-2})), 
        X_n \otimes X_{n-1} \otimes X_1 \otimes (X_2 + X_{n-2}) } ~dt \\
    & \quad + \int_0^1 \inner*{
        \nabla^4 \rho_{r,\phi}^\delta(X_{(1,n-3)} + t X_{n-3}), 
        X_n \otimes X_{n-1} \otimes X_{n-2} \otimes X_{n-3}} ~dt. \\
\end{aligned}
\end{equation*}
The third-order moments in \cref{eq:decompose_third_remainder_P} are further decomposed as in \cref{sec:third_lindeberg_swapping}. As a result,
\begin{equation*}
\begin{aligned}
    & \inner*{ \Exp\left[ \nabla^3 \rho_{r,\phi}^\delta(X_{(1,n-1)}) \right], \Exp[ X_n^{\otimes 3} ]} \\
    & = \inner*{ \Exp\left[ \nabla^3 \rho_{r,\phi}^\delta(Y_{(1,n-1)}) \right], \Exp[ X_n^{\otimes 3} ]}
    + \sum_{k=2}^{n-2} \Exp \left[ 
        \mathfrak{R}_{X,X_k}^{(6,0,1)}
        - \mathfrak{R}_{X,Y_k}^{(6,0,1)}
    \right], \\
\end{aligned}
\end{equation*}
\begin{equation*}
\begin{aligned}
    & \inner*{ \Exp\left[ \nabla^3 \rho_{r,\phi}^\delta(X_{(1,n-3)}) \right], \Exp[ X_{n-2} \otimes X_{n-1} \otimes X_n ]} \\
    & = \inner*{ \Exp\left[ \nabla^3 \rho_{r,\phi}^\delta(Y_{(1,n-3)}) \right], \Exp[ X_{n-2} \otimes X_{n-1} \otimes X_n ]}
    + \sum_{k=2}^{n-4} \Exp \left[ 
        \mathfrak{R}_{X,X_k}^{(6,0,2)}
        - \mathfrak{R}_{X,Y_k}^{(6,0,2)}
    \right], \\
\end{aligned}
\end{equation*}
\begin{equation*}
\begin{aligned}
    & \inner*{ \Exp\left[ \nabla^3 \rho_{r,\phi}^\delta(X_{(2,n-2)}) \right], \Exp[ X_{n-1} \otimes X_{n} \otimes X_1 ]} \\
    & = \inner*{ \Exp\left[ \nabla^3 \rho_{r,\phi}^\delta(Y_{(2,n-2)}) \right], \Exp[ X_{n-1} \otimes X_{n} \otimes X_1 ]}
    + \sum_{k=3}^{n-3} \Exp \left[ 
        \mathfrak{R}_{X,X_k}^{(6,0,3)}
        - \mathfrak{R}_{X,Y_k}^{(6,0,3)}
    \right], \\
\end{aligned}
\end{equation*}
\begin{equation*}
\begin{aligned}
    & \inner*{ \Exp\left[ \nabla^3 \rho_{r,\phi}^\delta(X_{(3,n-1)}) \right], \Exp[ X_{n} \otimes X_{1} \otimes X_2 ]} \\
    & = \inner*{ \Exp\left[ \nabla^3 \rho_{r,\phi}^\delta(Y_{(3,n-1)}) \right], \Exp[ X_{n} \otimes X_{1} \otimes X_2 ]}
    + \sum_{k=4}^{n-2} \Exp \left[ 
        \mathfrak{R}_{X,X_k}^{(6,0,4)}
        - \mathfrak{R}_{X,Y_k}^{(6,0,4)}
    \right], \\
\end{aligned}
\end{equation*}
\begin{equation*}
\begin{aligned}
    & \inner*{ \Exp\left[ \nabla^3 \rho_{r,\phi}^\delta(X_{(1,n-2)}) \right], \Exp[ X_{n-1} \otimes X_n \otimes (X_{n-1} + X_n) ]} \\
    & = \inner*{ \Exp\left[ \nabla^3 \rho_{r,\phi}^\delta(Y_{(1,n-2)}) \right], \Exp[ X_{n-1} \otimes X_n \otimes (X_{n-1} + X_n) ]} 
    + \sum_{k=2}^{n-3} \Exp \left[ 
        \mathfrak{R}_{X,X_k}^{(6,0,5)}
        - \mathfrak{R}_{X,Y_k}^{(6,0,5)}
    \right], \textand \\
\end{aligned}
\end{equation*}
\begin{equation*}
\begin{aligned}
    & \inner*{ \Exp\left[ \nabla^3 \rho_{r,\phi}^\delta(X_{(2,n-1)}) \right], \Exp[ X_{n} \otimes X_1 \otimes (X_{n-1} + X_n) ]} \\
    & = \inner*{ \Exp\left[ \nabla^3 \rho_{r,\phi}^\delta(Y_{(2,n-1)}) \right], \Exp[ X_{n} \otimes X_1 \otimes (X_{1} + X_n) ]} 
    + \sum_{k=3}^{n-2} \Exp \left[ 
        \mathfrak{R}_{X,X_k}^{(6,0,6)}
        - \mathfrak{R}_{X,Y_k}^{(6,0,6)}
    \right]. \\
\end{aligned}
\end{equation*}
Because $Y$ is Gaussian, by Lemma 6.2 in \citetalias{chernozhukov2020nearly} and Assumption (VAR-EV),
\begin{equation*}
\begin{aligned}
    \abs*{ \inner*{ 
        \Exp\left[ \nabla^3 \rho_{r,\phi}^\vareps(Y_{(1,n-1)}) \right], 
        \Exp[X_n^{\otimes 3}] } }
    & \leq \frac{C}{n^{3/2}} L_{3,n} \frac{(\log(ep))^{3/2}}{\underline{\sigma}^3} \\
    & \leq \frac{C}{n^{3/2}} L_{3,n} \frac{(\log(ep))^2}{\underline{\sigma}^2 \sigma_{\min}}.
\end{aligned}
\end{equation*} 
Putting all the terms together,
\begin{equation*}
\begin{aligned}
    & \abs*{ \mathfrak{R}_X^{(3)} - \mathfrak{R}_Y^{(3)} } 
    \leq \frac{C}{\sqrt{n}} \bar{L}_{3} \frac{(\log(ep))^2}{\underline\sigma^2 \sigma_{\min}}
    + \abs*{ \mathfrak{R}_X^{(4)} - \mathfrak{R}_Y^{(4)} }
    + \sum_{k=2}^{n-2} \abs*{\Exp \left[ 
        \mathfrak{R}_{X,X_k}^{(6)}
        - \mathfrak{R}_{X,Y_k}^{(6)}
    \right] },
\end{aligned}
\end{equation*}
where $\mathfrak{R}_{X,W_k}^{(6)} = \frac{1}{6} \mathfrak{R}_{X,W_k}^{(6,0,1)} + \mathfrak{R}_{X,W_k}^{(6,0,2)} + \mathfrak{R}_{X,W_k}^{(6,0,3)} + \mathfrak{R}_{X,W_k}^{(6,0,4)} + \frac{1}{2} \mathfrak{R}_{X,W_k}^{(6,0,5)} + \frac{1}{2} \mathfrak{R}_{X,W_k}^{(6,0,6)}$.

\subsection{First Lindeberg swapping under \texorpdfstring{$1$}{1}-dependence} \label{sec:first_lindeberg_swapping}

We apply Taylor's expansion to each term with $\varphi_r^\vareps$ of the Lindeberg swapping in Eq. (15). We only show the expansions for $j = 3, \dots, n-2$ here, but the calculations for $j= 1, 2, n-1$ and $n$ are similar. We recall the notations
\begin{equation*}
    W_{[i,j]}^\cmpl \equiv X_{[1,i)} + Y_{(j,n]} \textand
    W_{[i,j]}^{\indep} \equiv W_{[i-1,j+1]}^\cmpl.
\end{equation*}
First,
\begin{equation*}
\begin{aligned}
    & \varphi_r^\vareps(W^\cmpl_{[j,j]} + X_j)\\
    & = \varphi_r^\vareps(W^\cmpl_{[j,j]})
    + \inner*{ \nabla \varphi_r^\vareps(W^\cmpl_{[j,j]}), X_j}
    + \frac{1}{2} \inner*{ 
    \nabla^2 \varphi_r^\vareps(W^\cmpl_{[j,j]}), X_j^{\otimes 2}} \\
    & \quad + \frac{1}{2} \int_0^1 (1-t)^2 \inner*{
    \nabla^3 \varphi_r^\vareps(W^\cmpl_{[j,j]} + t X_j), X_j^{\otimes 3}} ~dt.
\end{aligned}
\end{equation*}
We further apply Taylor's expansion to the second and third terms:
\begin{equation*}
\begin{aligned}
    & \inner*{ \nabla \varphi_r^\vareps(W^\cmpl_{[j,j]}), X_j} \\
    & = \inner*{ \nabla \varphi_r^\vareps(W_{[j,j]}^{\indep}), X_j} 
    + \inner*{ 
    \nabla^2 \varphi_r^\vareps(W_{[j,j]}^{\indep}),   X_j 
    \otimes (X_{j-1} + Y_{j+1})} \\
    & \quad + \int_0^1 (1-t) \left\langle
        \nabla^3 \varphi_r^\vareps\left(
            W_{[j,j]}^{\indep} + t (X_{j-1} + Y_{j+1}) 
        \right), 
        X_j \otimes (X_{j-1} + Y_{j+1})^{\otimes 2}
    \right\rangle ~dt,
\end{aligned}
\end{equation*}
\begin{equation*}
\begin{aligned}
    & \inner*{ 
        \nabla^2 \varphi_r^\vareps(W^\cmpl_{[j,j]}), X_j^{\otimes 2}} \\
    & = \inner*{\nabla^2 \varphi_r^\vareps(W_{[j,j]}^{\indep}), X_j^{\otimes 2}} \\
    & + \int_0^1 \left\langle
        \nabla^3 \varphi_r^\vareps\left(
              W_{[j,j]}^{\indep} + t (X_{j-1} + Y_{j+1})
        \right), 
        X_j^{\otimes 2} \otimes (X_{j-1} + Y_{j+1})\right\rangle ~dt.
\end{aligned}
\end{equation*}

Last, 
\begin{equation*}
\begin{aligned}
    & \inner*{ 
        \nabla^2 \varphi_r^\vareps(W_{[j,j]}^{\indep}), 
        X_j \otimes (X_{j-1}+Y_{j+1})} \\
    & = \inner*{ 
        \nabla^2 \varphi_r^\vareps(W_{[j,j]}^{\indep}),   
        X_j \otimes X_{j-1}} 
    + \inner*{ 
        \nabla^2 \varphi_r^\vareps(W_{[j,j]}^{\indep}), 
        X_j \otimes Y_{j+1}} \\
    & = \inner*{ 
        \nabla^2 \varphi_r^\vareps(W_{[j-1,j]}^{\indep}), 
        X_j \otimes X_{j-1}}
    + \inner*{ 
        \nabla^2 \varphi_r^\vareps(W_{[j,j+1]}^{\indep}),   
        X_j \otimes Y_{j+1}} \\
    & + \int_0^1 \inner*{
        \nabla^3 \varphi_r^\vareps\left(
            W_{[j-1,j]}^{\indep} + t X_{j-2}
        \right), 
        X_j \otimes X_{j-1} \otimes X_{j-2}} ~dt \\
    & + \int_0^1 \inner*{
        \nabla^3 \varphi_r^\vareps\left(
            W_{[j,j+1]}^{\indep}+ t Y_{j+2}
        \right),
        X_j \otimes Y_{j+1} \otimes Y_{j+2})} ~dt.
\end{aligned}
\end{equation*}
In sum,
\begin{equation*}
\begin{aligned}
    & \Exp \left[
        \varphi_r^\vareps(W^\cmpl_{[j,j]} + X_j)\right] \\
    & = \Exp \left[
    \begin{aligned}
        & \varphi_r^\vareps(W^\cmpl_{[j,j]})
        + \inner*{ \nabla \varphi_r^\vareps(W_{[j,j]}^{\indep}), X_j } 
        + \frac{1}{2} 
        \inner*{\nabla^2 \varphi_r^\vareps(W_{[j,j]}^{\indep}), X_j^{\otimes 2}} \\
        & + \inner*{ 
            \nabla^2 \varphi_r^\vareps(W_{[j-1,j]}^{\indep}),  
            X_j \otimes X_{j-1}} 
        + \inner*{ 
            \nabla^2 \varphi_r^\vareps(W_{[j,j+1]}^{\indep}),
            X_j \otimes Y_{j+1}} 
        + \mathfrak{R}_{X_j}^{(3,1)}
    \end{aligned} \right] \\
    & = \Exp\left[ \varphi_r^\vareps(W^\cmpl_{[j,j]}) \right]
    + \inner*{ \Exp\left[ \nabla \varphi_r^\vareps(W_{[j,j]}^{\indep}) \right], 
        \Exp\left[ X_j \right]}
    +
        \frac{1}{2} 
        \inner*{ \Exp\left[ \nabla^2 \varphi_r^\vareps(W_{[j,j]}^{\indep}) \right], 
        \Exp\left[ X_j^{\otimes 2}\right]} \\
    & \quad + \inner*{ 
        \Exp\left[ \nabla^2 \varphi_r^\vareps(W_{[j-1,j]}^{\indep}) \right], 
        \Exp\left[ X_j \otimes X_{j-1} \right]}
    + \Exp[\mathfrak{R}_{X_j}^{(3,1)}],
\end{aligned}
\end{equation*}
where 
\begin{equation} \label{eq:R_3_1}
\begin{aligned}
    \mathfrak{R}_{X_j}^{(3,1)}
    & = \frac{1}{2} \int_0^1 (1-t)^2 \inner*{
        \nabla^3 \varphi_r^\vareps\left(
            W^\cmpl_{[j,j]} + t X_j
        \right),
        X_j^{\otimes 3}} ~dt \\
    & \quad + \int_0^1 (1-t) \left\langle
        \nabla^3 \varphi_r^\vareps\left(
            W_{[j,j]}^{\indep} + t (X_{j-1} + Y_{j+1}) 
        \right), 
        X_j \otimes (X_{j-1} + Y_{j+1})^{\otimes 2}
    \right\rangle ~dt \\
    & \quad + \int_0^1 \left\langle
        \nabla^3 \varphi_r^\vareps\left(
              W_{[j,j]}^{\indep} + t (X_{j-1} + Y_{j+1})
        \right), 
        X_j^{\otimes 2} \otimes (X_{j-1} + Y_{j+1})\right\rangle ~dt \\
    & \quad + \int_0^1 \inner*{
        \nabla^3 \varphi_r^\vareps\left(
            W_{[j-1,j]}^{\indep} + t X_{j-2}
        \right), 
        X_j \otimes X_{j-1} \otimes X_{j-2}} ~dt \\
    & \quad + \int_0^1 \inner*{
        \nabla^3 \varphi_r^\vareps\left(
            W_{[j,j+1]}^{\indep}+ t Y_{j+2}
        \right),
        X_j \otimes Y_{j+1} \otimes Y_{j+2})} ~dt.
\end{aligned}
\end{equation}
$\mathfrak{R}_{Y_j}^{(3,1)}$ is similarly derived. For $j = 1, 2, n-1$ and $n$, $\mathfrak{R}_{W_j}^{(3,1)}$ is the same as \cref{eq:R_3_1} but with zero in place of non-existing terms. 
Summing over $j = 1, \dots, n$,
\begin{equation*}
\begin{aligned}
    & \sum_{j=1}^n 
        \Exp \left[
        \rho_{r,\phi}^\delta(W^\cmpl_{[j,j]} + X_j)
        - \rho_{r,\phi}^\delta(W^\cmpl_{[j,j]} + Y_j)
    \right] 
    = \sum_{j=1}^n \Exp\left[\mathfrak{R}_{X_j}^{(3,1)}\right]
\end{aligned}
\end{equation*}

To further decompose $\mathfrak{R}_{X_j}^{(3,1)}$, we apply the Taylor expansion up to order $4$. For example, where $j=4,\dots,n-3$, 
\begin{equation*}
\begin{aligned}
    & \frac{1}{2} \int_0^1 (1-t)^2 \inner*{
        \nabla^3 \varphi_r^\vareps\left(
            W^\cmpl_{[j,j]} + t X_j
        \right),
        X_j^{\otimes 3}} ~dt \\
    & = \frac{1}{6} \inner*{ 
        \nabla^3 \rho_{r,\phi}^\delta(W^\cmpl_{[j,j]} ), 
        X_j^{\otimes 3}}
    + \frac{1}{6} \int_0^1 (1-t)^3 \inner*{
    \nabla^4 \rho_{r,\phi}^\delta(W^\cmpl_{[j,j]}  + t X_j), X_j^{\otimes 4}} ~dt.
\end{aligned}
\end{equation*}
Again, to break the dependency between $\nabla^3 \rho_{r,\phi}^\delta(X_{[1,n)})$ and $X_n^{\otimes 3}$, we re-apply the Taylor expansion centered at $X_1 + X_{n-1}$.
\begin{equation*}
\begin{aligned}
    & \inner*{ \nabla^3 \rho_{r,\phi}^\delta(W^\cmpl_{[j,j]}), X_j^{\otimes 3}} \\
    & = \inner*{ \nabla^3 \rho_{r,\phi}^\delta(W^\indep_{[j,j]}), X_j^{\otimes 3}} \\
    & \quad + \int_0^1 \inner*{
    \nabla^4 \rho_{r,\phi}^\delta(W^\indep_{[j,j]} + t (X_{j-1}+Y_{j+1})), X_n^{\otimes 3} \otimes (X_{j-1}+Y_{j+1})} ~dt.
\end{aligned}
\end{equation*}
Repeating this to the other third-order remainder terms, we get
\begin{equation*}
\begin{aligned}
    & \Exp\left[ \mathfrak{R}_{X_j}^{(3,1)} \right] \\
    & = \frac{1}{6} 
        \inner*{ \Exp\left[ \nabla^3 \rho_{r,\phi}^\delta(X_{[1,j-1)} + Y_{(j+1,n]}) \right], \Exp[ X_j^{\otimes 3} ]} \\
    & \quad + \frac{1}{2} \inner*{ 
        \Exp\left[ \nabla^3 \rho_{r,\phi}^\delta(X_{[1,j-2)} + Y_{(j+1,n]}) \right], 
        \Exp[ X_{j-1}^{\otimes 2} \otimes X_j ] + \Exp[ X_{j-1} \otimes X_j^{\otimes 2}]} \\
    & \quad + \inner*{ 
        \Exp\left[ \nabla^3 \rho_{r,\phi}^\delta(X_{[1,j-3)} + Y_{(j+1,n]}) \right], 
        \Exp\left[ X_{j-2} \otimes X_{j-1} \otimes X_{j} \right]} \\
    & \quad + \Exp\left[ \mathfrak{R}_{X_j}^{(4,1)} \right],
\end{aligned}
\end{equation*}
where 
\begin{equation} \label{eq:R_4_1}
\begin{aligned}
    \mathfrak{R}_{X_j}^{(4,1)}
    & = \frac{1}{6} \int_0^1 (1-t)^3 \inner*{
    \nabla^4 \rho_{r,\phi}^\delta(W^\cmpl_{[j,j]} + t X_j), X_j^{\otimes 4}} ~dt \\
    & \quad + \frac{1}{2} \int_0^1 (1-t)^2 \left\langle
        \nabla^4 \rho_{r,\phi}^\delta\left(
            W_{[j,j]}^{\indep} + t (X_{j-1} + Y_{j+1}) 
        \right), 
        X_j \otimes (X_{j-1} + Y_{j+1})^{\otimes 3}
    \right\rangle ~dt \\
    & \quad + \frac{1}{2} \int_0^1 (1-t) \left\langle
        \nabla^4 \rho_{r,\phi}^\delta\left(
              W_{[j,j]}^{\indep} + t (X_{j-1} + Y_{j+1})
        \right), 
        X_j^{\otimes 2} \otimes (X_{j-1} + Y_{j+1})^{\otimes 2}\right\rangle ~dt \\
    & \quad + \frac{1}{6} \int_0^1 \left\langle
        \nabla^4 \rho_{r,\phi}^\delta\left(
              W_{[j,j]}^{\indep} + t (X_{j-1} + Y_{j+1})
        \right), 
        X_j^{\otimes 3} \otimes (X_{j-1} + Y_{j+1})
    \right\rangle ~dt \\
    & \quad + \int_0^1 (1-t) \inner*{
        \nabla^4 \rho_{r,\phi}^\delta\left(
            W_{[j-1,j]}^{\indep} + t X_{j-2}
        \right), 
        X_j \otimes X_{j-1} \otimes X_{j-2}^{\otimes 2}} ~dt \\
    & \quad + \int_0^1 (1-t) \inner*{
        \nabla^4 \rho_{r,\phi}^\delta\left(
            W_{[j,j+1]}^{\indep}+ t Y_{j+2}
        \right),
        X_j \otimes Y_{j+1} \otimes Y_{j+2}^{\otimes 2}} ~dt \\
    & \quad + \frac{1}{2} \int_0^1 \inner*{
        \nabla^4 \rho_{r,\phi}^\delta\left(
            W_{[j-1,j]}^{\indep} + t X_{j-2}
        \right), 
        X_j \otimes X_{j-1}^{\otimes 2} \otimes X_{j-2}} ~dt \\
    & \quad + \frac{1}{2} \int_0^1 \inner*{
        \nabla^4 \rho_{r,\phi}^\delta\left(
            W_{[j,j+1]}^{\indep}+ t Y_{j+2}
        \right),
        X_j \otimes Y_{j+1}^{\otimes 2} \otimes Y_{j+2}} ~dt \\
    & \quad + \int_0^1 \inner*{
        \nabla^4 \rho_{r,\phi}^\delta\left(
            W_{[j-1,j+1]}^{\indep}+ t (X_{j-2} + Y_{j+2})
        \right),
        X_{j-1} \otimes X_j \otimes Y_{j+1} \otimes X_{j-2}} ~dt \\
    & \quad + \int_0^1 \inner*{
        \nabla^4 \rho_{r,\phi}^\delta\left(
            W_{[j-1,j+1]}^{\indep}+ t (X_{j-2} + Y_{j+2})
        \right),
        X_{j-1} \otimes X_j \otimes Y_{j+1} \otimes Y_{j+2}} ~dt \\
    & \quad + \frac{1}{2} \int_0^1 \inner*{
        \nabla^4 \rho_{r,\phi}^\delta\left(
            W_{[j-1,j]}^{\indep} + t X_{j-2}
        \right), 
        X_j^{\otimes 2} \otimes X_{j-1} \otimes X_{j-2}} ~dt \\
    & \quad + \frac{1}{2}  \int_0^1 \inner*{
        \nabla^4 \rho_{r,\phi}^\delta\left(
            W_{[j,j+1]}^{\indep}+ t Y_{j+2}
        \right),
        X_j^{\otimes 2} \otimes Y_{j+1} \otimes Y_{j+2}} ~dt \\
    & \quad + \int_0^1 \inner*{
        \nabla^4 \rho_{r,\phi}^\delta\left(
            W_{[j-2,j]}^{\indep} + t X_{j-3}
        \right), 
        X_j \otimes X_{j-1} \otimes X_{j-2} \otimes X_{j-3}} ~dt \\
    & \quad + \int_0^1 \inner*{
        \nabla^4 \rho_{r,\phi}^\delta\left(
            W_{[j,j+2]}^{\indep}+ t Y_{j+3}
        \right),
        X_j \otimes Y_{j+1} \otimes Y_{j+2} \otimes Y_{j+3}} ~dt.
\end{aligned}
\end{equation}
$\mathfrak{R}_{Y_j}^{(4,1)}$ is similarly derived. For $j = 1, 2, 3, n-2, n-1$ and $n$, $\mathfrak{R}_{W_j}^{(4,1)}$ is the same as \cref{eq:R_4_1} but with zero in place of proper terms.

\subsection{Second Lindeberg swapping} \label{sec:second_lindeberg_swapping}
For each $j = 2, \dots, n-2$,
\begin{equation*}
\begin{aligned}
    & \inner*{
            \nabla^2 \rho_{r,\phi}^\delta\left(
              X_{(1,j)} + X_j + Y_{(j,n-1)}
            \right),
            X_n^{\otimes 2}
        } \\
    & = \inner*{
            \nabla^2 \rho_{r,\phi}^\delta\left(
              X_{(1,j)} + Y_{(j,n-1)}
            \right),
            X_n^{\otimes 2}
        } 
    + \mathfrak{R}_{X_j}^{(3,2,1)} \\
\end{aligned}
\end{equation*}
where $\mathfrak{R}_{X_j}^{(3,2,1)} = \int_0^1 \inner{
    \nabla^3 \rho_{r,\phi}^\delta\left(
        X_{(1,j)} + t X_{j} + Y_{(j, n-1)}
    \right),
    X_n^{\otimes 2} \otimes X_j
} ~dt$. Furthermore,
\begin{equation*}
\begin{aligned}
    & \mathfrak{R}_{X_j}^{(3,2,1)} \\
    & = \inner*{
            \nabla^3 \rho_{r,\phi}^\delta\left(
              X_{(1,j)} + Y_{(j,n-1)}
            \right),
            X_n^{\otimes 2} \otimes X_j
        } \\
    & \quad + \int_0^1 (1-t) \inner*{
        \nabla^4 \rho_{r,\phi}^\delta\left(
            X_{(1,j)} + t X_{j} + Y_{(j, n-1)}
        \right),
        X_n^{\otimes 2} \otimes X_j^{\otimes 2}
    } ~dt \\
    & = \inner*{
            \nabla^3 \rho_{r,\phi}^\delta\left(
              X_{(1,j-1)} + Y_{(j+1,n-1)}
            \right),
            X_n^{\otimes 2} \otimes X_j
        } \\
    & \quad + \int_0^1 (1-t) \inner*{
        \nabla^4 \rho_{r,\phi}^\delta\left(
            X_{(1,j)} + t X_{j} + Y_{(j, n-1)}
        \right),
        X_n^{\otimes 2} \otimes X_j^{\otimes 2}
    } ~dt \\
    & \quad + \int_0^1 \inner*{
        \nabla^4 \rho_{r,\phi}^\delta\left(
            X_{(1,j-1)} + t (X_{j-1} + Y_{j+1}) + Y_{(j+1, n-1)}
        \right),
        X_n^{\otimes 2} \otimes X_j \otimes (X_{j-1} + Y_{j+1})
    } ~dt.
\end{aligned}
\end{equation*}
Because $\Exp[X_n^{\otimes 2} \otimes X_j] = \Exp[X_n^{\otimes 2} \otimes Y_j] = 0$, 
\begin{equation*}
\begin{aligned}
    & \inner*{ 
        \Exp\left[ \nabla^2 \rho_{r,\phi}^\delta(X_{(1,n-1)}) \right], 
        \Exp\left[ X_n^{\otimes 2} \right]}
    - \inner*{ 
        \Exp\left[ \nabla^2 \rho_{r,\phi}^\delta(Y_{(1,n-1)}) \right], 
        \Exp\left[ X_n^{\otimes 2} \right]} \\
    & = \sum_{j=2}^{n-2} \Exp\left[ \mathfrak{R}_{X_j}^{(3,2,1)} - \mathfrak{R}_{Y_j}^{(3,2,1)} \right] 
    = \sum_{j=2}^{n-2} \Exp\left[ \mathfrak{R}_{X_j}^{(4,2,1)} - \mathfrak{R}_{Y_j}^{(4,2,1)} \right],
\end{aligned}
\end{equation*}
where $\mathfrak{R}_{X_j}^{(4,2,1)}$ is the fourth-order remainder term above. Similarly,
\begin{equation*}
\begin{aligned}
    & \inner*{ 
        \Exp\left[ \nabla^2 \rho_{r,\phi}^\delta(X_{(2,n-1)}) \right], 
        \Exp\left[ X_n \otimes X_1 \right]}
    - \inner*{ 
        \Exp\left[ \nabla^2 \rho_{r,\phi}^\delta(Y_{(2,n-1)}) \right], 
        \Exp\left[ X_n \otimes X_1 \right]} \\
    & = \sum_{j=3}^{n-2} \Exp\left[ \mathfrak{R}_{X_j}^{(3,2,2)} - \mathfrak{R}_{Y_j}^{(3,2,2)} \right] 
    = \sum_{j=3}^{n-2} \Exp\left[ \mathfrak{R}_{X_j}^{(4,2,2)} - \mathfrak{R}_{Y_j}^{(4,2,2)} \right], \textand
\end{aligned}
\end{equation*}
\begin{equation*}
\begin{aligned}
    & \inner*{ 
        \Exp\left[ \nabla^2 \rho_{r,\phi}^\delta(X_{(1,n-2)}) \right], 
        \Exp\left[ X_n \otimes X_{n-1} \right]}
    - \inner*{ 
        \Exp\left[ \nabla^2 \rho_{r,\phi}^\delta(Y_{(1,n-2)}) \right], 
        \Exp\left[ X_n \otimes X_{n-1} \right]} \\
    & = \sum_{j=2}^{n-3} \Exp\left[ \mathfrak{R}_{X_j}^{(3,2,3)} - \mathfrak{R}_{Y_j}^{(3,2,3)} \right] 
    = \sum_{j=2}^{n-3} \Exp\left[ \mathfrak{R}_{X_j}^{(4,2,3)} - \mathfrak{R}_{Y_j}^{(4,2,3)} \right].
\end{aligned}
\end{equation*}
We define $\mathfrak{R}_{W_j}^{(3,2)} = \frac{1}{2} \mathfrak{R}_{W_j}^{(3,2,1)} + \mathfrak{R}_{W_j}^{(3,2,2)} + \mathfrak{R}_{W_j}^{(3,2,3)}$ and $\mathfrak{R}_{W_j}^{(4,2)} = \frac{1}{2} \mathfrak{R}_{W_j}^{(4,2,1)} + \mathfrak{R}_{W_j}^{(4,2,2)} + \mathfrak{R}_{W_j}^{(4,2,3)}$.

\subsection{Third Lindeberg swapping} \label{sec:third_lindeberg_swapping}

For $j = 3, \dots, n-1$,
\begin{equation*}
\begin{aligned}
    & \Big\langle 
        \Exp\left[ \nabla^3 \rho_{r,\phi}^\vareps(X_{[1,j-1)} + Y_{(j+1,n)}) \right] 
    - \Exp\left[ \nabla^3  \rho_{r,\phi}^\vareps(Y_{[1,j-1)} + Y_{(j+1,n)}) \right], \Exp[ X_j^{\otimes 3} ]\Big\rangle. \\
    & = \sum_{k=1}^{j-2} \Big\langle
        \Exp\left[ \nabla^3 \rho_{r,\phi}^\vareps(X_{[1,k)} + X_k + Y_{(k,j-1) \cup (j+1,n)}) \right] \\ 
    & \omit{\hfill $- \Exp\left[ \nabla^3   \rho_{r,\phi}^\vareps(X_{[1,k)} + Y_k + Y_{(k,j-1) \cup (j+1,n)}) \right], 
        \Exp[ X_j^{\otimes 3} ]
    \Big\rangle.$} \\
\end{aligned}
\end{equation*}
We first consider the case with $6 \leq j \leq n-1$.
For $3 \leq k \leq j-4$, the Taylor expansion centered at $X_{[1,k)} + Y_{(k,j-1) \cup (j+1,n)}$ implies
\begin{equation*}
\begin{aligned}
    & \inner*{
        \Exp\left[ \nabla^3 \rho_{r,\phi}^\vareps(X_{[1,k)} + X_k + Y_{(k,j-1) \cup (j+1,n)}) \right],
        \Exp[X_j^{\otimes 3}]
    } \\
    & = \inner*{
        \Exp[ \nabla^3 \rho_{r,\phi}^\vareps(X_{[1,k)} + Y_{(k,j-1) \cup (j+1,n)}) ],
        \Exp[X_j^{\otimes 3}] 
    } \\
    & \quad + \Exp\left[ \inner*{
        \nabla^4 \rho_{r,\phi}^\vareps(X_{[1,k)} + Y_{(k,j-1) \cup (j+1,n)}),
        \Exp[X_j^{\otimes 3}] \otimes X_k
    } \right] \\
    & \quad + \frac{1}{2} \Exp\left[ \inner*{
        \nabla^5 \rho_{r,\phi}^\vareps(X_{[1,k)} + Y_{(k,j-1) \cup (j+1,n)}),
        \Exp[X_j^{\otimes 3}] \otimes X_k^{\otimes 2}
    } \right] \\
    & \quad + \frac{1}{2} \Exp\left[ \int_0^1 (1-t)^2 \inner*{
        \nabla^6 \rho_{r,\phi}^\vareps(X_{[1,k)} + t X_k + Y_{(k,j-1) \cup (j+1,n)}) ,
        \Exp[X_j^{\otimes 3}] \otimes X_k^{\otimes 3}
    } ~dt \right].
\end{aligned}
\end{equation*}
For the inner product terms with dependent factors, we repeat the Taylor expansion:
\begin{enumerate}
\item \begin{equation*}
\begin{aligned}
    & \Exp\left[ \inner*{
        \nabla^4 \rho_{r,\phi}^\vareps(X_{[1,k)} + Y_{(k,j-1) \cup (j+1,n)}),
        \Exp[X_j^{\otimes 3}] \otimes X_k
    } \right] \\
    & = \Exp\left[ \inner*{
        \nabla^5 \rho_{r,\phi}^\vareps(X_{[1,k-1)} + Y_{(k+1,j-1) \cup (j+1,n)}),
        \Exp[X_j^{\otimes 3}] \otimes X_k \otimes(X_{k-1} + Y_{k+1})
    } \right] \\
    & \quad + \Exp\Big[ \int_0^1 (1-t) \Big\langle
        \nabla^6 \rho_{r,\phi}^\vareps(X_{[1,k-1)} + t (X_{k-1} + Y_{k+1}) + Y_{(k+1,j-1) \cup (j+1,n)}), \\
    & \omit{\hfill $\Exp[X_j^{\otimes 3}] \otimes X_k \otimes (X_{k-1} + Y_{k+1})^{\otimes 2}
    \Big\rangle dt \Big]$},
\end{aligned}
\end{equation*}

\item \begin{equation*}
\begin{aligned}
    & \Exp\left[ \inner*{
        \nabla^5 \rho_{r,\phi}^\vareps(X_{[1,k)} + Y_{(k,j-1) \cup (j+1,n)}),
        \Exp[X_j^{\otimes 3}] \otimes X_k^{\otimes 2}
    } \right] \\
    & = \inner*{
        \Exp\left[ \nabla^5 \rho_{r,\phi}^\vareps(X_{[1,k-1)} + Y_{(k+1,j-1) \cup (j+1,n)}) \right],
        \Exp[X_j^{\otimes 3}] \otimes \Exp\left[ X_k^{\otimes 2} \right]
    }  \\
    & \quad + \Exp\Big[ \int_0^1 \Big\langle
        \nabla^6 \rho_{r,\phi}^\vareps(X_{[1,k-1)} + t (X_{k-1} + Y_{k+1}) + Y_{(k+1,j-1) \cup (j+1,n)}), \\
    & \omit{\hfill $\Exp[X_j^{\otimes 3}] \otimes X_k^{\otimes 2} \otimes (X_{k-1} + Y_{k+1})
    \Big\rangle ~dt \Big],$}
\end{aligned}
\end{equation*}

\item \begin{equation*}
\begin{aligned}
    & \Exp\left[ \inner*{
        \nabla^5 \rho_{r,\phi}^\vareps(X_{[1,k-1)} + Y_{(k+1,j-1) \cup (j+1,n)}),
        \Exp[X_j^{\otimes 3}] \otimes X_k \otimes(X_{k-1} + Y_{k-1})
    } \right] \\
    & = \inner*{
        \Exp[ \nabla^5 \rho_{r,\phi}^\vareps(X_{[1,k-2)} + Y_{(k+1,j-1) \cup (j+1,n)}) ],
        \Exp[ X_j^{\otimes 3} ] \otimes \Exp[ X_k \otimes X_{k-1} ]
    } \\
    & \quad + \Exp\Big[ \int_0^1 \Big\langle
        \nabla^6 \rho_{r,\phi}^\vareps(X_{[1,k-2)} + t X_{k-2} + Y_{(k+1,j-1) \cup (j+1,n)}), 
        \Exp[X_j^{\otimes 3}] \otimes X_k \otimes X_{k-1} \otimes X_{k-2}
    \Big\rangle ~dt \Big] \\
    & \quad + \Exp\Big[ \int_0^1 \Big\langle
        \nabla^6 \rho_{r,\phi}^\vareps(X_{[1,k-1)} + t Y_{k+2} + Y_{(k+2,j-1) \cup (j+1,n)}), 
        \Exp[X_j^{\otimes 3}] \otimes X_k \otimes Y_{k+1} \otimes Y_{k+2}
    \Big\rangle ~dt \Big].
\end{aligned}
\end{equation*}
\end{enumerate}
For $\inner*{
    \Exp\left[ \nabla^3 \rho_{r,\phi}^\vareps(X_{[1,k)} + Y_k + Y_{(k,j-1) \cup (j+1,n)}) \right],
    \Exp[X_j^{\otimes 3}]
}$, the calculation is the same but with $Y_k$ in place of $X_k$. By the second moment matching,
\begin{equation*}
\begin{aligned}
    & \Big\langle
        \Exp\left[ \nabla^3 \rho_{r,\phi}^\vareps(X_{[1,k)} + X_k + Y_{(k,j-1) \cup (j+1,n)}) \right] \\ 
    & \omit{\hfill $- \Exp\left[ \nabla^3   \rho_{r,\phi}^\vareps(X_{[1,k)} + Y_k + Y_{(k,j-1) \cup (j+1,n)}) \right], 
        \Exp[ X_j^{\otimes 3} ]
    \Big\rangle$} \\
    & = \inner*{
        \Exp[ \nabla^5 \rho_{r,\phi}^\vareps(X_{[1,k-2)} + Y_{(k+1,j-1) \cup (j+1,n)}) ],
        \Exp[ X_j^{\otimes 3} ] \otimes \Exp[ X_k \otimes X_{k-1} ]
    } \\
    & \quad - \inner*{
        \Exp[ \nabla^5 \rho_{r,\phi}^\vareps(X_{[1,k-1)} + Y_{(k+2,j-1) \cup (j+1,n)}) ],
        \Exp[ X_j^{\otimes 3} ] \otimes \Exp[ Y_k \otimes Y_{k+1} ]
    } \\
    & \quad + \Exp\left[ \mathfrak{R}_{X_j,X_k}^{(6,1,1)} - \mathfrak{R}_{X_j,Y_k}^{(6,1,1)} \right], \\
\end{aligned}
\end{equation*}
where for $W_k = X_k$ or $Y_k$,
\begin{equation} \label{eq:R_6_1}
\begin{aligned}
    & \mathfrak{R}_{X_j,W_k}^{(6,1,1)} \\
    & = \frac{1}{2} \int_0^1 (1-t)^2 \inner*{
        \nabla^6 \rho_{r,\phi}^\vareps(X_{[1,k)} + t W_k + Y_{(k,j-1) \cup (j+1,n)}) ,
        X_j^{\otimes 3} \otimes W_k^{\otimes 3}
    } ~dt \\
    & \quad + \int_0^1 (1-t) \Big\langle
        \nabla^6 \rho_{r,\phi}^\vareps(X_{[1,k-1)} + t (X_{k-1} + Y_{k+1}) + Y_{(k+1,j-1) \cup (j+1,n)}), \\
    & \omit{\hfill $X_j^{\otimes 3} \otimes W_k \otimes (X_{k-1} + Y_{k+1})^{\otimes 2}
    \Big\rangle dt$} \\
    & \quad + \frac{1}{2} \int_0^1 \Big\langle
        \nabla^6 \rho_{r,\phi}^\vareps(X_{[1,k-1)} + t (X_{k-1} + Y_{k+1}) + Y_{(k+1,j-1) \cup (j+1,n)}), \\
    & \omit{\hfill $X_j^{\otimes 3} \otimes W_k^{\otimes 2} \otimes (X_{k-1} + Y_{k+1})
    \Big\rangle ~dt$} \\
    & \quad + \int_0^1 \Big\langle
        \nabla^6 \rho_{r,\phi}^\vareps(X_{[1,k-2)} + t X_{k-2} + Y_{(k+1,j-1) \cup (j+1,n)}), 
        X_j^{\otimes 3} \otimes W_k \otimes X_{k-1} \otimes X_{k-2}
    \Big\rangle ~dt \\
    & \quad + \int_0^1 \Big\langle
        \nabla^6 \rho_{r,\phi}^\vareps(X_{[1,k-1)} + t Y_{k+2} + Y_{(k+2,j-1) \cup (j+1,n)}), 
        X_j^{\otimes 3} \otimes W_k \otimes Y_{k+1} \otimes Y_{k+2}
    \Big\rangle ~dt.
\end{aligned}
\end{equation}
For $k = 1$,
\begin{equation*}
\begin{aligned}
    & \Big\langle
        \Exp\left[ \nabla^3 \rho_{r,\phi}^\vareps(X_1 + Y_{(1,j-1) \cup (j+1,n)}) \right] 
    - \Exp\left[ \nabla^3   \rho_{r,\phi}^\vareps(Y_1 + Y_{(1,j-1) \cup (j+1,n)}) \right], 
        \Exp[ X_j^{\otimes 3} ]
    \Big\rangle \\
    & = - \inner*{
        \Exp[ \nabla^5 \rho_{r,\phi}^\vareps(Y_{(3,j-1) \cup (j+1,n)}) ],
        \Exp[ X_j^{\otimes 3} ] \otimes \Exp[ Y_1 \otimes Y_2 ]
    } \\
    & \quad + \Exp\left[ \mathfrak{R}_{X_j,X_1}^{(6,1,1)} - \mathfrak{R}_{X_j,Y_1}^{(6,1,1)} \right], \\
\end{aligned}
\end{equation*}
where $\mathfrak{R}_{X_j,W_1}^{(6,1,1)}$ is the same as \cref{eq:R_6_1} but with $Y_n$ and $Y_{n-1}$ in place of $X_{k-1}$ and $X_{k-2}$, respectively.
For $k = 2$,
\begin{equation*}
\begin{aligned}
    & \Big\langle
        \Exp\left[ \nabla^3 \rho_{r,\phi}^\vareps(X_1 + X_2 + Y_{(2,j-1) \cup (j+1,n)}) \right] 
    - \Exp\left[ \nabla^3   \rho_{r,\phi}^\vareps(X_1 + Y_2 + Y_{(2,j-1) \cup (j+1,n)}) \right], 
        \Exp[ X_j^{\otimes 3} ]
    \Big\rangle \\
    & = \inner*{
        \Exp[ \nabla^5 \rho_{r,\phi}^\vareps(Y_{(3,j-1) \cup (j+1,n)}) ],
        \Exp[ X_j^{\otimes 3} ] \otimes \Exp[ X_2 \otimes X_1 ]
    } \\
    & \quad - \inner*{
        \Exp[ \nabla^5 \rho_{r,\phi}^\vareps(Y_{(4,j-1) \cup (j+1,n)}) ],
        \Exp[ X_j^{\otimes 3} ] \otimes \Exp[ Y_2 \otimes Y_3 ]
    } \\
    & \quad + \Exp\left[ \mathfrak{R}_{X_j,X_2}^{(6,1,1)} - \mathfrak{R}_{X_j,Y_2}^{(6,1,1)} \right], \\
\end{aligned}
\end{equation*}
where $\mathfrak{R}_{X_j,W_2}^{(6,1,1)}$ is the same as \cref{eq:R_6_1} but with $Y_n$ in place of $X_{k-2}$.
For $k = j-3$,
\begin{equation*}
\begin{aligned}
    & \Big\langle
        \Exp\left[ \nabla^3 \rho_{r,\phi}^\vareps(X_{[1,j-3)} + X_{j-3} + Y_{\{j-2\} \cup (j+1,n)}) \right] \\ 
    & \omit{\hfill $- \Exp\left[ \nabla^3   \rho_{r,\phi}^\vareps(X_{[1,j-3)} + Y_{j-3} + Y_{\{j-2\} \cup (j+1,n)}) \right], 
        \Exp[ X_j^{\otimes 3} ]
    \Big\rangle$} \\
    & = \inner*{
        \Exp[ \nabla^5 \rho_{r,\phi}^\vareps(X_{[1,j-5)} + Y_{(j+1,n)}) ],
        \Exp[ X_j^{\otimes 3} ] \otimes \Exp[ X_{j-3} \otimes X_{j-4} ]
    } \\
    & \quad - \inner*{
        \Exp[ \nabla^5 \rho_{r,\phi}^\vareps(X_{[1,j-4)} + Y_{(j+1,n)}) ],
        \Exp[ X_j^{\otimes 3} ] \otimes \Exp[ Y_{j-3} \otimes Y_{j-2} ]
    } \\
    & \quad + \Exp\left[ \mathfrak{R}_{X_j,X_{j-3}}^{(6,1,1)} - \mathfrak{R}_{X_j,Y_{j-3}}^{(6,1,1)} \right], \\
\end{aligned}
\end{equation*}
where $\mathfrak{R}_{X_j,W_2}^{(6,1,1)}$ is the same as \cref{eq:R_6_1} but with $0$ in place of $Y_{k+2}$.
For $k = j-2$,
\begin{equation*}
\begin{aligned}
    & \Big\langle
        \Exp\left[ \nabla^3 \rho_{r,\phi}^\vareps(X_{[1,j-2)} + X_{j-2} + Y_{(j+1,n)}) \right]
    - \Exp\left[ \nabla^3   \rho_{r,\phi}^\vareps(X_{[1,j-2)} + Y_{j-2} + Y_{(j+1,n)}) \right], 
        \Exp[ X_j^{\otimes 3} ]
    \Big\rangle \\
    & = \inner*{
        \Exp[ \nabla^5 \rho_{r,\phi}^\vareps(X_{[1,j-4)} + Y_{(j+1,n)}) ],
        \Exp[ X_j^{\otimes 3} ] \otimes \Exp[ X_{j-2} \otimes X_{j-3} ]
    } \\
    & \quad + \Exp\left[ \mathfrak{R}_{X_j,X_{j-2}}^{(6,1,1)} - \mathfrak{R}_{X_j,Y_{j-2}}^{(6,1,1)} \right], \\
\end{aligned}
\end{equation*}
where $\mathfrak{R}_{X_j,W_2}^{(6,1,1)}$ is the same as \cref{eq:R_6_1} but with $0$ in place of $Y_{k+1}$ and $Y_{k+2}$.
By the second moment matching, 
\begin{equation*}
\begin{aligned}
    & \inner*{ 
        \Exp\left[ \nabla^3 \rho_{r,\phi}^\vareps(X_{[1,j-1)} + Y_{(j+1,n)}) 
        - \nabla^3 \rho_{r,\phi}^\vareps(Y_{[1,j-1)} + Y_{(j+1,n)}) \right], 
        \Exp[ X_j^{\otimes 3} ]
    } \\
    & = \sum_{k=1}^{j-2} \Exp \left[ 
        \mathfrak{R}_{X_j,X_k}^{(6,1,1)}
        - \mathfrak{R}_{X_j,Y_k}^{(6,1,1)}
    \right]. \\
\end{aligned}
\end{equation*}
This is also the same for $j \in [1, 5]$, where some terms in $\mathfrak{R}_{X_j,W_k}^{(6,1,1)}$ are zero when appropriate. 
In sum, 
\begin{equation*}
\begin{aligned}
    & \sum_{j=1}^n  
        \inner*{ \Exp\left[ \nabla^3 \rho_{r,\phi}^\delta(X_{[1,j-1)} + Y_{(j+1,n)}) \right], \Exp[ X_j^{\otimes 3} ]} \\
    & = \sum_{j=1}^n  
        \inner*{ \Exp\left[ \nabla^3 \rho_{r,\phi}^\delta(Y_{[1,j-1)} + Y_{(j+1,n)}) \right], \Exp[ X_j^{\otimes 3} ]}
    + \sum_{j=3}^{n} \sum_{k=1}^{j-2} \Exp \left[ 
        \mathfrak{R}_{X_j,X_k}^{(6,1,1)}
        - \mathfrak{R}_{X_j,Y_k}^{(6,1,1)}
    \right]. \\
\end{aligned}
\end{equation*}
Similarly,
\begin{equation*}
\begin{aligned}
    & \sum_{j=1}^n \inner*{ 
        \Exp\left[ \nabla^3 \rho_{r,\phi}^\delta(X_{[1,j-1)} + Y_{(j+2,n)}) \right], 
        \Exp[ X_{j}^{\otimes 2} \otimes X_{j+1} ] 
        + \Exp[ X_{j} \otimes X_{j+1}^{\otimes 2}]} \\
    & = \sum_{j=1}^n \inner*{ 
        \Exp\left[ \nabla^3 \rho_{r,\phi}^\delta(Y_{[1,j-1)} + Y_{(j+2,n)}) \right], 
        \Exp[ X_{j}^{\otimes 2} \otimes X_{j+1} ] 
        + \Exp[ X_{j} \otimes X_{j+1}^{\otimes 2}]} \\
    & \quad + \sum_{j=4}^{n} \sum_{k=1}^{j-3} \Exp \left[ 
        \mathfrak{R}_{X_j,X_k}^{(6,1,2)}
        - \mathfrak{R}_{X_j,Y_k}^{(6,1,2)}
    \right], \textand 
    \\
    & \sum_{j=1}^n \inner*{ 
        \Exp\left[ \nabla^3 \rho_{r,\phi}^\delta(X_{[1,j-1)} + Y_{(j+3,n)}) \right], 
        \Exp\left[ X_{j} \otimes X_{j+1} \otimes X_{j+2} \right]} \\
    & = \sum_{j=1}^n \inner*{ 
        \Exp\left[ \nabla^3 \rho_{r,\phi}^\delta(Y_{[1,j-1)} + Y_{(j+3,n)}) \right], 
        \Exp\left[ X_{j} \otimes X_{j+1} \otimes X_{j+2} \right]} \\
    & \quad + \sum_{j=5}^{n} \sum_{k=1}^{j-4} \Exp \left[ 
        \mathfrak{R}_{X_j,X_k}^{(6,1,3)}
        - \mathfrak{R}_{X_j,Y_k}^{(6,1,3)}
    \right],
\end{aligned}
\end{equation*}
where $\mathfrak{R}_{X_j,W_k}^{(6,1,2)}$ and $\mathfrak{R}_{X_j,W_k}^{(6,1,3)}$ are similarly derived as in \cref{eq:R_6_1}.

\end{appendix}

\end{document}